\documentclass[12pt]{imsart}

\usepackage{xr}
\usepackage{amsthm,amsmath,amssymb}
\usepackage{parskip}
\usepackage[numbers]{natbib}
\usepackage{hypernat}
\usepackage{marvosym}
\usepackage[hang,small,bf]{caption}
\usepackage{mathtools}

\usepackage{hyperref}
\hypersetup{
    colorlinks,%
    citecolor=blue,%
    filecolor=black,%
    linkcolor=red,%
    urlcolor=black
}
\usepackage{hypernat}
\usepackage{enumerate}
\usepackage[top=2.8cm, bottom=2.8cm, left=2.8cm, right=2.8cm]{geometry}

\newtheoremstyle{exampstyle}
{6pt} % Space above
{6pt} % Space below
{\it} % Body font
{} % Indent amount
{\bfseries} % Theorem head font
{.} % Punctuation after theorem head
{.5em} % Space after theorem head
{} % Theorem head spec (can be left empty, meaning `normal')

\theoremstyle{exampstyle}

\newtheorem{theorem}{Theorem}[section]
\newtheorem{proposition}[theorem]{Proposition}
\newtheorem{lemma}[theorem]{Lemma}
\newtheorem{corollary}[theorem]{Corollary}
\newtheorem{remark}{Remark}[section]
\newtheorem{defn}{Definition}[section]

\newcommand{\E}{{\mathbb E}}

\newcommand{\R}{{\mathbb R}}
\newcommand{\rs}{{\mathfrak r}}
\renewcommand{\P}{{\mathbb P}}

\newcommand{\aff}{{\mathrm{aff}}}

\newcommand{\C}{{\mathcal{C}}}
\newcommand{\D}{{\mathbf{D}}}
\newcommand{\dist}{{\mathrm{dist}}}
\newcommand{\cone}{{\mathrm{cone}}}
\newcommand{\cras}{{\C_r(\{a_i\}, \{s_i\})}}
\newcommand{\I}{{\mathcal{I}}}

\newcommand{\K}{{\mathcal{K}}}
\newcommand{\krv}{{K^{(r)}(V)}}
\newcommand{\ktr}{{\mathbf{k}_r}}
\newcommand{\ellb}{{\mathbf{\ell}}}
\newcommand{\M}{{\mathcal{M}}}

\newcommand{\kb}{{\mathbf{k}}}
\newcommand{\m}{{\mathbf{m}}}
\newcommand{\q}{{\mathbf{q}}}
\newcommand{\lc}{{\mathcal{L}}}

\newcommand{\gf}{{\mathfrak{g}}}
\newcommand{\sgn}{{\mathrm{sgn}}}

\newcommand\norm[1]{\left\lVert#1\right\rVert}

\def\qt#1{\qquad\text{#1}}

\def\argmin{\mathop{\rm argmin}}
\def\argmax{\mathop{\rm argmax}}

\makeatletter

\newcommand{\Rmnum}[1]{\expandafter\@slowromancap\romannumeral #1@}
\makeatother

\newcounter{rcnt}[section]

\begin{document}
\begin{frontmatter}
\title{Adaptive Risk Bounds in Univariate Total Variation Denoising
  and Trend Filtering}         
\runtitle{Risk Bounds in Trend Filtering}  
\runtitle{Risk Bounds in  Trend Filtering\;\;\;\;\;\;\;}

\begin{aug}
\author{\fnms{Adityanand} \snm{Guntuboyina}\thanksref{t2}\ead[label=e2]{aditya@stat.berkeley.edu}}, 
\author{\fnms{Donovan}   \snm{Lieu}\ead[label=e3]{dlieu333@berkeley.edu}},
\author{\fnms{Sabyasachi} \snm{Chatterjee}\ead[label=e1]{sc1706@illinois.edu}},
\and
\author{\fnms{Bodhisattva} \snm{Sen}\thanksref{t4}\ead[label=e4]{bodhi@stat.columbia.edu}}

%\ead[label=e3]{third@somewhere.com}
%\ead[label=u1,url]{http://www.foo.com}

\thankstext{t2}{Supported by NSF CAREER Grant DMS-1654589}
\thankstext{t4}{Supported by NSF Grant DMS-1712822}

\runauthor{Guntuboyina, A., Lieu, D., Chatterjee, S. and Sen, B.}

\affiliation{University of California at Berkeley, University of
  California at Berkeley, University of Illinois at Urbana-Champaign and Columbia University}

\address{423 Evans Hall\\
Berkeley, CA 94720 \\
\printead{e2}\\
\phantom{E-mail:\ }}

\address{367 Evans Hall\\
Berkeley, CA 94720 \\
\printead{e3}\\
\phantom{E-mail:\ }}

\address{117 Illini Hall \\
725 S. Wright St. M/C 374 \\
Champaign, IL 61820 \\
\printead{e1}}

\address{1255 Amsterdam Avenue \\
New York, NY 10027\\
\printead{e4}
}
\end{aug}

\begin{abstract}
We study trend filtering, a relatively recent method for univariate nonparametric regression. For a given integer $r \geq 1$, 
the $r^{th}$ order trend filtering estimator is defined as the minimizer of the sum of squared errors when we constrain (or penalize)
the sum of the absolute $r^{th}$ order discrete derivatives of the fitted function at the design points. For $r = 1$, the estimator reduces to total variation regularization which has received much attention in the statistics and
image processing literature. In this paper, we study the performance
of the trend filtering estimator for every $r \geq 1$, both in the
constrained and penalized forms. Our main results show that in the
strong sparsity setting when the 
underlying function is a (discrete) spline with few ``knots'', the
risk (under the global squared error loss) of the trend filtering estimator
(with an appropriate choice of the tuning parameter) achieves the {\it 
  parametric} $n^{-1}$-rate, up to a logarithmic (multiplicative) 
factor. Our results therefore provide support for the use of trend
filtering, for every $r \geq 1$, in the strong sparsity setting. 
\end{abstract}

\begin{keyword}
Adaptive splines, discrete splines, fat shattering, higher order total variation regularization, metric entropy bounds, nonparametric function estimation, risk bounds, subdifferential, tangent cone
\end{keyword}
%\maketitle
\end{frontmatter}

\section{Introduction}
Consider the nonparametric regression problem where we observe data
generated according to the model:
\begin{equation}\label{eq:NPReg}
	Y_i = f^*({i}/{n}) + \xi_i, \qquad i = 1,\ldots, n,
\end{equation} 
where $f^*:[0,1] \to \R$ is the unknown regression function, and
$\xi_1,\ldots, \xi_n$ are unobserved independent errors
having the normal distribution with mean zero and variance $\sigma^2$. The
goal is to recover the underlying function  
$f^*$ from the measurements $Y_1, \dots, Y_n$. Alternatively, in the 
Gaussian sequence formulation,~\eqref{eq:NPReg} can be expressed as  
\begin{equation}\label{seqmo}
Y = \theta^* + \xi,
\end{equation}
where $\xi \sim N_n(0,\sigma^2 I_n)$, and $\theta^* :=
(f^*({1}/{n}), f^*({2}/{n}), \ldots, f^*(1))$ is unknown. Here $N_n(0,
\sigma^2 I_n)$ denotes the multivariate normal distribution with mean
vector zero and covariance matrix $\sigma^2 I_n$.   

In this paper, we study the performance of {\it trend filtering}, a
relatively new method for nonparametric regression with special
emphasis on its risk properties. For a given
integer $r \geq 1$, the $r^{th}$ order trend filtering estimator is defined
as the minimizer of the sum of squared errors when we constrain or
penalize the sum of the absolute $r^{th}$ order discrete derivatives of the fitted function
at the design points. Formally, given a fixed integer $r \geq 1$ and
a tuning parameter $V \geq 0$, the $r^{th}$ order trend filtering
estimator for $\theta^*$ in the constrained form is given by    
\begin{equation}\label{tf}
  \hat{\theta}^{(r)}_V := \argmin_{\theta \in \R^n}
  \left\{\frac{1}{2}\|Y - \theta\|^2 : \|D^{(r)} \theta\|_1 \leq V
    n^{1-r} \right\}      
\end{equation}
where $V >0$ is a tuning parameter (the multiplicative factor
$n^{1-r}$ is just for normalization), $D^{(0)} \theta := \theta$,
$D^{(1)} \theta := (\theta_2 - 
\theta_1, \dots, \theta_n - \theta_{n-1})$ and $D^{(r)} \theta$, for $r
\geq 2$, is recursively defined as $D^{(r)} \theta := D^{(1)}(D^{(r-1)}
\theta)$. Also $\|\cdot\|_1$ denotes the usual $L^1$ norm defined by
$\|x\|_1 := \sum_{i=1}^k |x_i|$ for $x = (x_1, \dots, x_k) \in
\R^k$. Note that $\|D^{(r)} \theta\|_1$  also 
equals $V(D^{(r-1)} \theta)$ where $V(\alpha) := \sum_{i=2}^k
|\alpha_i - \alpha_{i-1}|$ denotes the variation of a vector $\alpha =
(\alpha_1,\ldots, \alpha_k) \in \R^k$. For simplicity, we denote the
operator $D^{(1)}$ by simply $D$.  

Alternatively, the trend filtering estimator in the penalized form is 
\begin{equation}\label{ptf}
  \hat{\theta}^{(r)}_{\lambda} := \argmin_{\theta \in \R^n} \left(\frac{1}{2} \|Y
  - \theta\|^2 + \sigma n^{r-1} \lambda \|D^{(r)} \theta\|_1 \right)
\end{equation}
for $r \geq 1$ and tuning parameter $\lambda \geq 0$. There is an
abuse of notation here  in that we are using the same notation for
both the constrained and the penalized estimators. It may be noted,
however, that when the subscript of $\hat{\theta}^{(r)}$ is $V$, we
are referring to the constrained estimator \eqref{tf} while when the
subscript is $\lambda$, we are referring to the penalized estimator
\eqref{ptf}.   

For $r = 1$,~\eqref{ptf} reduces to the one-dimensional discrete
version of total variation regularization or total variation denoising
which was first proposed  by~\citet{rudin1992nonlinear} and has since  
been heavily used in the image processing community. The penalized
estimator~\eqref{ptf}, for general $r \geq 1$, was first proposed
by~\citet{steidl2006splines} in the image processing literature who
termed it \textit{higher order total variation regularization}. The
same estimator was later rediscovered by~\citet{kim2009ell_1} who
coined the name \textit{trend filtering} for it. Many properties of
the 
estimator have been studied in \citet{tibshirani2014adaptive} and
\citet{wang2014falling}.  It should also be mentioned here that a continuous version
of~\eqref{ptf}, where the discrete differences are replaced by
continuous derivatives, was proposed much earlier in the statistics
literature by~\citet{mammen1997locally} under the name \textit{locally
  adaptive   regression splines}.  

The presence of the $L^1$ norm in the constraint in \eqref{tf}
(resp. penalty in \eqref{ptf}) promotes sparsity of the vector $D^{(r)}
\hat{\theta}^{(r)}_V$ (resp. $D^{(r)}
\hat{\theta}^{(r)}_{\lambda}$). Now for every vector $\theta \in
\R^n$, $\|D^{(r)} \theta\|_0 = k$ if and only if $\theta$ equals
$(f(1/n), \dots, f(n/n))$ for a \textit{discrete spline} function $f$
that is made of $k+1$ polynomials each of degree $(r-1)$ (here 
$\|x\|_0$ denotes the number of entries of the vector $x$ that 
are non-zero). Discrete 
splines are piecewise polynomials with regularity at the knots. They
differ from the usual (continuous) splines in the form of the
regularity condition at the knots: for splines, the regularity
condition translates to (higher order) derivatives of adjacent
polynomials agreeing at the knots, while for discrete splines it
translates to discrete differences of adjacent polynomials agreeing at
the knots; see \citet{mangasarian1971discrete} for details. This fact
about the connection between $\|D^{(r)} \theta\|_0$ and  
discrete splines is standard (see e.g., \citet{steidl2006splines})
but we included a proof in Subsection \ref{ssds} for the convenience
of the reader.  

Thus the presence of the $L^1$ norm in \eqref{tf}
(resp. \eqref{ptf}) implies that  $\hat{\theta}^{(r)}_V$
(resp. $\hat{\theta}^{(r)}_{\lambda}$) can be written as
$(\hat{f}(1/n), \dots, \hat{f}(n/n))$ for a discrete spline $\hat{f}$
of degree $(r-1)$ made up of not too many polynomial pieces. Trend 
filtering thus presents a way of fitting (discrete) splines to the
data. Note that the knots of the discrete splines are automatically
chosen by the optimization algorithms underlying \eqref{tf} and
\eqref{ptf} without any input from the user (except for the
value of the tuning parameter $V$ or $\lambda$). Because of this
automatic selection of the knots, trend filtering can be regarded as a
spatially adaptive method (in the terminology of
\citet{donoho1994ideal}).  Note that such spatial adaptation is not
exhibited by classical nonparametric regression methods such as local
polynomials, kernels and splines, with a fixed tuning
parameter. On the other hand, methods such as CART
(\citet{B84}), MARS (\citet{F91}), variable-bandwidth
kernel/spline methods (see e.g.,~\citet{MS87, BGH93, P06} and
\citet{ZS01}) and wavelets (\citet{donoho1994ideal}) are also
spatially adaptive. 

The present paper studies the performance of the estimators
$\hat{\theta}_V^{(r)}$ and $\hat{\theta}_{\lambda}^{(r)}$ as
estimators of $\theta^*$ under the multivariate Gaussian model
\eqref{seqmo}. We shall use the squared error loss under which the
risk of an estimator $\hat{\theta}$ is defined as 
\begin{equation}\label{ridef}
  R(\hat{\theta}, \theta^*) := \frac{1}{n} \E_{\theta^*}
  \big\|\hat{\theta} - \theta^*\big\|^2. 
\end{equation}
Under natural sparsity assumptions on $\theta^*$, we provide upper
bounds on the risks $R(\hat{\theta}_V^{(r)}, \theta^*)$ and
$R(\hat{\theta}_{\lambda}^{(r)}, \theta^*)$ as well as high
probability upper bounds on the random loss functions
$\|\hat{\theta}_V^{(r)} - \theta^*\|^2/n$ and
$\|\hat{\theta}_{\lambda}^{(r)} - \theta^*\|^2/n$.  

It is natural to study the risk properties of \eqref{tf} and
\eqref{ptf} under the following two kinds of assumptions on
$\theta^*$: (a) $n^{r-1} \|D^{(r)} \theta^*\|_1 \leq V$ for some $V >
0$ (possibly dependent on $n$), and (b) $\|D^{(r)} \theta^* \|_0 \leq
k$ for some $k$ that is much smaller than $n$. We shall refer to these
two regimes as \textit{weak 
  sparsity} and \textit{strong sparsity} respectively. This breakdown
into weak and strong sparsity settings is inspired by corresponding
terminology in the study of risk properties of thresholding based
estimators in Gaussian sequence models \cite{johnstone2011gaussian}
and the prediction risk properties of the LASSO estimators in
regression  \cite{buhlmannLASSO}. Indeed, as demonstrated in
\citet{tibshirani2014adaptive}, there is a close connection 
between the trend filtering estimators and LASSO (more details are
provided in Subsection \ref{lasso}).  

A thorough study on the performance of the penalized  
trend filtering estimator \eqref{ptf} under weak sparsity 
has been done by \citet{tibshirani2014adaptive} and
\citet{wang2014falling} building on earlier results of
\citet{mammen1997locally}. It is proved there that, when the tuning
parameter $\lambda$ is appropriately chosen, the penalized estimator
\eqref{ptf} is minimax optimal in the weak sparsity setting. Actually,
the weak sparsity results of \cite{tibshirani2014adaptive, wang2014falling} are broader and hold under more general
settings (see Remark \ref{tiban} for more details).  

The present paper focuses on the strong sparsity setting. Compared to
available results in the weak sparsity setting, relatively little is
known about the performance of the trend 
filtering estimators in the strong sparsity setting. In fact, all
existing results \cite{dalalyan2017tvd, lin2016approximate,
ortelli2018total, van2018tight, levy2008catching,
harchaoui2012multiple} for strong sparsity deal 
with the case $r = 1$ (where trend filtering is the 
same as total variation denoising). To the best of our knowledge, the
present paper is the first to prove risk bounds for trend
filtering under strong sparsity for arbitrary $r \geq 1$. We also
improve, in certain aspects, existing results for $r = 1$.   

In order to motivate our results, let us consider the strong sparsity
setting where it is assumed that $D^{(r)}\theta^*$ is sparse. If
$\|D^{(r)} \theta^*\|_0 = k$, then, as mentioned previously, $\theta^* =
(f(1/n), \dots, f((n-1)/n), f(1))$ for a discrete spline function $f$
that is made of $k+1$ polynomials each of degree $(r-1)$. Given data
$Y \sim N_n(\theta^*, \sigma^2 I_n)$, an oracle  
piecewise polynomial estimator (having access to locations of the knots of
$\theta^*$) would put knots corresponding to $\theta^*$ and then fit a
polynomial of degree $(r-1)$ in each of the partitions given by the
knots. This would be a linear estimator with at most $(k+1)r$ degrees
vvof freedom and its risk (defined as in
\eqref{ridef}) will be bounded by $r \sigma^2 (k+1)/n$. This motivates
the following question which is the focus of this paper: When
$\|D^{(r)} \theta^* \|_0 = k$, how do the risks of properly tuned
trend filtering estimators \eqref{tf} and \eqref{ptf} compare with the oracle risk of
$r \sigma^2 (k+1)/n$?  

The main results of this paper for constrained trend filtering
(Theorem \ref{ada} and Corollary \ref{kco}) imply
that when $\|D^{(r)} \theta^*\|_0 = k$, the risk of
$\hat{\theta}^{(r)}_V$ satisfies
\begin{equation}\label{intror.1}
  R(\hat{\theta}^{(r)}_V, \theta^*) \leq C_r(c) \sigma^2 \frac{k+1}{n}
    \log \frac{en}{k+1},
\end{equation}
provided
\begin{enumerate}
\item[(i)]  the tuning  
parameter $V$ is non-random and close to $V^* := n^{r-1} \|D^{(r)}
\theta^*\|_1$, and  
\item[(ii)] (minimum length condition) each of the polynomial pieces of $\theta^*$
have length bounded below by $cn/(k+1)$ for a constant $c>0$
(in fact, our result requires a weaker version of this condition;
see~\eqref{nsa} and Remark~\ref{rem:nsa}). 
\end{enumerate}
Here $C_r(c)$ is a positive constant that depends only on $r$ and the
constant $c$ from the second assumption above. 

We also prove results for the penalized estimators. For $r = 1$, our
main result (Corollary \ref{expen1}) states that the risk of 
$\hat{\theta}^{(1)}_{\lambda}$ is also bounded by the right hand side
of \eqref{intror.1} under the minimum length condition provided
$\lambda$ is close to a theoretical choice $\lambda^*$ and $\lambda
\geq \lambda^*$. This choice $\lambda^*$ depends on $\theta^*$ and is
defined in \eqref{lams}. We provide an explicit upper bound for
$\lambda^*$ in Lemma \ref{explam1} which gives risk bounds for
$\hat{\theta}^{(1)}_{\lambda}$ under more explicit choices of $\lambda$
(see Corollary \ref{expen}). A comparison of these results to existing
results is given in Remarks \ref{comp} and \ref{compa}.   

For $r \geq 2$, we prove, in Corollary \ref{expenr}, that the
penalized estimator satisfies  
\begin{equation}\label{intro.r}
    R(\hat{\theta}^{(r)}_{\lambda}, \theta^*) \leq C_r(c) \sigma^2
    \left(\frac{k+1}{n} \log \frac{en}{k+1} + \frac{(k+1)^{2r}}{n}
    \right) 
\end{equation}
under the minimum length condition provided that $\lambda$ is 
close to $\lambda^*$ (defined in \eqref{lams}) and $\lambda \geq
\lambda^*$. Explicit upper bounds for $\lambda^*$ are in Lemma
\ref{explamr} and risk bounds for $\hat{\theta}^{(r)}_{\lambda}$ with
explicit penalty choices are in Corollary \ref{supno}. Note that
\eqref{intro.r} is weaker compared to \eqref{intror.1} in terms of the
dependence on $k$. 

The implication of our results is the following. As mentioned earlier,
the trend filtering estimators are given by discrete spline functions
of degree $r - 1$. The knots of these splines are chosen automatically
by the algorithm (the user only needs to specify the tuning parameter
$V$ or $\lambda$). Our results indicate that under the assumption
$\|D^{(r)} \theta^*\|_0 = k$ (i.e., $\theta^*$ is a discrete spline of
degree $r-1$ with $k+1$ polynomial pieces) with a minimum length
condition on the polynomial pieces of $\theta^*$, the automatic
selection of knots by the trend filtering estimators (when appropriate choices of $V$ or $\lambda$) happens in a way
that the overall risk is comparable to the 
oracle risk of $r \sigma^2 (k+1)/n$. In fact, when $k = O(1)$, the
risks of the ideally tuned trend filtering estimators is only off
compared to the oracle risk by a factor that is logarithmic in $n$ (we also prove
in Lemma \ref{simpo} that this logarithmic factor cannot be completely
removed in general). The automatic knot selection of trend filtering
can therefore be interpreted as being done \textit{adaptively}
depending on the structure of the unknown $\theta^*$ in order to
approximate the oracle risk. This is the reason why we refer to our
results as  adaptive risk bounds. It should be mentioned here that a
similar 
adaptation story can also be used to describe the weak sparsity
results \cite{tibshirani2014adaptive, wang2014falling} where the knots
are adaptively chosen to attain the minimax rate under the $L^1$
constraint on $D^{(r)} \theta^*$. Therefore, our results (together
with those of \cite{tibshirani2014adaptive, wang2014falling}) provide
support for the use of the trend filtering estimators in both weak and
strong sparsity settings. 

We would like to mention here that theoretical analysis of spatially
adaptive nonparametric regression methods under strong sparsity is
non-trivial. Indeed, among 
various such methods including CART, MARS, variable-bandwidth
kernel/spline methods and wavelets, rigorous theoretical risk results
under strong sparsity only exist for wavelets \cite{donoho1994ideal}
and variable-bandwidth kernel methods  \cite{LMS97, GN97}. The analysis of trend filtering estimators is more
involved compared to estimators based on wavelets and
variable-bandwidth kernels because the trend filtering estimators are
given by the output of an optimization algorithm and have no closed
form expressions.   

The rest of this paper is organized as follows. Our main results are
described in Section \ref{ms}: Subsection \ref{rece} deals with the
constrained estimator where we provide risk bounds under both weak
sparsity (which was not known previously) and strong
sparsity. Subsection \ref{repe} deals with the penalized estimator and
here we separate our presentation into two parts: results for $r = 1$
and results for $r \geq 2$; our results for $r \ge 2$ are weaker
(there is an additional $(k+1)^{2r}/n$ term in the risk) than the
results for $r =1$. Throughout, we focus on nonasymptotic upper bounds
for the risk (expected loss) although all our results can be converted into high
probability upper bounds on the loss (see Remark \ref{exphp}). All
proofs are given in the supplementary material at the end of the
paper and a high level overview of the proofs is provided  
in Section \ref{proids}. Section \ref{simu} contains some simulation
studies supporting some of our theoretical results. Finally several
interesting issues related to our results are described in Section
\ref{discu}.

\section{Main Results}\label{ms}
Throughout $C_r$ will denote a positive constant that depends on $r$ alone
although its precise value will change from equation to
equation. We shall assume that $n \geq 2r$ throughout the paper  
(many of our results also hold under the weaker condition $n \geq r +
1$). 

\subsection{Results for the Constrained Estimator}\label{rece}
We start with the bound of $n^{-2r/(2r + 1)}$ for
risk of  $\hat{\theta}^{(r)}_V$ under the condition that
the tuning parameter $V$ satisfies $\|D^{(r)} \theta^*\|_1 \leq V
n^{1-r}$. This result is similar to results
in~\citet{mammen1997locally}, \citet{tibshirani2014adaptive} and
\citet{wang2014falling} who focussed on the penalized estimator
\eqref{ptf} (see Remark \ref{tiban} for details). We also
explicitly state the dependence of the bound on $V$ and
$\sigma$.      

\begin{theorem}\label{woor}
  Fix $r \geq 1$. Suppose that the tuning parameter $V$ is 
  chosen so that $n^{r-1}\|D^{(r)}\theta^*\|_1 \leq V$. Then there 
  exists a positive constant $C_r$ depending on $r$ alone such that      
  \begin{equation}\label{woor.eq}
    R(\hat{\theta}^{(r)}_V, \theta^*) \leq C_r \max \left(
      \left(\frac{\sigma^2 V^{1/r}}{n} 
      \right)^{2r/(2r + 1)} ,
      \frac{\sigma^2}{n} \log (en)\right). 
  \end{equation}
  Also for every $x > 0$, we have 
  \begin{equation}\label{woor.peq}
    \frac{1}{n} \|\hat{\theta}^{(r)}_V - \theta^* \|^2 \leq C_r \max \left(
      \left(\frac{\sigma^2 V^{1/r}}{n} \right)^{2r/(2r + 1)} ,
        \frac{\sigma^2}{n} \log (en)\right)  + \frac{4 \sigma^2 x}{n} 
  \end{equation}
  with probability at least $1 - e^{-x}$. 
\end{theorem}

%\vspace{0.2in}

\begin{remark}\label{tiban}
As mentioned earlier, bounds similar to \eqref{woor.eq} and
\eqref{woor.peq} have been proved in \citet{mammen1997locally},
\citet{tibshirani2014adaptive}  and \citet{wang2014falling} 
for the penalized trend filtering estimator. Actually, the bounds in
these earlier papers hold under more general assumptions than the
assumptions of the current paper. For example, their analyses also holds
under the assumption that the (continuous) variation norm of the function $(f^*)^{(r-1)}$ (this
is the $(r-1)^{th}$ derivative of $f^*$) is at most $V$, where $f^*$ is
the true function with $\theta^* = 
(f^*(1/n), \dots, f^*(1))$. Note that there is subtle difference between this and our
assumption of an upper bound on $\|D^{(r)}
\theta^*\|_1$ in the sequence model~\eqref{seqmo}. An assumption on the variation norm of
$(f^*)^{(r-1)}$ does not directly lead to a bound on $\|D^{(r)}
\theta^*\|_1$ which makes the analysis difficult (see
\citet{wang2014falling} for more details on the relation between the
two variation norms). Also, the results in these earlier papers
studied the general setting with $\theta^* := (f^*(x_1), \dots,
f^*(x_n))$ where $x_1, \dots, x_n$ are design points that are not
necessarily equally spaced. We restrict ourselves to the equally
spaced design setting in this paper (see Subsection \ref{weakass}).  
\end{remark}

\begin{remark}\label{mopt}
$n^{-2r/(2r+1)}$ is the minimax rate of
estimation over the class of $\theta \in \R^n$ with
$\|D^{(r)}\theta\|_1 \leq V n^{1-r}$ (see e.g.,
\citet{donoho1998minimax}). This means that the constrained trend 
filtering estimator with tuning parameter $V$ is minimax optimal over
$\{\theta \in \R^n : \|D^{(r)}\theta\|_1 \leq V n^{1-r}\}$. This
result was known previously for the penalized estimator; see
\citet{tibshirani2014adaptive}. Note also that $V$ here can change
with $n$ as well and inequality \eqref{woor.eq}  implies that
$\hat{\theta}_V^{(r)}$ is minimax optimal even in terms of the
dependence of the rate on $V$. 
\end{remark}

Before we state results for strong sparsity, we need 
 some notation. Fix an integer $r \geq 1$ and let $n \geq
r+1$. For a vector $\theta \in \R^n$ and an index $2 \leq j \leq n-r+1$,
we say that $j$ is an $r^{th}$ \textit{order knot} (or \textit{knot of
  order $r$}) of $\theta$ provided $(D^{(r-1)} \theta)_{j-1} \neq (D^{(r-1)}
\theta)_j$. Note that first order knots are just jumps and second
order knots are points of change of slope. We also say
that an $r^{th}$ order knot $j$ has \textit{sign} $+1$ if $(D^{(r-1)}
\theta)_{j-1} < (D^{(r-1)} \theta)_j$ and \textit{sign} $-1$ if $(D^{(r-1)}
\theta)_{j-1} > (D^{(r-1)} \theta)_j$. For $\theta \in \R^n$, we let  
\begin{equation}
  \label{kv}
  \ktr(\theta) := \|D^{(r)} \theta\|_0 ~~~ \text{ and } ~~~ V^{(r)}(\theta)
  := n^{r-1}\|D^{(r)} \theta\|_1. 
\end{equation}
When $r = 1$, note that $V^{(1)}(\theta) = \|D \theta\|_1 = |\theta_2 -
\theta_1| + \dots + |\theta_n - \theta_{n-1}|$ which is simply the
variation of $\theta$. We therefore simply denote $V^{(1)}(\theta)$ by
$V(\theta)$. It also follows then that $V^{(r)}(\theta) = n^{r-1}
V(D^{(r-1)} \theta)$. 

It may be observed that $\ktr(\theta)$ equals precisely the number of
$r^{th}$ order knots of $\theta$. When the value of $r$ and $\theta
\in \R^n$ are clear from the context, we simply denote $\ktr(\theta)$
by $k$. Also, note that as $D^{(r)} \theta$ is a vector of length $n -
r$, we necessarily have $\ktr(\theta) = \|D^{(r)} \theta\|_0 \le n-r
\le n-1$.  

Suppose $\ktr(\theta) = k$   and let $2 \leq j_1 < \dots < j_k \leq
n-r+1$ denote all the $r^{th}$ order knots of $\theta$ with associated
signs $\rs_1, \dots, \rs_k \in \{-1, 1\}$. Also let $\rs_0 = \rs_{k+1} =
0$. Further, let $n_0 := j_1 + r-2$, $n_i := j_{i+1} - j_i$, for $1
\leq i \leq k-1$, and $n_k := n-r+2-j_k$, and observe that $\sum_{i=0}^k
n_i = n$. Finally, let     
\begin{equation*}
  n_{i*} := \min \left(n_i, \frac{n}{k+1} \right) \qt{for $i = 0, 1,
    \dots, k$}. 
\end{equation*}
We now define two quantities $\delta_r(\theta)$ and $\Delta_r(\theta)$
in the following way: 
\begin{equation}\label{drt}
  \delta_r(\theta) := \left(n^{1-2r}_{0*} + n^{1-2r}_{k*}  +
    \sum_{i=1}^{k-1} n_{i*}^{1 - 2r} I\{\rs_i \neq \rs_{i+1}\} 
  \right)^{1/2}
\end{equation}
and 
\begin{equation}\label{mrt}
  \Delta_r(\theta) := \frac{k+1}{n} \log \frac{en}{k+1} +
  \frac{\delta^2_r(\theta)}{n}  \left(\frac{n}{k+1} \right)^{2r-1}
  \log \frac{en}{k+1} + \left(\frac{\delta_r(\theta)}{\sqrt{n}}
  \right)^{1/r}
\end{equation}
where, in the definition of $\delta_r(\theta)$, the quantity $I\{\rs_i
\neq \rs_{i+1} \}$ denotes the indicator variable that equals $1$ if
$\rs_i \neq \rs_{i+1}$ and $0$ if $\rs_i = \rs_{i+1}$. Note that
trivially $\Delta_r(\theta) \geq (k+1)/n \geq 1/n$. 

Our results will show that the risk of the estimator
$\hat{\theta}^{(r)}_V$ for $\theta^*$ will essentially be controlled
by $\Delta_r(\theta^*)$. The key point to note about
$\Delta_r(\theta)$ is the fact (easy to check) that when  
\begin{equation}\label{nsa}
  \min_{0 \leq i \leq k : \rs_i \neq \rs_{i+1}} n_i \geq
  \frac{cn}{k+1} 
\end{equation}
for a positive constant $c \leq 1$ (here $\rs_1, \dots, \rs_k \in
\{-1, 1\}$ are the signs of the $r^{th}$ order knots of $\theta$ while
$\rs_0$ and $\rs_{k+1}$ are taken to be zero), then    
\begin{equation*}
  \delta^2_r(\theta) \leq \left(\frac{cn}{k+1} \right)^{1 - 2r} (k+1) 
\end{equation*}
and consequently
\begin{align}
  \Delta_r(\theta) &\leq \left\{1 + c^{1 - 2r} \right\} \frac{k+1}{n}
  \log \frac{en}{k+1} + c^{(1 - 2r)/(2r)} \frac{k+1}{n} \nonumber \\
&\leq \left\{1 + c^{1 - 2r} + c^{(1-2r)/(2r)} \right\} \frac{k+1}{n}
\log \frac{en}{k+1}.  \label{ghos} 
\end{align}
We say that $\theta$ satisfies the \textit{minimum length condition}
with constant $c$ if condition~\eqref{nsa} holds. We have just
observed that when $\theta$ satisfies the minimum length condition
with constant $c$ then $\Delta_r(\theta) \leq C_r(c) \frac{k+1}{n}
\log \frac{en}{k+1}$ for a constant $C_r(c)$ depending only on $c$ and
$r$.   

The following is our main result for the constrained
trend filtering estimator.  
\begin{theorem}\label{ada}
Fix $r \geq 1$ and $n \geq 2r$. Consider the estimator
$\hat{\theta}^{(r)}_V$ defined in \eqref{tf} with tuning parameter
$V \geq 0$. Then for every $\theta^* \in \R^n$, we have 
\begin{equation}\label{ada.eq}
  R(\hat{\theta}^{(r)}_V, \theta^*) \leq \inf_{\theta \in \R^n :
    V^{(r)}(\theta) = V} \left(\frac{1}{n} \|\theta^* - \theta\|^2 +
   C_r \sigma^2 \Delta_{r} (\theta) \right)
\end{equation}
for a positive constant $C_r$, depending only on $r$. 
\end{theorem}

\begin{remark}[High-probability bound]\label{exphp}
  Note that Theorem \ref{ada} gives an upper bound for
  $R(\hat{\theta}^{(r)}_V, \theta^*)$ which is the expectation of $\frac{1}{n}\|\hat{\theta}_V^{(r)} - \theta^*\|^2$.
%  $L(\hat{\theta}_V^{(r)}, \theta^*)$ where 
%  \begin{equation*}
%    L(\hat{\theta}_V^{(r)}, \theta^*) := \frac{1}{n}
%    \|\hat{\theta}_V^{(r)} - \theta^*\|^2. 
%  \end{equation*}
 Similarly as in Theorem \ref{woor},  the risk bound
 \eqref{ada.eq} can be supplemented by the  following high probability
 bound: for every $x > 0$, we have   
\begin{equation}\label{ada.peq}
  \frac{1}{n}
    \|\hat{\theta}_V^{(r)} - \theta^*\|^2 \le \inf_{\theta \in \R^n : 
    V^{(r)}(\theta) = V} \left(\frac{1}{n} \|\theta^* - \theta\|^2 +
   C_r \sigma^2 \Delta_{r} (\theta) \right) + \frac{4 \sigma^2 x}{n} 
\end{equation}
with probability at least $1 - e^{-x}$. This will be true in all the
results of this paper (namely that the bound on $R(\hat{\theta},
\theta^*)$ plus $4 \sigma^2 x/n$ will dominate $\frac{1}{n}
    \|\hat{\theta} - \theta^*\|^2$ with probability at least $1 - e^{-x}$). Thus, for ease of
presentation, we shall omit high probability statements and only
report risk results (i.e., bounds on $R(\hat{\theta}, \theta^*)$) in
the rest of the paper.   
\end{remark}

%\vspace{0.2in} 
Theorem \ref{ada} applies to every $\theta^* \in \R^n$ and is stated
in the sharp oracle form. It implies that the risk of
$\hat{\theta}^{(r)}_V$ is small provided
there exists some $\theta \in \R^n$ with $V^{(r)}(\theta) = V$ such
that (a) $\|\theta - \theta^*\|$ is small, and (b) $\Delta_r(\theta)$
is small. 

Theorem \ref{ada} yields the following corollary which is a
non-oracle inequality and is more readily interpretable. Recall from
\eqref{ghos}  that $\Delta_r(\theta)$ is bounded from above by a
constant multiple of $\frac{k+1}{n} \log \frac{en}{k+1}$ with
$\ktr(\theta) = k$ provided $\theta$ satisfies \eqref{nsa}.   

\begin{corollary}\label{kco}
Consider the estimator $\hat{\theta}^{(r)}_V$ with tuning parameter
$V$. Suppose $\theta^*$ satisfies the minimum length 
condition~\eqref{nsa} with constant $c$, then  
  \begin{equation}\label{kco.eq}
    R(\hat{\theta}^{(r)}_V, \theta^*) \leq \left(V - V^{(r)}(\theta^*)
    \right)^2 + C_r(c) \frac{\sigma^2 \left(\ktr(\theta^*)+1\right)}{n}  \log
    \frac{en}{\ktr(\theta^*)+1}  
  \end{equation}
where $C_r(c)$  is a positive constant that depends on $r$ and $c$
alone. Further, if $V$ is chosen so that  
  \begin{equation*}
   \left(V - V^{(r)}(\theta^*)\right)^2 \leq C \frac{\sigma^2
     (\ktr(\theta^*) + 1)}{n} \log \frac{en}{\ktr(\theta^*) + 1} 
  \end{equation*}
  for a positive constant $C$, then we have 
  \begin{equation}\label{kcoc}
    R(\hat{\theta}^{(r)}_V, \theta^*) \leq C_r(c, C) \frac{\sigma^2
      \left(\ktr(\theta^*)+1\right)}{n}  \log 
    \frac{en}{\ktr(\theta^*)+1}  
  \end{equation}
   for a positive constant $C_r(c, C)$ that depends on $r$, $c$ and
   $C$ alone. 
\end{corollary}

%\vspace{0.2in}

Note that Theorem \ref{ada} and Corollary
\ref{kco} both apply to every $r \ge 1$. On the other hand, existing
adaptation results for trend filtering all deal with the case $r = 1$
(which corresponds to total variation regularization). Even for $r =
1$, our results are stronger, in some respects, 
compared to the existing results in the literature (see Remark
\ref{comp} for a precise comparison).  

\begin{remark}[On the minimum length condition]\label{rem:nsa}
  The minimum length condition
  \eqref{nsa} required for Corollary \ref{kco} is weaker than existing
  minimum length conditions in the literature (this comparison is only
  for $r = 1$ because no results exist for $r \geq 2$) which are all
  of the  form
\begin{equation}\label{wn}
  \min_{0 \leq i \leq k} n_i \geq \frac{cn}{k+1}  \qt{where $k =
    \mathbf{k}_1(\theta^*)$}. 
\end{equation}
Indeed our condition \eqref{nsa} requires that $n_i \geq cn/(k +1)$  
  be true only for those $i$ for which $\rs_i \neq \rs_{i+1}$ while \eqref{wn}
  requires this for all $i$. To see why our condition can be
  substantially weaker, consider, for example, the situation when $D^{(r-1)} \theta^*$ is a monotonic vector (for $r = 1$, this means that $\theta^*$ is itself monotone while
  for $r = 2$, this means that $\theta^*$ is convex/concave). In this
  case, condition \eqref{nsa} is equivalent to requiring that $n_i \geq
  cn/(k+1)$ only for $i = 0$ and $i = k$ which is much weaker than
  requiring it for all $0 \leq i  \leq k$. 

The fact that our minimum length condition involves only those $i$ for
which $\rs_i \neq \rs_{i+1}$ as opposed to involving all $i \in \{0,
1, \dots, k\}$ is especially crucial for $r \geq 2$. To see this,
consider the piecewise linear function $f^*$ on $[0, 1]$  shown in
Figure \ref{fminmo}. This function clearly has three knots (points of
change of slope) in $(0, 1)$. However the vector $\theta^*$ obtained
as $(f^*(1/n), \dots, f^*(n/n))$ (with $n = 15$) has six second order
knots. The reason for the additional knots is due to the fact that the
original knots of $f^*$ are not at the design points $1/n, \dots,
n/n$. Note however that because of these additional knots, the minimum
length condition will not be satisfied over all $i = 0, 1, \dots,
k$. On the other hand, it should be clear that \eqref{nsa} will still
be satisfied because the additional linear pieces satisfy the property
that $\rs_i = \rs_{i+1}$. 
\end{remark}

\begin{figure}[h!]
\begin{center}
  \includegraphics[width = \textwidth]{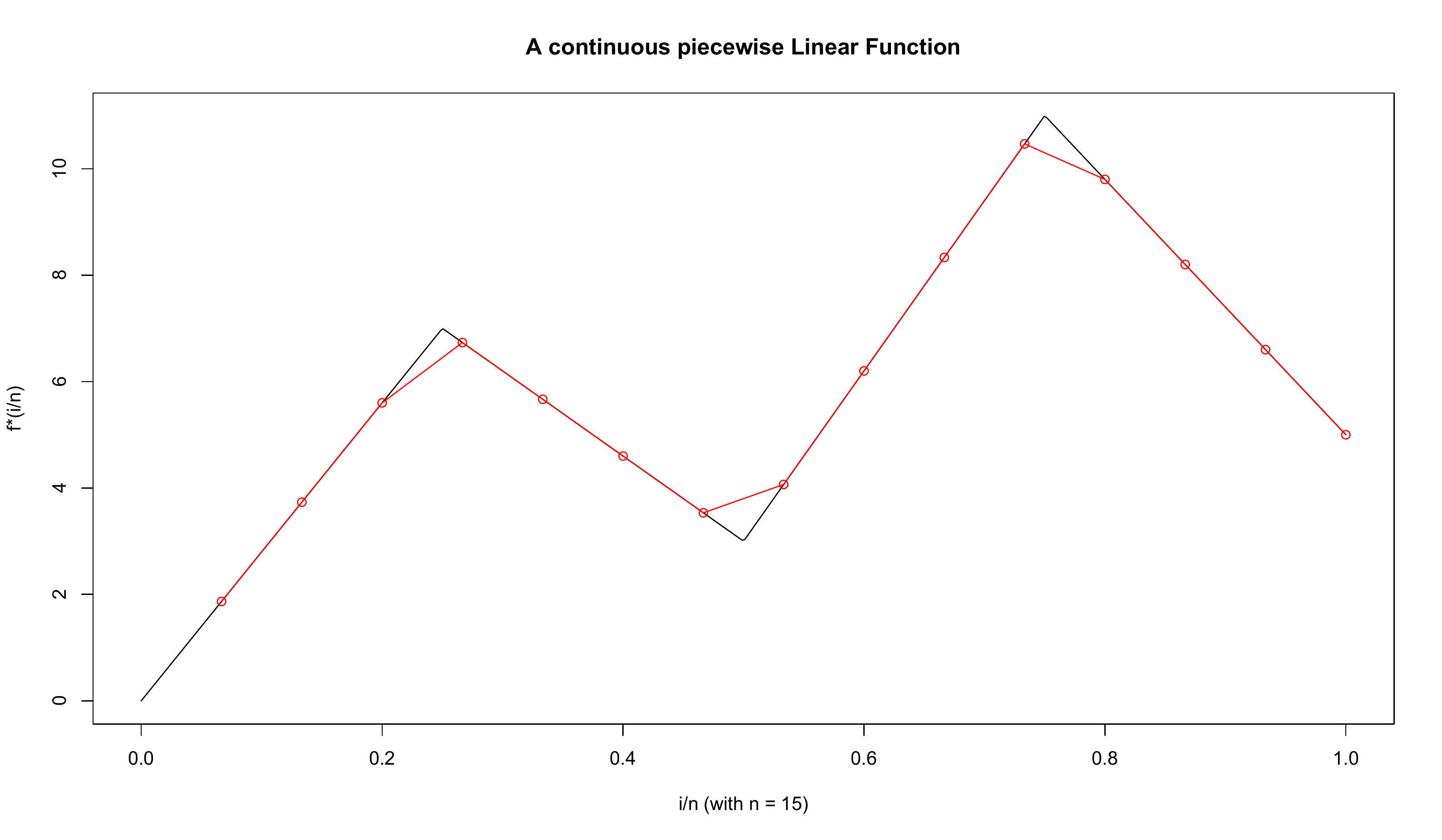}        
\caption{A piecewise linear function $f^*$ on 
$[0, 1]$ together with the vector $\theta^* := (f^*(1/n), \dots,
f^*(1))$ for $n = 15$ plotted in red. Note that $f^*$ has three knots
while $\theta^*$ has six first order knots. } 
\label{fminmo}
\end{center}
\end{figure}

%\vspace{0.2in}

\begin{remark}[The minimum length condition cannot be removed]  
We shall argue here via simulations that the minimum length condition
in Corollary \ref{kco} cannot be removed. Suppose that $\theta^*$ is
given by 
\begin{equation}\label{vio1}
  \theta_1^* = \dots = \theta_{n-1}^* = 0 ~~~ \text{ and } ~~~
  \theta_n^* = 5
\end{equation}
and consider estimating $\theta^*$ from an observation $Y \sim
N_n(\theta^*, I_n)$ (i.e., $\sigma = 1$) by $\hat{\theta}^{(1)}_V$ (i.e.,
$r = 1$) with tuning parameter $V = V^{(1)}(\theta^*) = 5$. It is
clear here that $\mathbf{k}_1(\theta^*) = 1$. The minimum length
condition \eqref{nsa} is not satisfied because $n_0 = n-1$ and $n_1 = 1$. The
risk $R(\hat{\theta}^{(1)}_V, \theta^*)$ can be computed via
simulation. In Figure \ref{f1} (left panel), we have plotted $\log
R(\hat{\theta}^{(1)}_V, \theta^*)$ against $\log n$ for values of $n$
between 1000 and 5000 (chosen to be equally spaced on the log-scale). For
each value of $n$, we calculated the risk using $100$ Monte Carlo
replications. The slope of the least squares line through these points
turned out to be close to $-2/3$ which indicates that the risk 
$R(\hat{\theta}^{(1)}_V, \theta^*)$ decays at the rate
$n^{-2/3}$. This rate is slower than the rate given by Corollary
\ref{kco} indicating that inequality \eqref{kco.eq} is not true for
this $\theta^*$. On the other hand, the $n^{-2/3}$ rate here makes
sense in light of Theorem  
\ref{woor}. Therefore, even though the vector $D \theta^*$ is sparse
(with $\|D \theta^*\|_0 = 1$), the rate of convergence of
$\hat{\theta}^{(1)}$ is equal to the $n^{-2/3}$ and not the
faster rate given by Corollary \ref{kco}. This points to the necessity of
the minimum length condition \eqref{nsa}. 
\begin{figure}[!htbp]
\begin{center}
 \includegraphics[height=2.0in,width=2.2in]{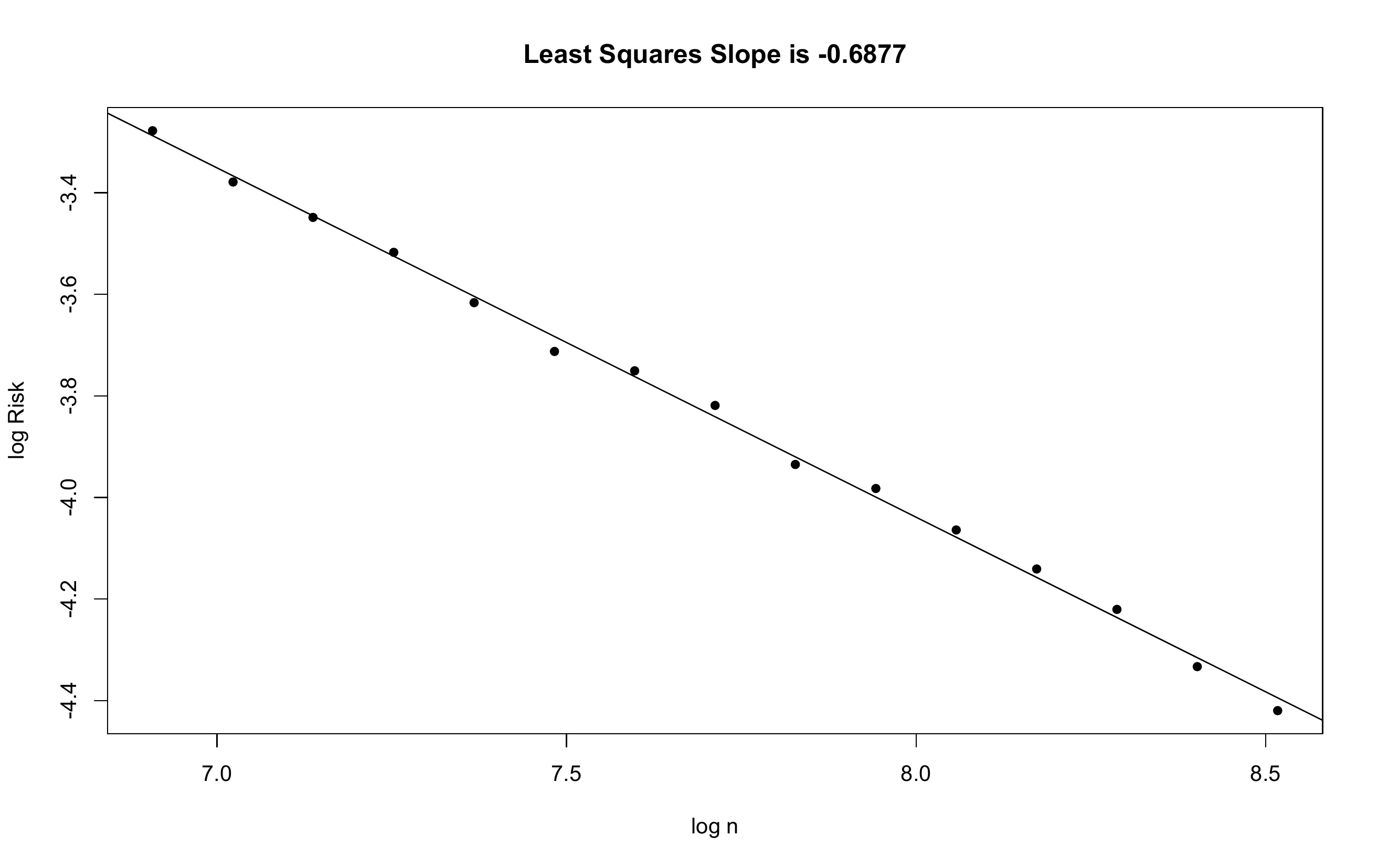}     \includegraphics[height=2.0in,width=2.2in]{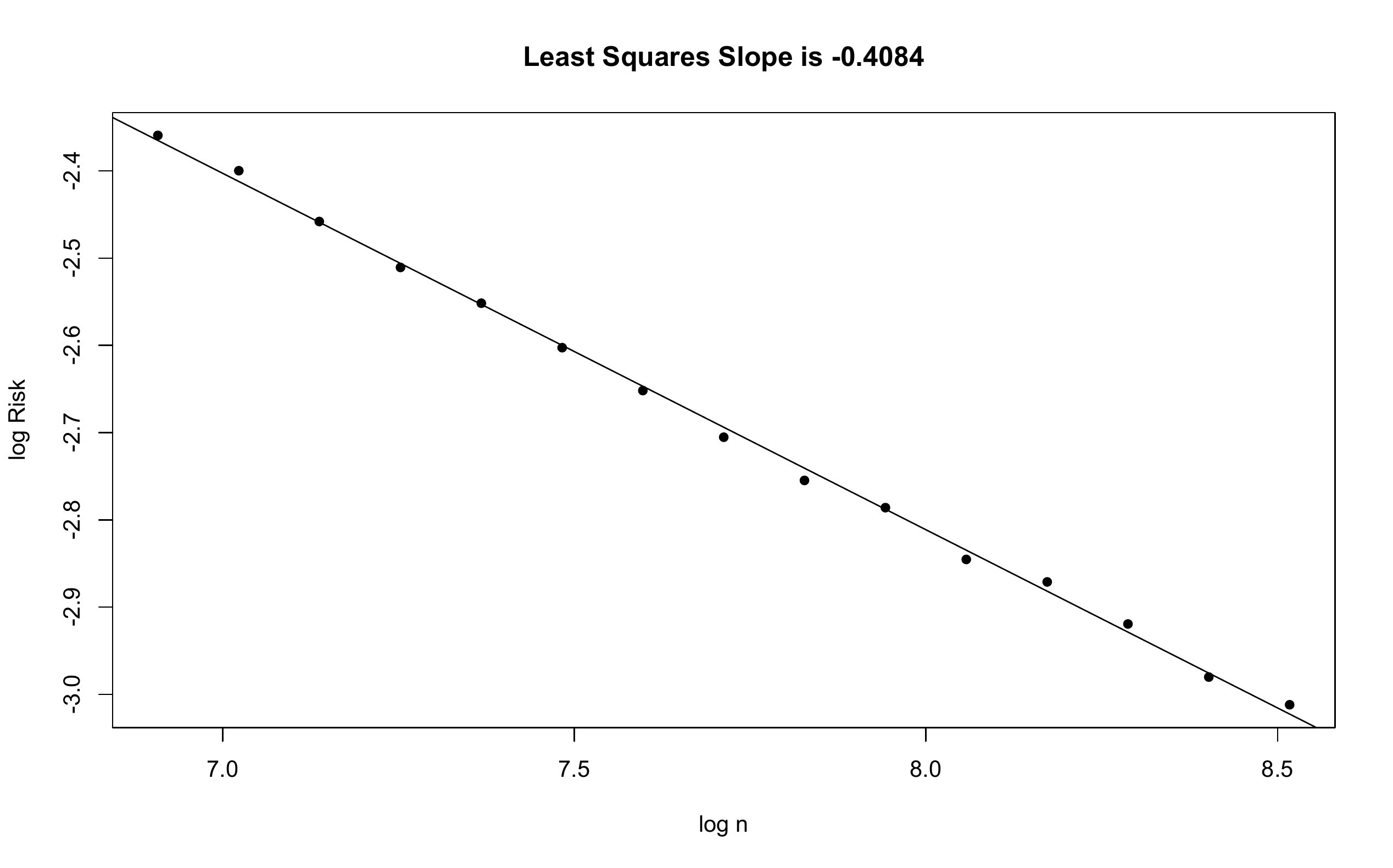}        
\caption{{\bf Left}: plot of $\log R(\hat{\theta}^{(1)}_V, \theta^*)$ against
  $\log n$ for $\theta^*$ as in \eqref{vio1}. The least
  squares slope is close to $-2/3$ which suggests that the risk decays
  as $n^{-2/3}$ instead of the faster rate given by Corollary
  \ref{kco}. {\bf Right}: plot of $\log R(\hat{\theta}^{(2)}_V, \theta^*)$
  against $\log n$ for $\theta^*$ defined in \eqref{vio2}. The slope is close to $-2/5$ which suggests that the risk decays
  as $n^{-2/5}$ instead of the faster rate given by Corollary
  \ref{kco}.} 
\label{f1}
\end{center}
\end{figure}

Another counterexample for the necessity of~\eqref{nsa} for Corollary
\ref{kco} is:  
\begin{equation}\label{vio2}
  \theta_1^* = \dots = \theta_{\lfloor n/2 \rfloor}^* = 0 ~~~ \text{
    and } ~~~ \theta_{\lfloor n/2 \rfloor + 1}^* = \theta_{\lfloor n/2
    \rfloor + 2}^* = \dots = \theta_{n}^* = 5. 
\end{equation}
Here consider the problem of estimating $\theta^*$ by the estimator
$\hat{\theta}^{(2)}_V$  (i.e., $r = 2$) with tuning parameter $V =
V^{(2)}(\theta^*) = 10 n$. It is clear that $\mathbf{k}_2(\theta^*) = 2$,
$n_0 = \lfloor n/2 \rfloor$, $n_1 = 1$ and $n_2 = n - \lfloor n/2
\rfloor - 1$. The minimum length condition \eqref{nsa} is not
satisfied as $n_1$ is too small. The risk $\log
R(\hat{\theta}^{(2)}_V, \theta^*)$ is plotted~against $\log n$ in the
right panel of Figure~\ref{f1} (the values of $n$ are chosen as
before). The slope of the least squares line here is close to $-2/5$
which suggests that the risk decays slowly than what is given by
Corollary \ref{kco}. Note that $n^{-2/5}$ is exactly the rate
given by Theorem \ref{woor} (take $r = 2$ and $V = 10 n$
in \eqref{woor.eq}).   
%\begin{figure}[htb]
%\begin{center}
% \includegraphics[scale = 0.65]{RPlotFigure2}           
%\caption{A plot of $\log R(\hat{\theta}^{(2)}, \theta^*)$ against
%  $\log n$ for $\theta^*$ defined as in \eqref{vio2}. The least
%  squares slope is close to $2/5$ which suggests that the risk decays
%  as $n^{-2/5}$ instead of the fast rate given by Corollary \ref{kco}.}
%\label{f2}
%\end{center}
%\end{figure}
\end{remark} 

It is natural to ask if the bound given by inequality \eqref{kcoc} can
be improved further by dropping the $\log \frac{en}{\ktr(\theta^*) +
  1}$ term. The following simple result shows that this cannot be done
in general. 

\begin{lemma}\label{simpo}
  Suppose $\theta^* := (0, \dots, 0, 1, \dots, 1)$ with jump at $j =
  \lceil n/2 \rceil$. Let 
  $\hat{\theta}^{(1)}_{V = 1}$  denote the estimator \eqref{tf} with
  $V = 1$. Then
  \begin{equation*}\label{simpo.eq}
    \lim_{\sigma \downarrow 0} \frac{1}{\sigma^2}
    R(\hat{\theta}_{V=1}^{(1)}, \theta^*) \geq 
    \frac{\log(n/2)}{2n}. 
  \end{equation*}
\end{lemma}

\subsection{Results for the Penalized Estimator}\label{repe} 
In this section, we present risk results for the penalized estimator
defined in \eqref{ptf}. An important role in these results will be
played by the subdifferential of the convex function $f(\theta) :=
\|D^{(r)} \theta\|_1$ at the true parameter value $\theta^*$. Recall
that the subdifferential of a convex function $g : \R^n \rightarrow
\R$ at a point $\theta \in \R^n$ is the set consisting of all
subgradients of $g$ at $\theta$ and will be denoted by $\partial
g(\theta)$. For every finite convex function $g$ on $\R^n$ and $\theta
\in \R^n$, the subdifferential $\partial g(\theta)$ is non-empty,
closed, convex and bounded (see, for example, \citet[Page
218]{Rockafellar70book}). 

The following is the reason why $\partial f(\theta^*)$ (for $f(\theta)
:= \|D^{(r)} \theta\|_1$) plays a key role in understanding the risk
of \eqref{ptf}. It has been proved by \citet[Theorem
2.2]{oymak2013sharp} that for a general penalized estimator:   
\begin{equation*}
  \hat{\theta}_{\lambda}^{g} := \argmin_{\theta \in \R^n}
  \left(\frac{1}{2} \|Y - \theta\|^2 + \sigma \lambda g(\theta)
  \right) 
\end{equation*}
where $g : \R^n \rightarrow \R$ is convex, its risk under the model $Y
\sim N_n(\theta^*, \sigma^2 I_n)$ satisfies: 
\begin{equation}\label{ohag}
  R(\hat{\theta}_{\lambda}^{g}, \theta^*) \leq \frac{\sigma^2}{n} \E \left(
  \inf_{v \in \lambda \partial g(\theta^*)} \|Z - v\|^2 \right) 
\end{equation}
where $\lambda \partial g(\theta^*) := \left\{\lambda v : v
  \in \partial g(\theta^*) \right\}$ and the expectation on the right
hand side is with respect to the standard Gaussian vector $Z \sim
N_n(0, I_n)$. Moreover, inequality 
\eqref{ohag} cannot in general be improved,
because, as proved in \cite[Proposition 4.2]{oymak2013sharp}, it is
tight in the low $\sigma$ limit, i.e., the limit (as $\sigma
\rightarrow 0$) of the left hand side of \eqref{ohag} scaled by
$\sigma^2/n$ equals the expectation on the right hand side of
\eqref{ohag}. Inequality \eqref{ohag} will be our main technical tool
for studying the risk of \eqref{ptf} and thus it will be important to
understand the subdifferentials of the function $\theta \mapsto
\|D^{(r)} \theta\|_1$. 

The next result (proved in Subsection \ref{sudi}) characterizes the 
subdifferential of $f(\theta) := \|D^{(r)}
\theta\|_1$. 

\begin{proposition}\label{characterization} 
Consider the function $f : \R^n \rightarrow \R$ defined by $f(\alpha)
:= \|D^{(r)} \alpha\|_1$. Fix  $\theta \in \R^n$. Then $\partial
f(\theta)$ consists of vectors $v  \in \mathbb{R}^n$ such that    
  \begin{equation}\label{ch1}
   \sum_{i=j}^n \binom{r+i-j-1}{r-1}v_i = 0 \qt{for $1 \leq j
      \leq r$,}
  \end{equation}
and 
\begin{equation}\label{ch2}
    \sum_{i=j}^n \binom{r+i-j-1}{r-1}v_i = \left\{
  \begin{array}{ll}
\sgn((D^{(r)}\theta)_{j-r}) &  \text{if $(D^{(r)}\theta)_{j-r}
  \neq 0$}\\
    \in [-1, 1] &  \text{otherwise}
  \end{array}
\right.
\end{equation}
for $r < j \leq n$. Here $\sgn(x)$ denotes the sign of $x$ for $x \neq
0$. 
\end{proposition}

It should be clear from the above proposition that $\partial
f(\theta^*)$ is always a convex polyhedron and is of a different
nature when $D^{(r)} \theta^* \neq 0$ 
as opposed to when $D^{(r)} \theta^* = 0$. For example, when $D^{(r)}
\theta^* = 0$, the zero vector belongs to $\partial f(\theta^*)$ and
moreover, the sets $\lambda \partial f(\theta^*) := \{\lambda v : v
\in \partial f(\theta^*)\}$ are increasing as $\lambda$
increases. Both these facts are not true when $D^{(r)} \theta^* \neq
0$. We thus separate our risk results into the two cases:
$D^{(r)} \theta^* \neq 0$ and $D^{(r)} \theta^* = 0$. First we deal
with the case $D^{(r)} \theta^* \neq 0$. The other (simpler) case is in Lemma \ref{rr0}. 

Assume therefore that $D^{(r)} \theta^* \neq 0$. The following
quantities (all defined in terms of the subdifferential $\partial
f(\theta^*)$) will play a key role in our risk bounds for the
penalized estimator \eqref{ptf}. Let
\begin{equation}\label{vvs}
  v^* := \argmin_{v \in \partial f(\theta^*)} \|v\| ~~ \text{ and }
  ~~ v_0 := \argmin_{v \in \aff(\partial f(\theta^*))} \|v\|  
\end{equation}
where $\aff(\partial f(\theta^*))$ denotes the affine hull of
$\partial f(\theta^*)$ (recall that for a subset $S \subseteq \R^n$,
its affine hull $\aff(S)$ consists of all vectors $w_1 x_1 + \dots +
w_m x_m$ such that $m \geq 1$, $x_i \in S$ and $w_1 + \dots + w_m =
1$). Note that 
$v^*$ and $v_0$ are uniquely defined because they are simply the
projections of the zero vector onto the closed convex sets $\partial
f(\theta^*)$ and $\aff(\partial f(\theta^*))$ respectively. Moreover,
they are both non-zero vectors because every vector $v$ in $\partial
f(\theta^*)$ (and consequently $\aff(\partial f(\theta^*))$) is
non-zero as it satisfies
\begin{equation*}
  \sum_{i=j}^n {r + i - j - 1 \choose r-1} v_i = \sgn((D^{(r)}
    \theta^*)_{j-r}) 
\end{equation*}
whenever $(D^{(r)} \theta^*)_{j-r} \neq 0$ (it should be kept in mind
that we are working under the assumption that $D^{(r)} \theta^* \neq
0$). It is helpful to note here that $v_0 = v^*$ when $r = 1$
(see Lemma \ref{gksd}) but for $r \geq 2$, they are not
necessarily the same.    

In addition to $v^*$ and $v_0$, we need the following quantity: 
\begin{equation}\label{lamzdef}
  \lambda_{\theta^*}(z) := \argmin_{\lambda \geq 0} \inf_{v
    \in \partial f(\theta^*)} \|z - \lambda v\| \qt{for $z \in
    \R^n$}. 
\end{equation}
In words, $\lambda_{\theta^*}(z)$ is the value of $\lambda$ which
minimizes the distance of the vector $z$ from the set
$\lambda \partial f(\theta^*)$. Lemma \ref{qfmg} proves that
$\lambda_{\theta^*}(z)$ is uniquely 
defined for each $z \in \R^n$ (under the assumption that $D^{(r)}
\theta^* \neq 0$) and also that $\E \lambda_{\theta^*}(Z) < \infty$
where the expectation is taken with respect to $Z \sim N_n(0,
I_n)$. We are now ready to state our first result on the risk of the
penalized trend filtering estimators (recall $\Delta_r(\theta)$ from
\eqref{mrt}).       
\begin{theorem}\label{prt1}
 Fix $r \geq 1$ and suppose $\theta^* \in \R^n$ with $D^{(r)} \theta^*
 \neq 0$. Let 
    \begin{equation}\label{lams}
    \lambda^* := n^{1-r} \left(\E \lambda_{\theta^*}(Z) +
      \frac{2}{\|v_0\|} \right)
  \end{equation}  
   where the expectation is taken with respect to the standard
   Gaussian vector $Z \sim N_n(0, I_n)$.  Then for every regularization
   parameter $\lambda \geq \lambda^*$, we have  
   \begin{equation}
     \label{prt1.eq}
     R(\hat{\theta}_{\lambda}^{(r)}, \theta^*) \leq C_r \sigma^2
     \Delta_r(\theta^*) + \frac{64\sigma^2}{n}
     \frac{\|v^*\|^2}{\|v_0\|^2} + \frac{4 \sigma^2}{n^{3 - 2r}}
     (\lambda - 
     \lambda^*)^2 \|v^*\|^2
   \end{equation}
   for a constant $C_r$ that only depends on $r$. 
\end{theorem}
The bound \eqref{prt1.eq} (which holds for every $\lambda \geq
\lambda^*$) is clearly smallest when $\lambda = \lambda^*$. To
simplify the right hand side of \eqref{prt1.eq} further, we need to
bound $\|v^*\|$ from above and $\|v_0\|$ from below. This is done in
the next result. 

\begin{lemma}\label{gksd}
Let $f : \R^n \rightarrow \R$ be given by $f(\theta) := \|D^{(r)}
\theta\|_1$ and let $\theta^* \in \R^n$ be such that $D^{(r)} \theta^*
\neq 0$. 
  \begin{enumerate}
  \item Suppose $r = 1$. Then $v_0 = v^*$. Further suppose that $\theta^*$
    has $k \geq 1$ jumps (first order knots) with signs $\rs_1, \dots,
    \rs_k$ and let $n_0, n_1, \dots, n_k$ denote the lengths of the 
    constant pieces of $\theta^*$. Then  
    \begin{equation}
      \label{voex}
      \|v_0\|^2 = \|v^*\|^2 =  \frac{1}{n_0} + \frac{1}{n_k} + 4 \sum_{i=1}^{k-1}
      \frac{I\{\rs_i \neq \rs_{i+1}\}}{n_i}. 
    \end{equation}
  \item For $r \geq 2$, we have
    \begin{equation}\label{gksd.low}
  \|v_0\| \geq \frac{(r -1)!}{(r+1) 2^{r-1}} n^{-r + 1/2}.
    \end{equation}
  \item Suppose $r \ge 2$ and $\theta^*$ satisfies the minimum length
    condition \eqref{nsa} with constant $c$, then
    \begin{equation}\label{gksd.up}
      \|v^*\| \leq C_r c^{-r + 1/2} (k+1)^r n^{-r + 1/2} 
    \end{equation}
   where $C_r$ is a constant depending only on $r$. 
  \end{enumerate}
\end{lemma}

We shall now present more explicit risk bounds by combining Theorem
\ref{prt1} and Lemma \ref{gksd}. Since the information provided by
Lemma \ref{gksd} about $\|v_0\|$ and $\|v^*\|$ is much more precise
for $r = 1$ compared to $r \geq 2$, we find it natural to state our
risk results separately in the two cases $r = 1$ and $r \geq 2$. The
following result deals with the $r = 1$ case. 
\begin{corollary}\label{expen1}
  Suppose $\theta^* \in \R^n$ has $k \geq 1$ jumps with signs $\rs_1,
  \dots, \rs_k$ and suppose that $n_0, n_1, \dots, n_k$ denote the
  lengths of the constant pieces of $\theta^*$. Then, with $\lambda^*$
  as in \eqref{lams}, we have 
  \begin{equation} \label{expen1.eq}
    R(\hat{\theta}_{\lambda}^{(1)}, \theta^*) \leq C \sigma^2
    \left(\Delta_1(\theta^*) + \frac{(\lambda - \lambda^*)^2}{n}
      \sum_{i=0}^k \frac{I\{\rs_i \neq \rs_{i+1}\}}{n_i} \right) 
  \end{equation}
   for every $\lambda \geq \lambda^*$. Here $C$ is a universal
   constant. Also, we use our usual convention
   $\rs_0 = \rs_{k+1} = 0$. 

   Further, if $\theta^*$ satisfies the minimum length condition
   \eqref{nsa} with constant $c$, then   
    \begin{equation}\label{expen1.seq}
     R(\hat{\theta}_{\lambda}^{(1)}, \theta^*) \leq C(c) \sigma^2
     \left(\frac{k+1}{n} \log \frac{en}{k+1} + (\lambda - \lambda^*)^2 
     \frac{k+1}{n^2} \sum_{i=0}^k I\{\rs_i \neq \rs_{i+1}\}\right)
   \end{equation}
    where $C(c)$ depends on $c$ alone. 
\end{corollary}

Inequality \eqref{expen1.seq} implies that, under the minimum length
condition, we have 
\begin{equation}  \label{prelim}
  R(\hat{\theta}^{(1)}_{\lambda}, \theta^*)
      \leq C(c) \sigma^2 \frac{k + 1}{n} \log
      \frac{en}{k + 1} \qt{for $\lambda = \lambda^*$}
\end{equation}
where $k$ is the number of jumps of $\theta^*$, i.e.,
$k = \mathbf{k}_1(\theta^*)$. Moreover, the logarithmic term above
cannot be removed in general. This is due to the following
reason. First, note that, for every non-random
$\lambda$ possibly depending on $\lambda^*$, the penalized estimator
$\hat{\theta}^{(1)}_{\lambda}$ has worse risk compared to the ideally
tuned constrained estimator i.e., $\hat{\theta}^{(1)}_{V}$ with $V =
V^{(r)}(\theta^*)$. This fact (which is noted and explained in
Subsection \ref{conpene}), together with Lemma \ref{simpo}, implies
clearly that the logarithmic factor in \eqref{prelim} cannot be
removed in general. 

\begin{remark}[Comparison to existing results]\label{comp} 
Among the class of existing results for the risk of
$\hat{\theta}^{(1)}_{\lambda}$, the strongest (in terms of giving the
smallest bound on the risk) is due to 
\citet{lin2016approximate} who proved that, when
$\lambda$ is appropriately selected (depending on $\theta^*$),
$\hat{\theta}^{(1)}_{\lambda}$ satisfies: 
\begin{equation}\label{tibb}
  R(\hat{\theta}_{\lambda}^{(1)}, \theta^*)  \leq C
  \frac{\sigma^2(k+1)}{n} \left(\left[\log
      (k+1) + \log \log n \right]  
    \log n + \sqrt{k+1} \right)  
\end{equation}
provided 
\begin{equation}\label{wn1}
  \min_{0 \leq i \leq k} n_i \geq \frac{cn}{k+1} 
\end{equation}
for a positive constant $c$. Here $n_0, \dots, n_k$ are
the lengths of the constant pieces of $\theta^*$. This bound from 
\citet{lin2016approximate} is smaller compared to an earlier result of
\citet{dalalyan2017tvd} and to a very recent result of
\citet{ortelli2018total} (although the results of
\cite{dalalyan2017tvd, ortelli2018total} apply to a universal choice
of the tuning parameter $\lambda$; see Remark \ref{compa}). The bound
\eqref{tibb} is weaker than  \eqref{prelim}  in two
respects: (a) there are additional terms in 
\eqref{tibb} involving 
$\log n$ and $k$ compared to \eqref{prelim},
and (b) our minimum length condition \eqref{nsa} is weaker than
\eqref{wn1}: \eqref{nsa} requires that $n_i \geq
cn/(k+1)$ only for those $i$ for which $\rs_i
\neq \rs_{i+1}$ while \eqref{wn1} requires this for all
$i$.   
\end{remark} 

Note that the regularization parameter $\lambda^*$ (for which the near
parametric risk bound \eqref{prelim} holds) depends on
$\theta^*$. Further, the exact nature of its dependence on $\theta^*$
is not apparent from its definition \eqref{lams}. In the next result,
we provide a more explicit upper bound for $\lambda^*$. For this, we
require a stronger length condition than \eqref{nsa}. Note that we are
still in the $r = 1$ case.  
\begin{lemma}\label{explam1}
   Consider the same setting as in Corollary \ref{expen1}. Assume that the
   length condition: 
  \begin{equation}\label{leon}
    \min_{0 \leq i \leq k: \rs_i \neq \rs_{i+1}} n_i \geq
    \frac{c_1n}{k+1} ~~ \text{ and } ~~ \max_{0 \leq i \leq k: \rs_i
      \neq \rs_{i+1}} n_i \leq  \frac{c_2n}{k+1}
  \end{equation}
 holds for two positive constants $c_1 \leq 1$ and $c_2 \geq 1$. Let
 $\lambda^*$ 
 be as defined in \eqref{lams}. Then there exists a positive constant 
  $C^*(c_1, c_2)$ (which depends only on $c_1$ and $c_2$) such that  
  \begin{equation}\label{rambo}
    \lambda^* \leq C^*(c_1, c_2) \sqrt{\frac{n}{\sum_{i=0}^k I\{\rs_i
        \neq \rs_{i+1}\}} \log \left(\frac{en}{k+1} \right)}.  
  \end{equation}
\end{lemma}
Lemma \ref{explam1} can be used, in conjunction with the risk bound
\eqref{expen1.seq} (which holds for every $\lambda \geq \lambda^*$) to
yield the following result which provides bounds similar to
\eqref{prelim} for explicit choices of $\lambda$. 

\begin{corollary}\label{expen}
 Consider the same setting as in Lemma \ref{explam1} and assume the
 length condition \eqref{leon}. Then if the regularization parameter
 $\lambda$ satisfies
\begin{equation}\label{lach}
    \lambda = \Gamma \sqrt{\frac{n}{\sum_{i=0}^k I\{\rs_i \neq
        \rs_{i+1}\}} \left(\log \frac{en}{k+1} \right)}, 
  \end{equation}
we have 
  \begin{equation}\label{jku}
    R(\hat{\theta}_{\lambda}^{(1)}, \theta^*) \leq C(c_1) \sigma^2 (1 +
    \Gamma^2)  \frac{k+1}{n} \log \frac{en}{k+1}
  \end{equation}
for every $\Gamma \geq C^*(c_1, c_2)$ (where $C^*(c_1, c_2)$ is the
constant given by Lemma \ref{explam1}). Also $C(c_1)$ depends
only on $c_1$. 

Also, if the regularization parameter $\lambda$ satisfies
    \begin{equation}\label{cdal}
      \lambda = \Gamma \sqrt{n \log (en)}, 
    \end{equation}
we have 
    \begin{equation}\label{cbal}
      R(\hat{\theta}_{\lambda}^{(1)}, \theta^*) \leq C(c_1)
      \frac{\sigma^2(k+1) (\log (en))}{n} \left(1 + \Gamma^2
\sum_{i=0}^k I\{\rs_i \neq \rs_{i+1}\} \right)
    \end{equation}
for every $\Gamma \geq C^*(c_1, c_2)$. 
\end{corollary}
In the bound \eqref{cbal}, the term $\sum_{i=0}^k I\{\rs_i \neq
\rs_{i+1}\}$ can be further bounded by its maximum possible value of
$k+1$ . However in certain instances (such as when
$\theta^*$ is monotone), $\sum_{i=0}^k I\{\rs_i \neq
\rs_{i+1}\}$ can be much smaller than $k+1$. 

\begin{remark}[Comparison to existing results]\label{compa} 
  We now compare Corollary \ref{expen} to existing results for the
  penalized estimator in \citet{lin2016approximate},
  \citet{dalalyan2017tvd} and \citet{ortelli2018total}. Note first
  that the choice \eqref{lach} of 
  $\lambda$ depends on certain aspects of $\theta^*$: in particular,
  it depends on $k$, $\sum_{i=0}^k I\{\rs_i \neq \rs_{i+1}\}$ and the
  values $c_1$ and $c_2$ in the length condition \eqref{leon}.  The
  bound \eqref{tibb} of \citet{lin2016approximate} holds for $ 
  \lambda_1 = \left(n \min_{0 \leq i \leq k} n_i \right)^{1/4}$ which
  also depends on the true vector $\theta^*$ through the 
   lengths $n_1, \dots, n_k$. If we assume that each $n_i$ is of order
   $n/(k+1)$, then
   \begin{equation}\label{lamlin}
     \lambda_1 \sim \sqrt{\frac{n}{\sqrt{k+1}}}. 
   \end{equation}
    Note that the leading term in our choice \eqref{lach} of $\lambda$ as well
    as in $\lambda_1$ is $\sqrt{n}$. Corollary \ref{expen}  also applies to the choice \eqref{cdal} for
    which the bound \eqref{cbal} holds. Note that \eqref{cdal}
    has considerably less dependence on $\theta^*$ as it only depends on
    the constants $c_1$ and $c_2$ appearing in the length
    condition \eqref{leon}. On the other hand, the bound \eqref{cbal} is weaker
    compared to \eqref{jku}. However, \eqref{cbal} needs to be
    compared to the results of \citet[Proposition
    3]{dalalyan2017tvd} and \citet[Corollary
    4.4]{ortelli2018total}. Indeed, \citet{dalalyan2017tvd} considered
    the  choice  
     \begin{equation}\label{lamdal}
       \lambda_2 := 2 \sqrt{2n \log (n/\delta)} 
     \end{equation}
     and proved that the following loss bound holds with probability at
     least $1 - \delta$:
   \begin{equation}\label{daldal}
       \frac{1}{n} \|\hat{\theta}_{\lambda}^{(1)} - \theta^*\|^2
       \leq C(c_1)  \left(\frac{(k+1)^2}{n} \log \frac{en}{\delta} + \frac{k+1}{n} \log
       (en) \log \frac{en}{\delta} \right). 
     \end{equation}
     This result has been improved slightly in the very recent paper
     \citet{ortelli2018total} (see also \citet{van2018tight}) where
     the $\log (en) \log (en/\delta)$ term in the right hand side above is
       replaced by $\log(en/(k+1)) \log(en/\delta)$ (i.e., one of the
       $\log (en)$ terms is relaced by $\log(en/(k+1))$). An
       expectation (risk) bound has not been proved in these two
       papers. Note the the choice of $\lambda$ in 
     \eqref{cdal} is similar to that of $\lambda_2$ in \eqref{lamdal}
     although our choice needs $\Gamma$ to be 
     sufficiently large while the choice $\lambda_2$ is universal
     (although it depends on $\delta$). On the other hand, the high
     probability  bound implied by \eqref{cbal} is (see Remark
     \ref{exphp}) the  statement that 
     \begin{align*}
       \frac{1}{n} \|\hat{\theta}_{\lambda}^{(1)} - \theta^*\|^2
       &\leq  C(c_1)
      \frac{\sigma^2(k+1) (\log (en))}{n} \left(1 + \Gamma^2
\sum_{i=0}^k I\{\rs_i \neq \rs_{i+1}\} \right) \\ &+ \frac{4 \sigma^2}{n} \log(\delta^{-1}) 
     \end{align*}
     holds with probability at least $1 - \delta$. This  is
     stronger compared to \eqref{daldal} because the right hand side
     of \eqref{daldal} has a $\log (en) \log(en/\delta) \geq (\log
     (en))^2$ term.  

     We reiterate here that our length condition
     \eqref{leon} involves an upper
     bound on $n_i$ for $\rs_i \neq \rs_{i+1}$. From an examination of
     the proof of Lemma \ref{explam1}, it will be clear that we will
     obtain a weaker upper bound for $\lambda^*$ in the sense of
     having additional multiplicative factors involving $k$ if this
     upper bound assumption on $n_i$ is removed. No such upper
     bound is needed for the results in \citet{lin2016approximate,
       dalalyan2017tvd, ortelli2018total}. On the other hand, our
     lower bound (and our upper bound in \eqref{leon}) involves only
     those $i$ satisfying $\rs_i \neq \rs_{i+1}$ while the assumptions
     in these earlier papers required a lower bound on every $n_i$.   
\end{remark}

We now state our risk results for \eqref{ptf} with $r \geq 2$ when
$D^{(r)} \theta^* \neq 0$. The following result is obtained by combining
Theorem \ref{prt1} and Lemma \ref{gksd}. 
\begin{corollary}\label{expenr}
  Fix $r \geq 2$. Suppose $D^{(r)} \theta^* \neq 0$ and $\theta^*$
  satisfies the minimum length condition \eqref{nsa} with constant
  $c$. Then, with $\lambda^*$ as in \eqref{lams}, we have 
  \begin{equation}\label{expenr.eq}
\begin{split}
    R(\hat{\theta}^{(r)}_{\lambda}, \theta^*) &\leq C_r(c) \sigma^2
    \left(\frac{k+1}{n} \log \frac{en}{k+1} + \frac{(k+1)^{2r}}{n} \right. \\
 &+ \left.
      (\lambda - \lambda^*)^2\frac{(k+1)^{2r}}{n^{2}} \right)
\end{split}
  \end{equation}
  for every $\lambda \geq \lambda^*$. Here $k :=
  \mathbf{k}_r(\theta^*)$ and $C_r(c)$ depends only on $c$.  
\end{corollary}

Corollary \ref{expenr} implies that when $\theta^*$ satisfies the
minimum length condition \eqref{nsa}, then (with $k = \mathbf{k}_r(\theta^*)$)
\begin{equation}
  \label{prelimr}
    R(\hat{\theta}^{(r)}_{\lambda}, \theta^*) \leq C_r(c) \sigma^2
    \left(\frac{k+1}{n} \log \frac{en}{k+1} + \frac{(k+1)^{2r}}{n}
    \right)   \qt{for $\lambda = \lambda^*$}. 
\end{equation}
It may be noted that the above result is weaker than our 
  corresponding risk bound for the constrained trend filtering
  estimator (Corollary \ref{kco}) because of the additional term
  involving $(k + 1)^{2r}$. We believe that this term is redundant and
  is an artifact of our proof. Specifically, this additional term
  comes from the fact that our upper bound for $\|v^*\|$ and lower
  bound for $\|v_0\|$ in Lemma \ref{gksd} are off by a factor of
  $(k+1)^r$. 

With the aim of providing an explicit value for $\lambda$ for which
the bound \eqref{prelimr} holds, the next result gives an upper bound
for $\lambda^*$. As in the case of Lemma \ref{explam1}, we need a
stronger length condition (compared to \eqref{nsa}) for this result. 

\begin{lemma}\label{explamr}
   Fix $r \geq 2$. Suppose $D^{(r)} \theta^* \neq 0$ and $\theta^*$
   satisfies the length condition:  
  \begin{equation}\label{leonr}
    \min_{0 \leq i \leq k: \rs_i \neq \rs_{i+1}} n_i \geq
    \frac{c_1n}{k+1} ~~ \text{ and } ~~ \max_{0 \leq i \leq k: \rs_i
      \neq \rs_{i+1}} n_i \leq  \frac{c_2n}{k+1}
  \end{equation}
for two positive constants $c_1 \leq 1$ and $c_2 \geq 1$. Here $n_0,
\dots, n_k$ have the same meaning as in \eqref{nsa}. Then $\lambda^*$
(defined as in \eqref{lams}) satisfies
\begin{equation}
  \label{explamr.eq}
  \lambda^* \leq C_r^*(c_1, c_2) \sqrt{n \log \left(\frac{en}{k+1} \right)}
\end{equation}
where $C_r^*(c_1, c_2)$ depends on $r$, $c_1$ and $c_2$ alone.  
\end{lemma}
Note that even though \eqref{leonr} and \eqref{leon} look exactly the
same, the difference is that \eqref{leon} applies to $r = 1$ while
\eqref{leonr} applies to $r = 2$. The meaning of $n_0, \dots, n_k$
depends on $r$. Indeed, the $n_i$'s refer to the lengths of the
constant pieces for $r = 1$, the lengths of the linear pieces for $r =
2$, etc.  

%{\color{red} Address the fact that (48) and (37) are exactly the same display.}

Compared to \eqref{rambo}, the bound
\eqref{explamr.eq} is weaker because there is no $\sum_{i=0}^k
I\{\rs_i \neq \rs_{i+1}\}$ in the denominator in \eqref{explamr.eq}.  

Combining Lemma \ref{explamr} with the risk bound \eqref{expenr.eq},
we obtain the following result which provides bounds similar to
\eqref{prelimr} for explicit choices of $\lambda$. 

\begin{corollary}\label{supno}
   Consider the same setting as in Lemma \ref{explamr}  and assume the
   length condition \eqref{leonr}. Then if the regularization
   parameter satisfies 
   \begin{equation}\label{lachr}
     \lambda = \Gamma \sqrt{n \log \left(\frac{en}{k+1} \right)}, 
   \end{equation}
    we have 
  \begin{equation}\label{jkur}
    R(\hat{\theta}_{\lambda}^{(r)}, \theta^*) \leq C_r(c_1) \sigma^2 (2 +
    \Gamma^2)  \frac{(k+1)^{2r}}{n} \log \frac{en}{k+1}
  \end{equation}
for every $\Gamma \geq C_r^*(c_1, c_2)$ (where $C_r^*(c_1, c_2)$ is the
constant given by Lemma \ref{explam1}). Also $C_r(c_1)$ only depends
only on $r$ and $c_1$. 

Further, if the regularization parameter $\lambda$ satisfies
    \begin{equation}\label{cdalr}
      \lambda = \Gamma \sqrt{n \log (en)}, 
    \end{equation}
we have 
    \begin{equation}\label{cbalr}
      R(\hat{\theta}_{\lambda}^{(r)}, \theta^*) \leq C_r(c_1) \sigma^2
      \left(2 + \Gamma^2 \right) \frac{(k+1)^{2r}}{n} \log (en)
    \end{equation}
for every $\Gamma \geq C_r^*(c_1, c_2)$. 
\end{corollary}

Finally we deal with the risk of the penalized estimator when $D^{(r)}
\theta^* = 0$. Here we have the following result which proves that the
risk is parametric (without any logarithmic factors) as long as the
tuning parameter $\lambda$ is larger than or equal to $\sqrt{6n
  \log(en)}$. This result holds for every $r \geq 1$. 
\begin{lemma}\label{rr0}
  Suppose $D^{(r)} \theta^* = 0$. Then for every $\lambda \geq \sqrt{6
    n \log (en)}$,  we have
  \begin{equation*}\label{rr0.eq}
      R(\hat{\theta}_{\lambda}^{(r)}, \theta^*) \leq \frac{C_r\sigma^2}{n}. 
  \end{equation*}
   for a constant $C_r$ that depends on $r$ alone. 
\end{lemma}

\section{Proof Ideas}\label{proids}
In this section, we provide a brief overview of the main ideas
underlying our proofs. Full proofs are in
the supplementary material at the end of the paper. For studying the constrained
trend filtering estimator $\hat{\theta}^{(r)}_V$, we invoke the
general theory of convex-constrained least squares
estimators. Convex-constrained least squares estimators are estimators
of the form 
\begin{equation*}
  \hat{\theta} := \argmin_{\theta \in \R^n} \left\{\frac{1}{2} \|Y -
    \theta\|^2 : \theta \in K \right\}. 
\end{equation*}
for a closed convex set $K$. Clearly $\hat{\theta}^{(r)}_V$ is a special case of
this estimator when $K$ is taken to be the set $K^{(r)}(V)$ defined as 
\begin{equation*}
  K^{(r)}(V) := \left\{\theta \in \R^n : \|D^{(r)} \theta\|_1 \leq V
    n^{1- r} \right\} . 
\end{equation*}
The general theory of convex-constrained least squares estimators 
(summarized in Section \ref{prems}) states that the accuracy of
$\hat{\theta}^{(r)}_V$ as an  
estimator for $\theta^*$ under the model $Y \sim N_n(\theta^*,
\sigma^2 I_n)$ can be deduced from bounds on the quantity:  
\begin{equation}\label{omgs}
  \E \sup_{\theta \in K^{(r)}_V : \|\theta - \theta^*\| \leq t}
  \left<\xi, \theta -\theta^* \right> 
\end{equation}
where $\xi \sim N_n(0, \sigma^2 I_n)$. To prove Theorem \ref{woor}, we
prove bounds on \eqref{omgs} in Lemma
\ref{svg}. Our strategy involves using Dudley's 
entropy bound to 
control \eqref{omgs} in terms of the metric entropy of the set:
\begin{equation*}
S_r(V, t) :=  \left\{\alpha \in \R^n : \|\alpha\| \leq t, \|D^{(r)}
  \alpha\|_1  \leq V n^{1-r}  \right\}.   
\end{equation*}
We then bound the metric entropy of $S_r(V, t)$ via its fat-shattering
dimension (it is well known that fat-shattering
dimension can be used to control metric entropy; see e.g.,
\citet{rudelson2006combinatorics}). Metric entropy and fat-shattering
dimension are formally defined in Subsection \ref{gwc1} and Subsection
\ref{ppgr} respectively. Our idea of using fat 
shattering to establish the metric entropy of $S_r(V, t)$ and thereby
bounding \eqref{omgs} seems novel. Previous bounds on quantities
similar to \eqref{omgs} in the context of trend filtering used
eigenvector incoherence (see, for example, \citet{wang2016trend}) and
the ideas here are quite different from our methods.   

To prove the strong sparsity risk bound, Theorem \ref{ada}, we use
another strand of results from the general theory of
convex-constrained least squares estimators. Specifically, a result
from \citet{oymak2013sharp} implies that the risk of
$\hat{\theta}^{(r)}_V$ at $V = V^* := V^{(r)}(\theta^*)$ can be
obtained by controlling the  \textit{Gaussian width} of the
\textit{tangent cone} of the convex set $K^{(r)}(V^*)$ at
$\theta^*$. These  
general results, along with the definitions of tangent cones and
Gaussian width, are again recalled in Subsection
\ref{prems}. Understanding the tangent cone to $K^{(r)}(V^*)$
at $\theta^*$ then becomes key to proving Theorem \ref{ada}.  

We provide a precise characterization of the tangent
cones of $K^{(r)}(V^*)$ in Lemma \ref{tcr}. These tangent cones
have a complicated structure (especially for $r \ge 2$) and
calculating their Gaussian width is non-trivial. Our idea behind these
calculations is the fact (proved in Lemma \ref{alta}) that,
under a unit norm constraint, 
every vector $\alpha$ in the tangent cone of $K^{(r)}(V^*)$ at
$\theta^*$ is nearly made up of two $(r-1)^{th}$ order convex/concave
sequences in each polynomial part of $\theta^*$ (note that a sequence
$\theta \in \R^n$ is said to be $(r-1)^{th}$ 
order convex/concave if the vector $D^{(r-1)} \theta$ is monotone;
see e.g., \citet{kuczma2009introduction}). The special case of
this observation for $r = 1$ implies that every vector $\alpha$ with
$\|\alpha\| \leq 1$ in the tangent cone to $K^{(1)}(V^*)$ at $\theta^*$
is nearly made up of two monotonic sequences in each constant piece of
$\theta^*$. For $r = 2$, it means that every vector $\alpha$ with
$\|\alpha\| \leq 1$ in the tangent cone to $K^{(2)}(V^*)$ at $\theta^*$
is nearly made up of two convex/concave sequences in each linear piece
of $\theta^*$.  

The above observation allows us to compute the Gaussian width of these
tangent cones using metric entropy results (established again via
connections between metric entropy and fat shattering) and also
available results (from \citet{bellec2015annals}) on the Gaussian widths of shape
constrained cones. The set of all $(r-1)^{th}$ order convex
sequences in $\R^n$ forms a convex cone in $\R^n$ and these cones have
been studied in the literature on shape constrained estimation.

For $r = 1$, the above idea bears strong similarities with the method
employed in \citet{lin2016approximate} for studying the penalized
estimator \eqref{ptf} for $r = 1$. In this paper, they use the key
observation that for appropriate $\lambda$, the vector $(I -
P_0)(\hat{\theta}^{(1)}_{\lambda} - \theta^*)$ is well-approximated by
a vector which is made of two monotonic 
sequences in each constant piece of $\theta^*$. Here $P_0$ is the
projection matrix onto the piecewise constant structure determined by
$\theta^*$ and $I$ is the identity matrix. This idea is similar
in spirit to our observation on the tangent cone of $K^{(1)}(V^*)$ at
$\theta^*$. The details differ though as we are working with the
vectors in the tangent cone while \citet{lin2016approximate} focus on
a functional of $\hat{\theta}^{(1)}_{\lambda} - 
\theta^*$ (note though that if $\hat{\theta}$ has variation $\leq V^*$,
then $\hat{\theta}- \theta^*$ does indeed belong to the tangent
cone). Also our method for dealing with the
Gaussian width of the set of these piecewise monotonic vectors is
sharper than the analysis of \citet{lin2016approximate} and our
analysis also extends to every $r \geq 2$. 

The results in Subsection \ref{repe} for the penalized  estimator are
all based on \eqref{ohag}. We use the precise characterization of the
subdifferential of the penalty function $\theta \mapsto \|D^{(r)}
\theta\|_1$ given in Proposition
\ref{characterization} to control the right side of
\eqref{ohag}. Our idea here is to relate the right side of
\eqref{ohag}   to the risk of the constrained
estimator (we use and extend ideas from \citet{foygel2014corrupted}
for this). This allows us to derive risk results for the
penalized trend filtering estimator as a corollary to our results for
the constrained estimator.    

\section{Simulations}\label{simu}
In this section, we present numerical evidence for
our theoretical results. We generate data from a
piecewise constant function $f_1^*$ and a continuous piecewise affine
function $f_2^*$ on $[0, 1]$ and evaluate the performance of the trend
filtering estimators for $r = 1$ (total variation denoising) and $r =
2$ respectively. The functions $f_1^*$ and $f_2^*$ (see Figure \ref{fcnsr1}) are  
given by {$$f_1^*(x) := 2 I_{(0.2, 0.4]}(x) 
 + 4 I_{(0.4, 0.6]} (x) + I_{(0.6, 0.8]}(x) + 4
I_{(0.8, 1]}(x)$$} and $$f_2^*(x) := -44 \max(x - 0.25, 0) + 48
\max(x - 0.5, 0) - 56 \max(x 
  - 0.75, 0) + 28 x.$$ The function $f_1^*$ was used in the simulation study of
\citet{lin2016approximate}. In addition to these functions, we also
performed a simulation study on another piecewise constant function
$f_3^*$ which is similar to the blocks function of
\citet{donoho1994ideal}; results for $f_3^*$ are in Section
\ref{adsims}.  

\begin{figure}
\begin{center}
 \includegraphics[width=\textwidth, height = 2in]{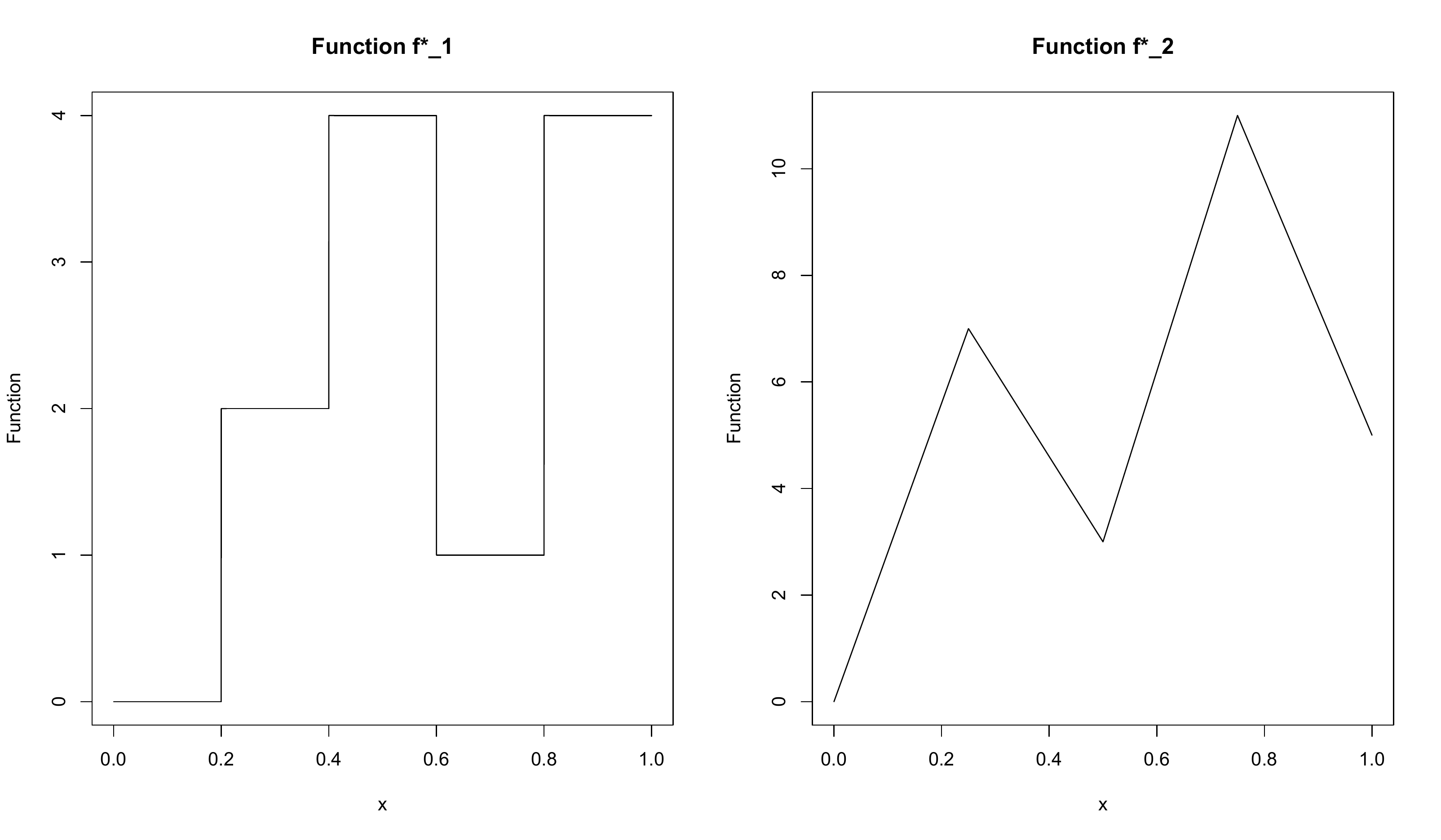}    
\caption{The two functions $f_1^*$ and $f_2^*$.}  
\label{fcnsr1}
\end{center}
\end{figure}
From $f_1^*$ and a value of $n$ (chosen from a grid of size 30 between 100 and 10000; the grid being equally spaced on the
logarithmic scale) we generated an $n \times 1$ observation vector $Y \sim N_n(\theta^*, I_n)$ where $\theta^*$ is the vector obtained by
sampling $f_1^*$ at $n$ equally spaced points with end-points 0 and
1. We then computed the following six estimators on the 
data vector $Y$: (a) the ideal constrained estimator \eqref{tf} with
$V = V^* = \|D \theta^*\|_1$, (b) the ideal penalized estimator
\eqref{ptf} with $\lambda = \lambda^*$ (as defined in
\eqref{lams}), (c) two cross-validation (CV) based estimators, (d) the
penalized estimator \eqref{ptf} with $\lambda$ of the form
\eqref{lach} with $\Gamma = 1$, and (e) the penalized estimator
\eqref{ptf} with $\lambda$  of the form \eqref{cdal} with
$\Gamma = 0.5$. Corollary
\ref{expen} proves that the risk with these $\lambda$ choices decays
as $(\log n)/n$ (ignoring terms involving $k$) provided $\Gamma$ is
taken to be a  large enough constant. In our simulations for $f_1^*$, we found
that $\Gamma = 1$ in \eqref{lach} and $\Gamma = 0.5$ in \eqref{cdal} were large enough to yield the desired performance. Higher
values of $\Gamma$ led to similar rates of decay of the risk with $n$ (even though the risk itself seemed to become larger with $\Gamma$). 

Here are some details behind the computation of these estimates. The constrained estimator was computed by the convex optimization
software MOSEK (via the R package \texttt{Rmosek}). The penalized
estimators were computed via the R package \texttt{tvd} for total
variation denoising. The computation of the ideal penalized estimator
requires computing the value of $\lambda^*$ and, for this, we need to
compute $\E \lambda_{\theta^*}(Z)$ (where $Z \sim N_n(0, I_n)$) and
$2/\|v_0\|$ (see \eqref{lams}). $2/\|v_0\|$ was calculated by the
formula \eqref{voex}. For $\E \lambda_{\theta^*}(Z)$, we used the fact
that $\lambda_{\theta^*}(z)$ can be calculated by convex optimization
for each $z \in \R^n$ which implies that the expectation can be
computed by Monte-Carlo averaging. More details behind this are
provided in Section \ref{adsims}. The CV estimators were calculated using the R package
\texttt{genlasso} which provides two penalized estimates based on CV: 
one based on choosing $\lambda$ so as to minimize the CV error
($CV_1$) and the other based on choosing $\lambda$ via the one
standard error rule ($CV_2$).   

\begin{figure}[h!]
\begin{center}
 \includegraphics[width = \textwidth]{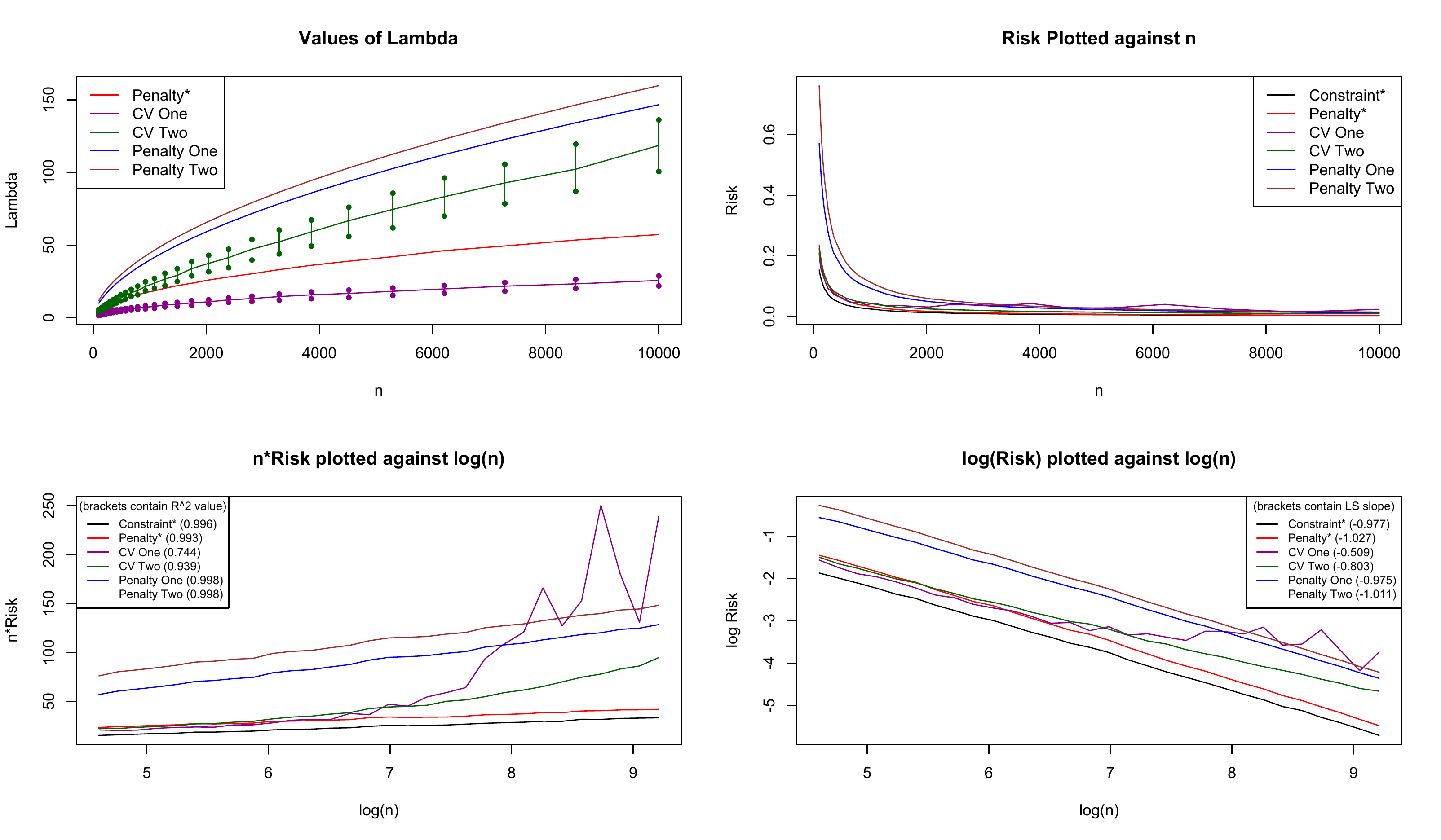}    
\caption{Plots when the true function is $f_1^*$. The top-left plot
  shows the $\lambda^*$ values, the CV $\lambda$ values (median and
  the first and third quartiles over 200 replications) and the values
  corresponding to the explicit penalties \eqref{lach} with $\Gamma =
  1$ and \eqref{cdal} with $\Gamma = 0.5$. The other three figures
  show the behavior of the risk as a function of $n$. In the last two
  plots, the legend shows the value of $R^2$ and the slope respectively for the curves corresponding to each estimator.}   
\label{f1p}
\end{center}
\end{figure}

For each data set, we computed the value of the loss $\|\hat{\theta} -
\theta^*\|^2/n$ for each of these six estimates. We generated 600
replications of the data for each value of $n$ to compute the average
value of the loss which is an approximation of the risk of each
estimator. Our results are provided in Figure \ref{f1p}. The top-left
plot shows the different values of $\lambda$ employed by the
estimators based on \eqref{ptf}. Here we plotted the $\lambda^*$
values as well as those corresponding to \eqref{lach} with $\Gamma =
1$ (penalty one) and \eqref{cdal} with $\Gamma = 0.5$ (penalty
two). In addition, we also plotted here the 
penality levels chosen by the CV estimators. These are random so we
plotted their median and quartile values over the 600
replications. The remaining three plots in Figure \ref{f1p} show the
risks of the six estimators. In the top-right plot, the risk is simply
plotted as a function of $n$ (from our theoretical results, the risk
is supposed to decay like the curve $n \mapsto (t_1/n) \log(t_2 n)$
for two constants $t_1$ and $t_2$). In the bottom-left plot, we plotted 
$n$ times the risk against $\log n$. These curves are supposed to be
linear so we provided the squared correlation ($R^2$) values of each of the
curves in this plot. One can see that the $R^2$ values are close to
one for every estimator except $CV_1$. Finally, in the bottom-right plot, we
plotted the logarithm of the risk against $\log n$. We expect the
curves here to have a near-linear relationship with negative slope of
$-1$. The least squares slope values for the different curves are
given in the legend in this and it is clear that, for the non-CV
estimators, the slope is indeed close to $-1$.  

The numerical results in Figure \ref{f1p} for the non-CV estimates
therefore clearly support our theoretical results. On the other hand,
the behavior of the CV estimators seems more complicated and a
theoretical study of their risk performance is beyond the scope of the
present paper. 

%{\color{red} Summarize the other finds from the plots;   e.g.,
%explain which estimator has the "smallest" risk, penalized   versus
%constrained estimator, etc. Any other take home points that   the
%practitioner can use when doing real data analysis?}  

We also show results for $f_2^*$ where we evaluated the performance of
trend filtering for $r = 2$. We did  a simplified
study here with the three estimators: (a) the ideal
constrained estimator \eqref{tf} with $V = V^* = n \|D^{2}
\theta^*\|_1$, (b) the penalized estimator \eqref{ptf} with $\lambda$
taken to be \eqref{lachr} with $\Gamma = 1/16$, and (c) the penalized
estimator \eqref{ptf} with $\lambda$ taken to be \eqref{cdalr} with
$\Gamma = 1/16$. Note that our theoretical
results apply to \eqref{lachr} and \eqref{cdalr} for a sufficiently
large $\Gamma$. For $f_2^*$, we found in simulations that
$\Gamma = 1/16$ was large enough to yield the desired rates. Higher
values of $\Gamma$ inflated risk but gave similar risk decay rates.
We could not compute the ideal penalized estimator with $\lambda = 
\lambda^*$ (defined in \eqref{lams}) here as the
convex optimization problem to compute $\lambda_{\theta^*}(z)$ was
highly ill-conditioned for $n \geq 1000$ so that MOSEK seemed
unable to find the global minimum (see Section \ref{adsims} for more
details). We also did not compute CV 
estimates here as these are not the focus of this paper. 

\begin{figure}[h!]
\begin{center}
 \includegraphics[width = \textwidth, height = 2.0in]{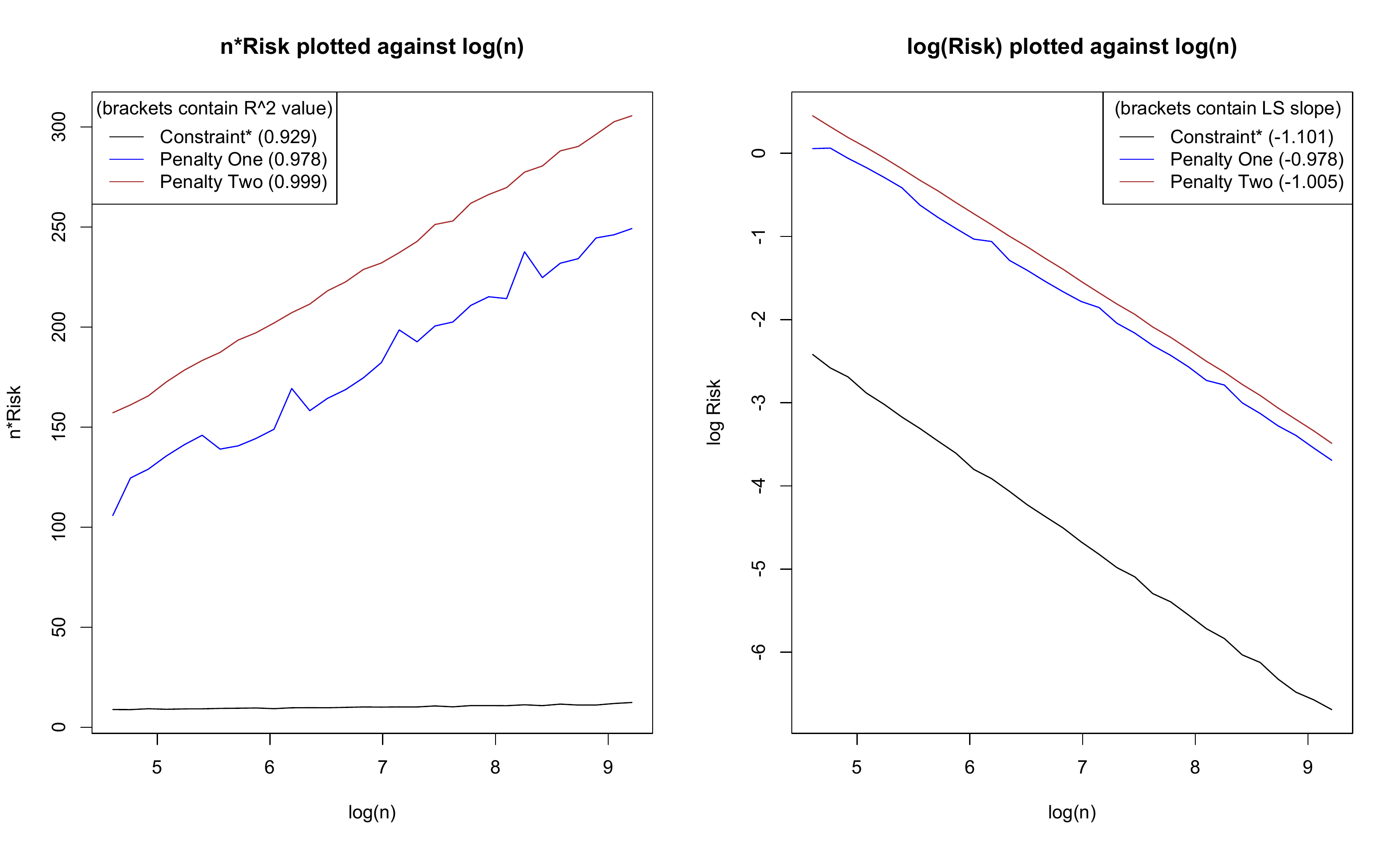}    
\caption{Risk plots when the true function is $f_2^*$.}  
\label{f3p}
\end{center}
\end{figure}

Our results are given in Figure \ref{f3p}. The left plot
  shows $n$ times the risk plotted against $\log n$. Our theory
  indicates that the curve corresponding to each estimator should be
  linear so we provided the squared correlation ($R^2$) values which
  are all close to 1. The right plot shows the behavior of log risk
  against $\log n$. These curves are expected to have a near-linear
  relationship with negative slope of $-1$. The legend shows the least
  squares slopes which are all close to $-1$. These plots therefore
  support our theoretical results.

\section{Discussion}\label{discu}  
In this section, we address various issues that are naturally
linked to our main results.  

\subsection{Weakening our assumptions}\label{weakass}
We emphasized the vector estimation setting~\eqref{seqmo} in this paper. Our results
can also be interpreted in the function estimation setting in
the following way. There is an unknown function $f^*$ and we observe
data $Y_1, \dots, Y_n$ according to the model: 
\begin{equation*}
  Y_i = f^*(x_i)  + \xi_i \qt{for $i = 1, \dots, n$}
\end{equation*}
where $f^* : [0, 1] \rightarrow \R$ is the unknown regression
function and $\xi_1, \dots, \xi_n$ are i.i.d.~$N(0, \sigma^2)$. We
focussed on the situation where $x_i = i/n$ for $i = 1, \dots,
n$. We can estimate $f^*$ by any discrete spline $\hat{f}$
of degree $r-1$ whose values at $i/n, 1= 1, \dots, n$, are given by
$\hat{\theta}_1, \dots, \hat{\theta}_n$ (with $\hat{\theta}$ defined
as in \eqref{tf} or \eqref{ptf}). We then evaluate the performance of
$\hat{f}$ as an estimator for $f^*$ via the loss $\frac{1}{n} \sum_{i=1}^n (f^*{(x_i)} - \hat{f}(x_i) )^2$ and prove bounds for the risk when $f^*$ is a discrete spline in
terms of the number of polynomials that make up $f^*$. 

This basic setting (which is standard and used in many
theoretical papers on univariate nonparametric regression) can be
generalized in many ways and we mention two extensions
involving the design points $x_1, \dots, x_n$ below. One is  
the situation where $x_1, \dots, x_n$ are not
equally spaced. In this case, note that the penalty terms in
\eqref{tf} and \eqref{ptf} need to be changed for $r \ge 2$; see e.g.,
\citet{tibshirani2014adaptive}. We believe that our results 
will still hold in this case provided $x_1, \dots, x_n$ satisfy
$\kappa_1/n \leq x_i - x_{i-1} \leq \kappa_2/n$ for two constants
$\kappa_1$ and $\kappa_2$. However, this would make the notation in
our proofs quite cumbersome. 

One can also study the setting where $x_1, \dots, x_n$ are generated
independently from a common distribution $\nu$ on $[0, 1]$ and/or we
measure the loss via $\int \left(\hat{f}(x) - f^*(x) \right)^2 d\nu(x).$
Analyzing this situation will require handling additional
approximation error terms and we will leave it for future work. 

\subsection{Constrained and penalized estimators}\label{conpene}
As mentioned in the Introduction, we have studied both constrained and
penalized versions of trend filtering while previous papers have
focussed on the penalized estimator alone. When the noise level
$\sigma$ tends to zero, it can be proved that the constrained estimator
with $V = V^* := V^{(r)}(\theta^*)$ is better than the penalized
estimator for every choice of the tuning parameter
$\lambda$. More precisely, 
\begin{equation}\label{oymha}
 \lim_{\sigma \downarrow 0} \frac{1}{\sigma^2}
 R(\hat{\theta}^{(r)}_{V^*}, \theta^*)  <  \lim_{\sigma \downarrow
   0} \frac{1}{\sigma^2}  R(\hat{\theta}^{(r)}_{\lambda}, \theta^*)
 \qt{for every $\lambda \in [0, \infty)$}. 
\end{equation}
Here $\lambda$ is even allowed to depend on $\theta^*$ as long as it
is non-random. Inequality \eqref{oymha} follows from the results of
\citet{oymak2013sharp} as described below: \citet[Theorem
2.1]{oymak2013sharp} implies   
\begin{equation}\label{yama1}
  \lim_{\sigma \downarrow 0} \frac{1}{\sigma^2}
 R(\hat{\theta}^{(r)}_{V^*}, \theta^*) = \frac{1}{n} \E \left(\inf_{v
     \in \cone(\partial g(\theta^*))} \|Z - v\|^2 \right) 
\end{equation}
and \citet[Theorem 1.1]{oymak2013sharp} implies
\begin{equation}\label{yama2}
  \lim_{\sigma \downarrow 0} \frac{1}{\sigma^2}
 R(\hat{\theta}^{(r)}_{\lambda}, \theta^*) = \frac{1}{n} \E \left(\inf_{v
     \in \lambda \partial g(\theta^*)} \|Z - v\|^2\right) 
\end{equation}
for every $\lambda \geq 0$. Here $g(\theta) := n^{r-1} \|D^{(r)}
\theta\|_1$, $\lambda \partial g(\theta^*) := \{\lambda v : v
\in \partial g(\theta^*)\}$, $\cone(\partial g(\theta^*)) :=
\cup_{\lambda \geq 0} \lambda \partial g(\theta^*)$ and $Z \sim N_n(0,
I_n)$. As $\cone(\partial g(\theta^*))$ is strictly larger than $\lambda
g(\theta^*)$ for every fixed $\lambda > 0$, the right hand side of
\eqref{yama1} will be strictly smaller than the right hand side of
\eqref{yama2} which proves \eqref{oymha}. 

The implication of this inequality is that there exist
settings (where $\sigma$ is small) where the constrained estimator
with $V = V^*$ is better than every penalized estimator. Therefore
it makes sense to study the constrained estimator in addition to the
penalized estimator.

\subsection{Results for data-dependent tuning parameters}\label{rddt}
From a practical point of view, a major limitation of the results of
this paper is that they only hold for ideal or oracle choices of the
tuning parameters. Indeed, our strong sparsity risk bounds for the
constrained estimator require $V$ to be close to
$V^* := V^{(r)}(\theta^*)$. On the other hand, our risk bounds for the
penalized estimator require knowledge of the noise level
$\sigma$ (note that the tuning parameter in \eqref{ptf} involves
$\sigma$) as 
well as certain aspects of $\theta^*$. For example, the choices 
\eqref{lams}, \eqref{lach} and \eqref{lachr}  depend on certain
properties of the locations and signs of the knots of $\theta^*$. The
choices  \eqref{cdal} and \eqref{cdalr} have lesser dependence on
$\theta^*$ but they still depend on the constants $c_1$ and $c_2$
from the condition \eqref{leonr}. 

We would like to note that this feature is also present in earlier
papers on the trend filtering estimators. The strong sparsity risk
results of \citet{lin2016approximate}  hold for the tuning choice
\eqref{lamlin} which depends on $\theta^*$. The results of
\citet{dalalyan2017tvd} and \citet{ortelli2018total} hold for the
tuning choice \eqref{lamdal} which does not depend on $\theta^*$ but
depends on the noise level $\sigma$ and the probability level
$\delta$ (note that these results of \cite{dalalyan2017tvd,
  ortelli2018total}  give only high probability statements and not
expectation (risk) bounds).   

We would like to highlight the problem of proving risk bounds under
strong sparsity for completely data-dependent choices of the tuning
parameters as a major open problem. One can approach this problem via
the constrained estimator which would require estimation of the
variation functional $V^{(r)}(\theta^*)$. Alternatively, one can
approach this problem via the penalized estimator which would require
estimation of $\sigma$ and $\lambda^*$ (defined in \eqref{lams}). It
will be  interesting to see if the risk of $\log(en)/n$ (up to
multiplicative factors depending on $k$) will be achieved for a
completely data dependent method~of~tuning.  

\subsection{Connections to results for the LASSO}\label{lasso}
The trend filtering estimators are closely related to the LASSO
estimator of \citet{tibshirani1996regression}. Indeed, for $r = 1$, it
is easy to see that the constrained estimator $\hat{\theta}^{(1)}_V$
is exactly equal to $X \hat{\beta}_V$ where $X$ is the $n \times n$
matrix whose $(i, j)^{th}$ entry equals $I\{i \geq j\}$ and $
\hat{\beta}_V := \argmin_{\theta \in \R^n} \left\{\|Y - X \beta\|^2 :
  \sum_{i=2}^n |\beta_i| \leq V\right\}$. Therefore our strong
sparsity risk results for $\hat{\theta}_V^{(1)}$ 
can simply be seen as results for the LASSO estimator
for this special design matrix $X$. This connection to LASSO also
holds for $r \geq 2$ (see \citet{tibshirani2014adaptive}). 

Based on this link to the LASSO, it might seem possible to believe
that our results might be derivable from general theorems about the
LASSO. However, existing strong sparsity risk bounds for the LASSO
impose stringent conditions on the design matrix (such as the
compatibility condition or the restricted eigenvalue condition) which
do not hold for this particular design matrix $X$ (see
\citet{dalalyan2017tvd}). The relaxed compatibility  
condition of \cite{dalalyan2017tvd} does hold who use this condition
to prove rates under strong sparsity but this argument is 
not strong enough to yield the $\frac{k+1}{n} \log \frac{en}{k+1}$
bound. More importantly, it is not clear if the relaxed compatibility
condition of \cite{dalalyan2017tvd} or a modified version of it holds
for $r \geq 2$.  

\subsection{Comparison to the $L^0$ estimators}
It is natural to compare the performance of the trend-filtering
estimators to the estimators obtained by replacing the $L^1$ norm in
\eqref{tf} by the $L^0$ norm: 
\begin{equation}\label{loes}
  \hat{\theta}^{(r)}_k := \argmin_{\theta \in \R^n} \left\{\frac{1}{2}
  \|Y - \theta\|^2 : \|D^{(r)} \theta\|_0 \leq k\right\}
\end{equation}
Under strong sparsity i.e., $\|D^{(r)}
\theta^*\|_0 \le k$, it should be possible to prove that 
\begin{equation}\label{lori}
  R(\hat{\theta}^{(r)}_k, \theta^*) \leq C_r \frac{\sigma^2(k+1)}{n}
  \log \frac{en}{k+1}. 
\end{equation}
A proof of this result for $r = 1$ can be found in the recent paper
\citet[Theorem 2.1]{gao2017minimax}. We could not find an exact
reference for $r \ge 2$ but we believe that \eqref{lori} should be
true based on the regression connection described in the previous
subsection and existing results for $L^0$-penalized estimators
in linear regression (see e.g., \cite[Theorem
4]{raskutti2011minimax}).   

From a comparison of \eqref{lori} with \eqref{kcoc}, it might seem
that the constrained trend filtering estimator (with $V = V^*$)
has similar performance under strong sparsity as that of the $L^0$
estimator. However, it must be 
kept in mind here that \eqref{kcoc} requires the minimum length
condition \eqref{nsa} while the bound \eqref{lori} for the $L^0$
estimator does not require any such minimum length condition. Without
the minimum length condition, the $L^1$ estimator performs much worse
compared to the $L^0$ estimator as proved in the recent paper
\citet{fan2017l_0}. Note, however, that the minimum length condition
is quite natural from the point of view of
estimating~piecewise~polynomial~functions.    

From a computational viewpoint, \eqref{loes} can
be efficiently computed for $r = 1$ via dynamic programming (see e.g.,
\citet{winkler2002smoothers}) but it is not clear how to 
compute it for $r \geq 2$. On the other hand, the trend
filtering estimators are efficiently computable for every $r \ge 2$
via convex optimization (see e.g., \citet{arnold2016efficient}
and \citet{kim2009ell_1} for details). 

\subsection{Connection to shape constrained estimators}
Shape constrained regression estimators are closely related to the
trend filtering estimators. Indeed, if one takes the constrained trend filtering
estimator \eqref{tf} and replaces the $L^1$ constraint by a
nonnegativity constraint on $D^{(r)} \theta$, then we obtain shape
constrained estimators. Specifically, consider
\begin{equation}\label{shes}
  \hat{\theta}^{(r)}_{\mathrm{shape}} := \argmin_{\theta \in \R^n}
  \left\{\frac{1}{2} \|Y - \theta\|^2 : D^{(r)} \theta \geq 0
  \right\}. 
\end{equation}
Here $D^{(r)} \theta \geq 0$ means that each component of the vector
$D^{(r)} \theta$ is nonnegative. When $r = 1$,  \eqref{shes}
coincides with the classical isotonic least squares estimator and when
$r = 2$, \eqref{shes} coincides with the convex least squares
estimator (see \citet{groeneboom2014nonparametric} for an introduction
to shape constrained estimation). Like the trend filtering estimators,
the shape constrained estimators enjoy the property that $D^{(r)}
\hat{\theta}^{(r)}_{\mathrm{shape}}$ is sparse. However, unlike the
trend filtering estimators, there is no tuning parameter in
\eqref{shes}  (of course, \eqref{shes} is only applicable in
situations where $\theta^*$ satisfies the constraint $D^{(r)} \theta^*
\geq 0$ exactly or in some approximate sense). 

The risk of \eqref{shes} under the strong sparsity assumption (and the
shape assumption $D^{(r)} \theta \geq 0$) has received much recent
attention (see 
\citet{guntuboyina2017nonparametric} for a recent survey). In
\citet{bellec2015annals}, it was proved that 
\begin{equation}\label{balma}
  R(\hat{\theta}^{(r)}_{\mathrm{shape}}, \theta^*) \leq \inf_{\theta : D^{(r)}
    \theta \geq 0} \left(\frac{1}{n} \|\theta^* - \theta\|^2 + C_r
 \frac{\sigma^2(k+ 1)}{n} \log
    \frac{en}{k + 1} \right). 
\end{equation}
where $k := \ktr(\theta) = \|D^{(r)} \theta\|_0$. This result is very
similar to our risk bounds for the constrained 
trend filtering estimator with the important difference that no
minimum length condition is required for \eqref{balma}. It is
interesting to note that we use the above result in the proof of
Theorem \ref{ada}.

\section*{Acknowledgements} We thank Ryan Tibshirani for informing us 
about the reference, \citet{steidl2006splines}, and for many other
helpful comments. We are also extremely thankful to the Associate
Editor and the anonymous referees for very detailed comments on an
earlier version of the paper. Their feedback greatly improved the
quality of the paper. 

\newpage

\section*{Supplementary Material (including proofs of main
  results)}  

\appendix

Here we provide proofs of the results in the paper and some additional
simulation results. The material is organized as follows. Section
\ref{prems} contains a summary of various existing results from the
literature on convex-constrained least squares estimators as well as
convex analysis and geometry that are needed for our main
proofs. Section \ref{pfmr} contains proofs of our main results in
Section \ref{ms} of the main paper. Section \ref{sutta} contains
proofs of various technical supporting results that were crucially
used in the proofs of Section \ref{pfmr}. Section \ref{appa} contains
additional technical results and proofs. Finally Section \ref{adsims}
contains some additional simulation results. 

\section{Preliminaries}\label{prems}
In this section, we state some existing general results on the risk of
constrained and penalized least squares estimators from the
literature. These results will be used in the proofs of our main
theorems from Section \ref{ms}. We shall also state some standard
results from convex analysis and convex geometry which will be used in
our arguments. 

Let us start with results for convex constrained least squares
estimators. These are estimators of the form
\begin{equation}\label{genk}
  \hat{\theta} := \argmin_{\theta \in \R^n} \left\{\frac{1}{2} \|Y -
    \theta\|^2 : \theta \in K \right\}. 
\end{equation}
for a closed convex set $K$. Note that the constrained trend filtering
estimator $\hat{\theta}_V^{(r)}$ is a special case of this estimator
when $K$ is taken to be the set $K^{(r)}(V)$ defined as
\begin{equation}\label{krv}
  K^{(r)}(V) := \left\{\theta \in \R^n : \|D^{(r)} \theta\|_1 \leq V
    n^{1- r} \right\} . 
\end{equation}
The general theory of convex-constrained least squares estimators 
has a long history and is, by now, well established (see e.g.,
\citet{VandegeerBook,   vaartwellner96book, hjort2011asymptotics,
  Chat14}). The following result, essentially from \citet{Chat14} (see
Remark \ref{rsket}) provides upper bounds for the risk of
$\hat{\theta}$. This result will be used in the proof of Theorem
\ref{woor}. 
\begin{theorem}\label{gencha}
  Suppose $Y \sim N_n(\theta^*, \sigma^2 I_n)$ for some $\theta^* \in
  K$ and consider the estimator \eqref{genk}. Then there exists a
  universal positive constant $C$ such that 
  \begin{equation}\label{gencha.eq}
  R(\hat{\theta}, \theta^*) :=  \frac{1}{n}\E_{\theta^*} \|\hat{\theta}
    - \theta^*\|^2 \leq \frac{C}{n} \max(t_0^2, \sigma^2)  
  \end{equation}
  for every $t_0 > 0$ which satisfies 
  \begin{equation}\label{tmup}
    \E \left[\sup_{\theta \in K : \|\theta - \theta^*\| \leq t_0}
      \left<\xi, \theta - \theta^* \right>\right] \leq \frac{t_0^2}{2}
    \qt{where $\xi \sim N_n(0, \sigma^2 I_n)$}.  
  \end{equation}
\end{theorem}

\begin{remark}\label{rsket}
  The purpose of this remark is to describe how Theorem \ref{gencha}
  follows from the results of \citet{Chat14} which are all stated for
  $\sigma = 1$. Extending \citet[Proof of Theorem 1.1]{Chat14} in a
  straightforward manner to the case of arbitrary $\sigma > 0$, one
  obtains that  
\begin{equation}\label{piq}
  \P \left\{\|\hat{\theta} - \theta^*\| - t_{\theta^*} \geq x
    \sqrt{t_{\theta^*}} \right\} \leq 3 \exp \left(\frac{-x^4}{32
      \sigma^2 \left(1 + x/\sqrt{t_{\theta^*}} \right)^2} \right)
\end{equation}
for every $x \geq 0$ where $t_{\theta^*}$ is defined as the maximizer
of 
\begin{equation*}
  t \mapsto     \E \left[\sup_{\theta \in K : \|\theta - \theta^*\| \leq t}
      \left<\xi, \theta - \theta^* \right>\right] - \frac{t^2}{2}
\end{equation*}
over $t \geq 0$. Inequality \eqref{piq} implies that whenever $x \geq
\sqrt{t_{\theta^*}}$, we obtain 
\begin{equation*}
  \P \left\{\|\hat{\theta} - \theta^*\| - t_{\theta^*} \geq x
    \sqrt{t_{\theta^*}} \right\} \leq 3 \exp \left(\frac{-t_{\theta^*}x^2}{128
      \sigma^2} \right). 
\end{equation*}
This is because $1 + x/\sqrt{t_{\theta^*}} \leq 2 x/\sqrt{t_{\theta^*}}$
under the assumption that $x \geq \sqrt{t_{\theta^*}}$. Replacing $x$ by
$u/\sqrt{t_{\theta^*}}$, we obtain 
\begin{equation*}
  \P \left\{\|\hat{\theta} - \theta^*\| - t_{\theta^*} \geq u \right\} \leq 3 \exp \left(\frac{-u^2}{128
      \sigma^2} \right) \qt{for $u \geq t_{\theta^*}$}. 
\end{equation*}
Multiplying both sides by $u$ and integrating from $u = t_{\theta^*}$ to
$u = \infty$, we get 
\begin{equation*}
  \E \left(\left(\|\hat{\theta} - \theta^*\| - t_{\theta^*} \right)^2 -
    t^2_{\theta^*} \right)_+ \leq 3 \int_{0}^{\infty} u \exp
  \left(\frac{-u^2}{128 \sigma^2} \right) du \leq C \sigma^2. 
\end{equation*}
This implies that (via $a^2 \le 2 (a - b)_+^2 + 2b^2$)
\begin{equation*}
 \E \left(\|\hat{\theta} - \theta^*\| - t_{\theta^*} \right)^2 \leq C
  \sigma^2 + 2t_{\theta^*}^2. 
\end{equation*}
which further implies that 
\begin{equation*}
  \E \|\hat{\theta} - \theta^*\|^2 \leq 6 t_{\theta^*}^2 + C \sigma^2 \leq
  C \max \left(t_{\theta^*}^2 , \sigma^2 \right). 
\end{equation*}
From here, we obtain \eqref{gencha.eq} by noting that $t_{\theta^*}
\leq t_0$ which follows from \citet[Proposition 1.3]{Chat14}.   
\end{remark}

The risk of $\hat{\theta}$ can also be related to the tangent cones of
the closed convex set $K$ at $\theta^*$. To describe these results, we
need some notation and terminology. The tangent cone of $K$ at $\theta
\in K$ is defined as  
\begin{equation}\label{tcod}
  T_K(\theta) := \mathrm{Closure}\{t(\eta - \theta) : t \geq 0, \eta
  \in K\}. 
\end{equation}
Informally, $T_K(\theta)$ represents all directions in which one can
move from $\theta$ and still remain in $K$. Note that $T_K(\theta)$ is
a cone which means that $a \alpha \in T_K(\theta)$ for every $\alpha
\in T_K(\theta)$ and $a \ge 0$. It is also easy to see that
$T_K(\theta)$ closed and convex. 

The statistical dimension of a closed convex cone $T \subseteq \R^n$
is defined as 
\begin{equation*}
  \delta(T) := \E \|\Pi_T(Z)\|^2 \qt{where $Z \sim N_n(0, I_n)$}
\end{equation*}
and $\Pi_T(Z) := \argmin_{u \in T} \|Z - u\|^2$ is the projection of
$Z$ onto $T$. The terminology of statistical dimension is due to
\citet{amelunxen2014living} and we refer the reader to this paper for
many properties of the statistical dimension. The statistical
dimension $\delta(T)$ is closely related to the Gaussian width of $T$
which is defined as 
\begin{equation}\label{gwef}
  w(T) := \E \left[\sup_{\theta \in T : \|\theta\| \leq 1} \left<Z, \theta
  \right> \right] \qt{where $Z \sim N_n(0, I_n)$}. 
\end{equation}
Indeed, it has been shown in \citet[Proposition
10.2]{amelunxen2014living} that 
\begin{equation}\label{gast}
  w^2(T) \leq \delta(T) \leq w^2(T) + 1
\end{equation}
for every closed convex cone $T$. 

The relevance of these notions to the estimator $\hat{\theta}$
(defined in \eqref{genk}) is that
the risk of $\hat{\theta}$ can be related to the statistical dimension
of tangent cones of $K$. This is the content of the following result
due to \citet[Corollary 2.2]{bellec2015annals}. 
\begin{theorem}\label{belma} 
Suppose $Y \sim N_n(\theta^*, \sigma^2 I_n)$ for some $\theta^* \in
\R^n$. Then 
\begin{equation}\label{belma.eq}
  R(\hat{\theta}, \theta^*) \leq \inf_{\theta \in K} \left[\frac{1}{n}
  \|\theta - \theta^*\|^2 + \frac{\sigma^2}{n}
  \delta(T_{K}(\theta))\right]. 
\end{equation}
Moreover for every $x > 0$, we have 
\begin{equation*}
  \frac{1}{n} \|\hat{\theta} - \theta^*\|^2 \leq \inf_{\theta \in K}
  \left[ \frac{1}{n} \|\theta - \theta^*\|^2 + \frac{2\sigma^2}{n}
    \delta(T_K(\theta)) \right] + \frac{4 \sigma^2 x}{n} 
\end{equation*}
with probability at least $1 - e^{-x}$. 
\end{theorem}

\begin{remark}
A useful lower bound corresponding to \eqref{belma.eq} has been proved
by \citet[Theorem 2.1]{oymak2013sharp}. This result states that when
$\theta^* \in K$, we have
\begin{equation}
  \label{losili}
\lim_{\sigma \downarrow 0}  \frac{1}{\sigma^2} R(\hat{\theta},
\theta^*) = \frac{1}{n} \delta(T_{K}(\theta^*))
\end{equation}
  which means that $\delta(T_K(\theta^*))/n$ provides a precise
  description of $R(\hat{\theta}, \theta^*)$ in the low $\sigma$
  limit. The fact \eqref{losili} will be used in the proof of Lemma
  \ref{simpo}. 
\end{remark}

An interesting aspect of Theorem \ref{belma} is that $\theta^*$ is
allowed to be any vector in $\R^n$; in particular, it is not necessary
that $\theta^* \in K$. Note that combining Theorem \ref{belma} with
the bound $\delta(T) \leq w^2(T) + 1$  from \eqref{gast}, we obtain
the following risk and loss bounds in terms of the Gaussian width of
tangent cones: 
 \begin{equation}\label{belma1}
  R(\hat{\theta}, \theta^*) \leq \inf_{\theta \in K} \left[\frac{1}{n}
  \|\theta - \theta^*\|^2 + \frac{\sigma^2}{n} + \frac{\sigma^2}{n} 
  w^2(T_{K}(\theta))\right]
\end{equation}
and for every $x > 0$, 
\begin{equation}\label{belma2}
  \frac{1}{n} \|\hat{\theta} - \theta^*\|^2 \leq \inf_{\theta \in K} \left[ \frac{1}{n} \|\theta
  - \theta^*\|^2 +\frac{2
    \sigma^2}{n} + \frac{2\sigma^2}{n} w^2(T_K(\theta)) \right] + \frac{4 \sigma^2 x}{n}  
\end{equation}
with probability at least $1 - e^{-x}$. The above pair of bounds will
be our starting points in the proof of Theorem \ref{ada}. With these
bounds, the main task for proving Theorem \ref{ada} (as well as
inequality \eqref{ada.peq} in Remark \ref{exphp}) will involve
showing the existence of a constant $C_r$ depending only on $r$ such
that  
\begin{equation}\label{rew}
  w(T_{K^{(r)}(V)}(\theta)) \le C_r \sqrt{n \Delta_r(\theta)} 
\end{equation}
for every $\theta \in \R^n$ with $V^{(r)}(\theta) = V$. Indeed,
combining the inequalities \eqref{belma1} and \eqref{rew}, we obtain 
\begin{equation*}
    R(\hat{\theta}, \theta^*) \leq \inf_{\theta \in K} \left[\frac{1}{n}
  \|\theta - \theta^*\|^2 + \frac{\sigma^2}{n} + C_r^2 \sigma^2
  \Delta_r(\theta) \right]. 
\end{equation*}
Because $\Delta_r(\theta) \geq (k+1)/n \geq 1/n$, the above bound
clearly implies \eqref{ada.eq}. Similarly, \eqref{belma2}, combined
with \eqref{rew}, implies \eqref{ada.peq}. The key therefore is to
prove \eqref{rew} which is accomplished in Subsection \ref{bsi}.  

Let us now describe results for penalized estimators of the form
$\hat{\theta}_{\lambda}^g$ defined as 
\begin{equation}\label{gepn}
  \hat{\theta}_{\lambda}^g := \argmin_{\theta \in \R^n}
  \left(\frac{1}{2} \|Y - \theta\|^2 + \sigma \lambda g(\theta) \right)
\end{equation}
where $g: \R^n \rightarrow \R$ is a convex function. The risk of
$\hat{\theta}^g_{\lambda}$ under $Y \sim N_n(\theta^*, \sigma^2 I_n)$
can be bounded by the Gaussian mean squared distance (defined next) of
the set $\lambda \partial g(\theta^*) := \{\lambda v : v \in \partial
g(\theta^*)\}$  where $\partial g(\theta^*)$ is the subdifferential of
$g$ at $\theta^*$. The Gaussian mean squared distance $\D(\C)$ of a
nonempty set $\C \subseteq \R^n$ is defined as 
\begin{equation}\label{deef}
  \D(\C) := \E \left[\dist^2(Z, \C) \right] \qt{where $\dist(Z, \C) :=
    \inf_{x \in \C} \|Z - x\|$}
\end{equation}
and $Z \sim N_n(0, I_n)$. The following result, due to \citet[Theorem
2.2]{oymak2013sharp} bounds the risk of $\hat{\theta}_{\lambda}^g$ in
terms of $\D(\lambda \partial g(\theta^*))$. 
\begin{theorem}\label{ohala}
  Suppose $Y \sim N_n(\theta^*, \sigma^2 I_n)$. Then 
  \begin{equation*}
    R(\hat{\theta}_{\lambda}^g, \theta^*) \leq \frac{\sigma^2}{n}
    \D(\lambda \partial g(\theta^*)). 
  \end{equation*}
\end{theorem}
Theorem \ref{ohala} will be our starting point for proving Theorem
\ref{prt1}.  Note that the penalized trend filtering estimator
$\hat{\theta}^{(r)}_{\lambda}$ is a special case of \eqref{gepn} with
$g(\theta) := n^{r-1} \|D^{(r)} \theta\|_1$ so that Theorem
\ref{ohala} will imply that the risk of $\hat{\theta}^{(r)}_{\lambda}$
will be bounded by $(\sigma^2/n)$ times $\D(\lambda n^{r-1} \partial
f(\theta^*))$ where $f(\theta) := \|D^{(r)} \theta\|_1$. The goal then
becomes that of bounding $\D(\lambda n^{r-1} \partial
f(\theta^*))$ from above in the case when $D^{(r)} \theta^* \neq 0$
(note that Theorem \ref{prt1} does not deal with the case $D^{(r)}
\theta^* = 0$; this case is dealt with in Lemma \ref{rr0} whose proof
is simpler and more direct).  

Our idea for bounding $\D(\lambda n^{r-1} \partial
f(\theta^*))$ is to relate it to the smaller quantity
$\D(\cone(\partial f(\theta^*)))$ where $\cone(\partial f(\theta^*))$
is the convex cone generated by $\partial f(\theta^*)$: 
\begin{equation*}
  \cone(\partial f(\theta^*)) := \cup_{\lambda \geq 0}
 \left[ \lambda \partial f(\theta^*) \right]. 
\end{equation*}
It is clear that $\cone(\partial f(\theta^*))$ contains the set
$\lambda n^{r-1} \partial f(\theta^*)$ for every $\lambda \geq 0$ and
thus by definition of $\D(\cdot)$, it follows that 
\begin{equation*}
\D(\lambda n^{r-1} \partial
f(\theta^*)) \geq  \D(\cone(\partial f(\theta^*))). 
\end{equation*}
However, we need an upper bound and not a lower bound for $\D(\lambda n^{r-1} \partial
f(\theta^*))$. It turns out that an upper bound can indeed be given
for $\D(\lambda n^{r-1} \partial
f(\theta^*))$ in terms of $\D(\cone(\partial f(\theta^*)))$ and
additional terms (involving $\lambda$ and the vectors $v_0$ and $v^*$
defined in \eqref{vvs}). This result (formally stated in Proposition \ref{qfmg}) can be
seen as a generalization of \citet[Proposition
1]{foygel2014corrupted}. The advantage of Proposition \ref{qfmg} is that it
reduces the task to upper bounding $\D(\cone(\partial
f(\theta^*)))$. As we shall outline below, by some standard facts from
convex  analysis, it follows that 
\begin{equation}\label{annpre}
\D(\cone(\partial f(\theta^*))) = \delta(T_{K^{(r)}(V^*)}(\theta^*))
\leq 1 + w^2(T_{K^{(r)}(V^*)}(\theta^*))
\end{equation}
where $V^* := V^{(r)}(\theta^*)$. This allows us to use the bound
\eqref{rew} (established in the course of the proof of Theorem
\ref{ada}) to bound $\D(\cone(\partial f(\theta^*)))$. 

We shall now explain why \eqref{annpre} is true. Note that we only
need to prove the first equality (the second inequality is a
consequence of \eqref{gast}). For this, we need to introduce the
notions of normal cone and polar cone from convex analysis (see, for
example, \citet{Rockafellar70book} for background on these standard
notions). The normal cone of a convex set $\C \subseteq \R^n$ at a
point $x \in \C$ is defined by 
\begin{equation*}
  N_{\C}(x) := \left\{u \in \R^n : \left<y - x, u
    \right> \le 0 \text{ for every } y \in \C  \right\} . 
\end{equation*}
Next let us define the notion of a polar cone. The polar $T^{\mathrm{o}}$ of a
nonempty closed convex cone $T \subseteq \R^n$ is defined as
\begin{equation}\label{pocodef} 
  T^{\mathrm{o}} := \left\{u \in \R^n : \left<u, x \right> \leq 0 \text{ for
      every } x \in T \right\}. 
\end{equation}
The following result (see, for example, \citet[Example
6.24]{rockafellar2009variational}) states that for every convex set
$\C$ and $x \in \C$, the normal cone $N_{\C}(x)$ equals the polar of
the tangent cone $T_{\C}(x)$ (recall that $T_{\C}(x)$ is defined in
\eqref{tcod}).  
\begin{lemma}\label{woki}
  For every convex set $\C \subseteq \R^n$ and $x \in \C$, we have 
  \begin{equation*}
    N_{\C}(x) = \left(T_{\C}(x) \right)^{\mathrm{o}}. 
  \end{equation*}
\end{lemma}
The next result states that $\cone(\partial f(\theta^*))$ equals
$N_{\C}(\theta^*)$ where
\begin{equation}\label{cbar}
 \C := \left\{\theta \in \R^n: f(\theta) \leq
  f(\theta^*) \right\} 
\end{equation}
under some conditions on the convex function $f$ and $\theta^*$. This
result follows from \citet[Theorem 23.7 and Corollary
23.7.1]{Rockafellar70book}. 
\begin{lemma}\label{roci}
  Suppose $f : \R^n \rightarrow \R$ is a convex function and $\theta^*
  \in \R^n$ is such that $\partial f(\theta^*)$ is a compact convex 
  set with $0 \notin \partial f(\theta^*)$. Then 
  \begin{equation*}
\cone(\partial f(\theta^*)) =  N_{\C}(\theta^*) 
  \end{equation*}
  where $\C$ is given by \eqref{cbar}.  
\end{lemma}
Observe now that when $f(\theta) := \|D^{(r)} \theta\|_1$ and
$\theta^*$ is such that $D^{(r)} \theta^* \neq 0$, the conditions in
Lemma \ref{roci} hold as can be seen from the characterization of
$\partial f(\theta^*)$ in Proposition \ref{characterization}. The
assumption that $0 \notin \partial f(\theta^*)$ holds because for every
$v \in \partial f(\theta^*)$, we must have
\begin{equation*}
  \sum_{i=j}^n {r+i-j-1 \choose r-1} v_i = \sgn((D^{(r)} \theta^*)_{j-r})
\end{equation*}
for every $j$ such that $(D^{(r)} \theta^*)_{j-r} \neq 0$ (there must
exist at least one such $j$ because of the assumption that $D^{(r)}
\theta^* \ne 0$). 

Further, for $f(\theta) := \|D^{(r)} \theta\|_1$, it is easy to see
that the set $\C$ in \eqref{cbar} satisfies
\begin{equation*}
  \C = \left\{\theta \in \R^n : \|D^{(r)} \theta\|_1 \leq \|D^{(r)}
    \theta^* \|_1 \right\} = K^{(r)}(V^{(r)}(\theta^*)) = K^{(r)}(V^*)
\end{equation*}
because $V^* := V^{(r)}(\theta^*)$ and $K^{(r)}(V)$ is defined as in
\eqref{krv}. Putting together the conclusions of Lemma \ref{woki} and
Lemma \ref{roci}, we therefore deduce that 
\begin{equation*}
  \cone(\partial f(\theta^*)) = \left(T_{K^{(r)}(V^*)}(\theta^*) 
  \right)^{\mathrm{o}}.  
\end{equation*}
From here, in order to prove \eqref{annpre}, we need another standard
fact from convex geometry (see, for example, \citet[Theorem
3.2.5]{hiriart2013convex}). This result states that for every closed
convex cone $T \subseteq \R^n$, we have 
\begin{equation}\label{moreau}
 \Pi_{T^{\mathrm{o}}}(z) = z - \Pi_T(z) \qt{for every $z \in \R^n$}
\end{equation}
where $\Pi_{K}(z)$ denotes the projection of $z$ onto $K$. 

Applying \eqref{moreau} to $T := \cone(\partial f(\theta^*))$ (which
is a closed convex cone when $D^{(r)} \theta^* \neq 0$; closedness
follows, for example, from \citet[Corollary
9.6.1]{Rockafellar70book}), we obtain 
\begin{equation}\label{morco}
  z - \Pi_{\cone(\partial f(\theta^*))}(z) =
  \Pi_{T_{K^{(r)}(V^*)}}(z). 
\end{equation}
From the above identity (and the definitions of $\D(\cdot)$ and $\delta(\cdot)$),
the fact \eqref{annpre} readily follows. The fact \eqref{annpre} will
be crucially used in the proof of Theorem \ref{prt1}. Also, the
identity \eqref{morco} will play a key role in the proofs of
Lemma \ref{explam1} and Lemma \ref{explamr}.  

\section{Proofs of the Main Results}\label{pfmr}
In the section, we provide the proofs of the following results in
Section \ref{ms}: Theorem \ref{woor}, Theorem \ref{ada} (and
inequality \eqref{ada.peq} in Remark \ref{exphp}), Corollary
\ref{kco}, Lemma \ref{simpo}, Theorem \ref{prt1}, Corollary
\ref{expen1}, Lemma 
\ref{explam1} and Corollary \ref{expen}, Corollary \ref{expenr}, Lemma
\ref{explamr} and Corollary \ref{supno} and finally, Lemma
\ref{rr0}. In addition to these results, Section \ref{ms} also
contains Proposition \ref{characterization} and Lemma
\ref{gksd}. These are proved in Subsection \ref{sudi}. 

Some of the proofs presented in this section will introduce and use
additional technical results. These technical results will be proved
in the Section \ref{sutta}.

\subsection{Proof of Theorem \ref{woor}} \label{pfwoor}
We prove Theorem \ref{woor} in this subsection. As mentioned at the
start of Section \ref{prems}, our starting point for this proof is
Theorem \ref{gencha}; note that $\hat{\theta}^{(r)}_V$ is the least
squares estimator subject to the constraint that $\theta \in
K^{(r)}(V)$ (recall that the set $K^{(r)}(V)$ is defined in
\eqref{krv}). Theorem \ref{gencha} implies that we can bound the risk
of $\hat{\theta}_V^{(r)}$ via upper bounds for   
\begin{equation}\label{gtd}
  G(t) := \E \left[\sup_{\theta \in K^{(r)}(V) : \|\theta - \theta^*\| \leq 
  t}  \left<\xi, \theta - \theta^* \right>  \right]
\end{equation}
for $t > 0$. Our upper bound for $G(t)$ is proved from the following
lemma. Let 
\begin{equation}
  \label{c2}
S_r(V, t) :=  \left\{\alpha \in \R^n : \|\alpha\| \leq t, \|D^{(r)}
  \alpha\|_1  \leq V n^{1-r}  \right\}.   
\end{equation}

\begin{lemma}\label{svg}
Fix an integer $r \geq 1$. Then there exists a positive constant $C_r$ such that
for every $n \geq r$, $t \geq 0$ and $V \geq 0$, we have     
\begin{equation}\label{svg.eq}
  \E \left[ \sup_{\theta \in S_r(V, t)} \left<\xi, \theta \right>
  \right] \leq C_r \sigma t 
      \left(\frac{\sqrt{n} V}{t} \right)^{1/(2r)} + C_r \sigma t
      \sqrt{\log (en)}.  
\end{equation}
\end{lemma}
Lemma \ref{svg} is proved in Subsection \ref{gwc} and the ideas behind 
its proof are as follows. By Dudley's entropy bound, the left hand
side of \eqref{svg.eq}  can be bounded from above by the metric
entropy numbers (formally defined in Subsection \ref{gwc}) of the set
$S_r(V, t)$ (defined in \eqref{c2}). The metric entropy of $S_r(V, t)$
will be bounded by controlling the fat shattering dimension (see
Subsection \ref{gwc1} for details).  

Below, we provide the proof of Theorem \ref{woor} based on Lemma
\ref{svg}. 
\begin{proof}[Proof of Theorem \ref{woor}]
As $\|D^{(r)} \theta^*\|_1 \leq V n^{1-r}$, it follows that
  $\theta^* \in \krv$ (the set $\krv$ is defined in
  \eqref{krv}). Theorem \ref{gencha} implies that 
\begin{equation}\label{chat}
  R(\hat{\theta}^{(r)}_V, \theta^*) \leq \frac{C}{n} \max \left(t_0^2,
    \sigma^2 \right)  
\end{equation}
for a universal positive constant $C$, where $t_0 > 0$ is such that
$G(t_0) \leq t_0^2/2$ with $G(t)$ defined as in \eqref{gtd}. In order
to apply this result, we need to bound the function $G(t)$ 
   from above. By triangle inequality, $\|D^{(r)}(\theta - \theta^*)
   \|_1 \leq \|D^{(r)} \theta\|_1 + \|D^{(r)} \theta^*\|_1$ so that    
  \begin{equation*}
    G(t) = \E \sup_{\theta \in \krv: \|\theta - \theta^*\| \leq t} 
    \left<\xi, \theta - \theta^* \right> \leq 
 \E \sup_{\alpha \in \R^n: \|\alpha\| \leq t, \|D^{(r)} \alpha\|_1
   \leq 2 V n^{1-r}} \left<\xi, \alpha \right>.    
  \end{equation*}
   The right hand side above is controlled in Lemma \ref{svg} from 
   which we deduce that 
\begin{equation*}
  G(t) \leq C_r \sigma t \left(\frac{\sqrt{n} V}{t} \right)^{\frac{1}{2
      r}} +  C_r \sigma t \sqrt{\log (e n)} 
\end{equation*}
for a constant $C_r$ depending on $r$ alone. We now observe that 
\begin{equation*}
C_r \sigma t \left(\frac{\sqrt{n} V}{t} \right)^{\frac{1}{2
      r}} \leq \frac{t^2}{4} ~~ \text{  iff  } ~~ t \geq (4 C_r)^{2r/(2r+1)} \sigma^{2r/(2r+1)}
  \left(V\sqrt{n} \right)^{1/(2r+1)} 
\end{equation*}
and 
\begin{equation*}
  C_r \sigma t \sqrt{\log (en)} \leq \frac{t^2}{4} ~~ \text{  iff  } ~~ t \geq 4 C_r \sigma \sqrt{\log(en)}. 
\end{equation*}
It follows therefore that $G(t_0) \leq t_0^2/2$ provided 
\begin{equation*}
  t_0 := \max \left( (4 C_r)^{2r/(2r+1)} \sigma^{2r/(2r+1)}
  \left(V\sqrt{n} \right)^{1/(2r+1)}, 4 C_r \sigma \sqrt{\log(en)}
\right).  
\end{equation*}
The proof of inequality \eqref{woor.eq} is therefore complete by
inequality \eqref{chat}.    

Inequality \eqref{woor.peq} can be derived as a consequence of
\eqref{woor.eq} and the fact that the map $y \mapsto
\|\hat{\theta}^{(r)}_V - \theta^*\|$ is 1-Lipschitz (see e.g.,~\citet[Section 2]{van2015concentration}). By the usual concentration
inequality for Lipschitz functions of Gaussian variables, this gives 
\begin{equation*}
  \P \left\{\|\hat{\theta}^{(r)}_V - \theta^*\| \geq \E_{\theta^*}
    \|\hat{\theta}^{(r)}_V - \theta^*\| + \sigma z \right\} \leq \exp
  \left(\frac{-z^2}{2} \right). 
\end{equation*}
This gives that
\begin{equation*}
  \frac{1}{n} \|\hat{\theta}^{(r)}_V - \theta^* \|^2 \leq 2
  R(\hat{\theta}^{(r)}_V, \theta^*) + \frac{4 \sigma^2 x}{n}
\end{equation*}
with probability at least $1 - e^{-x}$ so that inequality
\eqref{woor.peq} follows from \eqref{woor.eq}. 
\end{proof}

\subsection{Proof of Theorem \ref{ada}} \label{bsi}
Our starting points for proving Theorem \ref{ada} are the inequalities
\eqref{belma1} and \eqref{belma2} applied to $K = K^{(r)}(V)$. From
here, it is clear that inequality \eqref{ada.eq} (as well as
\eqref{ada.peq}) both follow from the inequality \eqref{rew}. Writing
explicitly the Gaussian width $w(T_{K^{(r)}(V)}(\theta))$, we see that
\eqref{rew}  is equivalent to proving that 
\begin{equation}\label{keytc}
  \E \left[\sup_{\alpha \in T_{\krv}(\theta): \|\alpha\| \leq 1} \left<Z,
    \alpha \right> \right] \leq C_r \sqrt{ n \Delta_r(\theta)} 
\end{equation}
for every $\theta \in \R^n$ with $V^{(r)}(\theta) = V$. For this, we
obviously need to understand the set $T_{\krv}(\theta)$ for $\theta
\in \R^n$ with $V^{(r)}(\theta) = V$. The following result provides a
necessary condition that is satisfied by every vector $\alpha \in 
T_{\krv}(\theta)$ with $\|\alpha\| \le 1$. The proof of this lemma is
given in Subsection \ref{tcc}. Recall,  from Section \ref{ms}, the notion of
$r^{th}$ order knots (along with their signs) of vectors in $\R^n$. We
shall also use the following notation. For $\alpha \in \R^m$ and $1
\leq a \leq b \leq m$, we let  
\begin{equation}\label{vab} 
  V_{a, b}(\alpha) := V(\alpha_a, \dots, \alpha_b) = |\alpha_{a+1} -
    \alpha_a| + \dots + |\alpha_b - \alpha_{b-1}|. 
\end{equation}

%\vspace{0.2in} 

\begin{lemma}\label{alta}
Fix $V > 0$, $r \ge 1$, $n \geq r$ and $\theta \in \R^n$ with
$V^{(r)}(\theta) = V$.  Suppose $2 \leq j_1 < \dots < j_k \leq n-r+1$
  denote any set of indices which contains all the $r^{th}$ order
  knots of $\theta$. Let $\rs_1, \dots, \rs_k$ be such that
  $\rs_i$ is the sign of the knot corresponding to $j_i$ if $j_i$ is a
  knot and $\rs_i$ is arbitrary in $\{-1, 0, 1\}$ if $j_i$ is not a
  knot. Also let $j_0 = 1$, $j_{k+1} = n-r+2$ and $\rs_0 = \rs_{k+1} =
    0$. The indices $j_0, j_1, \dots, j_k, j_{k+1}$ define a partition
      $\I_0, \dots, \I_k$ of $\{1, \dots, n\}$ in the following way:
      $\I_0 := \{j_0, \dots, j_1 + r - 2\}$ and     
\begin{equation*}
  \I_i = \{j_i + r - 1, \dots, j_{i+1} + r - 2\} \qt{for $i = 1,
    \dots, k$}.   
\end{equation*}
Let $n_i$ denote the cardinality of $\I_i$ for $i = 0, 1, \dots, k$
i.e., $n_0 := j_1 + r - 2$ and $n_i = j_{i+1} - j_i$ for $1 \leq i
\leq k$.  Then there exists a positive constant $C_r$ (that depends on
$r$ alone) such that for every $\alpha \in T_{\krv}(\theta)$ with
$\|\alpha\| \le 1$, there exist indices $\ell_0 \in \I_0, \ell_1 \in
\I_1, \dots, \ell_k \in \I_k$ such that  
\begin{equation}
  \label{alta.eq}
  \sum_{i=0}^k \Gamma_i(\alpha, \ell_i) \leq C_r \sqrt{ \sum_{i=0}^{k} n_i^{1 - 2r} I\{\rs_i \neq \rs_{i+1}\}}  
\end{equation}
where 
\begin{equation}\label{alta.de}
  \Gamma_i(\alpha, \ell_i) := V_{j_i, j_{i+1} - 1}(\Delta) -
  \rs_{i+1} (\Delta_{j_{i+1} - 1} - \Delta_{\ell_{i}}) - \rs_{i}
    (\Delta_{\ell_i} - \Delta_{j_{i}})
\end{equation}
with $\Delta = (\Delta_1,\ldots, \Delta_{n-r+1}):= D^{(r-1)} \alpha$.        
\end{lemma}

%\vspace{0.2in} 

\begin{remark}
  It may be noted that the indices $j_1, \dots, j_k$ in Lemma
  \ref{alta}  are not exactly the knots of $\theta$. They are any set
  of indices that contain the knots of $\theta$. We shall mostly work
  with the case when $j_1, \dots, j_k$ are exactly the set of knots of
  $\theta$ but we shall need this additional generality to deal with
  one special situation when some of the distances between the knots
  of $\theta$ are too large. In this case (see the last part of the proof
  of Theorem \ref{ada}), we shall add additional indices to the knots
  in order to keep the inter-distances manageable.  
\end{remark}

The insight provided by Lemma \ref{alta} into the structure
of $\{\alpha \in T_{\krv}(\theta) : \|\alpha\| \le 1 \}$ (note that
understanding this set is necessary for proving \eqref{keytc})  is as
follows. Suppose that $n_0, \dots, n_k$ are such that the right hand
side of \eqref{alta.eq} is small. In this case, Lemma \ref{alta}
implies that for every $\alpha \in T_{\krv}(\theta)$ with $\|\alpha\|
\le 1$, there exist indices $\ell_0, \dots, \ell_k$ for which
$\sum_{i=0}^k \Gamma_i(\alpha, \ell_i)$ is small. Figure \ref{tcplots}
displays two unit norm vectors in the tangent cone of a piecewise
constant vector $\theta$ (i.e., $r = 1$) and the corresponding indices
$\ell_0, \dots, \ell_k$. 

\begin{figure}[h!]
\begin{center}
 \includegraphics[height=2.8in,width=2.4in]{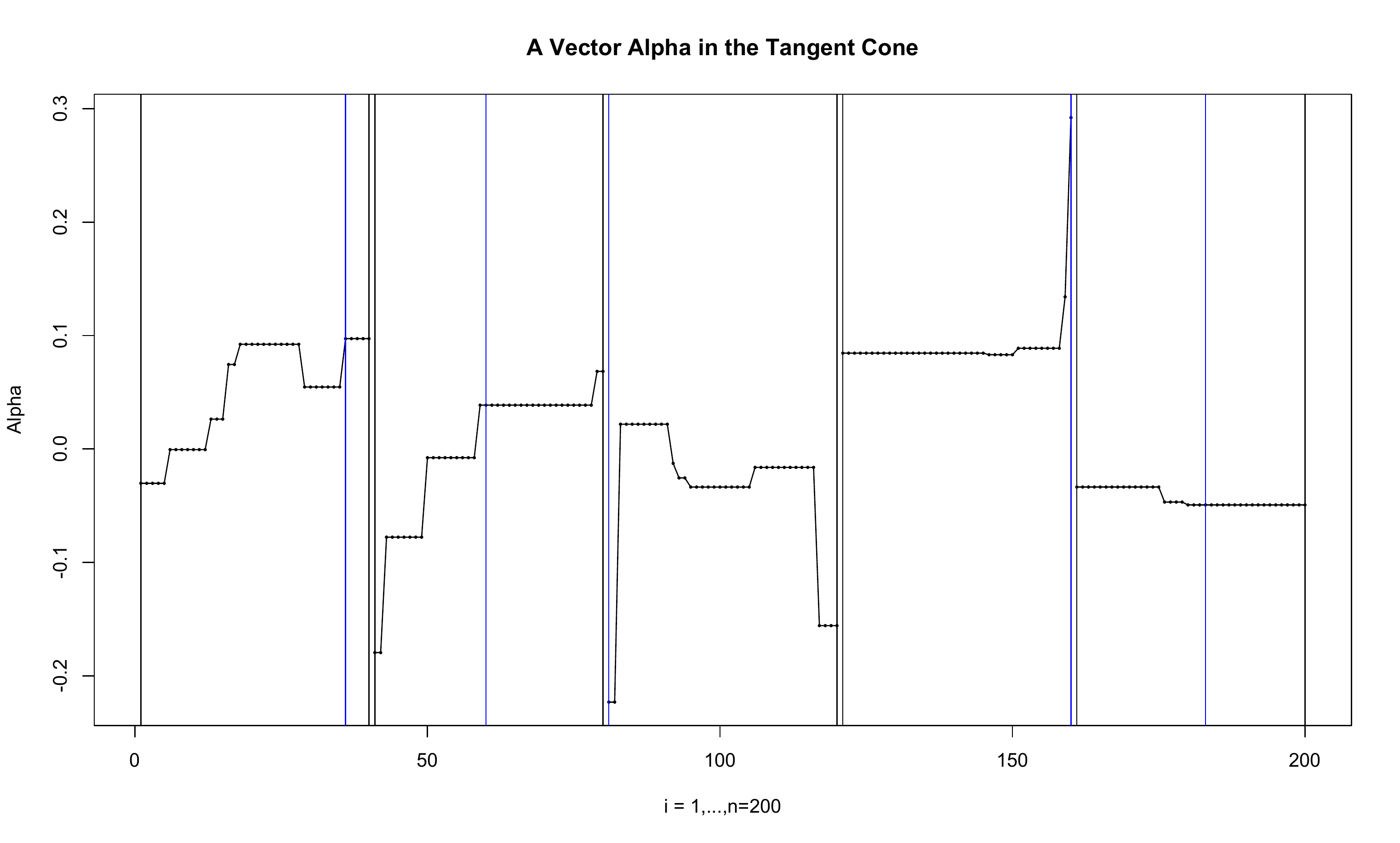}    
  \includegraphics[height=2.8in,width=2.4in]{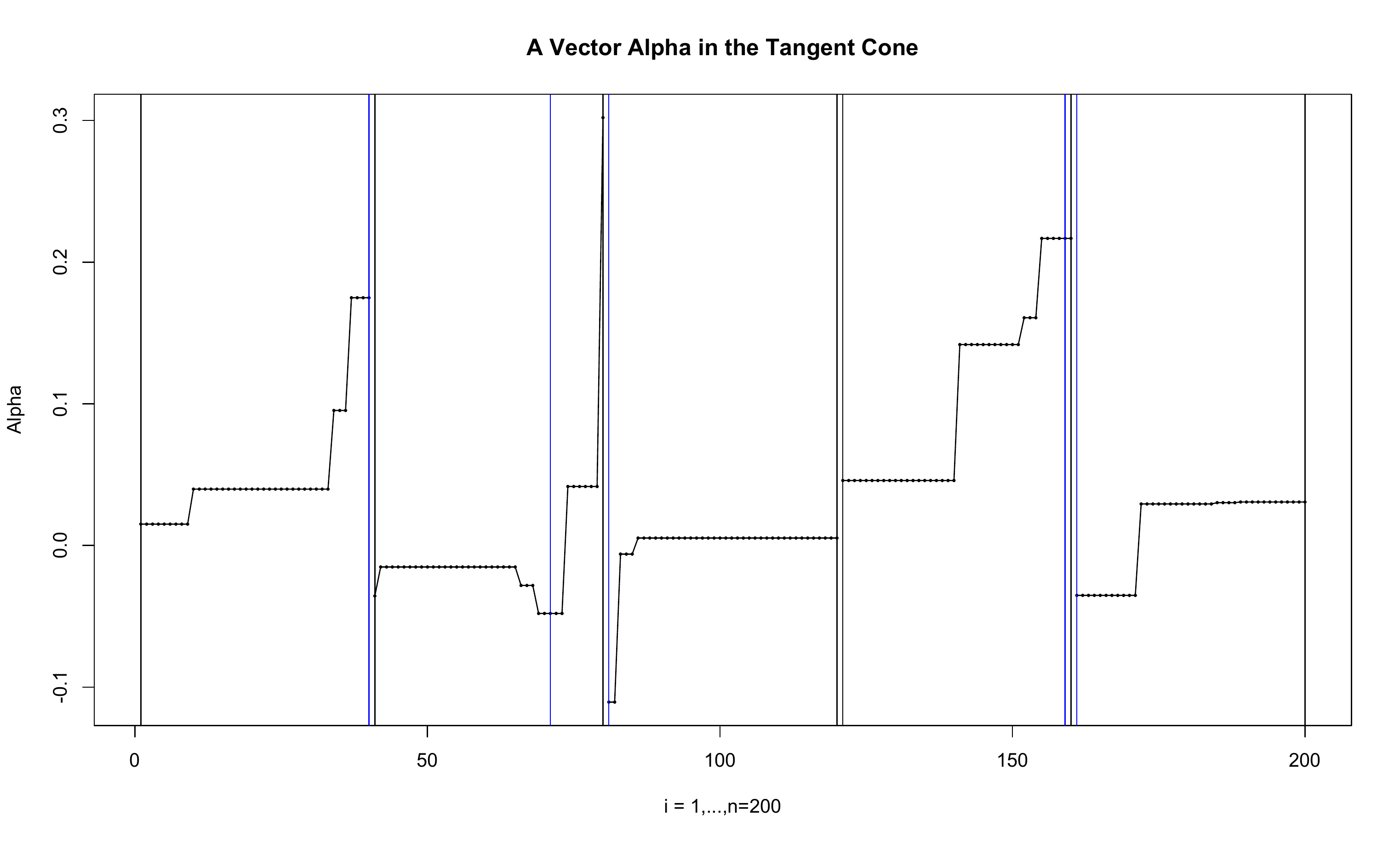}        
\caption{Let $r = 1$, $n = 200$ and let $\theta$ be the vector obtained
by sampling $f_1^*$ at $n$ equally spaced points with end points 0 and
1 (here $f_1^*$ is the piecewise constant vector from Section
\ref{simu} of the main paper). This vector $\theta$ has $k = 4$ jumps
at $j_1 = 41, j_2 = 81, j_3 = 121$ 
and $j_4 = 161$. These indices (and the indices $j_i - 1, i = 1, 2, 3,
4$) are plotted in black lines in the above pair of plots along with
vertical straight lines at $j_0 = 1$ and $j_5 = 200$. We then considered the
tangent cone, $T := T_{K^{(1)}(V)}(\theta)$, where $V$ is the variation of
$\theta$ and plotted two vectors $\alpha$ in $T$ with
$\|\alpha\| = 1$. For each of these two vectors $\alpha$, we also
plotted the integers $\ell_0, \dots, \ell_4$ as blue vertical
lines. Informally, in the five constant segments corresponding to
$\theta$, each vector $\alpha$ is approximately made of two monotone
segments.}   
\label{tcplots}
\end{center}
\end{figure}

It is helpful here
to observe that $\Gamma_i(\alpha, \ell_i)$ is always
nonnegative. Also,  if $\Gamma_i(\alpha, \ell_i) = 0$, then $D^{(r-1)}
\alpha$ is made of two monotone pieces in the interval from $j_i$ to
$j_{i+1} - 1$ (one piece 
  from $j_i$ to $\ell_i$ and the other from $\ell_i$ to $j_{i+1} -
    1$). When $r = 1$, this means that $\alpha$ is made of two
    monotone pieces in the interval from $j_i$ to $j_{i+1} - 1$. When
      $r = 2$, this means that $\alpha$ is made of two convex/concave
      pieces in the interval from $j_i$ to $j_{i+1} - 1$. For general
      $r$, this means that $\alpha$ is made of two $(r-1)^{th}$ order
      convex/concave functions in the interval from $j_i$ to $j_{i+1}
      - 1$. Extending this argument,  when
      $\Gamma_i(\alpha,  \ell_i)$ is small, $D^{(r-1)} \alpha$ is
      \textit{nearly} made of two monotone pieces in the interval from $j_i$  to $j_{i+1}
          - 1$; equivalently $\alpha$ is nearly made of two
          $(r-1)^{th}$ order convex/concave sequences in the interval
          from $j_i$ to $j_{i+1} - 1$.  This suggests therefore that in order to prove
          \eqref{keytc}, we need to prove bounds on the Gaussian suprema
          for vectors $\alpha \in \R^n$  for which $D^{(r-1)} \alpha$
          is nearly monotone. This is the content of the next lemma
          which is another main ingredient for the proof of Theorem
          \ref{ada}. 

%\vspace{0.2in} 

\begin{lemma}\label{fina} 
Fix $r \ge 1$, $n \geq r$, $1 \leq l \leq n-r+1$, $t > 0$ and $\delta
\geq 0$. For $\theta \in \R^n$, let $\Delta(\theta) =
(\Delta_1(\theta),\ldots, \Delta_{n-r+1}(\theta)) := D^{(r-1)}
\theta$. Also let $\xi \sim N_n(0, \sigma^2 I_n)$. For 
every $\rs_1, \rs_2 \in \{-1, 0, 1\}$, the 
  quantity    
  \begin{equation*}
    \E \sup_{\theta \in \R^n, \|\theta\| \leq t} \left\{ \left<\xi, \theta
    \right> : \Delta =
    \Delta(\theta) , V(\Delta) \leq \rs_1(\Delta_{\ell} - \Delta_1) +
    \rs_2(\Delta_{n-r+1} - \Delta_{\ell}) + \delta \right\} 
  \end{equation*}
  is bounded from above as 
  \begin{equation*}
G \leq C_r \sigma \left(t + \delta n^{(2r -1)/2} \right)
    \sqrt{\log (en)} + C_r \sigma t^{(2r-1)/(2r)} n^{(2r-1)/(4r)} \delta^{1/(2r)}
  \end{equation*}
  for a positive constant $C_r$ that depends on $r$ alone. 
\end{lemma}

%\vspace{0.2in} 

The proof of Lemma \ref{fina} is given in Subsection \ref{gwc} (in 
fact, in Subsection \ref{gwc}, we prove Lemma \ref{fina.full} which is
a more accurate result compared to Lemma \ref{fina} in the sense that
Lemma \ref{fina.full} gives a bound that depends on the actual values
of $\rs_1$ and $\rs_2$). The proof of Lemma \ref{fina.full}  uses
results on expected Gaussian suprema for classes of shape-constrained
vectors from \citet{bellec2015annals}.   

We are now ready to prove Theorem \ref{ada}. 

\begin{proof}[Proof of Theorem \ref{ada}] As the proof is rather long,
  we divide it into many steps.  

{\bf Step I}: We first note that the case when $V = 0$ is trivial. This is because
the set $\{\theta \in \R^n : \|D^{(r)} \theta\|_1 = 0\}$ is a subspace
of dimension $r$ so that $\hat{\theta}^{(r)}_{V=0}$ becomes a linear
projection onto a subspace. Thus, 
\begin{equation*}
  \frac{1}{n} \|\hat{\theta}^{(r)}_{V =0} - \theta^*\|^2 - \inf_{\theta \in
  \R^n : \|D^{(r)} \theta\|_1 = 0} \frac{1}{n} \|\theta - \theta^*\|^2
\sim \frac{\sigma^2}{n} \chi^2_r
\end{equation*}
where $\chi^2_r$ denotes the chi-squared distribution with $r$ degrees
of freedom. This and a standard tail bound for chi-squared random
variables such as (see e.g., \citet[Subsection
4.1]{laurent2000adaptive}) 
\begin{equation*}
  \P \left\{\chi^2_r \leq 2 r + 3 x \right\} \geq 1 - e^{-x} \qt{for
  every $x > 0$} 
\end{equation*}
prove inequalities \eqref{ada.eq} and \eqref{ada.peq} for $V= 0$; note
that when $\theta \in \R^n$ is such that $V^{(r)}(\theta) = V = 0$, we
have $\ktr(\theta) = 0$,  $\delta_r(\theta) = n^{(1/2) - r}$  and
$\Delta_r(\theta) = \frac{2}{n} \log(en) + \frac{1}{n}$.  

%{\color{red} Note that in this case $  \ktr(\theta) = 0$ and
%$\Delta_{r} (\theta)$ reduces to .... (cf.~\eqref{mrt})}.  

We shall assume from now on that $V > 0$. Based on the discussion at
the beginning of this subsection, it is enough to prove
\eqref{keytc}. We therefore fix $\theta \in \R^n$ with 
$V^{(r)}(\theta) = V$. We need to bound the quantity 
\begin{equation}\label{kann}
  G := \E \left[\sup_{\alpha \in T_{\krv}(\theta): \|\alpha\| \leq 1} 
  \left<Z, \alpha \right> \right] 
\end{equation}
where $Z$ is a standard $n$-dimensional Gaussian random vector. We
bound $G$ by breaking the set $\{\alpha \in T_{\krv}(\theta) : \|\alpha\|
\leq 1\}$ into smaller subsets. 

Let $2 \leq j_1 < \dots < j_k \leq n-r+1$ denote all the $r^{th}$
order knots of $\theta$. Also let $\rs_1, \dots, \rs_k \in \{-1, 1\}$
denote the signs of the knots. For convenience, we take $j_0 = 1, j_{k+1} = 
n-r+2$ and $\rs_0 = \rs_{k+1} = 0$. Let $n_0 = j_1 + r - 2$ and $n_u =
j_{u+1} - j_{u}$ for $u = 1, \dots, k$. Check that $\sum_{u=0}^k n_u =
n$. $\vspace{0.08in}$

{\bf Step II}: We shall prove~\eqref{keytc} first under the simplifying assumption
that   
\begin{equation}\label{sia}
  n_i \leq \frac{2n}{k+1} \qt{for every $i = 0, 1, \dots, k$}. 
\end{equation}
The goal of this step is to find a collection of sets whose union covers $  \left\{\alpha \in T_{\krv}(\theta) : \|\alpha\| \leq 1 \right\}$ (see~\eqref{cuc}). For every vector $\alpha \in \R^n$, let us define the vectors  
\begin{equation*}
  \alpha^{(0)} := (\alpha_{j_0}, \dots, \alpha_{j_1 + r - 2}) 
\end{equation*}
and 
\begin{equation*}
  \alpha^{(u)} := (\alpha_{j_u+r-1}, \dots, \alpha_{j_{u+1} + r - 2})
  \qt{for $u = 1, \dots, k$}. 
\end{equation*}
Note that the vector $\alpha^{(u)}$ has length exactly equal to $n_u$
for $u = 0, \dots, k$. 

Let $\M$ denote the class of all vectors $ \m := (\m_0, \dots,
\m_k)$ where each $\m_i$ is an integer with $1 \leq \m_i \leq k+1$ and
such that $\sum_{i=0}^{k} \m_i \leq 2(k+1)$. Because the number of
$(k+1)$-tuples of positive integers whose sum is equal to $p$ equals
$\binom{p-1}{k}$, it is easy to see that $\M$ is a finite set whose
cardinality $|\M|$ can be bounded as
\begin{align*}
  |\M| &\leq \sum_{p=k+1}^{2k+2} \binom{p-1}{k} \\
&= \sum_{l=k}^{2k+1} \binom{l}{l-k} \leq
  \sum_{l=k}^{2k+1} \binom{2k+1}{l-k} \leq 2^{2k+1} \leq 4^{k+1}. 
\end{align*}
Also let $\lc$ denote the class of all vectors $\ellb := (\ellb_0,
\dots, \ellb_k)$ where each $\ellb_i$ is an integer such that $j_0
\leq \ell_0 \leq j_1 - 1$ and $j_u+r-1 \leq \ell_u \leq j_{u+1} - 1$
for $u = 1, \dots, k$. The cardinality $|\lc|$ of $\lc$ is clearly
bounded   from above by $\prod_{u=0}^k n_u$.   

Let 
\begin{equation*}
  \delta := C_r \sqrt{n_0^{1 - 2r} + n_k^{1 - 2r} + \sum_{i=1}^{k-1}
    n_i^{1 - 2r} I \{\rs_i \neq \rs_{i+1}\}}  
\end{equation*}
where $C_r$ is the constant given by Lemma \ref{alta}. Note that
\begin{equation}\label{eq:delta}
\delta \leq C_r \delta_r(\theta)
\end{equation}
where $\delta_r(\theta)$ is as defined in \eqref{drt}. This is because
$n_i \geq n_{i*}$ for every $i$ and $1 - 2r < 0$. 

%Recall the quantity $\delta_r(\theta)$ from \eqref{drt} and note that
%because of the assumption \eqref{sia}, we have 
%\begin{equation*}
%  \delta_r(\theta) = \sqrt{\sum_{i=0}^{k}
%    n_i^{1 - 2r} I \{\rs_i \neq \rs_{i+1}\}}
%\end{equation*}
%Let $\delta := C_r \delta_r(\theta)$ where $C_r$ is as
%given by Lemma \ref{alta} and note then that $\delta$ represents
%precisely the right hand side of \eqref{alta.eq}. 
Also for $\ell \in
\lc$,  $\alpha \in \R^n$ and $0 \leq i \leq k$, let $\Gamma_i(\alpha, 
\ell_i)$ be defined (as in \eqref{alta.de}) as
\begin{align*}
& V_{j_i, j_{i+1} - 1}(D^{(r-1)} \alpha) -
  \rs_{i+1} \left\{(D^{(r-1)} \alpha)_{j_{i+1} - 1} - (D^{(r-1)}
      \alpha)_{\ell_{i}} \right\} - \\ & \rs_{i}
    \left\{(D^{(r-1)} \alpha)_{\ell_i} - (D^{(r-1)}
      \alpha)_{j_{i}} \right\}.  
\end{align*}
For every $\m, \q \in \M$ and $\ell \in \lc$, let $T(\m, \q, \ell)$
denote the set of all $\alpha \in \R^n$ with $\|\alpha\| \leq 1$ for
which  
\begin{equation*}
  \|\alpha^{(i)}\|^2 \leq \frac{\m_i}{k+1} ~~ \text{ and } ~~
  \Gamma_i(\alpha, \ell_i) \leq \frac{\q_i \delta}{k+1} 
\end{equation*}
for every $i = 0, 1, \dots, k$.
%For every $\m, \q \in \M$ and $\ellb \in \lc$, let $T(\m, \q, \ellb)$ 
%denote the set of all $\alpha \in T_{K^{(r)}}(\alpha)$ with
%$\|\alpha\| \le 1$ for which  x
%\begin{equation}\label{deftml}
%  \frac{\m_{u} - 1}{k+1}
%    \leq \|\alpha^{(u)} \|^2 \leq \frac{\m_u}{k+1} ~~~\text{ and } ~~~
%   (r_{u+1} - r_u) (D^{(r-1)} \alpha)_{\ell_u} \leq C n_u^{(1/2) - r} 
%   \sqrt{\frac{\m_u}{k+1}}  
%\end{equation}
%for all $0 \leq u \leq k$ where $C$ equals twice the universal
%constant that appears in Lemma \ref{delc}.  
We then claim that 
\begin{equation}\label{cuc}
  \left\{\alpha \in T_{\krv}(\theta) : \|\alpha\| \leq 1 \right\} 
  \subseteq \bigcup\limits_{\m, \q \in \M, \ellb \in \lc} T(\m, \q, \ellb).  
\end{equation}
To see \eqref{cuc}, note first that it follows from Lemma \ref{alta}
that for every $\alpha \in T_{\krv}(\theta)$ with $\|\alpha\| \le  
1$, there exists $\ell \in 
\lc$ such that $\sum_{i=0}^k \Gamma_{i}(\alpha, \ell_i) \leq
\delta$. This implies that for every $0 \leq i \leq k$, the inequality
$0 \leq \Gamma_i(\alpha, \ell_i) \leq \delta$ and so there exists 
an integer $1 \leq \q_i \leq k+1$ such that 
\begin{equation*}
\frac{(\q_i - 1) \delta}{k+1} \leq  \Gamma_i(\alpha, \ell_i) \leq
\frac{\q_i \delta}{k+1}.   
\end{equation*}
The integers $\q_0, \dots, \q_k$ would then have to satisfy
\begin{equation*}
  \delta \geq \sum_{i=0}^k \Gamma_i(\alpha, \ell_i) \geq \delta
  \sum_{i=0}^k \frac{\q_i - 1}{k+1} 
\end{equation*}
which is equivalent to $\sum_{i=0}^k \q_i \leq 2(k+1)$. Thus $\q =
(\q_0, \q_1, \dots, \q_k) \in \M$. Similarly, for each $0 \leq i \leq
k$, the inequality $1 \geq \|\alpha\|^2 \geq \|\alpha^{(i)}\|^2$ holds
so that there exists an integer $1 \leq \m_i \leq k+1$ such that  
\begin{equation*}
\frac{\m_i - 1}{k+1} \leq  \|\alpha^{(i)}\|^2 \leq \frac{\m_i}{k+1}. 
\end{equation*}
As $\sum_{i=0}^k \|\alpha^{(i)} \|^2 \leq 1$, the integers $\m_i,
0 \leq i \leq k$, satisfy $\sum_{i=0}^k \m_i \leq 2(k+1)$ which implies
that $\m = (\m_0, \dots, \m_k) \in \M$. This completes the proof of
\eqref{cuc}. $\vspace{0.08in}$

{\bf Step III}: In this step we find an upper bound of $G$ (in~\eqref{kann}) that depends on the collection $T(\m, \q, \ellb)$. 
Using~\eqref{cuc}, we can bound the quantity $G$ in \eqref{kann} via 
\begin{equation*}
  G \leq \E \left[ \max_{\m, \q \in \M, \ellb \in \lc} \sup_{\alpha \in T(\m,
    \q, \ellb)} \left<Z, \alpha \right> \right]. 
\end{equation*}
Since $Z$ is Gaussian, the first maximum in the right hand side
above can be taken outside the expectation up to an additional
correction term. We state this as a general result in Lemma
\ref{gt} (see the full statement and proof in Section
\ref{appa}). Note that each set $T(\m, \q, \ell)$ contains the zero 
vector and also that every vector in $T(\m, \q, \ell)$ has norm
bounded from above by 1. Therefore the quantity $D$ in Lemma \ref{gt}
can be taken to be 1 (also take $\sigma = 1$ in Lemma \ref{gt}). Lemma
\ref{gt} thus gives  
\begin{equation*}
  G \leq \max_{\m, \q \in \M, \ellb \in \lc} \E \left[ \sup_{\alpha \in T(\m,
    \q, \ellb)} \left<Z, \alpha \right> \right] +  \sqrt{4 \log  |\M| +
    2 \log |\lc|} +  \sqrt{\frac{\pi}{2}}.    
\end{equation*}
As $|\M| \leq 4^{k+1}$ and, by concavity of the logarithm,
\begin{align*}
  \log |\lc| \leq \sum_{i=0}^k \log n_i &= (k+1) \frac{1}{k+1}
  \sum_{i=0}^k \log n_i \\ &\leq (k+1) \log \left(\sum_{i=0}^k
  \frac{n_i}{k+1} \right) \\
&= (k+1) \log \frac{n}{k+1}  
  \leq (k+1) \log \frac{en}{k+1}, 
\end{align*}
we obtain (using also the fact that $\sqrt{2 + 4 \log 4} < 3$) that
\begin{equation}\label{neek}
  G \leq \max_{\m, \q \in \M, \ellb \in \lc} \E \left[ \sup_{\alpha \in T(\m,
    \q, \ellb)} \left<Z, \alpha \right> \right]+ 3  \sqrt{(k+1) \log
    \frac{en}{k+1}} +  \sqrt{\frac{\pi}{2}}.   
\end{equation}
We now fix $\m, \q \in \M$ and $\ellb \in \lc$ and attempt to bound 
\begin{equation*}%\label{gka}
  G(\m, \q, \ellb) := \E \left[ \sup_{\alpha \in T(\m, \q, \ellb)} \left<Z, \alpha
  \right> \right].  
\end{equation*}
We write $\left< Z, \alpha \right> = \sum_{i=0}^k \left<Z^{(i)}, 
  \alpha^{(i)} \right>$ so that $G(\m, \q, \ell) \leq \sum_{i=0}^k
G_i(\m, \q, \ell)$ where \\ $ G_i(\m, \q, \ell) :=   \E [\sup_{\alpha
  \in T(\m, \q, \ell)} \left<Z^{(i)}, \alpha^{(i)} \right> ].$ 
Let us now fix $0 \leq i \leq k$ and bound $G_{i}(\m, \q, \ell)$. By
the definition of $T(\m, \q, \ell)$, 
\begin{equation}\label{alad}
  G_i(\m, \q, \ell) \leq \E \left[\sup_{\alpha \in T^{(i)}(\m, \q, \ell)} \left<Z^{(i)},
    \alpha^{(i)} \right> \right]
\end{equation}
where $$ T^{(i)}(\m, \q,
\ell) :=\left\{\alpha \in \R^n: \|\alpha^{(i)}\|^2 \leq
  \frac{\m_i}{k+1}, 
    \Gamma_{i}(\alpha, \ell_i) \leq \frac{\q_i \delta}{k+1}
  \right\}. \vspace{0.08in}$$
%Because $T(\m, \q, \ellb)$ is a subset of $T_{K^{(r)}}(\alpha)$, we
%use  Lemma \ref{cort} here which gives a necessary condition for a
%vector to lie in $T_{K^{(r)}}(\alpha)$.  Lemma \ref{cort} implies
%that for every $\alpha \in T(\m, \ell)$, we have  
%  \begin{equation}\label{koti}
%    V_{j_u, j_{u+1} - 1} (D^{(r-1)} \alpha) \leq r_{u+1}
%    \left\{(D^{(r-1)} \alpha)_{j_{u+1} - 1} - (D^{(r-1)}
%      \alpha)_{\ell_{u}} \right\} + r_{u}
%    \left\{(D^{(r-1)} \alpha)_{\ell_u} - (D^{(r-1)}
%      \alpha)_{j_{u}} \right\} + \delta 
%  \end{equation}
% where $\delta := \sum_{u=0}^k \left(r_{u+1} - r_u \right) (D^{(r-1)}
%   \alpha)_{\ell_u}$. By definition of $T(\m, \ell)$ (specifically the second
%   inequality of \eqref{deftml}), we have 
% \begin{equation}\label{dbe}
%   \delta = \sum_{u=0}^k \left(r_{u+1} - r_u \right) (D^{(r-1)}
%   \alpha)_{\ell_u} \leq \frac{C}{\sqrt{k+1}} \sum_{u=0}^k n_u^{(1/2) - r}
%  \m_u^{1/2}. 
% \end{equation}
{\bf Step IV}: In this step, we describe how Lemma \ref{fina} can be
used to bound the right hand side in \eqref{alad}. Fix $0 \leq i \leq
k$. We do this by rewriting the underlying set  $ T^{(i)}(\m, \q,
\ell)$ in a form recognizable from Lemma \ref{fina}. For convenience,
let    
\begin{equation*}
  \delta_i := \frac{\q_i \delta}{k+1}. 
\end{equation*}
We claim that for every $0 \le i \leq k$ and $\alpha \in T^{(i)}(\m,
\q, \ell)$, we have  
 \begin{align}\label{lay}
  V(\Delta(\alpha^{(i)})) &\leq \rs_{i+1} \left((\Delta(\alpha^{(i)}))_{n_i - r + 1} - (\Delta(\alpha^{(i)}))_{\ell_i'}
  \right) \\ 
&+ \rs_i \left((\Delta(\alpha^{(i)}))_{\ell_i'} -
  (\Delta(\alpha^{(i)}))_{1} \right) + \delta_i \nonumber
\end{align} 
where $\Delta(\alpha^{(i)}) := D^{(r-1)} \alpha^{(i)}$ and $\ell_i'$
is related to $\ell_i$ via 
\begin{equation}\label{llp}
\ell_0 := \ell_0' ~~ \text{ and  } ~~ \ell_i = j_i + r - 2 + \ell_i'
~ \text{ for } ~ 1 \leq i \leq k. 
\end{equation}
Before proving \eqref{lay}, let us observe that the expected supremum
of $\left< Z^{(i)}, \alpha^{(i)} \right>$ over all $\alpha^{(i)}$
which satisfy the norm condition $\|\alpha^{(i)}\|^2 \leq \m_i/(k+1)$
and which satisfy \eqref{lay} can be controlled directly using Lemma
\ref{fina} (this is done in the next step). The argument for
\eqref{lay} goes as follows. Fix $\alpha \in T^{(i)}(\m, \q, \ell)$ and observe  that, from the
definition of $T^{(i)}(\m, \q, \ell)$, we have $\Gamma_i(\alpha,
\ell_i) \leq \delta_i$. For $i = 0$, inequality
\eqref{lay} is exactly the same as $\Gamma_0(\alpha, \ell_0) \leq  
\delta_0$. For $1 \leq i \leq k$, note that    
\begin{equation}\label{llpp}
  (D^{(r-1)} \alpha^{(i)})_{\ell} = (D^{(r-1)} \alpha)_{j_i + r - 2 +
    \ell}  \qt{for every $1 \leq \ell \leq n_i - r + 1$}, 
\end{equation}
which implies that $V(D^{(r-1)} \alpha^{(i)}) = V_{j_i + r - 1, j_{i+1}
- 1} (D^{(r-1)} \alpha)$ and that
\begin{equation*}
  V_{j_i, j_{i+1} - 1} (D^{(r-1)} \alpha) \geq V(D^{(r-1)}
  \alpha^{(i)}) + \rs_i \left((D^{(r-1)} \alpha)_{j_i + r - 1} -
    (D^{(r-1)} \alpha)_{j_i} \right). 
\end{equation*}
The notation $V_{a, b}(\cdot)$ may be recalled from \eqref{vab}. The
above inequality, together with $\Gamma_i(\alpha, \ell_i) \le 
\delta_i$, allows us to deduce that 
 \begin{align*}
  V(D^{(r-1)} \alpha^{(i)})  &\le \rs_{i+1}
    \left\{(D^{(r-1)} \alpha)_{j_{i+1} - 1} - (D^{(r-1)}
      \alpha)_{\ell_{i}} \right\} \\ 
&+ \rs_{i}
    \left\{(D^{(r-1)} \alpha)_{\ell_i} - (D^{(r-1)}
      \alpha)_{j_{i} + r - 1} \right\} + \delta_i.
\end{align*}
Using \eqref{llp} and \eqref{llpp}, it is now easy to see that the
above inequality is the same as \eqref{lay}. This proves
\eqref{lay}. $\vspace{0.08in}$ 

{\bf Step V}: Next, we use the characterization in~\eqref{lay} to
bound $G_i(\m, \q, \ell)$ using Lemma \ref{fina}. Indeed, we can take
$\sigma = 1, t = \sqrt{\m_i/(k+1)} \leq 1$, $n = n_i$, $\ell =
\ell_i'$ and $\delta = \delta_i$ in Lemma \ref{fina} to obtain      
 \begin{align*}
  G_i(\m, \q, \ell) &\leq C_r
   \left(\sqrt{\frac{\m_i}{k+1}} + \delta_i n_i^{(2r-1)/2} \right)
  \sqrt{\log (e n_i)} \\ &+ C_r  \left(\frac{\m_i}{k+1}
  \right)^{(2r-1)/(4r)} n_i^{(2r-1)/(4r)} \delta_i^{1/(2r)}
\end{align*}
for all $0 \leq i \leq k$.  Here $C_r$ is a constant that depends on
$r$ alone. This inequality, together with $G(\m, \q, \ell) \leq \sum_{i=0}^k
G_i(\m, \q, \ell)$, gives the following upper bound for $G(\m, \q,
\ell)/C_r$:    
\begin{align}\label{falc}
&\sum_{i=0}^k  \sqrt{\frac{\m_i}{k+1}} \sqrt{\log (e n_i)} +
\sum_{i=0}^k \delta_i 
    n_i^{(2r-1)/2} \sqrt{\log (e n_i)} \nonumber \\ &+ 
  \sum_{i = 0}^k \left(\frac{\m_i n_i}{k+1}
  \right)^{(2r-1)/(4r)}  \delta_i^{1/(2r)}  . 
\end{align}
We now bound separately each of the three terms above. For the first term, 
note that by the Cauchy-Schwarz inequality and the fact that $\sum_{i=0}^k
\m_i \leq 2 (k+1)$, we get
\begin{equation*}
  \sum_{i=0}^k  \sqrt{\frac{\m_i}{k+1}} \sqrt{\log (e n_i)} \leq
  \sqrt{\sum_{i=0}^k \frac{\m_i}{k+1}} \sqrt{\sum_{i=0}^k \log (e n_i)}
  \leq \sqrt{2} \sqrt{(k+1) \log \frac{en}{k+1}} 
\end{equation*}
where we have also used concavity of the logarithm function to claim
that $\sum_{i=0}^k \log (e n_i) \leq (k+1) \log \frac{en}{k+1}$. For
the second term in \eqref{falc}, we write 
 \begin{align*}
  \sum_{i=0}^k \delta_i n_i^{(2r-1)/2} \sqrt{\log (e n_i)} &\leq
  \max_{0 \leq i \leq k} \left[ n_i^{(2r-1)/2} \sqrt{\log (e n_i)} \right]
  \sum_{i=0}^k \delta_i \\ &\leq 2 \delta \max_{0 \leq i \leq k} \left[
  n_i^{(2r-1)/2} \sqrt{\log (e n_i)} \right]
\end{align*}
where we have used that $\sum_{i=0}^k \delta_i = \delta \sum_{i=0}^k \q_i/(k+1) \leq 2
\delta$. Assumption \eqref{sia} now gives 
\begin{equation*}
  \max_{0 \leq i \leq k} \left[
  n_i^{(2r-1)/2} \sqrt{\log (e n_i)} \right] \leq 2^r \left(\frac{n}{k+1}
\right)^{(2r - 1)/2} \sqrt{\log \frac{en}{k+1}}. 
\end{equation*}
We thus obtain 
\begin{equation*}
  \sum_{i=0}^k \delta_i n_i^{(2r-1)/2} \sqrt{\log (e n_i)} \leq 2^{1
  + r} \delta \left(\frac{n}{k+1}
\right)^{(2r - 1)/2} \sqrt{\log \frac{en}{k+1}}. 
\end{equation*}
For the third term in \eqref{falc}, we use the standard Holder's
inequality 
($\sum_i \alpha_i \beta_i \leq (\sum_i \alpha_i^p)^{1/p} (\sum_i
\beta_i^q)^{1/q}$ with $p = 2r/(2r - 1)$ and $q = 2r$) to obtain  
\begin{align*}
\sum_{i = 0}^k \left(\frac{\m_i n_i}{k+1}
  \right)^{(2r-1)/(4r)}  \delta_i^{1/(2r)} &\leq  \left(\sum_{i=0}^k \sqrt{\frac{\m_i n_i}{k+1}} \right)^{(2r-1)/(2r)}  \left(\sum_{i=0}^k \delta_i \right)^{1/(2r)} \\
%&\leq (2 \delta)^{1/(2r)} \left(\sum_{i=0}^k \sqrt{\frac{\m_i n_i}{k+1}} \right)^{(2r-1)/(2r)} \\
&\leq 2^{1/(2r)} \delta^{1/(2r)} \left(\sqrt{\sum_{i=0}^k \frac{\m_i}{k+1}
  \sum_{i=0}^k n_i} \right)^{(2r-1)/(2r)} \\
&\leq 2^{(2r+1)/(4r)}
\delta^{1/(2r)} n^{(2r-1)/(4r)}
\end{align*}
where, in the second inequality above, we used $\sum_{i=0}^k \delta_i
\le 2 \delta$ and the Cauchy-Schwarz inequality and, in the final
inequality, we used $\sum_{i=0}^k \m_i \leq 2(k+1)$ and $\sum_{i=0}^k
n_i = n$. Putting the bounds for the three terms in \eqref{falc}
together, we obtain 
\begin{eqnarray}\label{bade}
  \frac{G(\m, \q, \ell)}{C_r } & \leq & \sqrt{2 (k+1) \log
    \frac{en}{k+1}} + 2^{1 + r} \delta \left(\frac{n}{k+1}
      \right)^{(2r - 1)/2} \sqrt{\log \frac{en}{k+1}} \nonumber \\
      && \qquad \qquad \qquad \qquad  \qquad + \;
      2^{(2r+1)/(4r)} \delta^{1/(2r)} n^{(2r-1)/(4r)}
\end{eqnarray}
which gives (note also that $\delta \leq C_r \delta_r(\theta)$ by~\eqref{eq:delta}) 
\begin{equation*}
  G(\m, \q, \ell) \leq c_r \sqrt{n \Delta_r(\theta)},
\end{equation*}
for a suitable constant $c_r$ depending only on $r$; note that
$\sqrt{a} +\sqrt{b} + \sqrt{c} \le \sqrt{3} \sqrt{a + b + 
  c}$ for $a, b, c > 0$. Combined with \eqref{neek}, this completes
the proof of  \eqref{keytc} when assumption \eqref{sia} is true.
$\vspace{0.08in}$  

{\bf Step VI}: Now we work with the situation when the assumption \eqref{sia} is
violated. Our basic idea here is that we will add indices to the set
of knots $j_1, \dots, j_k$ to create a new set of indices which
contains all the knots of $\theta$ and which satisfies an assumption
similar to \eqref{sia}. Specifically for every $i \geq 1$ for which
$n_i$ is strictly larger than $2n/(k+1)$, we add the indices   
\begin{equation*}
  j_i + \left \lfloor \frac{2n}{k+1} \right \rfloor,   j_i + 2 \left \lfloor \frac{2n}{k+1}
 \right  \rfloor , \dots,   j_i + A_i \left \lfloor \frac{2n}{k+1}
\right \rfloor 
\end{equation*}
to the original set of knots, where $\lfloor x \rfloor$ denotes the largest integer less than or equal to $x$. Here $A_i$ is the
integer part of the ratio of $n_i$ to $\lfloor 2n/(k+1) \rfloor$ and
hence
\begin{equation*}
  A_i \leq n_i \left \lfloor \frac{2n}{k+1} \right \rfloor^{-1}. 
\end{equation*}
Similarly, if $n_0 \geq 2n/(k+1)$, then we add the indices  $2-r + \left \lfloor \frac{2n}{k+1} \right \rfloor,   2-r + 2 \left \lfloor \frac{2n}{k+1}
 \right  \rfloor , \dots,   2-r + A_0 \left \lfloor \frac{2n}{k+1}
\right \rfloor$ to the original set of knots $j_1, \dots, j_k$ where again $ A_0 \leq n_0 \left \lfloor \frac{2n}{k+1} \right \rfloor^{-1}.$
This construction will create a set of indices $j_1' < \dots <
j_{k'}'$ that contains all the original knots and which satisfy 
\begin{equation}\label{nip}
  n_i' \leq \frac{2n}{k+1} \qt{for every $i = 0, \dots, k'$}
\end{equation}
where $n_i'$ are defined with respect to $j_1' < \dots < j_{k'}'$ as
$n_0' = j_1' + r - 2$ and $n_i' := j_{i+1}' - j'_i$ for $i = 1, \dots,
k'$. We now note that the number of these new indices, $k'$, satisfies 
\begin{equation*}
  k' \leq k + \sum_{i=0}^k A_i \leq k + \left \lfloor \frac{2n}{k+1}
  \right \rfloor^{-1} \sum_{i=0}^k n_i = k + n \left \lfloor \frac{2n}{k+1}
  \right \rfloor^{-1}  \leq 2k + 1
\end{equation*}
where we have used $\left \lfloor \frac{2n}{k+1} \right \rfloor \geq \frac{2n}{k+1} - 1 \geq
  \frac{n}{k+1}.$
The inequality $k' \leq 2k + 1$, along with \eqref{nip}, implies that 
\begin{equation}\label{siam}
  n_i' \leq \frac{4n}{k'+1} \qt{for every $i = 0, \dots, k'$}. 
\end{equation}
For these indices $j_1', \dots, j_{k'}'$, we shall assign signs
$\rs'_1, \dots, \rs'_{k'} \in \{-1, 0, 1\}$ in the following way. If
$j_i'$ is a knot (i.e., it is one of $j_1, \dots, j_k$), then $\rs'_i$
equals the sign of the knot $j_i$. If $j_i'$ is not a knot, then we
assign $\rs_i'$ to be the sign of the nearest knot that is to the right
of $j_i'$ (if there is no knot to the right of $j_i'$, then we take
$\rs_i'$ to be zero). 

We now go through the previous proof (which was under the case when assumption \eqref{sia} is satisfied) with the set of indices
$j_1', \dots, j_{k'}'$ and signs $\rs'_1, \dots, \rs'_{k'}$. Instead
of \eqref{sia}, we use the inequality \eqref{siam} which has a
slightly worse constant $4$ instead of $2$. This argument will end
with an inequality similar to \eqref{bade} (but with slightly
different constants). Thus we obtain the following upper bound for
$G$:  
\begin{equation}\label{fion}
  \frac{G^2}{C_r} \le (k' + 1) \log \frac{en}{k' + 1} + (\delta')^2
  \left(\frac{n}{k'+1} \right)^{2r - 1} \log \frac{en}{k'+1} +
    (\delta')^{1/r} n^{(2r-1)/(2r)}
\end{equation}
for a constant $C_r$ where 
\begin{equation*}
  \delta' := \left(\sum_{i=0}^{k'}
  (n'_{i})^{1-2r} I\{\rs'_i \neq \rs'_{i+1}\} \right)^{1/2}. 
\end{equation*}
Now because $k \leq k' \leq 2k+1$, we have $k+1 \leq k'+1 \le 2(k+1)$
and thus we can replace the $k'$ on the right hand side of
\eqref{fion} by $k$ by enlarging the constant $C_r$ slightly. Finally,
to complete the proof, it suffices to observe that because of the
construction of the set of indices $j_1', \dots, j_{k'}'$ and the
choice of the signs, we have  
\begin{align*}
  (\delta')^2 &=  \sum_{i=0}^k \left[ \min\left(n_i, \left \lfloor
      \frac{2n}{k+1} \right \rfloor \right) \right]^{1-2r} I \{\rs_i \neq
    \rs_{i+1}\} \\
& \leq \sum_{i=0}^k n_{i*}^{1-2r} I \{\rs_i \neq \rs_{i+1}\} \\ &=
  n_{0*}^{1 - 2r} + n_{k*}^{1 - 2r} + \sum_{i=1}^{k-1} n_{i*}^{1-2r} I
  \{\rs_i \neq \rs_{i+1}\} = \delta^2_{r}(\theta). 
\end{align*}
This, along with replacing $k'$  by $k$ in \eqref{fion}, completes the
proof of Theorem \ref{ada}. 
\end{proof}

\subsection{Proof of Corollary \ref{kco}} 
We prove Corollary \ref{kco} as a consequence of Theorem \ref{ada} and
Theorem \ref{woor}. The following lemma (proved in Subsection
\ref{vava}) will be needed for this. 

%\vspace{0.2in} 

\begin{lemma}\label{lemco}
  Fix $r \geq 1$ and $n \geq r+1$. For every $\theta \in \R^n$, there
  exists $\eta \in \R^n$ such that 
  \begin{equation}\label{lemco.eq}
\|D^{(r)} \eta\|_1 = 0  ~~ \text{   and   } ~~  \|\theta -
\eta\|^2 \leq n^{2r - 1} \|D^{(r)} \theta\|_1^2. 
  \end{equation}
\end{lemma}

%%\vspace{0.2in} 

\begin{remark}
  When $r = 1$, the inequality in \eqref{lemco.eq} is equivalent to 
\begin{equation*}
  \sum_{i=1}^n \left(\theta_i - \bar{\theta} \right)^2 \leq  n \|D
  \theta\|^2 = n V^2(\theta)
\end{equation*}
which relates the variance of $\theta_1, \dots, \theta_n$ to the
variation (here $\bar{\theta} := (\theta_1 + \dots +
\theta_n)/n$). Therefore Lemma \ref{lemco} can be seen as a relation 
between variance and variation for general $r \geq 1$.  
\end{remark}

We are now ready to prove Corollary \ref{kco}. 
\begin{proof}[Proof of Corollary \ref{kco}]
  Recall from \eqref{kv} that $V^{(r)}(\theta^*) = n^{r-1}
  \|D^{(r)} \theta^* \|_1$. We first consider the case when $V^{(r)}(\theta^*) >
  0$. In this case, we use Lemma \ref{lemco} to claim the existence of
  $\eta^* \in \R^n$ such that $V^{(r)}(\eta^*) = 0$ and 
  \begin{equation}\label{ghul}
    \|\theta^* - \eta^*\|^2 \leq  n^{2r-1} \|D^{(r)} \theta^*\|_1^2
    = n \left(V^{(r)}(\theta^*) \right)^2.  
  \end{equation}
  Let now $\theta \in \R^n$ be defined as 
  \begin{equation*}
    \theta := \eta^* + \frac{V}{V^{(r)}(\theta^*)} \left(\theta^* - \eta^*
    \right). 
  \end{equation*}
As $V^{(r)}(\eta^*) = 0$, it follows that $V^{(r)}(\theta) =
  V$. We deduce therefore that 
  \begin{equation}\label{boish}
\inf_{\alpha \in \R^n : V^{(r)}(\alpha) = V} \left\{\frac{1}{n}
  \|\theta^* - \alpha\|^2 + C_r \sigma^2 \Delta_r(\alpha) \right\}
 \end{equation}
is bounded from above by 
  \begin{align*}
 \frac{1}{n} \|\theta^* - \theta\|^2 + C_r \sigma^2
\Delta_r(\theta) &= \frac{1}{n} \|\theta^* - \eta^*\|^2 \left(1 -
  \frac{V}{V^{(r)}(\theta^*)} \right)^2 + C_r \sigma^2 \Delta_r(\theta) \\
&\leq \left(V - V^{(r)}(\theta^*) \right)^2 + C_r \sigma^2
\Delta_r(\theta)
  \end{align*}
where the last inequality above follows from \eqref{ghul}. We now note
that, by construction, $\theta$ satisfies the minimum 
length condition \eqref{nsa} with the same constant $c$ because
$\theta^*$ does so. As a consequence, we have from \eqref{ghos} that 
\begin{equation*}
  \Delta_r(\theta) \leq C_r(c) \frac{k+1}{n} \log \frac{en}{k+1} 
\end{equation*}
where $C_r(c)$  depends on $r$ and $c$ alone and $k = \ktr(\theta) =
\ktr(\theta^*)$. We have thus shown that \eqref{boish} is bounded from
above by
\begin{equation*}
\left(V - V^{(r)}(\theta^*) \right)^2 + C_r(c) \sigma^2
\frac{k+1}{n} \log \frac{en}{k+1}. 
\end{equation*}
Inequality \eqref{kco.eq} then directly follows from Theorem
\ref{ada}.  

We now assume that $V^{(r)}(\theta^*) = 0$.  Here we have
$\ktr(\theta^*) = 0$ so that the second term on the right hand side of
\eqref{kco.eq} becomes $\frac{C_r\sigma^2}{n} \log (en)$. Note also
that because $V^{(r)}(\theta^*) = 0$ and $V \geq 0$, we can use
Theorem \ref{woor}. To complete the proof, we therefore only need to
prove that 
\begin{equation}\label{stta}
C_r \max \left( \left(\frac{\sigma^2 V^{1/r}}{n} 
      \right)^{2r/(2r + 1)} ,
      \frac{\sigma^2}{n} \log (en)\right) \leq V^2 +
    \frac{C^{(2r+1)/(2r)}_r \sigma^2}{n} \log (en).    
\end{equation}
To prove the above inequality, we may assume that 
\begin{equation*}
  \left(\frac{\sigma^2 V^{1/r}}{n}  \right)^{2r/(2r + 1)}  >
    C_r^{1/(2r)} \frac{\sigma^2}{n} \log (en)
\end{equation*}
for otherwise \eqref{stta} is trivial. It is now straightforward to
check that the above inequality is equivalent to 
\begin{equation*}
  \frac{\sigma^2}{n} < \frac{V^2}{(\log en)^{2r+1}} C_r^{-(2r+1)/(2r)}. 
\end{equation*}
From here, it is easy to show that 
\begin{equation*}
  C_r \left(\frac{\sigma^2 V^{1/r}}{n}  \right)^{2r/(2r + 1)} \leq
    \frac{V^2}{(\log en)^{2r}} \leq V^2
\end{equation*}
which proves \eqref{stta}. This completes the proof of \eqref{kco.eq} 
when $V^{(r)}(\theta^*) = 0$. Inequality
\eqref{kcoc} trivially follows from \eqref{kco.eq}. 
\end{proof}

\subsection{Proof of Lemma \ref{simpo}}
Here we provide the proof of the Lemma \ref{simpo} which implies that
the  $\log \frac{en}{\ktr(\theta^*) +
  1}$ appearing in our risk bounds cannot be completely removed. This
proof uses the fact \eqref{losili} as well as the precise
characterization of the tangent cones of the set $K^{(r)}(V)$ (defined
in \eqref{krv}) given in Lemma \ref{tcr} (in Subsection \ref{tcc}).  

\begin{proof}[Proof of Lemma \ref{simpo}]
Let $\theta^* = (0, \dots, 0, 1, \dots, 1)$ where the jump appears at
the index $j := \lceil n/2 \rceil$ (i.e., $\theta^*_j = 1$ and $\theta_{j-1}^* =
0$). The estimator $\hat{\theta}^{(1)}_{V = 1}$ is simply the least
squares projection of $Y$ onto the closed convex set $K^{(1)}(V)$
(defined in \eqref{krv}) with $V = 1$. The identity \eqref{losili}
therefore gives
\begin{equation*}
  \lim_{\sigma \downarrow 0} \frac{1}{\sigma^2} R(\hat{\theta}_{V
    =1}^{(1)}, \theta^*) = \frac{1}{n} \delta(T_{K^{(1)}(V)}(\theta^*)). 
\end{equation*}
 The characterization of $T := T_{K^{(1)}(V)}(\theta^*)$ from Lemma
 \ref{tcr} implies, for this specific $\theta^*$, that $T$ consists of
 all vectors $\alpha \in \R^n$ for which 
 \begin{equation}\label{tcgan}
   V_{1, j-1}(\alpha) + V_{j, n}(\alpha) \leq \alpha_{j-1} -
   \alpha_j
 \end{equation}
  where $V_{1, j-1}(\alpha)$ and $V_{j, n}(\alpha)$ are defined as in
  \eqref{vab}. Now let $\mathfrak{M}$ consist of all vectors $\alpha
  \in \R^n$ which satisfy: 
  \begin{equation*}
    0 = \alpha_1 \leq \alpha_2 \leq \dots \leq \alpha_{j-1} 
  \end{equation*}
   and 
   \begin{equation*}
     \alpha_j \le \alpha_{j+1} \leq \dots \leq \alpha_n = 0. 
   \end{equation*}
   Note that there is no relation between $\alpha_{j-1}$ and
   $\alpha_j$ in the definition of $\mathfrak{M}$. Then, it is easy to
   check directly that every vector $\alpha \in \mathfrak{M}$
   satisfies \eqref{tcgan} so that 
   \begin{equation*}
     \delta(T) \geq \delta(\mathfrak{M}). 
   \end{equation*}
    This follows from the fact that $\delta(C_1) \le \delta(C_2) $
    whenever $C_1$ and $C_2$ are two closed convex cones such that
    $C_1 \subseteq C_2$ (this fact is stated, for example,
    \citet[Subsection 3.1]{amelunxen2014living}). Now if
    \begin{equation*}
      \mathfrak{M}_1 := \left\{(\alpha_1, \dots, \alpha_{j-1}) : 0 =
        \alpha_1 \leq \alpha_2 \leq \dots \le \alpha_{j-1}  \right\} 
    \end{equation*}
     then it is clear that $\delta(\mathfrak{M}) \geq
     \delta(\mathfrak{M}_1)$ so that we have
     \begin{equation*}
     \delta(T) \geq \delta(\mathfrak{M}_1). 
     \end{equation*}
      We now use the fact that  $\delta(\mathfrak{M}_1)$ is precisely
      known to satisfy (see \citet[Equation (D.12) in Subsection
      D.4]{amelunxen2014living}) 
      \begin{equation*}
        \delta(\mathfrak{M}_1) = \frac{1}{2} \left(1 + \frac{1}{2} +
          \dots + \frac{1}{j-1} \right). 
      \end{equation*}
     This therefore implies (via $1 + (1/2) + \dots + (1/m) \geq
     \log(m+1)$) that 
\begin{equation*}
  \lim_{\sigma \downarrow 0} \frac{n}{\sigma^2} R(\hat{\theta}_{V
    =1}^{(1)}, \theta^*) = \delta(T_{K^{(1)}(V)}(\theta^*)) \geq
  \delta(\mathfrak{M}_1) \geq \frac{1}{2} \log(j) \geq \frac{1}{2} \log(n/2)
\end{equation*}
which proves Lemma \ref{simpo}.  
\end{proof}

\subsection{Proof of Theorem \ref{prt1}}
The following proposition is the key to proving Theorem
\ref{prt1}. This proposition provides an upper bound for
$\D(\lambda \partial g(\theta))$ for a convex function $g$ (recall
that $\D(\cdot)$ is defined as in \eqref{deef}  and also that
$\dist(z, \C) := \inf_{x \in \C} \|z - x\|$ for $z \in \R^n$ a subset
$\C \subseteq \R^n$) in terms of the smaller quantity
$\D(\cone(\partial g(\theta)))$. It is a generalization of
\citet[Proposition 1]{foygel2014corrupted}. Indeed, this latter result
of \cite{foygel2014corrupted} is the special case of Proposition
\ref{qfmg} under the additional assumption that $v_0 \in \partial
g(\theta)$ (this assumption does not necessarily hold for $g(\theta)
:= n^{r-1} \|D^{(r)} \theta\|_1$ when $r \ge 2$).  

\begin{proposition}\label{qfmg}
   Suppose $g : \R^n \rightarrow \R$ is a convex function and $\theta
  \in \R^n$. Suppose that the vector $v_0$ defined by 
  \begin{equation}\label{rdw} 
    v_0 := \argmin_{v \in \aff(\partial g(\theta))} \|v\|. 
  \end{equation}
 is a non-zero vector in $\R^n$.  Then for every $z \in \R^n$, 
   \begin{equation}\label{lzde}
     \lambda(z) := \argmin_{\lambda \geq 0} \dist(z, \lambda \partial
     g(\theta)) = \argmin_{\lambda \ge 0} \inf_{v \in \partial
       g(\theta)} \|z - \lambda v \|
   \end{equation}
    exists uniquely and, moreover, $\E \lambda(Z) < \infty$ where the
    expectation is taken with respect to $Z \sim N(0, I_n)$. 

    Further, let 
  \begin{equation*}
    \lambda^* := \E \lambda(Z) + \frac{2}{\norm{v_0}} \qt{where $Z
      \sim N(0, I_n)$}. 
  \end{equation*}
  Then for every $\lambda \geq \lambda^*$ and $v^* \in \partial
  g(\theta)$, we have 
  \begin{equation}\label{qfmg.eq}
    \D(\lambda \partial g(\theta)) \leq 4 +
    \left(\sqrt{\D(\cone(\partial g (\theta))} + 
      \frac{4 \|v^*\|}{\|v_0\|} + 2 +  (\lambda -
      \lambda^*) \norm{v^*} \right)^2.
  \end{equation}
\end{proposition}

Before proving Proposition \ref{qfmg}, let us first show how Theorem
\ref{prt1} follows from Proposition \ref{qfmg}. The fact
\eqref{annpre} and the bound \eqref{rew} (which was proved in
Subsection \ref{bsi}) will be used in the proof below. 

\begin{proof}[Proof of Theorem \ref{prt1}]
   Let $f(\theta) := \|D^{(r)} \theta\|_1$ and $g(\theta) := n^{r-1}
   f(\theta)$. Because $\hat{\theta}^{(r)}_{\lambda}$ equals the
   penalized estimator \eqref{gepn}, Theorem \ref{ohala} gives 
   \begin{equation}\label{ohalag}
     R(\hat{\theta}_{\lambda}^{(r)}, \theta^*) \leq \frac{\sigma^2}{n}
     \D(\lambda \partial g(\theta^*)). 
   \end{equation}
   We now use inequality \eqref{qfmg.eq} in Proposition \ref{qfmg} to
   bound the right hand side above. Note that under the assumption
   $D^{(r)} \theta^* \neq 0$, we observed (after \eqref{vvs}) that
   $v_0$ is non-zero so that Proposition \ref{qfmg} is
   applicable. This gives 
   \begin{equation*}
     \D(\lambda \partial g(\theta^*)) \leq 4 +
    \left(\sqrt{\D(\cone(\partial g (\theta^*))} + 
      \frac{4 \|v^*(g)\|}{\|v_0(g)\|} + 2 +  (\lambda -
      \lambda^*) \norm{v^*(g)} \right)^2
   \end{equation*}
for every $\lambda \geq \lambda^*(g)$ where 
\begin{equation*}
  \lambda^*(g) := \E \argmin_{\lambda \geq 0} \dist(Z,
  \lambda \partial g(\theta^*)) + \frac{2}{\|v_0(g)\|}, 
\end{equation*}
$v_0(g)$ is defined as in \eqref{rdw} and 
   \begin{equation*}
     v^*(g) := \argmin_{v \in\partial g(\theta)} \|v\|. 
   \end{equation*}
   Note that \eqref{qfmg.eq} holds for every $v^* \in \partial
   g(\theta^*)$ so it holds, in particular, for $v^*(g)$. Also note
   that $v^*(g)$ and $v_0(g)$ above are different from
   $v_0$ and $v^*$ in Theorem \ref{prt1} which are all
   defined in terms of $f$. Now the relation $g = n^{r-1} f$ implies
   that 
   \begin{equation*}
     v_0 = \frac{v_0(g)}{n^{r-1}}, ~~~ v^* = \frac{v^*(g)}{n^{r-1}}
     ~~~ \text{ and } ~~~ 
     \lambda^*(g)= \lambda^*.  
   \end{equation*}
   Note that $\lambda^*(g) = \lambda^*$ holds without any scaling
   factor because of the presence of the $n^{1-r}$ factor in the
   definition of $\lambda^*$ in \eqref{lams}. We have therefore proved
   that for every $\lambda \geq \lambda^*$, we have
   \begin{align*}
     \D(\lambda \partial g(\theta^*)) &\leq 4 +
    \left(\sqrt{\D(\cone(\partial g (\theta^*))} + 
      \frac{4 \|v^*\|}{\|v_0\|} + 2 +  \frac{(\lambda -
      \lambda^*)}{n^{1-r}} \norm{v^*} \right)^2 \\
&\leq 20 + 4 \D(\cone(\partial g (\theta^*)) + \frac{64
  \|v^*\|^2}{\|v_0\|^2}+ \frac{4 (\lambda - \lambda^*)^2}{n^{2 - 2r}} \|v^*\|^2
   \end{align*}
where, in the last inequality, we used the elementary fact $(a + b + c
+ d)^2 \leq 4(a^2 + b^2 + c^2 + d^2)$. Note now that 
\begin{equation*}
  \D (\cone(\partial g (\theta^*)) = \D (\cone(\partial f
  (\theta^*))
\end{equation*}
so that, by inequality \eqref{annpre}, we deduce that 
\begin{equation*}
  \D (\cone(\partial g (\theta^*)) \leq 1 + w^2
  (T_{K^{(r)}(V^*)}(\theta^*)). 
\end{equation*}
The bound \eqref{rew} then gives 
\begin{equation*}
  \D (\cone(\partial g (\theta^*)) \leq 1 + C_r^2 n
  \Delta_r(\theta^*). 
\end{equation*}
Putting the above pieces together (and the fact that
$\Delta_r(\theta^*) \geq 1/n$), we obtain
\begin{equation*}
  \D(\lambda \partial g(\theta^*)) \leq C_r n \Delta_r(\theta^*) + \frac{64
  \|v^*\|^2}{\|v_0\|^2}+ \frac{4 (\lambda - \lambda^*)^2}{n^{2 - 2r}}
\|v^*\|^2
\end{equation*}
for every $\lambda \geq \lambda^*$. Combining this with
\eqref{ohalag} gives \eqref{prt1.eq} and completes the proof of
Theorem \ref{prt1}. 
\end{proof}

We now give the proof of Proposition \ref{qfmg}. 
\begin{proof}[Proof of Proposition \ref{qfmg}]
   Note first that $\partial g(\theta)$ cannot contain the zero vector
   because we assumed that $v_0$ (defined by \eqref{rdw}) is
   non-zero. As a result, it follows from \citet[Corollary 
9.6.1]{Rockafellar70book} that  
   \begin{equation*}
     \cone(\partial g(\theta)) := \bigcup_{\lambda \geq 0}
    \left(\lambda \partial g(\theta) \right) 
   \end{equation*}
   is closed (and, of course, a convex cone). It follows therefore
   that 
   \begin{equation*}
     \Pi_{\cone(\partial g(\theta))}(z) := \argmin_{u \in
       \cone(\partial g(\theta))} \|z - u\|
   \end{equation*}
   exists uniquely. Let $\Pi_{\cone(\partial g(\theta))}(z) := \lambda_1
   v_1$ for some $\lambda_1 \geq 0$ and $v_1 \in \partial g(\theta)$. Then
   it is clear that $\lambda_1$ minimizes $\dist(z, \lambda \partial
   g(\theta))$ over $\lambda \geq 0$. To prove that $\lambda_1$ is the
   unique minimizer, assume, if possible, the existence of $\lambda_2
   \geq 0$ and $v_2 \in \partial g(\theta)$ such that
   $\lambda_1 v_1 = \lambda_2 v_2$. Note now that because $\aff(\partial
   g(\theta))$ is an affine set, the vector $v_0$ defined by
   \eqref{rdw} (which is the projection of the zero vector onto
   $\aff(\partial g(\theta))$) satisfies the orthogonality property: 
  \begin{equation}\label{fomc}
    \left<v - v_0, v_0 \right> = 0 \qt{for every $v \in \partial
      g(\theta)$}. 
  \end{equation}
   In particular, we have $\left<v, v_0 \right> = \|v_0\|^2$ for every
   $v \in \partial g(\theta)$. Applying this to $v = v_1$ and $v =
   v_2$, we obtain that 
   \begin{equation*}
     \lambda_1 \|v_0\|^2 = \left<\lambda_1 v_1 , v_0 \right> =
     \left<\lambda_2 v_2, v_0 \right> = \lambda_2 \|v_0\|^2
   \end{equation*}
   which implies that $\lambda_1 = \lambda_2$. This proves therefore
   that there is a unique $\lambda_1 \geq 0$ for which
   $\Pi_{\cone(\partial g(\theta))}(z) \in \lambda_1 \partial
   g(\theta)$  and this $\lambda_1$ clearly is equal to $\lambda(z)$
   defined in \eqref{lzde}. 
 
   To prove that $\E \lambda(Z) < \infty$ for $Z \sim N_n(0, I_n)$,
   we write $\Pi_{\cone(\partial g(\theta))} z = \lambda(z) v(z)$  for
   some $v(z) \in \partial g(\theta)$ and use \eqref{fomc} to obtain
   \begin{align*}
     \lambda(z) &= \frac{1}{\|v_0\|^2} \left<\lambda(z) v(z), v_0
     \right> = \frac{1}{\|v_0\|^2} \left<\Pi_{\cone(\partial
         g(\theta))}(z), v_0 \right> \leq \frac{\|\Pi_{\cone(\partial
         g(\theta))}(z) \|}{\|v_0\|}
   \end{align*}
    where the last inequality follows from the Cauchy-Schwarz
    inequality. The standard fact that the projection onto a closed
    convex cone reduces norm gives $\|\Pi_{\cone(\partial
         g(\theta))}(z) \| \leq \|z\|$ so that $\lambda(z) \leq
       \|z\|/\|v_0\|$ which implies obviously that $\E \lambda(Z) <
       \infty$ when $Z \sim N_n(0, I_n)$. 

  Let us now proceed to prove \eqref{qfmg.eq}. The first step for this
  is to observe that the map $z \mapsto \lambda(z) = 
  \argmin_{\lambda \geq 0} \dist(z, \lambda \partial g(\theta))$ is
  Lipschitz with parameter $1/\|v_0\|$ i.e., 
  \begin{equation}\label{dpat}
    |\lambda(z_1) - \lambda(z_2) | \leq \frac{\|z_1 -
      z_2\|}{\|v_0\|} \qt{for every $z_1, z_2 \in \R^n$}. 
  \end{equation}
   To see this, fix $z_1, z_2 \in \R^n$ and let $\Pi_{\cone(\partial
     g(\theta))} (z_i) = \lambda(z_i) v_i$ for two vectors $v_1, v_2
   \in \partial  g(\theta)$. Then,  by the contraction property for
   projections on closed convex cones, we have
   \begin{align*}
     \|z_1 - z_2\| &\geq  \|\lambda(z_1) v_1 - \lambda(z_2) v_2 \| \\
&= \|(\lambda(z_1) - \lambda(z_2)) v_0 + \lambda(z_1) (v_1 - v_0) -
  \lambda(z_2) (v_2 - v_0) \| \\
&= \|(\lambda(z_1) - \lambda(z_2)) v_0\| + \|\lambda(z_1) (v_1 - v_0) -
  \lambda(z_2) (v_2 - v_0) \|
   \end{align*}
   where the last equality follows from the orthogonality property
   \eqref{fomc}. Because the last term above is nonnegative, the
   inequality \eqref{dpat} follows. 
     
   The Lipschitz property of $z \mapsto \lambda(z)$ proved above
   implies, by standard Gaussian concentration, that
   \begin{align*}
     \P \left\{ \left|\lambda(z) - \E \lambda(Z) \right| <
     \frac{2}{\|v_0\|} \right\} \geq 1 - 2 e^{-2}. 
   \end{align*}
   Let $E := \{z \in \R^n : |\lambda(z) - \E \lambda(Z)| <
   2/\|v_0\|\}$ so that $\P\{z \in E\} \ge 1 - 2e^{-2}$.  Note that $0 \leq \lambda(z) < 
    \lambda^*$ when $z \in E$. This implies that for every $\lambda \geq \lambda^*$
    and vectors $v, v^*  \in \partial g(\theta)$, we have (by convexity of the
    subdifferential $\partial g(\theta)$)
    \begin{align*}
      \frac{\lambda(z)}{\lambda} v + \left(1 -
      \frac{\lambda(z)}{\lambda} \right) v^* \in \partial g(\theta). 
    \end{align*}
    In particular, this is true with $v = v(z)$ where
    $\Pi_{\cone(\partial g(\theta))}(z) := \lambda(z) v(z)$. As a
    result,  
    \begin{align*}
      \dist(z, \lambda \partial g(\theta))  &\leq \|z - \lambda(z) v(z) -
      (\lambda - \lambda(z)) v^* \| \\ 
&\leq \|z - \lambda(z) v(z)\| + \left(\lambda - \lambda(z) \right)
  \|v^*\| \\ 
&= \dist(z, \cone(\partial g(\theta))) + (\lambda -
  \lambda(z)) \|v^*\|. 
    \end{align*}
   Now, again for $z \in E$, we have $\lambda(z) > \E \lambda(Z) -
   2/\|v_0\|$ so that 
   \begin{align*}
     \lambda - \lambda(z) \leq \lambda - \E \lambda(Z) +
     \frac{2}{\|v_0\|} = \lambda - \lambda^* + \frac{4}{\|v_0\|}. 
   \end{align*}
   We have therefore proved that 
   \begin{align*}
     \dist(z, \lambda \partial g(\theta)) \leq \dist(z, \cone(\partial
     g(\theta))) + (\lambda - \lambda^*) \|v^*\| + \frac{4
     \|v_0\|}{\|v^*\|}
   \end{align*}
   for $z \in E$ which further implies that the probability 
   \begin{align*}
     \P \left\{\frac{1}{2} \dist(Z, \lambda \partial g(\theta)) -
     \frac{1}{2} \dist(Z, \cone(\partial g(\theta))) >
     \frac{2\|v^*\|}{\|v_0\|} + \frac{1}{2} (\lambda - \lambda^*)
     \|v^*\| \right\}
   \end{align*}
   is bounded from above by $2 e^{-2}$.  We now use \citet[Lemma
   4]{foygel2014corrupted} to claim that  
    \begin{align}
      \E \dist(Z, \lambda \partial g(\theta)) - \E \dist(Z,
      \cone(\partial g(\theta))) &\leq (\lambda - \lambda^*) \|v^*\| +
                                   \frac{4 \|v^*\|}{\|v_0\|} \nonumber \\ 
&+ 2                              \sqrt{-2 \log(1 - 2 e^{-2}) } \nonumber \\
&\leq (\lambda - \lambda^*) \|v^*\| + \frac{4 \|v^*\|}{\|v_0\|} + 2.  \label{eril}
    \end{align}
    To convert this into a bound on $\E \dist^2(Z, \lambda \partial
    g(\theta))$, we use the fact that $z \mapsto \dist(z,
    \lambda \partial g(\theta))$ is a $1$-Lipschitz function so that
    again by standard Gaussian concentration, we have 
    \begin{align*}
      \mathrm{var}(\dist(Z, \lambda \partial g(\theta))) &=
      \int_0^{\infty} \P \left\{\left|\dist(Z, \lambda \partial
      g(\theta)) - \E \dist(Z, \lambda \partial
      g(\theta))  \right| \geq \sqrt{t} \right\} dt  \\
&\leq 2 \int_0^{\infty} e^{-t/2} dt = 4. 
    \end{align*}
    This gives 
    \begin{align*}
    \D(\lambda \partial g(\theta)) &=  \E \dist^2(Z, \lambda \partial
      g(\theta))\\  &= \left( \E \dist(Z, \lambda \partial
      g(\theta))\right)^2 + \mathrm{var}(\dist(Z, \lambda \partial
                      g(\theta))) \\
&\leq \left( \E \dist(Z, \lambda \partial
      g(\theta))\right)^2 + 4
    \end{align*}
   which, combined with \eqref{eril} and the elementary fact  
   \begin{align*}
     \E \dist(Z, \cone(\partial g(\theta))) \leq \sqrt{\E \dist^2(Z,
     \cone(\partial g(\theta)))} = \sqrt{\D(\cone (\partial
     g(\theta)))},
   \end{align*}
completes the proof of Proposition \ref{qfmg}.    
\end{proof}

\subsection{Proofs of Corollary \ref{expen1}, Lemma \ref{explam1} and
  Corollary \ref{expen}} 
In this subsection, we shall provide the proofs of Corollary \ref{expen1}, Lemma \ref{explam1} and
  Corollary \ref{expen}. 

  \begin{proof}[Proof of Corollary \ref{expen1}]
    Corollary \ref{expen1} is a simple consequence of Theorem
    \ref{prt1} and Lemma \ref{gksd}. Indeed, Lemma \ref{gksd} states
    that for $r = 1$, we have $v^* = v_0$ and that 
    \begin{equation*}
      \|v^*\|^2 = \frac{1}{n_0} + \frac{1}{n_k} + 4 \sum_{i=1}^{k-1}
      \frac{I\{\rs_i \neq \rs_{i+1}\}}{n_i} \leq 4 \sum_{i=0}^{k}
      \frac{I\{\rs_i \neq \rs_{i+1}\}}{n_i}. 
    \end{equation*}
   Using this in the right hand side of \eqref{prt1.eq}, we get
   \begin{equation*}
     R(\hat{\theta}^{(1)}_{\lambda}, \theta^*) \leq C_1 \sigma^2
     \Delta_1(\theta^*) + \frac{64 \sigma^2}{n} + \frac{16 \sigma^2}{n}
     (\lambda - \lambda^*)^2 \sum_{i=0}^{k}
      \frac{I\{\rs_i \neq \rs_{i+1}\}}{n_i}
   \end{equation*}
   which implies \eqref{expen1.eq} as $\Delta_1(\theta^*) \geq
   1/n$. To prove \eqref{expen1.seq}, we further bound the right hand
   side above under the minimum length condition \eqref{nsa} by noting
   that $\Delta_1(\theta^*) \leq C(c) \frac{k+1}{n}\log
   \frac{en}{k+1}$ and also that 
   \begin{equation*}
     \sum_{i=0}^{k} \frac{I\{\rs_i \neq \rs_{i+1}\}}{n_i} \leq
     \frac{k+1}{cn} \sum_{i=0}^k I\{\rs_i \neq \rs_{i+1}\}. 
   \end{equation*}
  \end{proof}

  \begin{proof}[Proof of Lemma \ref{explam1}]
  From the formula \eqref{lams} for $\lambda^*$, it is clear that we
  need to bound both the terms $\E \lambda_{\theta^*}(Z)$ and
  $2/\|v_0\|$ from above in order to upper bound
  $\lambda^*$. For bounding $1/\|v_0\|$ from above, we use
  \eqref{voex}  to obtain
  \begin{align*}
    \|v_0\|^2 &= \frac{1}{n_0} + \frac{1}{n_k} + 4
                           \sum_{i=1}^{k-1} \frac{I\{\rs_i \neq
                           \rs_{i+1}\}}{n_i}  \geq  \sum_{i=0}^k
                           \frac{I\{\rs_i \neq \rs_{i+1}\}}{n_i} \\
&\geq \frac{1}{n} \sum_{i=0}^k I\{\rs_i
    \neq \rs_{i+1}\} 
  \end{align*}
where, in the last inequality above, we used $n_i \leq n$. This gives 
\begin{equation}\label{neda}
  \frac{2}{\|v_0\|} \leq \sqrt{\frac{4n}{\sum_{i=0}^k
      I\{\rs_i \neq \rs_{k+1}\}}}. 
\end{equation}
We shall now bound $\E \lambda_{\theta^*}(Z)$. Note that $\theta^* \in
\R^n$ is such that $D \theta^* \neq 0$. Throughout this proof,
$f(\theta) := \|D \theta\|_1$. As observed in the proof of Proposition
\ref{qfmg}, $\cone(\partial f(\theta^*))$ is a closed convex cone and
for every $z \in \R^n$, we have 
\begin{equation}\label{lamju}
  \Pi_{\cone(\partial f(\theta^*))}(z) = \lambda_{\theta^*}(z) v(z)
\end{equation}
for some vector $v(z) \in \partial f(\theta^*)$. Suppose now that
$\theta^*$ has the $k$ jumps $2 \leq j_1 < \dots < j_k \leq n$ with
associated signs $\rs_1, \dots, \rs_k$. Also let $j_0 = 1, j_{k+1} =
n+1$ and $\rs_0 = \rs_{k+1} = 0$. Then by the characterization of
$\partial f(\theta^*)$ from Proposition \ref{characterization}, we
have 
\begin{equation*}
  \sum_{u = j_i}^n v_u(z) = \rs_i  \qt{for every $i = 0, \dots, k+1$}
\end{equation*}
where $(v_1(z), \dots, v_n(z))$ are the components of the vector
$v(z)$. This implies, via \eqref{lamju}, that 
\begin{equation*}
  \rs_i \lambda_{\theta^*}(z) = \sum_{u=j_i}^n (\Pi z)_u
\end{equation*}
where $(\Pi z)_1, \dots, (\Pi z)_n$ denote the components of
$\Pi z := \Pi_{\cone(\partial f(\theta^*))}(z)$. As a
consequence (by subtracting the above identity for $i$ from the
corresponding identity for $i+1$), we obtain 
\begin{equation*}
 \left( \rs_{i} - \rs_{i+1} \right) \lambda_{\theta^*}(z) = (\Pi
 z)_{j_i} + \dots + (\Pi z)_{j_{i+1} - 1}
\end{equation*}
for every $i = 0, \dots, k$. Multiplying both sides above by $(\rs_i -
\rs_{i+1})$, we get 
\begin{equation*}
  \left( \rs_{i} - \rs_{i+1} \right)^2  \lambda_{\theta^*}(z) = \left(
    \rs_i - \rs_{i+1} \right) \left((\Pi 
 z)_{j_i} + \dots + (\Pi z)_{j_{i+1} - 1} \right) 
\end{equation*}
for every $i = 0, \dots, k$. Adding these for $i = 0, \dots, k$, we
obtain
\begin{equation*}
  \lambda_{\theta^*}(z) \sum_{i=0}^k (\rs_i - \rs_{i+1})^2 =
  \sum_{i=0}^k (\rs_i - \rs_{i+1}) \left((\Pi 
 z)_{j_i} + \dots + (\Pi z)_{j_{i+1} - 1}  \right) 
\end{equation*}
We now use the important identity \eqref{morco} which gives
\begin{equation*}
\Pi z = \Pi_{\cone(\partial f(\theta^*))}(z) = z - \Pi_{T_{K^{(1)}(V^*)}}(z)
\end{equation*}
where $V^* := \|D \theta^*\|_1$. This gives (below we write $\Pi_T z$
as shorthand for $\Pi_{T_{K^{(1)}(V^*)}}(z)$)  
\begin{align*}
  \lambda_{\theta^*}(z) \sum_{i=0}^k (\rs_i - \rs_{i+1})^2 &=
  \sum_{i=0}^k (\rs_i - \rs_{i+1}) \left(z_{j_i} + \dots + z_{j_{i+1}
      - 1}  \right) \\ &-  
  \sum_{i=0}^k (\rs_i - \rs_{i+1}) \left((\Pi_T
 z)_{j_i} + \dots + (\Pi_T z)_{j_{i+1} - 1}  \right) . 
\end{align*}
This equality holds for all vectors $z \in \R^n$. Applying this to
$Z \sim N(0, I_n)$ and taking expectations on both sides with respect
to $Z$, we obtain
\begin{equation*}
\E \lambda_{\theta^*}(Z)  \sum_{i=0}^k (\rs_i - \rs_{i+1})^2 = -
  \sum_{i=0}^k (\rs_i - \rs_{i+1}) \left((\E \Pi_T
 Z)_{j_i} + \dots + (\E \Pi_T Z)_{j_{i+1} - 1}  \right). 
\end{equation*}
Using the Cauchy-Schwarz inequality on the right hand side above, we
deduce
\begin{align*}
\E \lambda_{\theta^*}(Z)  \sum_{i=0}^k (\rs_i - \rs_{i+1})^2 &\leq \| \E \Pi_T Z \|
  \sqrt{\sum_{i=0}^k (\rs_i - \rs_{i+1})^2 n_i} \\ &\leq 2 \| \E \Pi_T Z \|
  \sqrt{\sum_{i=0}^k n_i I\{\rs_i \neq \rs_{i+1}\}}
\end{align*}
where we used the fact that $|\rs_i - \rs_{i+1}| \leq 2$ when $\rs_i
\ne \rs_{i+1}$. This gives (also using $|\rs_i - \rs_{i+1}| \geq 1$
when $\rs_i \neq \rs_{i+1}$ on the left hand side)
\begin{equation}\label{eadc}
  \E \lambda_{\theta^*}(Z) \leq 2  \| \E \Pi_T Z \|
  \sqrt{\frac{\sum_{i=0}^k n_i I\{\rs_i \neq \rs_{i+1}\}}{\left(\sum_{i=0}^k
    I\{\rs_i \neq \rs_{i+1}\} \right)^2}}. 
\end{equation}
To bound $\|\E \Pi_T Z\|$, we use Jensen's inequality and inequality
\eqref{gast} (recall the notions of statistical dimension and Gaussian
width from Subsection \ref{prems}) to obtain
 \begin{equation*}
   \|\E \Pi_TZ\|^2 \leq \E \|\Pi_T Z\|^2 = \delta(T) \leq 1 + w^2(T) =
   1 + w^2(T_{K^{(1)}(V^*)})
 \end{equation*}
 Inequality \eqref{rew} now gives 
 \begin{equation*}
w^2(T_{K^{(1)}(V^*)}) \leq C_1^2 n \Delta_1(\theta^*)
 \end{equation*}
for a positive constant $C_1^2$. This implies (note that
$\Delta_1(\theta^*) \geq 1/n$) that 
\begin{equation*}
  \|\E \Pi_TZ\| \leq C \sqrt{n \Delta_1(\theta^*)}. 
\end{equation*}
Combining this with \eqref{eadc}, we get
\begin{equation*}
    \E \lambda_{\theta^*}(Z) \leq C  \sqrt{n \Delta_1(\theta^*)}
  \sqrt{\frac{\sum_{i=0}^k n_i I\{\rs_i \neq \rs_{i+1}\}}{\left(\sum_{i=0}^k
    I\{\rs_i \neq \rs_{i+1}\} \right)^2}}. 
\end{equation*}
We now use the length condition \eqref{leon}. Under this condition, we
know that  
\begin{equation*}
  n \Delta_1(\theta^*) \leq C(c_1) (k+1) \log \left(\frac{en}{k+1}
  \right). 
\end{equation*}
Using this (and the fact that $n_i \leq c_2 n/(k+1)$), we obtain
\begin{equation*}
  \E \lambda_{\theta^*}(Z) \leq C(c_1, c_2)
  \sqrt{\frac{n}{\sum_{i=0}^k I\{\rs_i \neq \rs_{i+1}\}} \log
    \frac{en}{k+1}}. 
\end{equation*}
The proof of \eqref{rambo} is now completed by the combining the above 
bound with \eqref{neda}. 
  \end{proof}

  \begin{proof}[Proof of Corollary \ref{expen}]
    Suppose $\lambda$ is as in \eqref{lach} for $\Gamma \geq C^*(c_1,
    c_2)$ (where $C^*(c_1, c_2)$ comes from Lemma
    \ref{explam1}). Then, by Lemma \ref{explam1}, $\lambda \geq
    \lambda^*$. We can therefore apply Corollary \ref{expen1}
    (specifically, inequality \eqref{expen1.seq} as $\theta^*$
    satisfies the length condition \eqref{leon} which implies the
    minimum length condition with constant $c_1$) to obtain
    \begin{align}
           R(\hat{\theta}_{\lambda}^{(1)}, \theta^*) &\leq C(c_1) \sigma^2
     \left(\frac{k+1}{n} \log \frac{en}{k+1} + (\lambda - \lambda^*)^2
     \frac{k+1}{n^2} \sum_{i=0}^k I\{\rs_i \neq \rs_{i+1}\}\right) \nonumber \\
&\leq C(c_1) \sigma^2
     \left(\frac{k+1}{n} \log \frac{en}{k+1} + \lambda^2
     \frac{k+1}{n^2} \sum_{i=0}^k I\{\rs_i \neq \rs_{i+1}\}\right) \label{kke}
    \end{align}
   for a constant $C(c_1)$ depending only on $c_1$. In the last
   inequality above, we used the trivial fact that $(\lambda -
   \lambda^*)^2 \leq \lambda^2$. Plugging in the value of $\lambda$
   from \eqref{lach} in the bound \eqref{kke}, we obtain \eqref{jku}.  

  We shall now prove \eqref{cbal} assuming that $\lambda$ is as in
  \eqref{cdal} with $\Gamma \geq C^*(c_1, c_2)$. For this, note first
  that \eqref{kke} holds for this $\lambda$ as well because $\lambda \geq
  \lambda^*$. Plugging in $\lambda = \Gamma \sqrt{n \log(en)}$ in
  \eqref{kke}, we obtain
  \begin{equation*}
           R(\hat{\theta}_{\lambda}^{(1)}, \theta^*) \leq  C(c_1) \sigma^2
     \left(\frac{k+1}{n} \log \frac{en}{k+1} + \Gamma^2 (\log (en))
     \frac{k+1}{n} \sum_{i=0}^k I\{\rs_i \neq \rs_{i+1}\}\right). 
  \end{equation*}
  The trivial bound $\log(en/(k+1)) \leq \log(en)$ now gives
  \eqref{cdal}. The proof of Corollary \ref{expen} is complete. 
  \end{proof}

\subsection{Proofs of Corollary \ref{expenr}, Lemma \ref{explamr}
  and Corollary \ref{supno}}
In this subsection, we provide the proofs of Corollary \ref{expenr},
Lemma \ref{explamr} and Corollary \ref{supno}. 

\begin{proof}[Proof of Corollary \ref{expenr}]
  Corollary \ref{expenr} is a simple consequence of Theorem \ref{prt1}
  and Lemma \ref{gksd}. Indeed, plugging in the lower bound on
  $\|v_0\|$ from \eqref{gksd.low} and the upper bound on $\|v^*\|$
  from \eqref{gksd.up} in inequality \eqref{prt1.eq}, we obtain 
  \begin{equation*}
    R(\hat{\theta}^{(r)}_{\lambda}, \theta^*) \leq C_r \sigma^2
    \Delta_r(\theta^*) + C_r(c) \frac{\sigma^2}{n} (k+1)^{2r} + C_r(c) \sigma^2
    (\lambda - \lambda^*)^2 \frac{(k+1)^{2r}}{n^2}. 
  \end{equation*}
  for a constant $C_r(c)$ depending only on $c$ ($c$ appears in the
  minimum length condition \eqref{nsa}). From here, inequality
  \eqref{expenr.eq} immediately follows from the observation that
  $\Delta_r(\theta^*) \leq C_r(c) \frac{k+1}{n} \log \frac{en}{k+1}$
  under the minimum length condition.  
\end{proof}

\begin{proof}[Proof of Lemma \ref{explamr}]
  Recall that 
  \begin{equation*}
    \lambda^* = n^{1-r} \left(\E \lambda_{\theta^*}(Z) + \frac{2}{\|v_0\|} \right)
  \end{equation*}
  with 
  \begin{equation*}
    \lambda_{\theta^*}(z) := \argmin_{\lambda \geq 0} \inf_{v
      \in \partial f(\theta^*)} \|z - \lambda v\|
  \end{equation*}
   where $f(\theta) := \|D^{(r)} \theta\|_1$ and we have assumed that
   $D^{(r)} \theta^* \neq 0$. To bound 
  $\lambda^*$ from above, we therefore need to bound both the terms
  $\E \lambda_{\theta^*}(Z)$ and $2/\|v_0\|$ from above. To bound
  $2/\|v_0\|$, we simply used inequality \eqref{gksd.low} which gives
  \begin{equation}\label{fro8}
    \frac{2}{\|v_0\|} \leq C_r n^{r - 1/2} 
  \end{equation}
  for a constant $C_r$. The main task therefore is to bound $\E
  \lambda_{\theta^*}(Z)$. We follow a strategy similar to that
  employed in the proof of Lemma \ref{explam1}. As observed in the
  proof of Proposition \ref{qfmg}, $\cone(\partial f(\theta^*))$ is a
  closed convex cone (because $D^{(r)} \theta^* \neq 0$) and for every
  $z \in \R^n$, we can write
\begin{equation}\label{lamjur}
  \Pi_{\cone(\partial f(\theta^*))}(z) = \lambda_{\theta^*}(z) v(z)
\end{equation}
for some vector $v(z) \in \partial f(\theta^*)$. Suppose now that
$\theta^*$ has the $k$ knots (or order $r$): $2 \leq j_1 < \dots < j_k
\leq n-r+1$ with associated signs $\rs_1, \dots, \rs_k$. Then by the 
characterization of $\partial f(\theta^*)$ from Proposition
\ref{characterization} (specifically using \eqref{ch2} with $j = j_k +
r - 1$), we obtain   
\begin{equation*}
  \sum_{i = j_k + r - 1}^n {i-j_k \choose r-1} v_i(z) = \rs_k
\end{equation*}
where $v_1(z), \dots, v_n(z)$ are the  components of the vector
$v(z)$. This implies, via \eqref{lamjur}, that 
\begin{equation*}
  \sum_{i = j_k + r - 1}^n {i-j_k \choose r-1} (\Pi z)_i= \rs_k
  \lambda_{\theta^*}(z)
\end{equation*}
where $(\Pi z)_1, \dots, (\Pi z)_n$ denote the components of $\Pi z :=
\Pi_{\cone(\partial f(\theta^*))}(z)$. Using \eqref{morco}, we can write
\begin{equation*}
  \Pi z = \Pi_{\cone(\partial f(\theta^*))}(z) = z -   \Pi_{T_{K^{(r)}(V^*)}}(z). 
\end{equation*}
We thus obtain (using $\Pi_T z$ as shorthand for
$\Pi_{T_{K^{(r)}(V^*)}}(z)$), 
\begin{equation*}
  \rs_k \lambda_{\theta^*}(z) =   \sum_{i = j_k + r - 1}^n {i-j_k
    \choose r-1} z_i -   \sum_{i = j_k + r - 1}^n {i-j_k \choose r-1}
  (\Pi_T z)_i. 
\end{equation*}
Applying this to $Z \sim N_n(0, I_n)$, we get
\begin{equation*}
  \rs_k \E \lambda_{\theta^*}(Z) =  \sum_{i = j_k + r - 1}^n {i-j_k \choose r-1}
  (\E \Pi_T z)_i
\end{equation*}
so that 
\begin{equation*}
  \E \lambda_{\theta^*}(Z) = \left| \sum_{i = j_k + r - 1}^n {i-j_k \choose r-1}
  (\E \Pi_T z)_i \right|. 
\end{equation*}
By the Cauchy-Schwarz inequality, we  now get
\begin{align*}
 \left( \E \lambda_{\theta^*}(Z) \right)^2 &\leq \left[ \sum_{i=j_k + r - 1}^n
 {i-j_k \choose r-1}^2 \right] \left[ \sum_{i=j_k + r -1}^n \left( (\E
   \Pi_T Z)_i \right)^2 \right] \\
&\leq \left[ \sum_{i=j_k + r - 1}^n  {i-j_k \choose r-1}^2 \right]  \|
  \E \Pi_T Z \|^2 \\ &\leq \left[ \sum_{i=j_k + r - 1}^n  {i-j_k \choose
  r-1}^2 \right]  \E \| \Pi_T Z \|^2 . 
\end{align*}
Note now that for every $i = j_k + r - 1, \dots, n$, clearly
\begin{equation*}
  {i-j_k \choose r-1} \leq {n \choose r-1} \leq n^{r-1}. 
\end{equation*}
As a result, we have
\begin{equation*}
  \left( \E \lambda_{\theta^*}(Z) \right)^2 \leq n^{2r-2} (n - j_k - r
  + 2)  \E \| \Pi_T Z \|^2 . 
\end{equation*}
Noting that $n_k = n-r+2-j_k$, we have proved that 
\begin{equation*}
  \left( \E \lambda_{\theta^*}(Z) \right)^2 \leq n^{2r-2} n_k \E \|
  \Pi_T Z \|^2.  
\end{equation*}
Inequality \eqref{gast} (recall the notions of statistical dimension
and Gaussian width from Subsection \ref{prems}) now gives
 \begin{equation*}
\E \|\Pi_T Z\|^2 = \delta(T) \leq 1 + w^2(T) =
   1 + w^2(T_{K^{(r)}(V^*)}). 
 \end{equation*}
Using inequality \eqref{rew}, we get 
 \begin{equation*}
w^2(T_{K^{(r)}(V^*)}) \leq C_r^2 n \Delta_r(\theta^*)
 \end{equation*}
for a positive constant $C_r^2$. We have therefore proved that
\begin{equation*}
\E \lambda_{\theta^*}(Z)  \leq n^{r-1} \sqrt{n_k \E \|
  \Pi_T Z \|^2} \leq n^{r-1} \sqrt{n_k \left(1 + C_r^2 n
    \Delta_r(\theta^*) \right)}. 
\end{equation*}
We now invoke the length condition \eqref{leonr}. Under this
condition, we first have
\begin{equation*}
  n \Delta_r(\theta^*) \leq C_r(c_1) (k+1) \log \frac{en}{k+1}
\end{equation*}
and also $n_k \leq c_2n/(k+1)$ so that
\begin{align*}
  \E \lambda_{\theta^*}(Z)  &\leq C_r(c_1, c_2) n^{r-1}
  \sqrt{\frac{n}{k+1} (k+1) \log \frac{en}{k+1}} \\
&= C_r(c_1, c_2) n^{r-1}
  \sqrt{n \log \frac{en}{k+1}}. 
\end{align*}
Combining this with \eqref{fro8}, we get
\begin{align*}
  \lambda^* &= n^{1-r} \left(\E \lambda_{\theta^*}(Z) +
              \frac{2}{\|v_0\|} \right) \\
&\leq n^{1-r} \left(C_r(c_1, c_2) n^{r-1}
  \sqrt{n \log \frac{en}{k+1}}  + C_r n^{r-1/2} \right) \\ &\leq
                                                             C^*_r(c_1,
                                                             c_2)
                                                             \sqrt{n
                                                             \log
                                                             \frac{en}{k+1}}.  
\end{align*}
This finishes the proof of Lemma \ref{explamr}. 
\end{proof}

\begin{proof}[Proof of Corollary \ref{supno}]
    Suppose $\lambda$ is as in \eqref{lachr} for $\Gamma \geq C_r^*(c_1,
    c_2)$ (where $C_r^*(c_1, c_2)$ comes from Lemma
    \ref{explamr}). Then, by Lemma \ref{explamr}, $\lambda \geq
    \lambda^*$. We can therefore apply Corollary \ref{expenr}
    (note that $\theta^*$ satisfies the length condition \eqref{leonr}
    which implies the minimum length condition with constant $c_1$) to
    obtain 
    \begin{align}
           R(\hat{\theta}_{\lambda}^{(1)}, \theta^*) &\leq C_r(c_1) \sigma^2
     \left(\frac{k+1}{n} \log \frac{en}{k+1} + \frac{(k+1)^{2r}}{n} + (\lambda - \lambda^*)^2
     \frac{(k+1)^{2r}}{n^2} \right) \nonumber \\
&\leq C_r(c_1) \sigma^2
     \left(\frac{k+1}{n} \log \frac{en}{k+1} + \frac{(k+1)^{2r}}{n} + \lambda^2
     \frac{(k+1)^{2r}}{n^2}\right) \label{kker}
    \end{align}
   for a constant $C_r(c_1)$ depending only on $r$ and $c_1$. In the last
   inequality above, we used the trivial fact that $(\lambda -
   \lambda^*)^2 \leq \lambda^2$. Plugging in the value of $\lambda$
   from \eqref{lachr} in the bound above, we obtain
   \begin{align*}
       R(\hat{\theta}_{\lambda}^{(1)}, \theta^*) &\leq C_r(c_1)
       \sigma^2 \left(\frac{k+1}{n} \log \frac{en}{k+1} +
         \frac{(k+1)^{2r}}{n} + \frac{\Gamma^2 (k+1)^{2r}}{n} \log
         \frac{en}{k+1} \right)  \\
&\leq C_r(c_1) \sigma^2 (2 + \Gamma^2) \frac{(k+1)^{2r}}{n} \log
  \frac{en}{k+1} 
   \end{align*}
  which proves \eqref{jkur}. We shall now prove \eqref{cbalr} assuming
  that $\lambda$ is as in \eqref{cdalr} with $\Gamma \geq C^*(c_1, c_2)$. For this, note first
  that \eqref{kker} holds for this $\lambda$ as well because $\lambda \geq
  \lambda^*$. Plugging in $\lambda = \Gamma \sqrt{n \log(en)}$ in
  \eqref{kker}, we obtain
   \begin{align*}
       R(\hat{\theta}_{\lambda}^{(1)}, \theta^*) &\leq C_r(c_1)
       \sigma^2 \left(\frac{k+1}{n} \log \frac{en}{k+1} +
         \frac{(k+1)^{2r}}{n} + \frac{\Gamma^2 (k+1)^{2r}}{n} \log
                                                   (en) \right)  \\ 
&\leq C_r(c_1) \sigma^2 (2 + \Gamma^2) \frac{(k+1)^{2r}}{n} (\log (en))
   \end{align*}
which proves \eqref{cdalr} and completes the proof of Corollary
\ref{supno}. 
\end{proof}

\subsection{Proof of Lemma \ref{rr0}} \label{rr0.pf}
The proof of Lemma \ref{rr0}, which deals with the case when $D^{(r)}
\theta^* = 0$, is provided here. 

\begin{proof}[Proof of Lemma \ref{rr0}]
  Let $f(\theta) := \|D^{(r)} \theta\|_1$  and $g(\theta) := n^{r-1}
  f(\theta)$. The estimator $\hat{\theta}_{\lambda}^{(r)}$ is then
  given by
  \begin{equation*}
    \hat{\theta}_{\lambda}^{(r)} = \argmin_{\theta \in \R^n}
    \left(\frac{1}{2} \|Y - \theta\|^2 + \sigma \lambda g(\theta)
    \right). 
  \end{equation*}
  The risk result \eqref{ohala} gives
  \begin{equation}\label{oym7}
    R(\hat{\theta}_{\lambda}^{(r)}, \theta^*) \leq \frac{\sigma^2}{n}
    \D(\lambda \partial g(\theta^*)) = \frac{\sigma^2}{n} 
    \D(n^{r-1}\lambda \partial f(\theta^*)). 
  \end{equation}
   Because $D^{(r)} \theta^* = 0$, the subdifferential of $f$ at
   $\theta^*$ consists precisely of all vectors $v \in \R^n$ for which   
  \begin{equation}\label{equ}
    \sum_{i=j}^n {r+i-j-1 \choose r-1} v_i = 0 \qt{for $1 \leq j \leq r$}
  \end{equation}
  and 
  \begin{equation*}
    \max_{r < j \leq n} \left|\sum_{i=j}^n {r+i-j-1 \choose r-1}
      v_i\right| \leq 1. 
  \end{equation*}
   This is a consequence of the characterization of the
   subdifferential given in Proposition \ref{characterization}.
 
  Now let $S_r$ denote the set consisting of all vectors $v \in \R^n$
  such that \eqref{equ}  holds. Clearly $S_r$ is a subspace in $\R^n$
  of dimension exactly equal to $n - r$. Let $\Pi_{S_r}$ denote the
  projection matrix onto $S_r$ and let
  \begin{equation*}
    \lambda(z) := n^{1-r}\max_{r <  j \leq n} \left|\sum_{i=j}^n
      {r+i-j-1 \choose r-1} (\Pi_{S_r} z)_i \right| \qt{for $z \in
      \R^n$}.  
  \end{equation*}
   For $Z \sim N(0, I_n)$, we can write
   \begin{align*}
     \D(n^{r-1}\lambda \partial f(\theta^*)) &= \E \dist^2(Z, n^{r-1}
                                               \lambda \partial
                                               f(\theta^*)) \\
&= \E \dist^2(Z, n^{r-1}
                                               \lambda \partial
                                               f(\theta^*))
  I\{\lambda(Z) \leq \lambda\} \\ &+ \E \dist^2(Z, n^{r-1}
                                               \lambda \partial
                                               f(\theta^*))I\{\lambda(Z)
  > \lambda\} 
   \end{align*}
   From the characterization of $\partial f(\theta^*)$ given above, it
   is clear that when $\lambda(Z) \leq \lambda$, the vector $\Pi_{S_r}
   Z$ belongs to $n^{r-1} \lambda \partial f(\theta^*)$. On the other
   hand, the zero vector always belongs to $\partial f(\theta^*)$
   (note that we are working under the assumption that $D^{(r)}
   \theta^* = 0$). This allows us to deduce that 
   \begin{align*}
     \D(n^{r-1}\lambda \partial f(\theta^*)) &\leq \E \|Z - \Pi_{S_r}
                                               Z \|^2 + \E \|Z\|^2
                                               I\{\lambda(Z) >
                                               \lambda\}. 
   \end{align*}
   Because $S_r$ is a subspace of dimension $n - r$, the first term
   above equals $r$. For the second term, we use Cauchy-Schwarz
   inequlity (and the elementary fact that $\E \|Z\|^4 = n^2 + 2n$) to
   obtain
   \begin{align}\label{aara}
     \D(n^{r-1}\lambda \partial f(\theta^*)) &\leq r + \sqrt{n^2 + 2n}
                                               \sqrt{\P
                                               \left\{\lambda(Z) >
                                               \lambda \right\}}. 
   \end{align}
   To bound $\P \{\lambda(Z) > \lambda\}$, we write (via the union bound)
   \begin{align*}
     \P \{\lambda(Z) > \lambda\} \leq \sum_{r < j \leq n} \P \left\{\left|\sum_{i=j}^n
      {r+i-j-1 \choose r-1} (\Pi_{S_r} Z)_i \right| > n^{r-1} \lambda
  \right\}. 
   \end{align*}
   For each fixed $r < j \leq n$, the random variable 
   \begin{align*}
     \sum_{i=j}^n
      {r+i-j-1 \choose r-1} (\Pi_{S_r} Z)_i
   \end{align*}
   is easily seen to be normally distributed with mean zero and
   variance equal to $\|\Pi_{S_r} u\|^2$ where $u$ is the vector whose
   $i^{th}$ entry is ${r + i - j - 1 \choose r-1}$ for $i \geq j$ and
   0 for $i < j$. Note that
   \begin{align*}
     \|\Pi_{S_r} u\|^2 &\leq \|u\|^2 = \sum_{i=j}^n {r + i - j - 1
     \choose r-1}^2  \\ &\leq n {n-j+r-1 \choose r-1}^2 \leq n \times
     (n^{r-1})^2 = n^{2r-1}. 
   \end{align*}
   Using this (and the Gaussian tail bound: $\P\{|N(0, 1)| \geq
   t\} \leq \exp(-t^2/2)$), we obtain
   \begin{align*}
     \P \{\lambda(Z) > \lambda\} &\leq \sum_{r < j \leq n} \P \left\{\left|\sum_{i=j}^n
      {r+i-j-1 \choose r-1} (\Pi_{S_r} Z)_i \right| 
> n^{r-1} \lambda
  \right\} \\
&\leq n \exp \left( \frac{-(n^{r-1} \lambda)^2}{2 n^{2r-1}} \right)
  = n \exp \left(\frac{-\lambda^2}{2n} \right). 
   \end{align*}
    Combining the above inequality with \eqref{aara}, we obtain
    \begin{align*}
      \D(n^{r-1}\lambda \partial f(\theta^*)) &\leq r + \sqrt{n^3 +
                                                2n^2} \exp
                                                \left(\frac{-\lambda^2}{4n}
                                                \right). 
    \end{align*}
    Now for $\lambda \geq \sqrt{6n \log (en)}$, we obtain
    \begin{align*}
      \D(n^{r-1}\lambda \partial f(\theta^*)) &\leq r + \sqrt{n^3 +
                                                2n^2} (en)^{-3/2} \leq
                                                r + e^{-3/2} \sqrt{1 +
                                                \frac{2}{n}} \leq
                                                C_r
    \end{align*}   
    where $C_r$ only depends on $r$. This bound and inequality
    \eqref{oym7} together complete the proof of Lemma \ref{rr0}.  
\end{proof}

\section{Proofs of Key Technical Results}\label{sutta}
Our main proofs presented in Section \ref{pfmr} were crucially reliant
on the following technical results: Lemma \ref{svg} (used in the proof
of Theorem \ref{woor}), Lemma \ref{alta} and Lemma \ref{fina} (used in
the proof of Theorem \ref{ada}). The proofs of these results are given
in this section. In addition, this section also contains the proofs of
Proposition \ref{characterization} and Lemma \ref{gksd} from Section
\ref{ms} of the main paper. The proofs of this section will further
involve other technical results which (together with some other
supporting results from the previous section such as Lemma \ref{lemco}
which was used in the proof of Corollary \ref{kco}) will be proved in
Section \ref{appa}.   

The organization of this section is as follows. We first prove Lemma
\ref{svg} in Subsection \ref{gwc1}. Next Lemma \ref{alta} is proved in
Subsection \ref{tcc} and this requires a precise understanding 
of the tangent cones $T_{\krv}(\theta)$. Subsection \ref{gwc}
is devoted to the proof of Lemma \ref{fina}. In Subsection
\ref{sudi}, we study the subdifferential of $\theta \mapsto \|D^{(r)}
\theta\|_1$ and provide proofs of Proposition \ref{characterization}
and Lemma \ref{gksd}. 

\subsection{Proof of Lemma \ref{svg}} \label{gwc1}
In this subsection, we shall provide the proof of Lemma \ref{svg}
(which was crucially used for the proof of Theorem \ref{woor}). Our
strategy is to use Dudley's entropy bound to control the left hand
side of \eqref{svg.eq} in terms of the metric entropy of $S_r(V, t)$
(defined in \eqref{c2}). Let us first formally define the notion of
metric entropy. For a set $K \subset \R^n$ and $\epsilon > 0$, we
define $N(\epsilon, K)$ to be the 
smallest integer $m$ for which there exist points $a_1, \dots, a_m \in
\R^n$ satisfying   
\begin{equation*}
  \sup_{a \in K} \inf_{1 \leq i \leq m} \|a - a_i\| \leq \epsilon
\end{equation*}
where, as usual, $\|\cdot\|$ denotes the Euclidean norm. The
$\epsilon$-metric entropy of $K$ is the logarithm of $N(\epsilon, K)$. 

Dudley's entropy bound bounds the left hand side of \eqref{svg.eq} via
$\log N(\epsilon, S_r(V, t))$. The following theorem then provides
upper bounds on $\log N(\epsilon, S_r(V, t))$. 

\begin{theorem}\label{dprg}
For $r \ge 1$, $t > 0$, $V > 0$ and $n \geq r$, let    
\begin{equation*}
  S_r(V, t) := \left\{\theta \in \R^n : \|\theta\| \leq t, V(D^{(r-1)}\theta)
  \leq V n^{1-r}\right\}. 
\end{equation*}
Then for every $\epsilon > 0$, we have 
\begin{equation}\label{dprg.eq}
  \log N(\epsilon, S_r(V, t)) \leq C_r \left(\frac{V
      \sqrt{n}}{\epsilon} \right)^{1/r}
  + r \log \left(2 + \frac{2^rn^{r} t}{\epsilon \sqrt{n}} \right) + C_r
\end{equation}
for a constant $C_r$ that depends only on $r$.     
\end{theorem}

Let us first complete the proof of Lemma \ref{svg} assuming that
Theorem \ref{dprg}. The proof of Theorem \ref{dprg} will be provided
following the proof of Lemma \ref{svg}.  

\begin{proof}[Proof of Lemma \ref{svg}]
  Let $G$ denote the left hand side of \eqref{svg.eq}. We use
  Dudley's entropy bound to deduce that  
  \begin{equation*}
    G \leq  C \sigma \int_0^{t} \sqrt{\log N(\epsilon, S_r(V, t))}
    \, d\epsilon 
  \end{equation*}
   where the set $S_r(V, t)$ is defined as $\{\theta \in \R^n :
   \|\theta\| \leq t, V(D^{(r-1)}\theta) \leq  V n^{1-r}\}$ and
   $N(\epsilon, S_r(V, t))$ denotes the $\epsilon$-covering number of
   $S_r(V, t)$ under the Euclidean metric. These covering numbers are
   bounded in Theorem \ref{dprg} which furnishes a constant $C_r$ such that  
   \begin{equation*}
     \sqrt{\log N(\epsilon, S_r(V, t))} \leq C_r \left(\frac{V
         \sqrt{n}}{\epsilon} \right)^{1/(2r)} + \sqrt{r \log  \left(2
         + \frac{2^r n^{r} t}{\epsilon \sqrt{n}} \right)} + C_r 
   \end{equation*}
for every $\epsilon >  0$. Note that the square root of the right hand
side of \eqref{dprg.eq} is bounded from above by the right hand side
above via the elementary inequality $\sqrt{a_1 + a_2 +
  a_3} \leq \sqrt{a_1} + \sqrt{a_2} + \sqrt{a_3}$ for $a_1, a_2, a_3
\geq 0$. It follows therefore that    
   \begin{equation*}
     G \leq C_r \sigma t \left(\frac{V \sqrt{n}}{t} \right)^{1/(2r)} + C_r
     \sigma t + C_r \sigma  \int_0^t \sqrt{\log \left(2 + \frac{2^r
           n^{r} t}{\epsilon \sqrt{n}}  \right)} d\epsilon.  
   \end{equation*}
   The last integral above can be controlled in the following way: 
   \begin{align*}
   \frac{1}{t}  \int_0^t \sqrt{\log \left(2 + \frac{2^r n^{r}
         t}{\epsilon \sqrt{n}}  
       \right)} d\epsilon &=  \int_0^1 \sqrt{\log \left(2 + \frac{2^r
           n^{r}}{u \sqrt{n}} \right)} du \\  
&= \int_0^{\sqrt{n} n^{-r}} \sqrt{\log \left(2 + \frac{2^r
           n^{r}}{u \sqrt{n}} \right)} du \\
&+ \int_{\sqrt{n} n^{-r}}^1
     \sqrt{\log  \left(2 + \frac{2^r n^{r}}{u \sqrt{n}} \right)} du.   
   \end{align*}
For the second integral above, we use $u \geq \sqrt{n} n^{-r}$ to
argue that it is bounded from above by $\sqrt{\log(2 +
  2^rn^{2r-1})} \leq C_r \sqrt{\log (en)}$. For the first integral, we
use
\begin{equation*}
  \log \left(2 + \frac{2^r n^{r}}{u \sqrt{n}} \right) \leq 1 +
  \frac{2^r n^{r}}{u \sqrt{n}} \leq \frac{2^{r+1} n^{r}}{u \sqrt{n}}
\end{equation*}
to obtain 
\begin{equation*}
  \int_0^{\sqrt{n} n^{-r}} \sqrt{\log \left(2 + \frac{2^r
           n^{r}}{u \sqrt{n}} \right)} du \leq C_r.
\end{equation*}
We have therefore proved that 
\begin{equation*}
  G \leq C_r \sigma t \left(\frac{V \sqrt{n}}{t} \right)^{1/(2r)}
  + C_r \sigma t \sqrt{\log(en)}  
\end{equation*}
for a constant $C_r$ which completes the proof of Lemma \ref{svg}.  
\end{proof}

Let us now provide the proof of Theorem \ref{dprg}. For this, let us
first introduce the following definition. 
\begin{defn}\label{cras}
For $r \ge 1$, $n \geq r$, real numbers $a_0, \dots, a_{r-1}$ and
non-negative real numbers $s_0, \dots, s_{r-1}$, let $\cras$ denote
the class of all $\theta \in \R^n$  for which $a_i \leq (D^{(i)}
\theta)_1 \leq a_i + s_i, i = 0, 1, \dots, r-2$,  and  
\begin{equation*}
  a_{r-1} \leq (D^{(r-1)} \theta)_1 \leq \dots \leq (D^{(r-1)}
  \theta)_{n-r+1} \leq a_{r-1} + s_{r-1}. 
\end{equation*}  
\end{defn}

\begin{remark}
  Note that when $r = 1$, the condition $a_i \leq (D^{(i)} \theta)_1
  \leq a_i + s_i, i = 0, \dots, r-2$ is vacuous so that vectors in
  $\C_1(\{a_i\}, \{s_i\})$ are required to only satisfy the inequality
  \begin{equation*}
    a_0 \leq \theta_1 \le \theta_2 \le \dots \leq \theta_n \leq a_0 +
    s_0. 
  \end{equation*}
\end{remark}

Our strategy for proving Theorem \ref{dprg} is to derive it from another
result on the metric entropy of $\cras$. The following lemma gives an
upper bound on the metric entropy of $\cras$. This is the most
important ingredient for the proof of Theorem \ref{dprg}. The proof of
this lemma is given in Subsection \ref{ppgr} and is based on an upper
bound on the fat shattering dimension of the classes $\cras$ and a
standard result (from \citet{rudelson2006combinatorics}) relating fat
shattering dimension to metric entropy. See Subsection \ref{ppgr} for
full details including the definition of fat shattering dimension. 

%\vspace{0.2in} 

\begin{lemma}\label{pgr}
  For every $\epsilon > 0$, $r \geq 1$, $n \geq r$, $a_0, \dots,
  a_{r-1} \in \R$ and $s_0, \dots, s_{r-1} \geq 0$, we have   
  \begin{equation*}
    \log N(\epsilon, \cras) \leq C_r \left(\frac{\sqrt{n} \sum_{j=1}^r
      n^{j-1} s_{j-1}}{\epsilon} \right)^{1/r}
  \end{equation*}
  where $C_r$ is a positive constant that depends on $r$ alone. 
\end{lemma}

We are now ready to prove Theorem \ref{dprg}. 

\begin{proof}[Proof of Theorem \ref{dprg}] 
  Fix $\delta > 0$ and let 
  \begin{equation*}
    K_i := \max \left\{u \geq 0 \mbox{ integer}: u \delta \leq 2^i t \right\} \qt{for
    $0 \leq i < r$}. 
  \end{equation*}
  It is then clear that $K_i \leq 2^i t/\delta < K_i + 1$ for every
  $0 \leq i < r$. Let $\K$ denote the class of all vectors $\kb := (\kb_0, \dots, \kb_{r-1})$ where each $\kb_i$ is an integer satisfying $-(K_i + 1) \leq \kb_i \leq K_i$. For every $\kb = (\kb_0, \dots, \kb_{r-1}) \in \K$, let 
 \begin{equation*}
   \M(\kb) := \left\{\theta \in S_r(V, t) : \kb_i \delta 
   \leq (D^{(i)} \theta)_1 \leq (\kb_i + 1) \delta \text{ for } 0 \leq
   i < r\right\} .  
 \end{equation*}
As 
\begin{align*}
    \left|(D^{(i)} \theta)_1 \right| = \left|\sum_{j=1}^{i+1} (-1)^j
      \binom{i}{j-1} \theta_j \right| &\leq \left(\binom{i}{0}^2 +
      \dots + \binom{i}{i}^2 \right)^{1/2} \|\theta\| \\
   &= \binom{2i}{i}^{1/2} \|\theta\| \leq 2^i \|\theta\| \leq 2^i t
  \end{align*}
  for $\theta \in S_r(V, t)$ and $0 \leq i < r$, it follows that
  $S_r(V, t) \subseteq \cup_{\kb \in \K} \M(\kb)$. As a result
\begin{equation*}
    N(\epsilon, S_r(V, t)) \leq \sum_{\kb \in \K} N(\epsilon, \M(\kb)) \leq 2^r \prod_{i=0}^{r-1} (K_i + 1) \sup_{\kb \in \K} N(\epsilon, \M(\kb)). 
\end{equation*}
Since $K_i \leq 2^i t/\delta \leq 2^{r-1} t/\delta$, we deduce
\begin{equation}\label{aa1}
    \log N(\epsilon, S_r(V, t)) \leq r  \log \left(2 +
      \frac{2^r t}{\delta} \right) + \sup_{\kb \in \K} \log
    N(\epsilon, \M(\kb)). 
\end{equation}
We now bound $\log N(\epsilon, \M(\kb))$ from above for a fixed $\kb \in \K$. For every $\theta \in \R^n$, let us define two vectors $\alpha(\theta) := (\alpha_1(\theta), \dots,
   \alpha_n(\theta))$ and $\beta(\theta) := (\beta_1(\theta), \dots,
   \beta_n(\theta))$ in $\R^n$ via   
   \begin{equation}\label{bhim1}
     \alpha_i(\theta) := \sum_{j=1}^{i-r} \binom{i-j-1}{r-1} (D^{(r)} \theta)^+_j
    + \sum_{j=1}^r \binom{i-1}{j-1} (D^{(j-1)} \theta)^+_1   
   \end{equation}
and 
   \begin{equation*}
     \beta_i(\theta) := \sum_{j=1}^{i-r} \binom{i-j-1}{r-1} (D^{(r)} \theta)^-_j
    + \sum_{j=1}^r \binom{i-1}{j-1} (D^{(j-1)} \theta)^-_1
   \end{equation*}
where $x^+ := \max(x, 0)$ and $x^- = x^+ - x$. It then follows from
Lemma \ref{dfi} that $\theta = \alpha(\theta) - \beta(\theta)$ and,
consequently,  
\begin{equation}\label{aa2}
  \log N(\epsilon, \M(\kb)) \leq \log N(\epsilon/2, \M_{\alpha}(\kb))
  + \log N(\epsilon/2, \M_{\beta}(\kb))
\end{equation}
where 
\begin{equation*}
  \M_{\alpha}(\kb) := \left\{\alpha(\theta) : \theta \in \M(\kb)
  \right\} ~~ \text{  and  } ~~ \M_{\beta}(\kb) :=
  \left\{\beta(\theta) : \theta \in \M(\kb) \right\}. 
\end{equation*}
We now show how to control $\log N(\epsilon/2, \M_{\alpha}(\kb))$
below. The argument for $\log N(\epsilon/2, \M_{\beta}(\kb))$ will be
similar. The main idea here (recall the definition of $\cras$ from
Definition \ref{cras}) is to note that
\begin{equation}\label{bm}
\M_{\alpha}(\kb)  \subseteq \cras  
\end{equation}
with
\begin{equation*}
  a_i = \kb_i^+ \delta \qt{for $i = 0, \dots, r-1$}, 
\end{equation*}
and 
\begin{equation*}
  s_i = \delta ~~ \text{ for } i = 0, \dots, r-2 ~~ \text{ and } ~~
  s_{r-1} = V n^{1-r} + \delta. 
\end{equation*}
To see \eqref{bm}, first note that from the definition of
$\alpha(\theta)$ in \eqref{bhim1}, it is straightforward to check that  
\begin{equation}\label{ka1}
  (D^{(r)} \alpha(\theta))_j = (D^{(r)} \theta)^+_j \qt{for $j = 1, \dots, n-r$}
\end{equation}
and
\begin{equation}\label{ka2}
  (D^{(i)} \alpha(\theta))_1 = (D^{(i)} \theta)^+_1 \qt{for $0 \leq i < r$}. 
\end{equation}
From these identities, it is easy to verify \eqref{bm} in the
following way. Let $\theta \in \M(\kb)$ so that $\alpha(\theta) \in
\M_{\alpha}(\kb)$. Then $\kb_i \delta \leq (D^{(i)} \theta)_1 \leq
(\kb_i + 1) \delta$ for $0 \leq i < r$. This implies (because the map
$x \mapsto x^+$ is non-decreasing and subadditive) via \eqref{ka2}
that    
\begin{equation}\label{h1}
\kb_i^+ \delta \leq  (D^{(i)} \alpha(\theta))_1 = (D^{(i)} \theta)^+_1
\leq \kb_i^+ \delta + \delta . 
\end{equation} 
Also the identity \eqref{ka1} implies that $D^{(r)} \alpha(\theta)
\geq 0$  which, together with \eqref{h1}, means that
\begin{align*}
\kb_{r-1}^+ \delta \leq (D^{(r-1)} \alpha(\theta))_1 \leq \dots &\leq
(D^{(r-1)} \alpha(\theta))_{n-r+1} \\ &= V(D^{(r-1)} \alpha(\theta)) +
(D^{(r-1)} \alpha(\theta))_1 \\
&\leq V(D^{(r-1)} \alpha(\theta)) + \kb_{r-1}^+ \delta + \delta. 
\end{align*}
The statement \eqref{bm} will therefore be proved if we establish that
$V(D^{(r-1)} \alpha(\theta)) \leq V n^{1-r}$. This follows since
\begin{align*}
  V(D^{(r-1)} \alpha(\theta)) = \|D^{(r)} \alpha(\theta)\|_1 &=
  \|(D^{(r)} \theta)^+\|_1 \\ &\leq \|D^{(r)} \theta\|_1 = V(D^{(r-1)}
  \theta) \leq V n^{1-r}. 
\end{align*}
This proves \eqref{bm}. We can thus use Lemma \ref{pgr} to
bound $\log N(\epsilon/2, \M_{\alpha}(\kb))$ as 
\begin{equation*}
  \log N(\epsilon/2, \M_{\alpha}(\kb)) \leq C_r n^{1/(2r)}
 \left( \frac{\delta n^{r-1} + V}{\epsilon} \right)^{1/r}. 
\end{equation*}
Using the elementary inequality $(a + b)^{1/r} \leq a^{1/r} +
b^{1/r}$, we obtain the simpler inequality
\begin{equation}\label{aa3}
  \log N(\epsilon/2, \M_{\alpha}(\kb)) \leq C_r \frac{\delta^{1/r} n^{1 - 1/2r}}{\epsilon^{1/r}} + C_r \left(\frac{V \sqrt{n}}{\epsilon} \right)^{1/r} . 
\end{equation}
Combining \eqref{aa1}, \eqref{aa2} and \eqref{aa3}, we obtain 
\begin{equation*}
\log N(\epsilon, S_r(V, t))   \leq r \log \left(2 + \frac{2^r
    t}{\delta} \right) + C_r \frac{\delta^{1/r}
    n^{1 - 1/2r}}{\epsilon^{1/r}} + C_r \left(\frac{V
      \sqrt{n}}{\epsilon} \right)^{1/r} . 
\end{equation*}
Note that $\delta > 0$ above is arbitrary. Taking $\delta = \epsilon
\sqrt{n} n^{-r}$, we obtain \eqref{dprg.eq} which completes the proof
of Theorem \ref{dprg}. 
\end{proof}

\subsection{Study of the tangent cones $T_{\krv}(\theta)$ and the
  proof of Lemma \ref{alta}} \label{tcc} 
This section deals with the tangent cone (see \eqref{tcod} for the
definition of tangent cone) of the convex set
$K^{(r)}(V)$ (defined in \eqref{krv}) at $\theta \in \R^n$ for which
$V^{(r)} (\theta) = V$. This tangent cone is denoted by
$T_{\krv}(\theta)$. The ultimate goal of this subsection is to prove
Lemma \ref{alta} which was crucial for the proof of Theorem \ref{ada}.  

We start with the statement and proof of a lemma (Lemma \ref{tcr})
which gives a precise characterization of
$T_{\krv}(\theta)$. Recall the notation $V_{a, b}(\alpha)$ (from
\eqref{vab}) for $1 \leq a \leq b \leq m$ and $\alpha \in \R^m$. Also
recall, from Section \ref{ms}, the notion of $r^{th}$ order knots
(along with their signs) of vectors in $\R^n$.  

\begin{lemma}\label{tcr}
  Fix $r \geq 1$, $n \geq r + 1$ and let $K^{(r)}(V)$ be as in
  \eqref{krv}. Let  $\theta\in K^{(r)}(V)$ be such that $V^{(r)}
  (\theta) = V$.  
  \begin{enumerate}[(i)]
  \item  Let $2 \leq j_1 < \dots < j_k \leq n-r+1$
  denote all the $r^{th}$ order knots of $\theta$ along with
  associated signs $\rs_1, \dots, \rs_k \in \{-1, 1\}$. Then 
   \begin{equation}\label{tcre.eq}
\begin{split}
    T_{K^{(r)}(V)}(\theta) &= \left\{\alpha \in \R^n : \sum_{i=0}^{k} 
      V_{j_i, j_{i+1} - 1} (D^{(r-1)}\alpha) \right. \\
& \left. \leq \sum_{i=1}^k \rs_i
      \left((D^{(r-1)} \alpha )_{j_i-1} - (D^{(r-1)}
          \alpha )_{j_i} \right)  \right\} 
\end{split}
  \end{equation}
  with the convention $j_0 = 1$ and $j_{k+1} = n-r+2$. 
\item Suppose $2 \leq j_1 < \dots < j_k \leq n-r+1$
  denote any set of indices which contains all the $r^{th}$ order
  knots of 
  $\theta$. Let $\rs_1, \dots, \rs_k$ be such that
  $\rs_i$ is the sign of the knot corresponding to $j_i$ if $j_i$ is a
  knot and $\rs_i \in \{-1, 0, 1\}$ is arbitrary if $j_i$ is not a 
  knot. Then  
 \begin{equation}\label{tcr.eq}
\begin{split}
    T_{K^{(r)}(V)}(\theta) &\subseteq \left\{\alpha \in \R^n :
      \sum_{i=0}^{k}  V_{j_i, j_{i+1} - 1} (D^{(r-1)}\alpha)
    \right. \\ & \left. \leq
      \sum_{i=1}^k \rs_i 
      \left((D^{(r-1)} \alpha )_{j_i-1} - (D^{(r-1)}
          \alpha )_{j_i} \right)  \right\} 
\end{split}
  \end{equation}
where again $j_0 = 1$ and $j_{k+1} = n-r+2$. 
  \end{enumerate}
\end{lemma}

%\vspace{0.2in} 

\begin{remark}
  Lemma \ref{tcr} only deals with those $\theta \in K^{(r)}(V)$ for 
  which $V^{(r)}(\theta)  = V$. On the other hand, it is
  easy to see that when $V^{(r)}(\theta) < V$, the
  tangent cone $T_{K^{(r)}(V)}(\theta)$ equals $\R^n$.     
\end{remark}

%\vspace{0.2in} 

\begin{remark}
  It must be clear from the right hand side of \eqref{tcr.eq} that the
  tangent cone $T_{K^{(r)}(V)}(\theta)$ only depends on the knot
  indices $j_1, \dots,  j_k$ and the knot signs $\rs_1, \dots,
  \rs_k$. For example, the exact values of $\theta$ at $j_1, \dots,
  j_k$ are not relevant for the determination of the tangent
  cone. 
\end{remark}

\begin{proof}[Proof of Lemma \ref{tcr}]
  We start with the proof of the first part of the lemma. Let $T$
  denote the set on the right hand side of \eqref{tcre.eq}. Let 
  us first prove that $T \subseteq T_{K^{(r)}(V)}(\theta)$. For 
  this, we fix $\alpha \in T$ and argue that $\alpha \in 
  T_{K^{(r)}(V)}(\theta)$, i.e., we show that there exists $c >0$ such
  that $\theta + c \alpha \in K^{(r)}(V)$. For $c > 0$, first note
  that, by the  
  definition of $V(\cdot)$, the variation $v := V(D^{(r-1)}(\theta + 
  c\alpha))$ can be written as   
  \begin{equation*}
    \begin{split}
  v &=  \sum_{i=0}^k V_{j_i, j_{i+1} - 1} (D^{(r-1)} (\theta + c\alpha))
   \\ &+ \sum_{i=1}^k 
  \left|(D^{(r-1)} (\theta + c\alpha))_{j_i} - (D^{(r-1)}(\theta +
    c\alpha))_{j_i - 1} \right| 
    \end{split}
  \end{equation*}
  Because $\theta$ has no $r^{th}$ order knots except at $j_1,
  \dots, j_k$, first term above can be simplified to obtain 
  \begin{equation*}
\begin{split}
    v &= c \sum_{i=0}^k V_{j_i, j_{i+1} - 1} (D^{(r-1)} \alpha) \\
   & + \sum_{i=1}^k 
  \left|(D^{(r-1)} (\theta + c\alpha))_{j_i} - (D^{(r-1)}(\theta +
    c\alpha))_{j_i - 1} \right|.  
\end{split}
  \end{equation*}
  Now when $c > 0$ is sufficiently small, we can rewrite the above as 
   \begin{align*}
    v &= c \sum_{i=0}^k V_{j_i, j_{i+1} - 1} (D^{(r-1)} \alpha) \\
    &+ \sum_{i=1}^k \rs_i \left\{(D^{(r-1)} (\theta + c\alpha))_{j_i} - (D^{(r-1)}(\theta +
    c\alpha))_{j_i - 1} \right\} \\
&= V(D^{(r-1)} \theta) \\ &+ c \left\{\sum_{i=0}^{k}
      V_{j_i, j_{i+1} - 1} (D^{(r-1)}\alpha) - \sum_{i=1}^k \rs_i
      \left((D^{(r-1)} \alpha )_{j_i-1} - (D^{(r-1)}
          \alpha )_{j_i} \right)  \right\} \\
&\leq V n^{1-r}
  \end{align*}
where the last step follows from the fact that $\alpha \in T$ and $V(D^{(r-1)} \theta) = V n^{1-r}$. This
proves $T \subseteq T_{K^{(r)}(V)}$.  

We shall now verify that $T_{K^{(r)}(V)} \subseteq T$. As $T$ is a closed convex cone, it is enough to show that $\alpha - \theta \in T$ for every $\alpha \in K^{(r)}(V)$. For this, as $D^{(r-1)} (\alpha - \theta) = D^{(r-1)} \alpha - D^{(r-1)} \theta$, we need to show
that  
\begin{equation}\label{ve1}
  \sum_{i=0}^k V_{j_i, j_{i+1} - 1} (D^{(r-1)}(\alpha - \theta)) +
  \sum_{i=1}^k \rs_i \left((D^{(r-1)} \alpha)_{j_i} - (D^{(r-1)} 
    \alpha)_{j_i - 1} \right) 
\end{equation}
is not larger than 
\begin{equation} \label{ve2}
  \sum_{i=1}^k  \rs_i \left((D^{(r-1)}
    \theta)_{j_i} - (D^{(r-1)}  \theta)_{j_i - 1} \right). 
\end{equation}
This is easy because \eqref{ve2} equals
$V(D^{(r-1)} \theta) = V n^{1-r}$ and \eqref{ve1} is clearly bounded 
from above by $V(D^{(r-1)} \alpha) \leq V n^{1-r}$. This proves the first part of the lemma.  

The second part is an easy consequence of the first part of the
lemma and the following trivial observation. If $j_i$ and $j_{i+1}$
denote two consecutive knots of $\theta$ and if $j_i'$ is any
integer with $j_i < j_i' < j_{i+1}$, then 
\begin{equation*}
  V_{j_i, j_{i+1} - 1}(\Delta) \geq V_{j_i, j_i'-1}(\Delta) + V_{j_i',
  j_{i+1} - 1}(\Delta) + \rs_i' \left(\Delta_{j_i'} - \Delta_{j_i'-1}
\right) 
\end{equation*}
for every $\alpha \in \R^n$ and $\rs_i' \in \{-1, 0, 1\}$ where
$\Delta := D^{(r-1)} \alpha$.   
\end{proof}

The following corollary to Lemma \ref{tcr} gives a simple necessary
condition for a vector $\alpha$ to belong to $T_{K^{(r)}(V)}(\theta)$.   

%\vspace{0.2in} 

\begin{corollary}\label{cort}
Fix $r \geq 1$ and let $K^{(r)}(V)$ be as in \eqref{krv}. Let 
  $\theta$ be any point in $K^{(r)}(V)$ for which $V^{(r)}(\theta) =
  V$. Let $2 \leq j_1 < \dots <  j_k \leq n-r+1$ 
  and $\rs_1, \dots, \rs_k \in \{-1, 0, 1\}$ be as in Lemma
  \ref{tcr}(ii). For every $0 \leq i \leq k$, let $\ell_i$ denote an
  arbitrary index lying in the set $\{j_i, \dots, j_{i+1} - 1\}$. Then
  for every $\alpha \in T_{K^{(r)}(V)}(\theta)$ we have (with  the
  convention that $j_0 = 1$, $j_{k+1} = n-r+2$, $\rs_0 = 0$ and
  $\rs_{k+1} = 0$) 
\begin{equation}
  \label{cort.eq}
  \sum_{i=0}^k \Gamma_i(\alpha, \ell_i) \leq \sum_{i=0}^k (\rs_{i+1} -
  \rs_i) (D^{(r-1)} \alpha)_{\ell_i} 
\end{equation}
where 
\begin{equation*}
  \Gamma_i(\alpha, \ell_i) := V_{j_i, j_{i+1} - 1}(\Delta) -
  \rs_{i+1} \left(\Delta_{j_{i+1} - 1} - \Delta_{\ell_{i}} \right) - \rs_{i}
    \left(\Delta_{\ell_i} - \Delta_{j_{i}} \right) 
\end{equation*}
with $\Delta = (\Delta_1,\ldots, \Delta_{n-r+1}) := D^{(r-1)} \alpha$. 
\end{corollary}

%\vspace{0.2in} 

\begin{proof}[Proof of Corollary \ref{cort}]
  Fix $\alpha \in T_{K^{(r)}(V)}(\theta)$. Lemma \ref{tcr} gives
  that  
  \begin{equation}\label{abc}
\sum_{i=0}^{k}
      V_{j_i, j_{i+1} - 1} (\Delta) \leq \sum_{i=1}^k \rs_i
      \left(\Delta_{j_i-1} - \Delta_{j_i} \right).     
  \end{equation}
Writing 
\begin{equation*}
V_{j_i, j_{i+1} - 1}(\Delta) = 
\Gamma_i(\alpha, \ell_i) +  \rs_{i+1} \left(\Delta_{j_{i+1}
    - 1} - \Delta_{\ell_{i}} \right) + \rs_{i}
    \left(\Delta_{\ell_i} - \Delta_{j_{i}} \right) 
\end{equation*}
in \eqref{abc}, we deduce that $\sum_{i=0}^k \Gamma_i(\alpha, \ell_i)$
is bounded from above by 
\begin{equation*}
  \sum_{i=1}^k \rs_i \left(\Delta_{j_i - 1} - \Delta_{j_i} \right) -
  \sum_{i=0}^k \rs_{i+1} \left(\Delta_{j_{i+1} - 1} - \Delta_{\ell_i}
  \right) - \sum_{i=0}^k \rs_i \left(\Delta_{\ell_i} - \Delta_{j_i}
  \right). 
\end{equation*}
It is now trivial to check that the expression above equals the right
hand side of \eqref{cort.eq} which completes the proof of Corollary
\ref{cort}.  
\end{proof}

%Corollary \ref{cort} can be understood in the following way. Suppose
%that the indices $\ell_0, \dots, \ell_k$ are chosen in such a way
%that $\delta$ is small (this is indeed possible if $\alpha$ satisfies
%additional conditions such as having a bounded norm; see Lemma
%\ref{delc}). In that case,  Corollary \ref{cort} says that on every
%set $\{j_u, \dots, j_{u+1} - 1\}$, the sequence $D^{(r-1)} \alpha$ is
%essentially made of at most two monotonic parts (the part from $j_u$
%to $\ell_u$ and the part from $\ell_u$ to $j_{u+1} - 1$). When $r =
%1$, this means that $\alpha$ itself is essentially made of at most
%two monotonic parts while for $r = 2$, this means that $\alpha$ is
%essentially made of at most two convex/concave parts. This intuition
%allows us to use results on the gaussian widths of sets of monotone
%and convex sequences to control the gaussian widths of the tangent
%cones of $K^{(1)}$ and $K^{(2)}$ respectively.  

We next show that under the assumption that $\|\alpha\| \leq 1$, the
right hand side of \eqref{cort.eq} can be made small by choosing
$\ell_0, \dots, \ell_k$ appropriately. This is the content of the next
lemma. Let $2 \leq j_1 < \dots < j_k \leq n-r+1$ and $\rs_1, \dots,
\rs_k \in \{-1, 0, 1\}$ be as in Lemma \ref{tcr}(ii). Also let $j_0 =
1$, $j_{k+1} = n-r+2$ and $\rs_0 = \rs_{k+1} = 0$. The indices $j_0,
j_1, \dots, j_k, j_{k+1}$ can be used to define a partition of $\{1,
\dots, n\}$ in the following way: $\I_0 := \{j_0, \dots, j_1 + r -
2\}$ and    
\begin{equation*}
  \I_i = \{j_i + r - 1, \dots, j_{i+1} + r - 2\} \qt{for $i = 1,
    \dots, k$}.   
\end{equation*}
Observe that the length of $\I_i$ equals $n_i$ where $n_0 :=
j_1 + r - 2$ and $n_i = j_{i+1} - j_i$ for $1 \leq i \leq k$. 

%\vspace{0.2in} 

\begin{lemma}\label{cr1}
Let $\theta \in \R^n$ and let $2 \leq j_1 < \dots < j_k \leq n-r+1$
and $\rs_1, \dots, \rs_k \in \{-1, 0, 1\}$ be as in Lemma
\ref{tcr}(ii). Also let $j_0 = 1$, $j_{k+1} = n-r+2$ and $\rs_0 =
\rs_{k+1} = 0$. Further let $\I_0, \dots, \I_k$ and $n_0, \dots, n_k$
be as described above. For every $\alpha \in \R^n$ with $\|\alpha\|
\leq 1$, there  exist indices $\ell_0 \in \I_0, \dots, \ell_k \in
\I_k$ such that 
 \begin{equation}\label{cr1.eq}
   \sum_{i=0}^k (\rs_{i+1} - \rs_i) (D^{(r-1)} \alpha)_{\ell_i} \leq C_r
   \sqrt{\sum_{i=0}^{k} n_i^{1 - 2r} I \{\rs_i
     \neq \rs_{i+1}\}}
 \end{equation}
where $C_r$ is a positive constant that depends only on $r$. 
\end{lemma}

Note that the role of $\theta \in \R^n$ in the above lemma is just to
define the $j_i$'s and the $\rs_i$'s as in Lemma \ref{tcr}(ii). 

The proof of Lemma \ref{cr1} is given next. A crucial role in this
proof is played by the following result on the magnitude of $\min_{1
  \leq i \leq n-r+1} (D^{(r-1)} \theta)_i$ for a vector $\theta$ with
$\|\theta\| \leq 1$. This result (proved in Subsection \ref{pdo})
might be of independent interest.  

\begin{lemma}\label{do}
  Fix $r \geq 1$. There exists a positive constant $C_r$
  depending only on $r$ such that for every $n \geq 2r$, $t > 0$ and
  $\theta \in \R^n$ with $\|\theta\| \leq t$, there exist indices
  $\ell_1, \ell_2 \in \{1, \dots, n-r+1\}$ such that   
  \begin{equation}\label{do.eq}
    (D^{(r-1)} \theta)_{\ell_1} \leq C_r n^{(1/2) - r} t ~~ \text{   and
    } ~~ (D^{(r-1)} \theta)_{\ell_2} \geq -C_r n^{(1/2) - r} t. 
  \end{equation}
\end{lemma}

\begin{remark}\label{dore}
  Lemma \ref{do} is trivial for $r = 1$ (when it holds with $C_1 = 1$)
  but the extension to $r \geq 2$ is non-trivial. Also, for general $r
  \geq 2$, the two indices $\ell_1$ and $\ell_2$ will be different and
  it will be incorrect to claim that for every $\theta \in \R^n$ with
  $\|\theta \| \leq 1$, there exists a single index $\ell \in \{1,
  \dots, n-r+1\}$ for which   $|(D^{(r-1)} \theta)_{\ell}| \leq C_r
  n^{(1/2) - r} t$. One may define $\ell_1$ and $\ell_2$ as 
  \begin{equation*}
    \ell_1 := \argmin_{1 \leq j \leq n-r+1} (D^{(r-1)} \theta)_j ~~
    \text{ and } \ell_2 := \argmax_{1 \leq j \leq n-r+1} (D^{(r-1)} 
    \theta)_j. 
  \end{equation*}
\end{remark}

We are now ready to prove Lemma \ref{cr1}.

\begin{proof}[Proof of Lemma \ref{cr1}]
  The proof of Lemma \ref{cr1} is crucially reliant on Lemma \ref{do}
  (proved in Section \ref{pdo}) which essentially says that
  $$\sup_{\alpha \in \R^n: \|\alpha\| \leq 
    t} \min_{1 \leq i \leq n-r+1} (D^{(r-1)} \alpha)_i \leq C_r
  n^{(1/2) - r} t$$ for every $t > 0$ and $n \geq r$. 

  Fix $\alpha \in \R^n$. Define
  \begin{equation*}
    \alpha^{(0)} := (\alpha_{j_0}, \dots, \alpha_{j_1 + r  - 2}) 
  \end{equation*}
 and
 \begin{equation*}
   \alpha^{(u)} := (\alpha_{j_u + r - 1}, \dots, \alpha_{j_{u+1} + r
      - 2}) 
 \end{equation*}
for $u = 1, \dots, k$.  Note that the vector $\alpha^{(u)}$ has length
exactly equal to $n_u$, for $u = 0, \dots, k$.  

  Fix $0 \leq u \leq k$ and let $t_u := \|\alpha^{(u)} \|$. By Lemma
  \ref{do}, there exists an index $\ell'_u \in \{1, \dots, n_u - r +
  1\}$ such that 
  \begin{equation}\label{pb}
    (\rs_{u+1} - \rs_u) (D^{(r-1)} \alpha^{(u)})_{\ell_u'} \leq 2 C_r
    n_u^{1/2 - r} t_u I\{\rs_u \neq \rs_{u+1}\}
  \end{equation}
for a constant $C_r$ depending on $r$ alone. Taking  
  \begin{equation*}
    \ell_0 := \ell_0' ~~ \text{ and } ~~ \ell_u := j_u + r - 2 +
    \ell_u' ~~ \text{ for } 1 \leq u \leq k, 
  \end{equation*}
  and using the fact that $(D^{(r-1)} \alpha^{(u)})_{\ell_u'} = (D^{(r-1)}
  \alpha)_{\ell_u}$, we deduce from \eqref{pb} that
  \begin{equation*}
    (\rs_{u+1} - \rs_u) (D^{(r-1)} \alpha)_{\ell_u} \leq 2 C_r 
    n_u^{1/2 - r} t_u I\{\rs_u \neq \rs_{u+1}\}  
  \end{equation*}
 for every $u = 0, 1,  \dots, k$. The left hand side of \eqref{cr1.eq}
 can therefore be bounded as 
   \begin{align*}
    \sum_{i=0}^k (\rs_{i+1} - \rs_i) (D^{(r-1)} \alpha)_{\ell_i} &\leq
   2 C_r \sum_{i=0}^k n_i^{1/2 - r} t_i I\{\rs_i \neq \rs_{i+1}\} \\
&\leq 2 C_r \sqrt{\sum_{i=0}^k n_i^{1 - 2r} I\{\rs_i \neq \rs_{i+1}
\}}  \sqrt{\sum_{i=0}^k t_i^2}  \\ &\leq 2 C_r \sqrt{\sum_{i=0}^k
                                     n_i^{1 - 2r} I\{\rs_i \neq
                                     \rs_{i+1} \}}  
  \end{align*}
where we have used Cauchy-Schwarz inequality and the fact that
$\sum_{i=0}^k t_i^2 = \|\alpha\|^2 \leq 1$. This completes the proof
of Lemma \ref{cr1}. 
\end{proof}

We now have all the ingredients to complete the proof of Lemma
\ref{alta}. 

\begin{proof}[Proof of Lemma \ref{alta}]
The result clearly follows by combining Corollary \ref{cort} and
Lemma \ref{cr1}.  
\end{proof}

\subsection{Proof of Lemma~\ref{fina}}\label{gwc}    
The goal of this subsection is to prove Lemma \ref{fina} which was
crucial for the proof of Theorem \ref{ada}. We shall actually prove
the following more precise result from which Lemma \ref{fina} easily
follows.  

%\vspace{0.2in} 

\begin{lemma}\label{fina.full}
  Fix $r \ge 1$, $n \geq r$, $1 \leq \ell \leq n-r+1$, $t > 0$ and
  $\delta \geq 0$. For $\theta \in \R^n$, let $\Delta(\theta) =
  (\Delta_1(\theta),\ldots, \Delta_{n-r+1}(\theta)):= D^{(r-1)}
  \theta$. For every $\rs_1, \rs_2 \in \{-1, 0, 1\}$, the quantity     
  \begin{eqnarray*}
    G := \E \Big[\sup \big\{\left<\xi, \theta \right>
      & : & \theta \in \R^n,  \|\theta 
        \| \le t, \mbox{ and } \\
        \vspace{2in}& &V(\Delta(\theta)) \leq \rs_1 (\Delta_{\ell}(\theta) -
        \Delta_1(\theta)) \\ & & + \rs_2 ( \Delta_{n-r+1}(\theta) -
        \Delta_{\ell}(\theta)) + \delta \big\} \Big]   
  \end{eqnarray*}
  is bounded from above in the following way. When $\rs_1 = \rs_2 =
  0$, we have
  \begin{equation*}
    G \leq C_r \sigma t^{(2r-1)/(2r)} \delta^{1/(2r)} n^{(2r-1)/(4r)} + 
    C_r \sigma t \sqrt{\log (en)}. 
  \end{equation*}
  When $\rs_1 = 0, \rs_2 \neq 0$, we have
  \begin{align*}
    G &\leq C_r     \sigma \left\{t^{(2r-1)/(2r)} \ell_1^{(2r-1)/4r} \delta^{1/(2r)}
  + t \sqrt{\log(e \ell_1)} \right\}  \\ &+  C_r \sigma \left(t + \delta
  \ell^{(2r-1)/2}_2 \right) \sqrt{\log(e \ell_2)}. 
  \end{align*}
  When $\rs_1 \neq 0, \rs_2 = 0$, we have
  \begin{align*}
G &\leq C_r   \sigma \left(t + \delta
  \ell^{(2r-1)/2}_1 \right) \sqrt{\log(e \ell_1)} \\
&+ C_r \sigma
\left\{t^{(2r-1)/(2r)} \ell_2^{(2r-1)/4r} \delta^{1/(2r)} + t
  \sqrt{\log(e \ell_2)} \right\}  . 
  \end{align*}
  Finally when $\rs_1 \neq 0, \rs_2 \neq 0$, we have
  \begin{align*}
    G \leq C_r \sigma \left(t + \delta
  \ell^{(2r-1)/2}_1 \right) \sqrt{\log(e \ell_1)} + C_r\sigma \left(t +
  \delta \ell^{(2r-1)/2}_2 \right) \sqrt{\log(e \ell_2)}. 
  \end{align*}
In each case, $\ell_1 := \ell + r - 1$, $\ell_2 := n - \ell - r + 1$
and $C_r$ is a constant depending on $r$ alone.  
\end{lemma}

%\vspace{0.2in} 

\begin{remark}\label{fincon}
  It is easy to see that Lemma \ref{fina.full} implies Lemma
  \ref{fina}. This is a consquence of the fact that the integers
  $\ell_1$ and $\ell_2$ appearing in Lemma \ref{fina.full} are both
  bounded from above by $n$. 
\end{remark}

The rest of this subsection is dedicated to the proof of Lemma
\ref{fina.full}. As described in Remark above, Lemma
\ref{fina.full} implies Lemma \ref{fina}. Before proceeding to prove
Lemma \ref{fina.full}, we prove an auxiliary result below which will
considerably simplify the proof of Lemma \ref{fina.full}.   

\begin{lemma}\label{am}
  For every $r \geq 1$, $n \geq r, t > 0$ and $\delta \geq 0$, we have
  that 
\begin{equation*}
    \E \left[\sup_{\substack{\theta \in \R^n: \|\theta\| \leq t \\  V(D^{(r-1)}\theta) \leq
      (D^{(r-1)}\theta)_{n-r+1} - (D^{(r-1)}\theta)_1 + \delta}}
    \left<\xi, \theta \right> \right]
\end{equation*}
is bounded from above by 
\begin{equation*}
C_r \sigma \left(t + \delta n^{r-1/2} 
    \right) \sqrt{\log (en)} 
  \end{equation*}
for a constant $C_r$ that depends on $r$ alone. 
\end{lemma}

Lemma \ref{am} is proved below. This proof will use Lemma \ref{dfi}
(stated and proved in Subsection \ref{dfisec}) which provides a
formula for an arbitrary vector $\theta$  in terms of $D^{(r)} \theta$
and \citet[Theorem 1 in the supplementary material]{bellec2015annals}
which provides a bound for the statistical dimension of the cone of
all $\gamma \in \R^n$ which satisfy $\min_{1 \leq i  \leq n-r}
(D^{(r)} \gamma)_i \geq 0$.  

\begin{proof}[Proof of Lemma \ref{am}] 
We can assume without loss of generality that $t = 1$ (which is
ensured by scaling and replacing $\delta$ by $\delta/t$). The idea of
this proof is to write $\theta$ as the difference of two vectors
$\alpha(\theta)$ and $\beta(\theta)$  which satisfy $\min_{1 \leq i
  \leq n-r} (D^{(r)} \alpha(\theta))_i \geq 0$ and $\min_{1 \leq i
  \leq n-r} (D^{(r)} \beta(\theta))_i \geq 0$. \citet[Theorem 1 in the
supplementary material]{bellec2015annals}  will then be used to
control the Gaussian width of the cone of all $\gamma \in \R^n$ which
satisfy $\min_{1 \leq i  \leq n-r} (D^{(r)} \gamma)_i \geq 0$.  

To construct the sequences $\alpha(\theta)$ and $\beta(\theta)$, we
use Lemma \ref{dfi} which gives the following formula for expressing a
vector $\theta \in \R^n$ in terms of $D^{(r)} \theta$ and $(D^{(i)}
\theta)_1$ for $i = 0, \dots, r-1$: 
  \begin{equation*}
    \theta_i = \sum_{j=1}^{i-r} \binom{i-j-1}{r-1} (D^{(r)} \theta)_j
    + \sum_{j=1}^r \binom{i-1}{j-1} (D^{(j-1)} \theta)_1  
  \end{equation*}
  where we take the convention that $\binom{a}{b} = 0$ for $b > a$, 
  $\binom{0}{0} = 1$ so that the first term in the right hand side is
  zero unless $i > r$. Motivated by the above expression, we define
  $\alpha(\theta) := (\alpha_1(\theta), \dots, \alpha_n(\theta))$ and
  $\beta(\theta) := (\beta_1(\theta), \dots, \beta_n(\theta))$ in the
  following way:  
\begin{equation*}
  \alpha_i(\theta) := \sum_{j=1}^{i-r} \binom{i-j-1}{r-1} (D^{(r)}
  \theta)_{j}^+ + \sum_{j=1}^r \binom{i-1}{j-1} (D^{(j-1)} \theta)_1 
\end{equation*}
and 
\begin{equation*}
  \beta_i(\theta) := \sum_{j=1}^{i-r} \binom{i-j-1}{r-1} (D^{(r)}
  \theta)_{j}^-
\end{equation*}
where $x^+ := \max(x, 0)$ and $x^- := x^+ - x$. It
is easy then to observe the following: (a) $\theta = \alpha(\theta) -
\beta(\theta)$, (b) $(D^{(r)} \alpha(\theta))_i = (D^{(r)}
\theta)_i^+$, $(D^{(r)} \beta(\theta))_i = (D^{(r)}
\theta)_i^-$, (c) both vectors $\alpha(\theta)$ and
$\beta(\theta)$ belong to $S_n^{[r]}$ where 
\begin{align*}
  S_n^{[r]} &:= \left\{\gamma \in \R^n: \min_{1 \leq i \leq n-r} (D^{(r)}
    \gamma)_i \geq 0 \right\} \\
&= \left\{\gamma \in \R^n: (D^{(r-1)} \gamma)_1
    \leq \dots \leq (D^{(r-1)} \gamma)_{n-r+1} \right\},
\end{align*}
and (d) $(D^{(j-1)} \beta(\theta))_1 = 0$ for $1 \le j \leq r$. From these, it
follows that 
\begin{align*}
  V(D^{(r-1)} \theta) = \|D^{(r)} \theta\|_1 &= \sum_{i=1}^{n-r}
  |(D^{(r)} \theta)_i| \\
&= \sum_{i=1}^{n-r} (D^{(r)} \alpha(\theta))_i + \sum_{i=1}^{n-r}
(D^{(r)} \beta(\theta))_i \\
&= (D^{(r-1)} \alpha(\theta))_{n-r+1} - (D^{(r-1)} \alpha(\theta))_1 \\ &+
(D^{(r-1)} \beta(\theta))_{n-r+1} - (D^{(r-1)} \beta(\theta))_1 \\
&= (D^{(r-1)} \alpha(\theta))_{n-r+1} - (D^{(r-1)} \alpha(\theta))_1 \\ &+
(D^{(r-1)} \beta(\theta))_{n-r+1}. 
\end{align*}
From the above (and the fact that $ D^{(r-1)} \theta = D^{(r-1)}
\alpha(\theta) - D^{(r-1)} \beta(\theta)$), it is straightforward to
observe that the condition  
\begin{equation*}
V(D^{(r-1)}\theta) \leq
      (D^{(r-1)}\theta)_{n-r+1} - (D^{(r-1)}\theta)_1 + \delta  
\end{equation*}
is equivalent to 
\begin{equation}\label{bth}
  (D^{(r-1)} \beta(\theta))_{n-r+1} \leq \frac{\delta}{2}.  
\end{equation}
Now for $\beta(\theta) \in S_n^{[r]}$, $(D^{(j-1)} \beta(\theta))_1 = 0$ for
$1 \le j \leq r$, and satisfying \eqref{bth}, we can use Lemma
\ref{dfi} (with $r$ replaced by $r - 1$) to observe that  
\begin{equation}\label{bmnu}
 0 \leq \beta_i(\theta) \leq \frac{\delta}{2} \sum_{j=1}^{i-r+1}
 \binom{i-j-1}{r-2} = \frac{\delta}{2} \binom{i-1}{r-1} \leq
 \frac{\delta}{2} i^{r-1}
\end{equation}
where we have used the following elementary identity involving
binomial coefficients: for every two integers $a$ and $b$ with $0 \leq
b < a$, we have 
\begin{equation}\label{bnom}
  \binom{b}{b} + \binom{b+1}{b} + \dots + \binom{a}{b} =
  \binom{a+1}{b+1}. 
\end{equation}
Note the presence of the
term $r - 2$ in some of the binomial coefficients in \eqref{bmnu}
which will be negative when $r = 1$. But the inequality $0 \leq
\beta_i(\theta) \leq \delta/2$ is also true for $r = 1$ which can
directly be seen from $\beta_n(\theta) \leq \delta/2$ (inequality
\eqref{bth} for $r = 1$), the fact that $S_n^{[1]}$ consists of
monotone sequences (so that $\beta_i(\theta) \leq \beta_n(\theta)$)
and the fact that $(D^{(j-1)} \beta(\theta))_1 = 0$ for $1 \leq j \leq
r$ (which for $r = 1$ gives $\beta_1(\theta) = 0$).   
 
A consequence of \eqref{bmnu}  is that
\begin{equation*}
  \|\beta(\theta)\|^2 \leq \frac{\delta^2}{4} \sum_{i=1}^n i^{2r - 2} \leq
  \frac{\delta^2}{4} n^{2r - 1}
\end{equation*}
or $\|\beta(\theta)\| \leq \delta n^{r - 1/2}/2$. Because
$\|\theta\| \leq 1$, we further deduce that 
\begin{equation*}
  \|\alpha(\theta)\| \leq \|\theta\| + \|\beta(\theta)\| \leq 1 +
  \frac{\delta}{2} n^{r - 1/2}.  
\end{equation*}
Based on these observations, if 
\begin{equation*}
G :=   \E \left[\sup_{\substack{\theta \in \R^n: \|\theta\| \leq t \\  V(D^{(r-1)}\theta) \leq
      (D^{(r-1)}\theta)_{n-r+1} - (D^{(r-1)}\theta)_1 + \delta}}
    \left<\xi, \theta \right> \right],
\end{equation*}
we can write 
\begin{equation*}
  G \leq \E \left[\sup_{\alpha \in S_n^{[r]} : \|\alpha\| \leq 1 + \delta
    n^{r - 1/2} /2}  \left<\xi, \alpha \right> \right]+ \E \left[\sup_{\beta \in
      S_n^{[r]} : \|\beta\| \leq  \delta n^{r - 1/2} /2}  \left<\xi,
      -\beta \right>\right]. 
\end{equation*}
By an elementary scaling property and the fact that $\xi$ and $-\xi$ have
the same distribution, we deduce that 
\begin{equation*}
  G \leq \left(1 + {\delta} n^{r - 1/2} \right) w(S_n^{[r]}) 
\end{equation*}
where $w(S_n^{[r]})$ is the Gaussian width of $S_n^{[r]}$ (defined in
\eqref{gwef}). The right hand side above can be bounded using
\citet[Theorem 1 in the supplementary material]{bellec2015annals}
which implies that  
\begin{equation*}
w(S_n^{[r]})\leq C_r \sigma \sqrt{\log (en)} 
\end{equation*}
for a constant $C_r$. To be precise, \citet[Equation (5) in the
supplementary material]{bellec2015annals} gives a bound for
$\delta(S_n^{[r]})$. The connection \eqref{gast} between Gaussian
width and statistical dimension then leads to the above stated
bound. We therefore have   
\begin{equation*}
  G \leq C_r \left(1 + {\delta} n^{r - 1/2} \right)
  \sqrt{\log (en)}. 
\end{equation*}
which completes the proof of Lemma \ref{am}. 
\end{proof}

We are now ready to prove Lemma \ref{fina.full}. 
\begin{proof}[Proof of Lemma \ref{fina.full}]
  The case when $\rs_1 = \rs_2 = 0$ follows directly from Lemma
  \ref{svg} so we assume that at least one of $\rs_1$ and $\rs_2$ is
  non-zero.  

  For $\theta \in \R^n$, let $\theta^{(1)} := (\theta_1, \dots,
  \theta_{\ell+r-1})$ and $\theta^{(2)} := (\theta_{\ell+r}, \dots,
  \theta_n)$. We analogously define $\xi^{(1)}$ and
  $\xi^{(2)}$. Recall that $\Delta \equiv \Delta(\theta) =
  (\Delta_1(\theta),\ldots, \Delta_{n-r+1}(\theta)):= D^{(r-1)}
  \theta$. We first 
  claim that under the assumption $V(\Delta) \leq \rs_1 (\Delta_{\ell} -
  \Delta_1) + \rs_2(\Delta_{n-r+1} - \Delta_{\ell}) + \delta$, we have  
  \begin{equation}\label{kac1}
    V(D^{(r-1)} \theta^{(1)}) = V(\Delta_1, \dots, \Delta_{\ell}) \leq
    \rs_1(\Delta_{\ell} - \Delta_1) + \delta
  \end{equation}
  and 
  \begin{equation}\label{kac2}
    V(D^{(r-1)} \theta^{(2)}) = V(\Delta_{\ell+r}, \dots, \Delta_{n -
      r + 1}) \leq \rs_2(\Delta_{n-r+1} - \Delta_{\ell + r}) + \delta. 
  \end{equation}
Inequality \eqref{kac1} is a consequence of 
  \begin{align*}
   \rs_1 (\Delta_{\ell} -
  \Delta_1) + \rs_2(\Delta_{n-r+1} - \Delta_{\ell}) + \delta \geq
  V(\Delta) \geq V(\Delta_1, \dots, \Delta_{\ell}) +
  \rs_2(\Delta_{n-r+1} - \Delta_{\ell}) 
  \end{align*}
while \eqref{kac2} is a consequence of 
  \begin{align*}
   \rs_1 (\Delta_{\ell} -
  \Delta_1) + \rs_2(\Delta_{n-r+1} - \Delta_{\ell}) + \delta &\geq
    V(\Delta) \geq \rs_1(\Delta_{\ell} - \Delta_1) \\ &+ V(\Delta_{\ell+r},
  \dots, \Delta_{n-r+1}) + \rs_2(\Delta_{\ell+r} - \Delta_{\ell}). 
  \end{align*}
From inequalities \eqref{kac1} and \eqref{kac2}, and the fact that
$\left<\xi,  \theta \right> = \sum_{i=1}^2 \left<\xi^{(i)},
  \theta^{(i)}\right>$, it follows that $G \leq G_1 + G_2$ where  
\begin{equation*}
\begin{split}
  G_1 &:= \E \left[ \sup \left\{\left<\xi^{(1)}, \theta^{(1)} \right> :
    \|\theta^{(1)} \| \leq t, \right. \right. \\ & \left. \left. V(D^{(r-1)} \theta^{(1)}) \leq \rs_1
    ( (D^{(r-1)} \theta^{(1)})_{\ell} - (D^{(r-1)} \theta^{(1)})_1) +
    \delta \right\}  \right]
\end{split}
\end{equation*}
 and 
 \begin{equation*}
\begin{split}
   G_2 &:= \E \left[\sup \left\{\left<\xi^{(2)}, \theta^{(2)} \right> :
    \|\theta^{(2)} \| \leq t, \right. \right. \\ & \left. \left.
    V(D^{(r-1)} \theta^{(2)}) \leq \rs_2 
    ( (D^{(r-1)} \theta^{(2)})_{n-\ell-2r+2} - (D^{(r-1)} \theta^{(2)})_1) +
    \delta \right\} \right].  
\end{split}
 \end{equation*}
Note now that when $\rs_1 = 0$, we have
\begin{align*}
  G_1 &\leq C_r \sigma \left\{t^{(2r-1)/(2r)} (\ell+r-1)^{(2r-1)/(4r)} \delta^{1/(2r)}
 \right. \\ & \left. + t \sqrt{\log(e(\ell+r-1))} \right\}
\end{align*}
as this bound simply follows from Lemma \ref{svg}. On the other hand,
when $\rs_1 \neq 0$, we have 
\begin{align*}
  G_2 &\leq C_r \sigma \left(t + \delta (\ell + r - 1)^{(2r-1)/2} \right)
 \sqrt{\log(e(\ell + r  - 1))}. 
\end{align*}
This follows from Lemma \ref{am} when $\rs_1 = 1$. When $\rs_1 = -1$,
we can switch from $\theta^{(1)}$ to $-\theta^{(1)}$ so that the above
bound will again follow from Lemma \ref{am}. An identical argument
also gives that 
\begin{align*}
  G_2 &\leq C_r \sigma \left\{t^{(2r-1)/(2r)} (n-\ell-r+1)^{(2r-1)/(4r)} \delta^{1/(2r)}
\right. \\ & \left.  + t \sqrt{\log(e(n-\ell-r+1))} \right\} 
\end{align*}
when $\rs_2 = 0$ and 
\begin{align*}
  G_2 &\leq  C_r \sigma \left(t + \delta (n-\ell - r + 1)^{(2r-1)/2} \right)
 \sqrt{\log(e(n-\ell-r + 1))}   
\end{align*}
when $\rs_2 \neq 0$. By putting together the above bounds for $G_1$
and $G_2$ the proof of Lemma \ref{fina.full} is complete.  
\end{proof}

\subsection{Subdifferential of $\theta \mapsto \|D^{(r)} \theta\|_1$
  and proof of Lemma \ref{gksd}}\label{sudi}
This subsection provides a study of the subdifferential $\partial
f(\theta)$ where $f(\theta) := \|D^{(r)} \theta\|_1$ with an aim to
prove Proposition \ref{characterization} and Lemma \ref{gksd} in
Section \ref{ms}. We start by proving Proposition
\ref{characterization} which gives a precise characterization of the
subdifferential.  

\begin{proof}[Proof of Proposition \ref{characterization}]
  Let us first construct an $n \times n$ matrix $M$ such that for
  every $\beta \in \R^n$, we have
\[
    (M \beta)_i = \left\{
  \begin{array}{ll}
(D^{(i-1)} \beta)_1 &  \text{for $i = 1, \dots, r$}\\
(D^{(r)} \beta)_{i-r} &  \text{for $i = r+1, \dots, n$.}
  \end{array}
\right.
\]
This is of course possible because $\beta \mapsto (D^{(i)} \beta)_j$
is a linear mapping. More specifically, it can be checked that
$M = (M_{ij})$ defined by 
 \[
   M_{ij} = \left\{
  \begin{array}{ll}
(-1)^{i-j} \binom{i-1}{i-j} I\{1 \leq j \leq i \leq n\} &  \text{for
  $1 \leq i \leq r, 1 \leq j \le n$}\\
(-1)^{i-j} \binom{r}{i-j} I \{i - r \leq j \leq i\} &  \text{for $r+1
  \leq i \leq n, 1 \leq j \leq n$}
  \end{array}
\right.
\]
satisfies the requirement. This is a consequence of the expression: 
\begin{equation*}
  (D^{(r)} \beta)_j = \sum_{k=j}^{j+r} (-1)^{j+r-k} \binom{r}{k-j}
  \beta_k \qt{for $1 \leq j \leq n - r$}. 
\end{equation*}
It is easy to see from the formula for $M$ that it is lower triangular
with positive diagonal entries and hence invertible. 

Now a vector $v \in \R^n$ is in $\partial f(\theta)$ if and only if it
satisfies 
\begin{equation}\label{defsg}
  f(\theta + \beta) - f(\theta) \geq \left<v,  \beta \right> \qt{for every $\beta
    \in \R^n$}. 
\end{equation}
The left hand side above can be written as 
\begin{equation}\label{eq:f_theta}
  f(\theta + \beta) - f(\theta) = \sum_{j=1}^{n-r} \left[ |(D^{(r)}
    \theta)_j + (M \beta)_{j+r}| - |(D^{(r)} \theta)_j| \right].  
\end{equation}
The right hand side in \eqref{defsg} can be written using Lemma
\ref{dfi} as  
  \begin{equation*}
  \begin{split}
\left<v, \beta \right> &= \sum_{i=1}^n v_i\beta_i \\
      &= \sum_{i=1}^n
          v_i\sum_{j=1}^{i-r} \binom{i-j-1}{r-1}(D^{(r)}\beta)_j +
        \sum_{i=1}^n
          v_i\sum_{j=1}^r \binom{i-1}{j-1}(D^{(j-1)}\beta)_1 \\
      &= \sum_{j=1}^{n-r} (D^{(r)}\beta)_j\sum_{i=r+j}^n
          \binom{i-j-1}{r-1}v_i +
        \sum_{j=1}^r (D^{(j-1)}\beta)_1\sum_{i=j}^n
          \binom{i-1}{j-1}v_i \\
      &= \sum_{j=1}^{n-r} a_{r+j}(M\beta)_{r+j} +
        \sum_{j=1}^r b_j(M\beta)_j
  \end{split}
  \end{equation*}
where 
\begin{equation*}
  b_j := \sum_{i=1}^n \binom{i-1}{j-1} v_i \qt{for $1 \leq j \leq r$} 
\end{equation*}
and 
\begin{equation*}
  a_{r+j} := \sum_{i=r+j}^n  \binom{i-j-1}{r-1} v_i \qt{for $1 \leq j
    \leq n-r$}. 
\end{equation*}
We now set $\beta = \pm M^{-1}\mathbf{e}_j$ for $1 \le j \le r$, where  
  $\mathbf{e}_j$ is the $j$'th standard basis vector of
  $\mathbb{R}^n$. Then, using~\eqref{eq:f_theta}, $f(\theta+\beta) - f(\theta) = 0$, so we
  must have $\left<v, \beta \right> = b_j = 0$. Now set
  $\beta = \lambda M^{-1}\mathbf{e}_{r+j}$ for $1 \le j \le n - r$.
  If $(D^{(r)}\theta)_j > 0$, then
  $f(\theta+\beta) - f(\theta) = \lambda$ for
  $\lambda \ge -(D^{(r)}\theta)_j$, and
  $\left<v, \beta \right> = \lambda a_{r+j}$. In particular, $a_{r+j} \le 1$ by
  taking $\lambda > 0$, and $a_{r+j} \ge 1$ by taking
  $0 > \lambda \ge -(D^{(r)}\theta)_j$, so we must have
  $a_{r+j} = 1$. Similarly, if $(D^{(r)}\theta)_j < 0$, then we must
  have $a_{r+j} = -1$. If $(D^{(r)}\theta)_j = 0$, then
  $f(\theta+\beta) - f(\theta) = |\lambda|$, so we must have
  $a_{r+j} \in [-1, 1]$. We have thus proved that if $v \in \partial
  f(\theta)$, then $b_j = 0$ for $1 \leq j \leq r$ and 
\[
   a_{r+j} = \left\{
  \begin{array}{ll}
\sgn((D^{(r)}\theta)_{j}) &  \text{if $(D^{(r)}\theta)_{j}
  \neq 0$}\\
    \in [-1, 1] &  \text{otherwise}
  \end{array}
\right.
\] 
for $1 \leq j \leq n-r$. On the other hand, it is easy to see that if
these two conditions are satisfied, then $v \in \partial
f(\theta)$. The proof of Lemma \ref{characterization} will then 
be complete by the observation that $b_j = 0$ for $1 \leq j \leq r$ is
equivalent to $a_j = 0$ for $1 \leq j \leq r$, where $a_j$ is the left
hand side of \eqref{ch1}. To see this, just note that 
  \begin{equation*}
  \begin{split}
    \sum_{k=j}^r \binom{r-j}{r-k}b_k
      &= \sum_{k=j}^r \binom{r-j}{r-k}\sum_{i=k}^n
        \binom{i-1}{k-1}v_i = \sum_{i=j}^n v_i\sum_{k=j}^i
        \binom{r-j}{r-k}\binom{i-1}{k-1} \\
      &= \sum_{i=j}^n v_i\sum_{k=1}^r
        \binom{r-j}{r-k}\binom{i-1}{k-1} = \sum_{i=j}^n v_i\binom{r+i-j-1}{r-1}
        = a_j.
  \end{split}
  \end{equation*}
  so that $(a_j)_{j=1}^r$ is related to $(b_j)_{j=1}^r$ by a
  triangular linear system. This completes the proof of Proposition
  \ref{characterization}. 
\end{proof}

We are now ready to prove Lemma \ref{gksd}. 

\begin{proof}[Proof of Lemma \ref{gksd}] We start with proof of the
  assertions for $r = 1$ (including inequality \eqref{voex}) and then
  proceed to the proofs of inequalities \eqref{gksd.low} and
  \eqref{gksd.up}.  

{\bf Proofs for $r = 1$. } Assume that $r = 1$ and that
$D \theta^* \neq 0$. Let $2 \leq
j_1 < \dots < j_k \leq n$ denote the jumps (first order knots) of
$\theta$ with signs are $\rs_1, \dots, \rs_k$. Also let $j_0 =
1$, $j_{k+1} = n+1$ and $\rs_0 = \rs_{k+1} = 0$. Then $n_i :=  j_{i+1}
- j_i$ for $0 \leq i \leq k$ denote the lengths 
of the $k+1$ constant pieces of $\theta^*$. 

Define the vector $v_0 = (v_{01}, \dots, v_{0n}) \in \R^n$  in the
following way. For $1 \le i \leq n$, let $0 \leq l \leq k$ be the
unique integer such that $j_l \leq i < j_{l+1}$. Then we take $v_{0i}
:= (\rs_l -  \rs_{l+1})/n_l$. We first claim that $v_0 \in \partial
f(\theta^*)$ where $f(\theta) := \|D \theta\|_1$. By the
characterization of $\partial f(\theta^*)$ given in Proposition
\ref{characterization}, to prove that $v_0 \in \partial
f(\theta^*)$, we need to prove that 
\begin{equation*}
  v_{01} + \dots + v_{0n} = 0 , 
\end{equation*}
\begin{equation*}
  v_{0j} + \dots + v_{0n} \in [0, 1] \qt{for every $1 \leq j \leq n$}
\end{equation*}
and
\begin{equation*}
  v_{0j_u} + \dots v_{0n} = \rs_u \qt{for $u = 1, \dots, k$}. 
\end{equation*}
Each of three conditions follow from the calculation below. Fix $1
\leq j \leq n$ and let $0 \leq l \leq k$  be the unique integer such
that $j_l \leq i < j_{l+1}$. Then 
\begin{align*}
  \sum_{i = j}^n v_{oi} &= \sum_{i=j}^{j_{l+1} - 1} v_{0i} +
                          \sum_{u=l+1}^k \sum_{i=j_u}^{j_{u+1} - 1}
                          v_{0i} \\
&= \sum_{i=j}^{j_{l+1} - 1} \frac{\rs_{l} - \rs_{l+1}}{n_l} +
                          \sum_{u=l+1}^k \sum_{i=j_u}^{j_{u+1} - 1}
                          \frac{\rs_u - \rs_{u+1}}{n_u}  \\
&= \frac{\rs_l - \rs_{l+1}}{n_l} (j_{l+1} - j) + \sum_{u=l+1}^k (\rs_u
  - \rs_{u+1}) \\
&= \frac{\rs_l - \rs_{l+1}}{n_l} (j_{l+1} - j) + \rs_{l+1}  = \rs_l
  \left(\frac{j_{l+1} - j}{n_l} \right) + \rs_{l+1} \left(\frac{j -
  j_l}{n_l} \right). 
\end{align*}
This proves $v_0 \in \partial f(\theta^*)$. We shall next prove that
$v_0$ minimizes $\|v\|$ over $v \in \aff(\partial f(\theta^*))$. This will
automatically (because $v_0 \in \partial f(\theta^*)$) also prove that
$v_0$ minimizes $\|v\|$ over $v \in \partial f(\theta^*)$ so that $v_0
= v^*$.  Because $\aff(\partial f(\theta^*))$ is an affine set and
$v_0 \in \partial f(\theta^*)$, the fact that $v_0$ minimizes $\|v\|$
over $\aff(\partial f(\theta^*))$ is equivalent to the condition:  
\begin{align}\label{wd.con}
  \left<v - v_0, v_0 \right> = 0 \qt{for every $v \in \partial
  f(\theta^*)$}. 
\end{align}
Therefore we only need to verify \eqref{wd.con}. For this, write
\begin{align*}
  \left<v - v_0, v_0 \right> &= \sum_{u=0}^k \sum_{i = j_u}^{j_{u+1} -
                               1}
  \left(v_i - \frac{\rs_u - \rs_{u+1}}{n_u} \right) \left(\frac{\rs_u
      - \rs_{u+1}}{n_u} \right) \\ 
&= \sum_{u=0}^k \frac{\rs_u - \rs_{u+1}}{n_u} \left(\sum_{i  =
  j_u}^{j_{u+1} - 1} 
  v_i \right) - \sum_{u=0}^k \frac{(\rs_u - \rs_{u+1})^2}{n_u}. 
\end{align*}
The quantity above equals zero because, by the characterization of the
subdifferential $\partial f(\theta^*)$, we have $\sum_{i
  =j_u}^{j_{u+1} - 1} v_i =
\rs_u - \rs_{u+1}$ for every $v \in \partial f(\theta^*)$ and $0 \leq
u \leq k$. This proves that the condition \eqref{wd.con} holds. 

We now prove inequality \eqref{voex}. For this, simply write
\begin{align*}
  \|v_0\|^2 &= \sum_{u=0}^k \sum_{i=j_u}^{j_{u+1} - 1} \left(
  \frac{\rs_u - \rs_{u+1}}{n_u}\right)^2 \\
&=  \sum_{u=0}^k \frac{(\rs_u - \rs_{u+1})^2}{n_u} = \frac{1}{n_0} +
  \frac{1}{n_k} + 4 \sum_{u=1}^{k-1} \frac{I\{\rs_u \neq
  \rs_{u+1}\}}{n_u} 
\end{align*}
because $(\rs_u - \rs_{u+1})^2$ equals 1 for $u = 0, k$ and $4I\{\rs_u
\neq \rs_{u+1}\}$ for all other $u$. This proves \eqref{voex} and
completes the proof of the first part of Lemma \ref{gksd} (for $r = 1$). 

{\bf Proof of inequality \eqref{gksd.low}.}
  Fix $\theta^* \in \R^n$ with $D^{(r)} \theta^* \neq
  \mathbf{0}_{n-r}$. Note that $v_0$ is the projection of the zero  
  vector $\mathbf{0}_{n}$ onto $\aff(\partial f(\theta^*))$. 

  Because $\partial f(\theta^*)$ is given by a finite number of linear
  inequalities (i.e., it is a polyhedron), its affine hull is given by
  the intersection of the inequalities which are actually equalities
  (see, for example, \citet[Chapter 8]{Schrijver}). Therefore,
  $\aff(\partial f(\theta^*))$ is given by the vectors $v \in \R^n$
  for which \eqref{ch1} holds and for which 
  \begin{equation*}
    a_j = \sum_{i=j}^{n} \binom{r+i-j-1}{r-1} v_i = \sgn((D^{(r)}
    \theta^*)_{j-r}) 
  \end{equation*}
  for $r < j \leq n$ such that $(D^{(r)} \theta^*)_{j-r} \neq
  0$. Let the number of $r^{th}$ order knots of $\theta^*$ be $k$ so
  that the number of equalities in $\aff(\partial f(\theta^*))$ is $k
  + r$. We can represent these equalities in matrix form as $B v = b$
  where $B$ is $(k+r) \times n$ and $b \in \R^{k+r}$ with $\|b\|_1 = 
  k$. Note also that $\max_{i, j} |B_{ij}| \leq \binom{n+r-2}{r-1}$ so
  that 
  \begin{equation*}
    \|B\|_1 := \sup_{x \ne 0} \frac{\|B x\|_1}{\|x\|_1} = \max_{1 \leq
    j \leq n} \sum_{i=1}^{k+r} |B_{ij}| \leq (k+r) \binom{n+r-2}{r-1}
  \leq \frac{(r+1)k}{(r -1)!} (2n)^{r-1}. 
  \end{equation*}
  As a result, because the vector $v_0$ satisfies $B v_0 = b$, we
  obtain 
  \begin{equation*}
    \|v_0\| \geq \frac{\|v_0\|_1}{\sqrt{n}} \geq
    \frac{\|b\|_1}{\sqrt{n} \|B\|_1} \geq \frac{k}{\sqrt{n}}
    \frac{(r-1)!}{(r+1) k (2n)^{r-1}} = \frac{(r-1)!}{(r+1) 2^{r-1}}
    n^{-r + 1/2}. 
  \end{equation*}
  This proves \eqref{gksd.low}. 

{\bf Proof of Inequality \eqref{gksd.up}.}
  This proof is rather long. Fix $\theta^* \in \R^n$ with $D^{(r)}
  \theta^* \neq \mathbf{0}_{n-r}$. Let $2 \leq j_1
  < \dots < j_k \leq n-r+1$ be the $r^{th}$ order knots of $\theta^*$
  along with associated signs $\rs_1, \dots, \rs_k \in \{-1,
  1\}$. Also let $j_0 = 1, j_{k+1} = 
  n-r+2$ and $\rs_0  = \rs_{k+1} = 0$. It will be
  convenient below to take $m_i := j_i + r - 1$ for $l = 0,   \dots,
  k$. Also let $n_0 = j_1 + r - 2$ and $n_i = j_{i+1} - j_i$ for $i =
  1, \dots, k$.    

  Because it is assumed that the minimum length condition \eqref{nsa}
  holds for $\theta^*$ with constant $c$, it follows that  $n_i \geq
  cn/(k+1)$ whenever $\rs_i \neq \rs_{i+1}$.  

   Let $\gf : \R \rightarrow \R$ be a smooth (i.e., $C^{\infty}$)
   function such that
   \begin{enumerate}
   \item $\gf(0) = 0$, $\gf(1) = 1$. 
   \item $\gf^{(j)}(0) = \gf^{(j)}(1) = 0$ for $j \geq 1$. 
   \item $\gf(t) \in [0, 1]$ for $t \in [0, 1]$. 
   \end{enumerate}
   where $\gf^{(j)}$ is the $j^{th}$ order derivative of $\gf$. For
   example, the function $\gf(x) := \int_0^{x} \phi(t) dt$ where
\[
   \phi(t) = \left\{
  \begin{array}{ll}
\gamma \exp \left(\frac{-1}{t(1-t)} \right) &  \text{for $t  \in (0, 1)$}\\
0 &  \text{otherwise}
  \end{array}
\right.
\]
where $\gamma$  is chosen so that $\int_0^1 \phi(t) dt = 1$ will
satisfy the requirements for $\gf$. 

Let us now define a function $S : [1, n+r] \rightarrow \R$ as follows: 
\[
   S(t) = \left\{
  \begin{array}{ll}
0 &  \text{for $t  \in [1, r] \cup [n+1, n+r]$}\\ 
\rs_i \left(1 - \gf \left(\frac{t-m_i}{n_i} \right) \right) + \rs_{i+1}
\gf \left(\frac{t - m_i}{n_i} \right) &  \text{for $m_i \leq t \leq
  m_{i+1}$,  $0 \leq i \leq k$}
  \end{array}
\right.
\]
By an abuse of notation, we shall also denote by $S$, the
$n+r$-dimensional vector $ (S(1), \dots, S(n+r))$. It will be
clear from the context whether we are referring to the vector $S$ or
the function $S$. From the properties of $\gf$, it is easy to deduce
that $S(m_i) = \rs_i$ and $S^{(j)}(m_i) = 0$ for all $j$ and $0 \leq i
\leq k+1$. Also $\sup_{t \in [1, n+r]}|S(t)| \leq 1$. 

The first key observation is that the vector $v^* \in \R^n$ defined by  
\begin{equation*}
  v^*_j := (-1)^r (D^{(r)} S)_j = \sum_{k=j}^{j+r} (-1)^{k -j}
  \binom{r}{k-j} S_k \qt{for $1 \leq j \leq n$}
\end{equation*}
belongs to the subdifferential $\partial f(\theta^*)$. To see this, we
need to use Proposition \ref{characterization}. Note that  for $1 \leq
j \leq n$, 
{\small \begin{align*}
  a_j^* := \sum_{i=j}^{n} \binom{r+i-j-1}{r-1} v_i^* 
&= \sum_{i=j}^{n} \binom{r+i-j-1}{r-1} \sum_{k=i}^{i+r} (-1)^{k-i} \binom{r}{k-i} S_k \\
&= \sum_{k=j}^{n+r} S_k \sum_{i=k-r}^{\min(k, n)} (-1)^{k-i}
\binom{r}{k-i} \binom{r+i-j-1}{r-1} \\
&= \sum_{k=j}^{n} S_k \sum_{i=k-r}^{k} (-1)^{k-i}
\binom{r}{k-i} \binom{r+i-j-1}{r-1}  
\end{align*}}
where the last equality follows because $S_k = 0$ for $k = n+1, \dots,
n+r$. Now let
\begin{equation*}
  \beta_i := \binom{r+i-j-1}{r-1} \qt{for $i = \dots, -2, -1, 0, 1, 2,
    \dots $}
\end{equation*}
where the binomial coefficient is taken to be zero if $r + i - j - 1 <
r-1$. Then 
\begin{equation*}
  a_j^* = \sum_{k=j}^n S_k \sum_{i=k-r}^{k} (-1)^{k-i}
\binom{r}{k-i} \beta_i = \sum_{k=j}^n S_k (D^{(r)} \beta)_{k-r}. 
\end{equation*}
It is now easy to see that $\beta_i$ is a polynomial in $i$ for $i
\geq j+1-r$ which implies that $(D^{(r)} \beta)_{k-r} = 0$ for $k \geq
j+1$. It can also be checked that $(D^{(r)} \beta)_{j-r} = 1$. This
therefore gives $a_j^* = S_j$ for $j = 1, \dots, n$. Proposition
\ref{characterization} and the fact that $S_j = S(j) = 0$ for $1 \leq
j \leq r$, $S(m_i) = \rs_i$ and $|S(t)| \le 1$ for all $t$ proves that
$v^* \in \partial f(\theta^*)$. 
   
We shall now bound $\|v^*\|$ by writing
  \begin{equation*}
  \begin{split}
    \|v^*\|^2 &= \sum_{l=0}^{k+1} \sum_{j=m_l-r+1}^{m_l-1}
    {v^*_j}^2 + \sum_{l=0}^k \sum_{j=m_l}^{m_{l+1}-r} {v^*_j}^2 \\
      &= \sum_{l=0}^{k+1} \sum_{j=m_l-r+1}^{m_l-1}
          ((D^{(r)}S)_j)^2 +
        \sum_{l=0}^k \sum_{j=m_l}^{m_{l+1}-r}
          ((D^{(r)}S)_j)^2
  \end{split}
  \end{equation*}
  Let
  \begin{equation*}
    M_r := \sup_{t \in [0, 1]} \left|\gf^{(r)}(t)\right|,
  \end{equation*}
  and note that
  \begin{equation*}
    \left|S^{(r)}(t)\right|
      \le |\mathfrak{r}_{l+1}-\mathfrak{r}_l|M_rn_l^{-r}
      \le 2M_rn_{min}^{-r}
  \end{equation*}
  for $t \in [m_l, m_{l+1}]$ and $0 \leq l \leq k$ where 
\begin{equation*}
  n_{\min} :=   \min_{0 \leq i \leq k : \rs_i \neq \rs_{i+1}} n_i . 
\end{equation*}
Then for
  $m_l - r < j < m_l + r$ and $0 \le l \le k + 1$ we have
  \begin{equation*}
    |S(j)|
      \le \frac{2M_rn_{\min}^{-r}}{r!}|j-m_l|^r
      \le \frac{2r^r}{r!}M_rn_{\min}^{-r}
      \le \frac{2e^r}{\sqrt{2\pi r}}M_rn_{\min}^{-r}
  \end{equation*}
  by $(r-1)$-th order Taylor expansion about $m_l$ and Stirling's
  approximation. (The bound trivially holds if $j < r$ or $j > n$;
  if $j \notin (m_{l-1}, m_{l+1})$, then the bound holds by
  expansion about the nearest $m_i$). Thus for $m_l - r < j < m_l$ 
  and $0 \le l \le k + 1$, again by Stirling's approximation, we have
  \begin{equation*}
    |(D^{(r)}S)_j| \le \sum_{i=0}^r
          \binom{r}{i}\frac{2e^r}{\sqrt{2\pi r}}M_rn_{\min}^{-r} \le
          \frac{2^{r+1}e^r}{\sqrt{2\pi r}}M_rn_{\min}^{-r} 
  \end{equation*}
and so
  \begin{equation}\label{truncation}
    \begin{split}
    \sum_{l=0}^{k+1} \sum_{j=m_l-r+1}^{m_l-1}
        ((D^{(r)}S)_j)^2
      &\le \sum_{l=0}^{k+1}
        (r-1)\frac{2^{2r+2}e^{2r}}{2\pi r}M_r^2n_{\min}^{-2r} 
        \\
      &\le \frac{2(k+2)}{\pi}(2 e)^{2r} M_r^2n_{\min}^{-2r} \\
&\le
      (2e)^{2r}M_r^2(k+1)n_{\min}^{-2r}. 
     \end{split}
  \end{equation}
  We now proceed to the second term for bounding $\|v^*\|$. For this,
  let 
  \begin{equation*}
    N_r = \sup_{t \in [0, 1]} \left|\gf^{(r+1)}(t)\right|,
  \end{equation*}
  and note that
  \begin{equation*}
    \left|S^{(r+1)}(t)\right|
      \le |\mathfrak{r}_{l+1}-\mathfrak{r}_l|N_rn_l^{-r-1}.
  \end{equation*}
  for $t \in [m_l, m_{l+1}]$ and $0 \le l \le k$. Then for
  $m_l \le j \le m_{l+1} - r$ and $0 \le l \le k$,
  \begin{equation*}
  \begin{split}
    \left|(-1)^r(D^{(r)}S)_j-S^{(r)}(j)\right|
      &\le \sum_{i=0}^r
        \binom{r}{i}
          \frac{|\mathfrak{r}_{l+1}-\mathfrak{r}_l|N_rn_l^{-r-1}}
            {(r+1)!}
          i^{r+1} \\
&\le \frac{2^rr^{r+1}}{(r+1)!}
        |\mathfrak{r}_{l+1}-\mathfrak{r}_l|N_rn_l^{-r-1} \\
      &\le \frac{2^re^{r+1}}{\sqrt{2\pi(r+1)}}
        |\mathfrak{r}_{l+1}-\mathfrak{r}_l|N_rn_l^{-r-1} \\
&\le \frac{2^{r-1}e^{r+1}}{\sqrt{\pi}}
        |\mathfrak{r}_{l+1}-\mathfrak{r}_l|N_rn_l^{-r-1},
  \end{split}
  \end{equation*}
  by $r$-th order Taylor expansion about $j$ and
  Stirling's approximation, using the fact that the $r$-th order
  forward difference approximates the $r$-th derivative up to an
  error depending on the $(r+1)$-th derivative (i.e. all lower order
  terms in the Taylor expansion cancel). Then the trivial inequality
  $|a^2 - b^2| \leq (a-b)^2 + 2|b||a-b|$ gives, for 
  \begin{equation*}
    T_j := \left|((D^{(r)}S)_j)^2-\left(S^{(r)}(j)\right)^2\right|,
  \end{equation*}
  the upper bound
  \begin{equation*}
  \begin{split}
T_j  &\le \left|(-1)^r(D^{(r)}S)_j-S^{(r)}(j)\right|^2 +
        2\left|S^{(r)}(j)\right|
          \left|(-1)^r(D^{(r)}S)_j-S^{(r)}(j)\right| \\
      &\le \frac{2^{2r-2}e^{2r+2}}{\pi}
          (\mathfrak{r}_{l+1}-\mathfrak{r}_l)^2N_r^2n_l^{-2r-2} +
        \frac{2^re^{r+1}}{\sqrt{\pi}}
          (\mathfrak{r}_{l+1}-\mathfrak{r}_l)^2M_rN_rn_l^{-2r-1} \\
      &\le (2e)^{2r}(\mathfrak{r}_{l+1}-\mathfrak{r}_l)^2
        (M_r+N_r)N_rn_l^{-2r-1}.
  \end{split}
  \end{equation*}
  So for $0 \le l \le k$ we have,
  \begin{equation*}
  \begin{split}
   &  \left|\sum_{j=m_l}^{m_{l+1}-r} ((D^{(r)}S)_j)^2 -
      \sum_{j=m_l}^{m_{l+1}-r} \left(S^{(r)}(j)\right)^2\right| \\
      &\le (n_l-r+1)(2e)^{2r}(\mathfrak{r}_{l+1}-\mathfrak{r}_l)^2
        (M_r+N_r)N_rn_l^{-2r-1} \\
      &\le (2e)^{2r}(\mathfrak{r}_{l+1}-\mathfrak{r}_l)^2
        (M_r+N_r)N_rn_l^{-2r} \le 4(2e)^{2r}(M_r+N_r)N_rn_{\min}^{-2r}
  \end{split}
  \end{equation*}
  (the above bound trivially holds if $n_l < r$). Thus 
  \begin{equation}\label{finitedifference}
\begin{split}
    \sum_{l=0}^k \sum_{j=m_l}^{m_{l+1}-r} ((D^{(r)}S)_j)^2
      &\leq \sum_{l=0}^k \sum_{j=m_l}^{m_{l+1}-r}
          \left(S^{(r)}(j)\right)^2 \\ &+
        4(2e)^{2r}(M_r+N_r)N_r(k+1)n_{\min}^{-2r}.
\end{split}
  \end{equation}
  Now let
  \begin{equation*}
    K_r = \sup_{t \in [0, 1]}
      \left|\frac{d}{dt}
        \left(\left(\gf^{(r)}(t)\right)^2\right)\right|,
  \end{equation*}
  and note that
  \begin{equation*}
    \left|\frac{d}{dt} \left(\left(S^{(r)}(t)\right)^2\right)\right|
      \le (\mathfrak{r}_{l+1}-\mathfrak{r}_l)^2K_rn_l^{-2r-1}
  \end{equation*}
  for $t \in [m_l, m_{l+1}]$ and $0 \le l \le k$, regarding the
  derivative as one-sided at the endpoints. Then for
  $m_l \le j \le m_{l+1} - r $ and $0 \le l \le k$,
  \begin{equation*}
  \begin{split}
    \left|\left(S^{(r)}(j)\right)^2-
        \int_j^{j+1} \left(S^{(r)}(t)\right)^2 dt\right|
      &\le \int_j^{j+1} (\mathfrak{r}_{l+1}-\mathfrak{r}_l)^2K_r
        n_l^{-2r-1}(t-j) dt \\
      &= \frac{1}{2}(\mathfrak{r}_{l+1}-\mathfrak{r}_l)^2K_r
        n_l^{-2r-1}
  \end{split}
  \end{equation*}
  by a zeroth order Taylor expansion about $j$. So for $0 \le l \le k$
  we have
  \begin{equation*}
  \begin{split}
   & \left|\sum_{j=m_l}^{m_{l+1}-r}\left(S^{(r)}(j)\right)^2-
        \int_{m_l}^{m_{l+1}-r+1} \left(S^{(r)}(t)\right)^2 dt\right| \\
      &\le (n_l-r+1)\frac{1}{2}(\mathfrak{r}_{l+1}-\mathfrak{r}_l)^2
        K_rn_l^{-2r-1} \\
      &\le \frac{1}{2}(\mathfrak{r}_{l+1}-\mathfrak{r}_l)^2K_r
        n_l^{-2r} \le 2K_rn_{\min}^{-2r}
  \end{split}
  \end{equation*}
  (the bound trivially holds if $n_l < r$.) Thus
  \begin{equation}\label{rectanglemethod}
\begin{split}
    \sum_{l=0}^k \sum_{j=m_l}^{m_{l+1}-r} \left(S^{(r)}(j)\right)^2
     & \le \sum_{l=0}^k \int_{m_l}^{m_{l+1}-r+1}
          \left(S^{(r)}(t)\right)^2 dt \\ &+
        2K_r(k+1)n_{\min}^{-2r}.
\end{split}
  \end{equation}
  Let
  \begin{equation*}
    I_r = \int_0^1 \left(\gf^{(r)}(t)\right)^2 dt,
  \end{equation*}
  and note that for $0 \le l \le k$,
  \begin{equation*}
  \begin{split}
    \int_{m_l}^{m_{l+1}-r+1} \left(S^{(r)}(t)\right)^2 dt
      &\le \int_{m_l}^{m_{l+1}} \left(S^{(r)}(t)\right)^2 dt \\ &= \int_{m_l}^{m_{l+1}}
        (\mathfrak{r}_{l+1}-\mathfrak{r}_l)^2
          \left(\gf^{(r)}\left(\frac{t-m_l}{n_l}\right)\right)^2
          n_l^{-2r} dt \\
      &= (\mathfrak{r}_{l+1}-\mathfrak{r}_l)^2I_rn_l^{-2r+1} \le
      4I_rn_{\min}^{-2r+1}. 
  \end{split}
  \end{equation*}
  Thus
  \begin{equation}\label{integral}
    \sum_{l=0}^k \int_{m_l}^{j_{m+1}-r+1}
        \left(S^{(r)}(t)\right)^2 dt
      \le 4I_r(k+1)n_{\min}^{-2r+1}.
  \end{equation}
  Combining bounds (\ref{truncation}), (\ref{finitedifference}),
  (\ref{rectanglemethod}), and (\ref{integral}), we have
 \begin{equation*}
  \begin{split}
    \|v^*\|^2
      &\le (2e)^{2r}M_r^2(k+1)n_{\min}^{-2r} +
        4(2e)^{2r}(M_r+N_r)N_r(k+1)n_{\min}^{-2r} \\
&+ 2K_r(k+1)n_{\min}^{-2r} + 4I_r(k+1)n_{\min}^{-2r+1} \\
      &\le \left((2e)^{2r}(M_r+2N_r)^2+2K_r+4I_r\right)(k+1)
        n_{\min}^{-2r+1}.
  \end{split}
  \end{equation*}
  This proves \eqref{gksd.up} with
  $C_r = \sqrt{(2e)^{2r}(M_r+2N_r)^2+2K_r+4I_r}$ (because of the fact
  that $n_{\min} \geq cn/(k+1)$ under assumption \eqref{nsa}).  
\end{proof}

\section{Additional technical results and proofs}\label{appa}
\subsection{A result on Gaussian suprema}
The following result was used in the proof of Theorem \ref{ada}. 
\begin{lemma}\label{gt}
Suppose $p, n \geq 1$ and let $\Theta_1, \dots, \Theta_p$ be subsets
of $\R^n$ each containing the origin and each contained in the closed
Euclidean ball of radius $D$ centered at the origin. Then, for $\xi 
  \sim N(0, \sigma^2 I)$, we have 
  \begin{equation}\label{gt.eq}
    \E \left(\max_{1 \leq i\le p} \sup_{\theta \in \Theta_i} \left<\xi,
        \theta \right> \right) \leq \max_{1 \leq i \leq p} \E
    \sup_{\theta \in \Theta_i} \left<\xi, \theta \right> + D \sigma
    \left(\sqrt{2 \log p} + \sqrt{\frac{\pi}{2}}\right). 
  \end{equation}
\end{lemma}

\begin{proof}[Proof of Lemma \ref{gt}]
  For every $t \geq 0$, by the union bound
  \begin{equation*}
\begin{split}
    \P \left\{\max_{1 \leq i \leq p} \sup_{\theta \in \Theta_i}
      \left<\xi, \theta \right>  \geq \max_{1 \leq i \leq p} \E
      \sup_{\theta \in \Theta_i} \left<\xi, \theta \right> + t \sigma \right\}
    &\leq \sum_{i=1}^p \P \left\{\sup_{\theta \in \Theta_i} \left<\xi,
        \theta \right>  \right.\\ & \left.\geq \E \sup_{\theta \in \Theta_i} \left<\xi,
        \theta \right> + t \sigma \right\}. 
\end{split}
  \end{equation*}
  Now by hypothesis, every vector in $\Theta_i$ has norm bounded by
  $D$. As a result, the map $\xi \mapsto \sup_{\theta \in \Theta_i}
  \left<\xi, \theta \right>$ is Lipschitz with constant $D$. By the
  Gaussian concentration inequality, we deduce therefore that 
  \begin{equation*}
    \P \left\{\sup_{\theta \in \Theta_i} \left<\xi,
        \theta \right> \geq \E \sup_{\theta \in \Theta_i} \left<\xi,
        \theta \right> + \sigma t  \right\} \leq \exp \left(- \frac{t^2}{2
        D^2} \right) 
  \end{equation*}
for every $1 \leq i \leq p$. Consequently, 
  \begin{equation*}
    \P \left\{\max_{1 \leq i \leq p} \sup_{\theta \in \Theta_i}
      \left<\xi, \theta \right>  \geq \max_{1 \leq i \leq p} \E
      \sup_{\theta \in \Theta_i} \left<\xi, \theta \right> + t \sigma \right\}
    \leq \min \left\{p \exp\left(- \frac{t^2}{2 D^2} \right), 1
    \right\} 
  \end{equation*}
for every $t \geq 0$. Integrating both sides of this inequality from $t = 0$ to $t =
 \infty$, we obtain
 \begin{equation*}
   \E \left(\max_{1 \leq i \leq p} \sup_{\theta \in \Theta_i} \left<\xi,
     \theta\right> - \max_{1 \leq i \leq p} \E \sup_{\theta \in
     \Theta_i} \left<\xi, \theta\right> \right)^+ \leq \sigma
 \int_0^{\infty}  \min \left\{p \exp\left(- \frac{t^2}{2 D^2} \right), 1
    \right\} dt. 
 \end{equation*}
The trivial inequality $a \leq b + (a - b)^+$ therefore gives 
\begin{equation*}
\E \left(\max_{1 \leq i\le p} \sup_{\theta \in \Theta_i} \left<\xi,
        \theta \right> \right) \leq \max_{1 \leq i \leq p} \E
    \sup_{\theta \in \Theta_i} \left<\xi, \theta \right> +   \sigma
 \int_0^{\infty}  \min \left\{p \exp\left(- \frac{t^2}{2 D^2} \right), 1
    \right\} dt.  
\end{equation*}
We will now bound the integral from above. For this,
we simply write
\begin{align*}
  \int_0^{\infty}  \min \left\{p \exp\left(- \frac{t^2}{2 D^2} \right), 1
    \right\} dt &\leq \int_0^{D \sqrt{2 \log p}} 1 dt \\ &+ p \int_{D
      \sqrt{2 \log p}}^{\infty} \exp \left(- \frac{t^2}{2 D^2} \right)
    dt \\
&= D \sqrt{2 \log p} + \sqrt{2 \pi} p D  \left(1 - \Phi(\sqrt{2 \log
    p}) \right). 
\end{align*}
We now complete the proof of \eqref{gt.eq} via the Gaussian tail bound
$1 - \Phi(x) \leq \exp(-x^2/2)/2$ for $x = \sqrt{2 \log p}$ (see e.g.,~\citet{dumbgen2010bounding}). 
\end{proof}

\subsection{A formula for $\theta$ in terms of $D^{(r)}
  \theta$} \label{dfisec}
The following result provides a formula for expressing a
vector $\theta \in \R^n$ in terms of $D^{(r)} \theta$ and $(D^{(i)}
\theta)_1$ for $i = 0, \dots, r-1$. This result is quite useful and
we have used it in multiple places in our proofs. 
\begin{lemma}\label{dfi}
  Fix $r \geq 1$ and $n \geq r$. For every $\theta \in \R^n$ and $1
  \leq i \leq n$, we have 
  \begin{equation}
    \label{rle.eq}
    \theta_i = \sum_{j=1}^{i-r} \binom{i-j-1}{r-1} (D^{(r)} \theta)_j
    + \sum_{j=1}^r \binom{i-1}{j-1} (D^{(j-1)} \theta)_1  
  \end{equation}
  where we take the convention that $\binom{a}{b} = 0$ for $b > a$, 
  $\binom{0}{0} = 1$ and that the first term in the right hand side is
  zero unless $i > r$. 
\end{lemma}

\begin{proof}[Proof of Lemma \ref{dfi}]
  We shall use induction on $r \geq 1$. For $r = 1$, the formula
  \eqref{rle.eq} becomes
  \begin{equation}\label{r1}
    \theta_i = \sum_{j=1}^{i-1} (D \theta)_j + \theta_1 
  \end{equation}
  which is trivial because $(D \theta)_j = \theta_{j+1} - \theta_j$. 
 
  Let us now assume that \eqref{rle.eq} is true for some $r = \ell
  \geq 1$ and we shall then prove it for $r = \ell + 1$. Because
  \eqref{rle.eq} is true for $r = \ell$, we have 
  \begin{equation}\label{ih}
    \theta_i = \sum_{j=1}^{i-\ell} \binom{i-j-1}{\ell-1} (D^{(\ell)} \theta)_j
    + \sum_{j=1}^{\ell} \binom{i-1}{j-1} (D^{(j-1)} \theta)_1. 
  \end{equation}
   Inequality \eqref{r1} for $\theta$ replaced by $D^{\ell} \theta$
   gives 
   \begin{equation*}
     (D^{\ell} \theta)_j = (D^{\ell} \theta)_1 + \sum_{k=1}^{j-1}
     (D^{(\ell + 1)} \theta)_k. 
   \end{equation*}
   Using the above identity in \eqref{ih}, we obtain 
   \begin{align}  
    \theta_i &= \sum_{j=1}^{i-\ell} \binom{i-j-1}{\ell-1} \left((D^{\ell} \theta)_1 + \sum_{k=1}^{j-1}
     (D^{(\ell + 1)} \theta)_k \right) + \sum_{j=1}^{\ell} \binom{i-1}{j-1}
   (D^{(j-1)} 
    \theta)_1 \nonumber \\
&= \sum_{j=1}^{i - \ell} \sum_{k=1}^{j-1} \binom{i-j-1}{\ell-1}
(D^{(\ell + 1)} \theta)_k + (D^{\ell} \theta)_1 \sum_{j=1}^{i - \ell}
\binom{i-j-1}{\ell-1} \nonumber \\ &+ \sum_{j=1}^{\ell} \binom{i-1}{j-1} 
   (D^{(j-1)}  \theta)_1 \nonumber \\  
&= \sum_{k=1}^{i - \ell - 1} (D^{(\ell + 1)} \theta)_k
\sum_{j=k+1}^{i-\ell} \binom{i-j-1}{\ell-1} + (D^{\ell} \theta)_1
\sum_{j=1}^{i - \ell} 
\binom{i-j-1}{\ell-1} \nonumber \\ &+ \sum_{j=1}^{\ell} \binom{i-1}{j-1} 
   (D^{(j-1)}  \theta)_1 \label{kal}
  \end{align}
We now use the elementary identity \eqref{bnom} involving binomial 
coefficients to obtain
\begin{equation*}
  \sum_{j=k+1}^{i-\ell} \binom{i-j-1}{\ell-1} = \binom{i - k -
    1}{\ell} ~~ \text{ and } ~~ \sum_{j=1}^{i - \ell} 
\binom{i-j-1}{\ell-1} = \binom{i-1}{\ell}. 
\end{equation*}
From the above and \eqref{kal}, we deduce that 
\begin{equation*}
  \theta_i = \sum_{k=1}^{i-\ell-1} (D^{(\ell + 1)} \theta)_k
  \binom{i-k-1}{\ell} + \sum_{j=1}^{\ell + 1} \binom{i-1}{j-1}
  (D^{(j-1)} \theta)_1 
\end{equation*}
which is exactly \eqref{rle.eq} for $r = \ell + 1$. This completes the
proof of Lemma \ref{dfi}. 
\end{proof}

\subsection{Strong Sparsity and Discrete Splines} \label{ssds}
The following result gives a connection between sparsity of the vector
$D^{(r)} \theta$ and discrete splines. 
\begin{proposition}\label{prssds}
  Suppose $\theta \in \R^n$ with $\|D^{(r)} \theta\|_0 = k$. Then
  $\theta$ equals $(p(1/n), \dots, p((n-1)/n), p(1))$ for a discrete
  spline $p$ that is made of $k+1$ polynomials each of degree $(r-1)$.   
\end{proposition}

The proof of Proposition \ref{prssds} is given below. Note that the
result is trivial when $r = 1$. So it may well be assumed that $r \ge
2$ in the rest of this subsection. In fact, the argument below will
also hold for $r =1$  provided the involved binomial coefficients are
interpreted correctly for $r = 1$.  

The following lemma will be used in the proof of Proposition
\ref{prssds}. 
\begin{lemma}\label{fn}
Let $r \geq 1$, $n \geq r$  and $1 \leq a \leq b-1 \leq
n-r+1$. Suppose that 
\begin{equation}\label{fn.con}
  (D^{(r-1)} \theta)_a = \dots = (D^{(r-1)} \theta)_{b-1} = c. 
\end{equation}
Then 
\begin{equation}\label{fn.eq}
  \theta_i = c \binom{i-a}{r-1} + \sum_{j=1}^{r-1} \binom{i-a}{j-1}
  (D^{(j-1)} \theta)_a 
\end{equation}
for every $i = a, \dots, r+b-2$. 
\end{lemma}

\begin{proof}[Proof of Lemma \ref{fn}]
  Let $\alpha$ be the $b-a+r-1$-dimensional vector defined by
  \begin{equation*}
    \alpha = (\theta_a, \theta_{a+1}, \dots, \theta_{b+r-2}). 
  \end{equation*}
  Then $(D^{(r-1)} \alpha)_u = (D^{(r-1)} \theta)_{a+u-1}$ for $u = 1, \dots, b-a$
  and hence we have $(D^{(r-1)} \alpha)_1 = \dots = (D^{(r-1)}
  \alpha)_{b-a} = c$ because of \eqref{fn.con}. An application of
  Lemma \ref{dfi} now gives  
  \begin{equation*}
    \alpha_u = c \sum_{j=1}^{u-r+1} \binom{u - j - 1}{r-2} +
    \sum_{j=1}^{r-1} \binom{u-1}{j-1} (D^{(j-1)} \alpha)_1  
  \end{equation*}
  for $u = 1, \dots, r-1+b-a$. The elementary inequality \eqref{bnom}
  applied to $a = u-2$ and $b = r-2$ allows us to deduce   
   \begin{equation*}
     \alpha_u = c \binom{u-1}{r-1} + \sum_{j=1}^{r-1} \binom{u-1}{j-1}
     (D^{(j-1)} \alpha)_1  
   \end{equation*}
    for $u = 1, \dots, r-1+b-a$. Applying the above to $u = i+1-a$, we
    obtain inequality \eqref{fn.eq}. This completes the proof of Lemma
    \ref{fn}.  
\end{proof}
We now prove Proposition \ref{prssds}. 

\begin{proof}[Proof of Proposition \ref{prssds}]
  Suppose $\theta \in \R^n$ and let $2 \leq j_1 < \dots < j_k \leq n-r+1$
denote all the $r^{th}$ order knots of $\theta$ with $j_0 = 1$ and
$j_{k+1} = n-r+2$. We then have 
\begin{equation*}
  (D^{(r-1)} \theta)_{j_u} = \dots = (D^{(r-1)} \theta)_{j_{u+1} - 1}
  = c_u \qt{for $u = 0, \dots, k$} 
\end{equation*}
for some real numbers $\{c_u, 0 \leq u \leq k\}$. 

Lemma \ref{fn} applied to $a = j_u$ and $b = j_{u+1}$ then implies
that for every $0 \leq u \leq k$ and $i = j_u, \dots, r+j_{u+1}-2$,
we have 
\begin{equation}\label{bir}
  \theta_i = c_u \binom{i-j_u}{r-1} + \sum_{j=1}^{r-1}
  \binom{i-j_u}{j-1} 
  (D^{(j-1)} \theta)_{j_u} 
\end{equation}
Now, for each $0 \leq u \leq k$, let $p_u$ denote the polynomial in
$x$ defined by 
\begin{align*}
  p_u(x) &:= \frac{c_u}{(r-1) !} (nx - j_u) \dots (nx - j_u -
  r+2) \\ &+ \sum_{j=1}^{r-1} \frac{(nx - j_u) \dots (nx - j_u -
    j+2)}{(j-1)!} (D^{(j-1)} \theta)_{j_u}. 
\end{align*}
It is clear that $p_u(x)$ is a polynomial in $x$ of degree
$(r-1)$. Also the identity \eqref{bir} is equivalent to  
\begin{equation}
  \label{gill}
\theta_i = p_u(i/n) \qt{for $0 \leq u \leq k$ and $j_u \leq i \leq r +
  j_{u+1} - 2$}.   
\end{equation}
We now define a function $p$ via
\[
 p(x) = 
  \begin{cases} 
   p_0(x) & \text{for } x < \frac{r + j_1 -2}{n} \\
   p_u(x)       & \text{for } \frac{r + j_{u} - 2}{n} \leq x <
   \frac{r + j_{u+1} - 2}{n}, u = 1, \dots, k-1 \\
   p_k(x)  & \text{for } x \geq \frac{r + j_k - 2}{n}. 
  \end{cases}
\]
Clearly $p$ is a piecewise polynomial of degree $(r-1)$. Also, it is 
trivial to see from \eqref{gill} that $p(i/n) = \theta_i$ for every $1
\leq i \leq n$. Moreover, using \eqref{gill}, it is easy to show that
one has 
\begin{equation}
  \label{kds}
   p_{u-1} \left(\frac{i}{n} \right) = p_u
 \left(\frac{i}{n} \right) \qt{for $1 \leq u \leq k$ and $j_u \leq i
   \leq r + j_u -2 $}
\end{equation}
for $r \ge 2$. Thus if $x_u := (r + j_u - 2)/n$ denotes the knots of
the piecewise polynomial $p$, then we have 
\begin{equation}
  \label{kdsdef}
  p_{u-1} \left(x_u - \frac{i}{n} \right) = p_u \left(x_u -
    \frac{i}{n} \right) \qt{for $i = 0,
    1, \dots, r-2$}. 
\end{equation}
This means that the function $p$ is a discrete spline of degree
$(r-1)$ having $k+1$ polynomial pieces which proves Proposition
\ref{prssds}. 
\end{proof}

\subsection{A result on the magnitude of $\min_{1 \leq i \leq n-r+1}
  (D^{(r-1)} \theta)_i$ when $\|\theta\| \leq 1$}\label{pdo}
This section is devoted to the proof of the Lemma \ref{do} which was
crucially used in the proof of Lemma \ref{cr1}.  

\begin{proof}[Proof of Lemma \ref{do}]
  We only need to prove the first inequality in \eqref{do.eq}. The
  second inequality follows by applying the first inequality to
  $-\theta$. 

  Via Lemma \ref{dfi}, we can write the following for every $\theta
  \in \R^n$ with $\|\theta\| \leq t$: 
  \begin{equation*}
    t^2 \geq \|\theta\|^2 = \sum_{i=1}^n \left(\sum_{j=1}^{i-r}
    \binom{i-j-1}{r-1} s_j + \sum_{j=1}^r \binom{i-1}{j-1} (D^{(j-1)}
    \theta)_1  \right)^2 
  \end{equation*}
where $s_j := (D^{(r)} \theta)_j$ for $ = 1, \dots, n-r$. It follows
from here that
\begin{equation*}
  t^2 \geq \inf_{\beta_1, \dots, \beta_r \in \R} \sum_{i=1}^n
\left(\sum_{j=1}^{i-r} \binom{i-j-1}{r-1} s_j - \sum_{j=1}^r
  \binom{i-1}{j-1} \beta_j  \right)^2.  
\end{equation*}
  We now define two matrices. Let $X$ be the $n \times r$ matrix whose
$(i, j)^{th}$ entry equals $\binom{i-1}{j-1}$. Let $S$ be the $n
\times (n-r)$ matrix whose $(i, j)^{th}$ entry equals
$\binom{i-j-1}{r-1}$.  Throughout we use the convention that
$\binom{a}{b} = 0$ when $a < b$. Also let $s := D^{(r)} \theta = (s_1,
\dots, s_{n-r})^T$ and $\beta := (\beta_1, \dots, \beta_r)$ . It is then
easy to see from the previous inequality that  
\begin{equation}\label{gan}
  t^2 \geq \inf_{\beta_1, \dots, \beta_r \in \R} \|S s - X \beta\|^2 =
  s^T S^T (I - P_X) S s 
\end{equation}
where $P_X = X(X^T X)^{-1} X^T$ is the projection matrix on to the
column space of $X$. 

We now need the following two facts about the matrix $A := S^T (I -
P_X) S$. These facts (whose proofs are long) are proved in Proposition
\ref{rowsum} and Proposition \ref{positive} respectively. 
\begin{enumerate}
\item If $\mathbf{1}$ denotes the $n-r$ vector consisting of ones,
  then $\mathbf{1}^T A \mathbf{1} \geq C_r n^{2r + 1}$  for a constant
    $C_r$ depending on $r$ alone.  
\item Every entry of the matrix $A$ is positive. 
\end{enumerate}
We shall now complete the proof of Lemma \ref{do} assuming the above
two facts about the matrix $A$. Let $\delta := \min_{1 \leq j \leq
  n-r} s_j$. Our goal is to prove that $\delta \leq C_r t n^{-r-1/2}$ so
we can assume that $\delta \geq 0$ for otherwise there is nothing to
prove. In that case,  inequality \eqref{gan} and the second fact about
$A$ together imply   
\begin{equation*}
  t^2 \geq \delta^2 \mathbf{1}^T S^T (I - P_X) S \mathbf{1} = \delta^2
  \mathbf{1}^T A \mathbf{1} . 
\end{equation*}
The first fact about $A$ then gives $t^2 \geq C_r \delta^2 n^{2r+1}$
  and this completes the proof of Lemma \ref{do}. 
\end{proof}

The remainder of this subsection is devoted to proving the two facts
about the matrix $A := S^T (I - P_X) S$ stated in the proof of Lemma
\ref{do}. These proofs are tedious and long. We adopt the convention
that $\binom{n}{k} = \frac{(n)_k}{k!}$ if $k \geq 0$ and $0$
otherwise, where $(n)_k$ is the falling factorial, extending the
definition of the binomial coefficient to integer arguments. We will
make judicious use of the identities $\binom{n}{k} = \binom{n}{n-k}$
and $\binom{n}{k} = (-1)^k\binom{k-n-1}{k}$, as well as the
Chu-Vandermonde identity, $\binom{m+n}{r} =
  \sum_{k=0}^r\binom{m}{k}\binom{n}{r-k}$, in its equivalent form 
$\binom{m+n}{r-s} = \sum_{k=s}^r\binom{m}{k-s}\binom{n}{r-k}$.

Recall that $X$ is the $n \times r$ matrix with
$X_{ij} = \binom{i-1}{j-1} = \binom{i-1}{i-j}$, $S$ is the $n
\times (n-r)$ matrix with $S_{ij} = \binom{i-j-1}{r-1} =
\binom{i-j-1}{i-j-r}$ if $i - j \ge r$ and $0$ otherwise, and $A =
S^T(I - P_X)S$ where $P_X$ is the projection onto the column space of
$X$. Our first step is to compute the inverse of the matrix
$A$ explicitly. This is the content of the following Proposition. 

%\vspace{0.2in} 

\begin{proposition}\label{inverse}
  Let $T$ be the $(n-r) \times (n-r)$ matrix with
  $T_{ij} = (-1)^{i-j}\binom{2r}{r+i-j}$. Then $T = A^{-1}$.
\end{proposition}
In order to prove Proposition \ref{inverse}, we need the following
lemma. 

%\vspace{0.2in} 

\begin{lemma}\label{decomposition}
  Let $Y$ be the $r \times (n-r)$ matrix with
  $Y_{ij} = (-1)^{r+i-j}\binom{r+i-1}{i-j}$, and let $U$ be
  $n \times (n-r)$ matrix with
  $U_{ij} = (-1)^{r+i-j}\binom{r}{i-j}$. Then $XY + ST = U$.
\end{lemma}

\begin{proof}[Proof of Lemma \ref{decomposition}]
  We have
  \begin{equation*}
  \begin{split}
    (XY+ST)_{ij} &= \sum_{k=1}^r X_{ik}Y_{kj} +
        \sum_{l=1}^{n-r} S_{il}T_{lj} \\
      &= \sum_{k=j}^r(-1)^{r+k-j}
          \binom{i-1}{i-k}\binom{r+k-1}{k-j} \\ &+
        \sum_{l=1}^{i-r} (-1)^{l-j}
          \binom{i-l-1}{i-l-r}\binom{2r}{r+l-j} \\
      &= (-1)^r\sum_{k=j}^r \binom{i-1}{i-k}\binom{-r-j}{k-j} \\ &+
        (-1)^{r+i-j}\sum_{l=1}^{i-r}
          \binom{-r}{i-l-r}\binom{2r}{r+l-j}.
  \end{split}
  \end{equation*}
  If $i < j$, then at least one of $i - k$, $k - j$ is negative,
  since $(i-k) + (k-j) = i - j < 0$. Hence $(XY)_{ij} = 0$, and
  similarly $(ST)_{ij}$ = 0, so $(XY+ST)_{ij} = 0 = U_{ij}$.
  Otherwise, there are three cases. If $j \le i \le r$, then
  $(ST)_{ij} = 0$ since the sum is empty and
  \begin{align*}
    (XY)_{ij} &= (-1)^r\sum_{k=j}^i
        \binom{i-1}{i-k}\binom{-r-j}{k-j} \\
    &  = (-1)^r\binom{-r+i-j-1}{i-j}
        = (-1)^{r+i-j}\binom{r}{i-j}
        = U_{ij}.
  \end{align*}
  If $r < j \le i$, then $(XY)_{ij} = 0$ since the sum is empty, and
  \begin{align*}
    (ST)_{ij} &= (-1)^{r+i-j}\sum_{l=j-r}^{i-r}
        \binom{-r}{i-l-r}\binom{2r}{r+l-j} \\ &=
          (-1)^{r+i-j}\binom{r}{i-j} = U_{ij}.
  \end{align*}
  Finally, if $j \le r < i$, then
  \begin{equation*}
  \begin{split}
    (U-XY)_{ij} &= U_{ij} -
        (-1)^r\sum_{k=j}^i \binom{i-1}{i-k}\binom{-r-j}{k-j} \\ &+
        (-1)^r\sum_{k=r+1}^i \binom{i-1}{i-k}\binom{-r-j}{k-j} \\
      &= (-1)^r\sum_{k=1}^{i-r}
          \binom{i-1}{i-k-r}\binom{-r-j}{r+k-j} \\
      &= \sum_{k=1}^{i-r} (-1)^{k-j}
          \binom{i-1}{i-k-r}\binom{2r+k-1}{r+k-j} \\
      &= \sum_{k=1}^{i-r} (-1)^{k-j}
          \binom{i-1}{i-k-r}\binom{2r+k-1}{r+j-1} \\
      &= \sum_{k=1}^{i-r} (-1)^{k-j}\binom{i-1}{i-k-r}
          \sum_{l=1}^{r+j} \binom{2r}{r+j-l}\binom{k-1}{l-1} \\
      &= \sum_{k=1}^{i-r} \sum_{l=1}^k
          (-1)^{k-j}\binom{i-1}{i-k-r}\binom{2r}{r+l-j}
            \binom{k-1}{k-l} \\
      &= \sum_{l=1}^{i-r} \binom{2r}{r+l-j}
          \sum_{k=l}^{i-r}
           (-1)^{k-j}\binom{i-1}{i-k-r}\binom{k-1}{k-l} \\
      &= \sum_{l=1}^{i-r} (-1)^{l-j}\binom{2r}{r+l-j}
          \sum_{k=l}^{i-r} \binom{i-1}{i-k-r}\binom{-l}{k-l} \\
      &= \sum_{l=1}^{i-r}
          (-1)^{l-j}\binom{2r}{r+l-j}\binom{i-l-1}{i-l-r}
        = (ST)_{ij}
  \end{split}
  \end{equation*}
where the sixth equality above follows from the fact that 
  $\binom{2r}{r+j-l} = 0$ for $l > r + j$ and $\binom{k-1}{l-1} = 0$
  for $l > k$.
\end{proof}

We are now ready to prove Proposition \ref{inverse}. 

\begin{proof}[Proof of Proposition \ref{inverse}]
    Let $Y$ and $U$ be defined as in Lemma \ref{decomposition}. Note
  that
  \begin{equation*}
  \begin{split}
    (X^TU)_{ij} = \sum_{k=1}^n X_{ki}U_{kj} &= \sum_{k=i}^{r+j}
        (-1)^{r+k-j}\binom{k-1}{k-i}\binom{r}{r+k-j} \\
      &= (-1)^{r+i-j}\sum_{k=i}^{r+j}
        \binom{-i}{k-i}\binom{r}{r+j-k}.
  \end{split}
  \end{equation*}
  If $r + j < i$, then $(X^TU)_{ij} = 0$ since the sum is empty.
  Otherwise,
  \begin{equation*}
    (X^TU)_{ij} = (-1)^{r+i-j}\binom{r-i}{r+j-i} = 0
  \end{equation*}
  since $0 \le r - i < r + j - i$ for $1 \le i \le r$. That is,
  $X^TU = \mathbf{0}_{r\times(n-r)}$; then each column of $U$ is in
  $\mathcal{N}(X^T) = \mathcal{C}(X)^\perp$, so $(I-P_X)U = U$.
  Lemma \ref{decomposition} gives $(I-P_X)ST = (I-P_X)(U-XY) = U$.
  Also,
  \begin{align*}
    (U^TU)_{ij} &= \sum_{k=1}^n U_{ki}U_{kj} \\ &=
    (-1)^{i-j}\sum_{k=j}^{r+i} \binom{r}{k-i}\binom{r}{k-j} \\ &=
      (-1)^{i-j}\sum_{k=j}^{r+i} \binom{r}{r+i-k}\binom{r}{k-j}. 
  \end{align*}
  If $r + i < j$, then $(U^TU)_{ij} = 0 = T_{ij}$ since the sum is
  empty. Otherwise,
  \begin{equation*}
    (U^TU)_{ij} = (-1)^{i-j}\binom{2r}{r+i-j} = T_{ij}.
  \end{equation*}
  That is, $U^TU = T$. Then
  $TAT = TS^T(I-P_X)ST = ((I-P_X)ST)^T(I-P_X)ST = U^TU = T$, and $T$
  is invertible since $U$ is lower triangular with full column rank
  $n-r$ and $\text{rank}(U^TU) = \text{rank}(U)$. Thus
  $AT = I_{n-r}$, i.e. $T = A^{-1}$.
\end{proof}

We shall now prove the first fact about $A$ in the proof of Lemma
\ref{do}: the bound on $\mathbf{1}^T A \mathbf{1}$. In fact, the
result below gives a precise formula for this quantity from which the
stated bound trivially follows. 

%\vspace{0.2in} 

\begin{proposition}\label{rowsum}
  $\mathbf{1}_{n-r}^TA\mathbf{1}_{n-r}
     = \binom{2r}{r}^{-1}\binom{n+r}{2r+1}
     = \binom{2r}{r}^{-1}\binom{n+r}{n-r-1}$.
\end{proposition}

\begin{proof}[Proof of Proposition \ref{rowsum}]
   Let us first complete the proof of Proposition \ref{rowsum}
   assuming that the following claim is true. We shall subsequently
   give the proof of this claim. 
   \begin{equation}\label{pill}
     A \mathbf{1}_{n-r} = b
   \end{equation}
  where $b$ is the $(n-r)$-dimensional vector with $b_i =
  \binom{2r}{r}^{-1}\binom{n-i}{r}\binom{r+i-1}{r}$. 

By \eqref{pill}, for the claimed expression of $\mathbf{1}_{n-r}^T A
\mathbf{1}_{n-r}$, it is equivalent to show that
$\binom{2r}{r}\mathbf{1}_{n-r}^Tb = \binom{n+r}{n-r-1}$. To see this,
  write
  \begin{equation*}
  \begin{split}
    \binom{2r}{r}\mathbf{1}_{n-r}^Tb
      &= \binom{2r}{r}\sum_{i=1}^{n-r} b_i = \sum_{i=1}^{n-r} \binom{n-i}{r}\binom{r+i-1}{r} 
    \\ & = \sum_{i=1}^{n-r} \binom{n-i}{n-r-i}\binom{r+i-1}{i-1} \\
      &= (-1)^{n-r-1}\sum_{i=1}^{n-r}
        \binom{-r-1}{n-r-i}\binom{-r-1}{i-1}  \\ &=
        (-1)^{n-r-1}\binom{-2r-2}{n-r-1} = \binom{n+r}{n-r-1}
  \end{split}
  \end{equation*}
which proves Proposition \ref{rowsum} assuming that \eqref{pill} is
true. We shall now prove \eqref{pill}. By Proposition \ref{inverse},
it is equivalent to show that $\binom{2r}{r}Tb =
\binom{2r}{r}\mathbf{1}_{n-r}$. We have 
  \begin{equation*}
  \begin{split}
    \binom{2r}{r}(Tb)_i
      &= \binom{2r}{r}\sum_{j=1}^{n-r} T_{ij}b_j \\
      &= \sum_{j=1}^{n-r}
        (-1)^{i-j}\binom{2r}{r+i-j}\binom{n-j}{r}\binom{r+j-1}{r} \\
      &= \sum_{j=1}^{n-r}
        (-1)^{i-j}\binom{2r}{r+j-i}\binom{n-j}{n-r-j}
          \binom{r+j-1}{r} \\
      &= \sum_{j=r+1}^n
        (-1)^{r+i-j}\binom{2r}{j-i}\binom{n+r-j}{n-j}
          \binom{j-1}{r} \\
      &= (-1)^{n-r+i}\sum_{j=r+1}^n
        \binom{2r}{j-i}\binom{-r-1}{n-j}\binom{j-1}{r}.
  \end{split}
  \end{equation*}
  Since $\binom{j-1}{r}$ is a degree $r$ polynomial with leading
  coefficient $\frac{1}{r!}$ and $\left((j-i)_k\right)_{k=0}^r$ is a
  basis for degree $r$ polynomials, we can write
  $\binom{j-1}{r} = \sum_{k=0}^r c_k(j-i)_k$ with
  $c_r = \frac{1}{r!}$. Then
  \begin{equation*}
  \begin{split}
    \binom{2r}{r}(Tb)_i &= (-1)^{n-r+i}\sum_{j=i}^n
        \binom{2r}{j-i}\binom{-r-1}{n-j}\sum_{k=0}^r c_k(j-i)_k \\
      &= (-1)^{n-r+i}\sum_{j=i}^n \sum_{k=0}^r
        c_k(2r)_k\binom{2r-k}{j-i-k}\binom{-r-1}{n-j} \\
      &= (-1)^{n-r+i}\sum_{k=0}^r c_k(2r)_k
        \sum_{j=i+k}^n \binom{2r-k}{j-i-k}\binom{-r-1}{n-j} \\
      &= (-1)^{n-r+i}\sum_{k=0}^r c_k(2r)_k\binom{r-k-1}{n-i-k} \\
      &= (-1)^{n-r+i}c_r(2r)_r\binom{-1}{n-r-i}
        = \frac{(2r)_r}{r!}\binom{n-r-i}{n-r-i}
        = \binom{2r}{r}.
  \end{split}
  \end{equation*}
  The first equality follows from the fact that
  $\binom{2r}{j-i} = 0$ for $j < i$ and $\binom{j-1}{r} = 0$ for
  $j \le r$. The second equality follows from the identity
  $\binom{2r}{j-i}(j-i)_k = (2r)_k\binom{2r-k}{j-i-k}$. The third
  equality follows from the fact that $\binom{2r-k}{j-i-k} = 0$ for
  $j < i + k$. This completes the proof of \eqref{pill}. 
\end{proof}

We now turn to the second claimed fact about $A$ in the proof of
Lemma \ref{do}. This is the content of the following proposition. 

%\vspace{0.2in} 

\begin{proposition}\label{positive}
 Every entry of the matrix $A$ is positive.
\end{proposition}

We need the following lemma for the proof of Proposition
\ref{positive}.

%\vspace{0.2in} 
 
\begin{lemma}\label{firstcolumn}
  Let $x$ be the $(n-r)$-dimensional vector with $i^{th}$ component: 
  $x_i =
     \binom{n+r-1}{n-1}^{-1}\binom{r+i-2}{r-1}\binom{n-i}{n-r-i}$.
  Then $x$ is the first column of $A$.
\end{lemma}

\begin{proof}[Proof of Lemma \ref{firstcolumn}]
  By Proposition \ref{inverse}, it is equivalent to show that
  $Tx = \mathbf{e}_1$, where $\mathbf{e}_1$ is the first standard
  basis vector of $\mathbb{R}^{n-r}$. We have
  \begin{equation*}
  \begin{split}
    \binom{n+r-1}{n-1}(Tx)_i
      &= \binom{n+r-1}{r}\sum_{j=1}^{n-r} T_{ij}x_j \\
      &= \sum_{j=1}^{n-r}
          (-1)^{i-j}\binom{2r}{r+i-j}\binom{r+j-2}{r-1}
            \binom{n-j}{n-r-j} \\
      &= (-1)^{n-r-i}\sum_{j=1}^{n-r}
          \binom{2r}{r+j-i}\binom{r+j-2}{r-1}\binom{-r-1}{n-r-j} \\
      &= (-1)^{n-r-i}\sum_{j=r+1}^n
          \binom{2r}{j-i}\binom{j-2}{r-1}\binom{-r-1}{n-j} \\
      &= (-1)^{n-r-i}\sum_{j=i}^n
          \binom{2r}{j-i}\binom{j-2}{r-1}\binom{-r-1}{n-j} \\ 
          & \qquad -
        (-1)^{n-r-i}\binom{2r}{1-i}\binom{-1}{r-1}
          \binom{-r-1}{n-1}\delta_{i1},
  \end{split}
  \end{equation*}
  where $\delta_{ij}$ is the Kronecker delta. The last equality
  follows from the fact that $\binom{2r}{j-i} = 0$ for $j < i$ and
  $\binom{j-2}{r-1} = 0$ for $2 \le j \le r$. Now
  \begin{align*}
    -(-1)^{n-r-i}\binom{2r}{1-i}\binom{-1}{r-1}\binom{-r-1}{n-1}
     &= (-1)^{i-1}\binom{r-1}{r-1}\binom{n+r-1}{n-1}\delta_{i,1} \\
     &= \binom{n+r-1}{n-1}\delta_{i,1}.
  \end{align*}
  Writing $\binom{j-2}{r-1} = \sum_{k=0}^{r-1} c_k(j-i)_k$,
  similarly to the proof of Proposition \ref{rowsum},
  \begin{equation*}
  \begin{split}
    \sum_{j=i}^n \binom{2r}{j-i}\binom{j-2}{r-1}\binom{-r-1}{n-j}
      &= \sum_{j=i}^n
        \binom{2r}{j-i}\binom{-r-1}{n-j}
          \sum_{k=0}^{r-1} c_k(j-i)_k \\
      &= \sum_{j=i}^n \sum_{k=0}^{r-1}
        c_k(2r)_k\binom{2r-k}{j-i-k}\binom{-r-1}{n-j} \\
      &= \sum_{k=0}^{r-1}
        c_k(2r)_k\sum_{j=i+k}^n
          \binom{2r-k}{j-i-k}\binom{-r-1}{n-j} \\
      &= \sum_{k=0}^{r-1} c_k(2r)_k\binom{r-k-1}{n-i-k}
        = 0.
  \end{split}
  \end{equation*}
  The second and third equalities follow from the same reasoning as
  in the proof of Proposition \ref{rowsum}. The last equality
  follows from the fact that $0 \le r - k - 1 < n - i - k$ for
  $i \le n - r$. Thus $(Tx)_i = \delta_{i,1}$, i.e.
  $Tx = \mathbf{e}_1$.
\end{proof}

We are now ready to prove Proposition \ref{positive}.
\begin{proof}[Proof of Proposition \ref{positive}]
  Let $x$ be defined as in Lemma \ref{firstcolumn}. Observe that
  \begin{equation*}
    \frac{x_{k+1}}{x_k} = \frac{r+k-1}{k}\cdot\frac{n-r-k}{n-k}
  \end{equation*}
  and
  \begin{equation*}
    \frac{x_{n-r-k+1}}{x_{n-r-k}}
      = \frac{n-k}{n-r-k+1}\cdot\frac{k-1}{r+k-1},
  \end{equation*}
  so
  \begin{equation*}
    \frac{x_{k+1}}{x_k}\cdot\frac{x_{n-r-k+1}}{x_{n-r-k}}
     = \frac{k-1}{k}\cdot\frac{n-r-k}{n-r-k+1} < 1.
  \end{equation*}
  Then for $i \le \frac{n-r+1}{2}$,
  \begin{equation*}
    \frac{x_{n-r-i+1}}{x_i}
      = \prod_{k=i}^{n-r-i} \frac{x_{k+1}}{x_k}
      = \prod_{k=i}^{n-r-i} \frac{x_{n-r-k+1}}{x_{n-r-k}}
      = \sqrt{\prod_{k=i}^{n-r-i}
          \frac{x_{k+1}}{x_k}\frac{x_{n-r-k+1}}{x_{n-r-k}}}
      \leq 1
  \end{equation*}
  is increasing in $i$ since the number of terms in the product
  decreases as $i$ increases. Let $1 \le i \le j \le n - r$ such
  that $i + j \le n - r + 1$. If $j \le \frac{n-r+1}{2}$, then
  \begin{equation*}
    \frac{x_{n-r-i+1}}{x_i}\cdot\frac{x_{n-r-j+1}}{x_j} \le 1.
  \end{equation*}
  Otherwise, let $j' = n - r - j + 1$, so that
  $i \le j' \le \frac{n-r+1}{2}$. Then
  \begin{align*}
    \frac{x_{n-r-i+1}}{x_i}\cdot\frac{x_{n-r-j+1}}{x_j}
      &= \frac{x_{n-r-i+1}}{x_i}\cdot
        \left(\frac{x_{n-r-j'+1}}{x_j'}\right)^{-1} \\
     & \le \frac{x_{n-r-j'+1}}{x_j'}\cdot
        \left(\frac{x_{n-r-j'+1}}{x_j'}\right)^{-1}
      = 1.
  \end{align*}
  Thus
  \begin{equation*}
    x_ix_j - x_{n-r-i+1}x_{n-r-j+1} \ge 0.
  \end{equation*}
  
  Observe that $T$ is a symmetric Toeplitz matrix. By Lemma
  \ref{firstcolumn}, $x$ is the first column of $A$, so the
  symmetric Gohberg-Semencul formula (see, for example,
  \citet{gohberg1972inversion}) gives 
  \begin{equation*}
  \begin{split}
    A = \frac{1}{x_1}
      &\left(
        \begin{bmatrix}
        x_1     & 0         & \cdots & 0 \\
        x_2     & x_1       & \cdots & 0 \\
        \vdots  & \vdots    & \ddots & \vdots \\
        x_{n-r} & x_{n-r-1} & \cdots & x_1
        \end{bmatrix}
        \begin{bmatrix}
        x_1    & x_2    & \cdots & x_{n-r} \\
        0      & x_1    & \cdots & x_{n-r-1} \\
        \vdots & \vdots & \ddots & \vdots \\
        0      & 0      & \cdots & x_1
        \end{bmatrix}\right. \\
        &\qquad\left. - 
        \begin{bmatrix}
        0       & \cdots & 0       & 0 \\
        x_{n-r} & \cdots & 0       & 0 \\
        \vdots  & \ddots & \ddots  & \vdots \\
        x_2     & \cdots & x_{n-r} & 0
        \end{bmatrix}
        \begin{bmatrix}
        0      & x_{n-r} & \cdots & x_2 \\
        \vdots & \vdots  & \ddots & \vdots \\
        0      & 0       & \ddots & x_{n-r} \\
        0      & 0       & \cdots & 0
        \end{bmatrix}
      \right),
  \end{split}
  \end{equation*}
  or
  \begin{equation*}
  \begin{split}
    A_{ij}
      &= \frac{1}{x_1}\left[\sum_{k=1}^i x_kx_{k+j-i} -
          \sum_{k=1}^{i-1} x_{n-r-k+1}x_{n-r+i-j-k+1}\right] \\
      &= \frac{1}{x_1}\left[x_ix_j +
          \sum_{k=1}^{i-1} \left(x_kx_{k+j-i} -
            x_{n-r-k+1}x_{n-r+i-j-k+1}\right)\right].
  \end{split}
  \end{equation*}
  Since $T$ is symmetric Toeplitz, in particular it is symmetric
  persymmetric; by Proposition \ref{inverse}, $T = A^{-1}$, so $A$
  is symmetric persymmetric as well. It suffices then to consider
  $A_{ij}$ for $1 \le i \le j \le n - r$ satisfying
  $i + j \le n - r + 1$. Now $1 \le k \le k + j - i \le n - r$ and
  $k + (k + j - i) \le n - r + 1$ for $1 \le k \le i - 1$, hence
  $x_kx_{k+j-i} - x_{n-r-k+1}x_{n-r+i-j-k+1} \ge 0$ and
  $A_{ij} \ge \frac{x_ix_j}{x_1} > 0$. Thus $A_{ij} > 0$ for all $1
  \le i,j \le n-r$ which completes the proof of Proposition
  \ref{positive}. 
\end{proof}

\subsection{A Result on Variance and Variation (Lemma
  \ref{lemco})}\label{vava} 
In this subsection, we provide the proof of Lemma \ref{lemco} which
was used in the proof of Corollary \ref{kco}. 

\begin{proof}[Proof of Lemma \ref{lemco}]
  Note that first that for $r = 1$, the result follows by taking $\eta
  = \bar{\theta} \mathbf{1}_{n}$ (where $\bar{\theta} := (\theta_1 +
  \dots + \theta_n)/n$) and using the inequality  
  \begin{equation}\label{gaao}
    \sum_{i=1}^n \left(\theta_i - \bar{\theta} \right)^2 \leq n \|D
    \theta\|_1^2 = n V^2(\theta),
  \end{equation}
  which is a consequence of the fact that $|\theta_i - \bar{\theta}|
  \leq \max_{k, l} |\theta_k - \theta_l| \leq V(\theta)$ for every $1
  \leq i \leq n$. 

  Let us therefore assume that $r \geq 2$. We may assume without loss
  of generality that the vector $D^{(r-1)} \theta$ has mean zero (if
  not, we will work with $\tilde{\theta}$ instead of $\theta$ where
  $\tilde{\theta}$ is created by subtracting a suitable polynomial
  sequence of degree $(r-1)$ from $\theta$; this will ensure that
  $D^{(r-1)} \tilde{\theta}$ has mean zero and that $D^{(r)} \theta =
  D^{(r)} \tilde{\theta}$). Let $X$ be the $n \times
  (r-1)$ matrix whose $(i, j)^{th}$ entry equals
  $\binom{i-1}{j-1}$. Let $S$ be the $n \times (n-r+1)$ matrix whose
  $(i, j)^{th}$ entry equals $\binom{i-j-1}{r-2}$.  Throughout we use
  the convention that $\binom{a}{b} = 0$ when $a < b$.  Let $\eta$
  denote the projection of $\theta$ on to the column space of $X$. We
  shall prove that the conditions of Lemma \ref{lemco} are satisifed
  for this choice of $\eta$. 

  Note first that $\eta$ belongs to the column space of $X$ which
  implies that the entries $\eta_i$ of $\eta$ will be given by a
  polynomial in $i$ of degree at most $r-2$ so that $D^{(r-1)} \eta =
  \mathbf{0}_{n-r+1}$. The reader may observe that $D^{(r-1)} \eta =
  \mathbf{0}_{n-r+1}$ is stronger than the statement of Lemma
  \ref{lemco} which reads  $D^{(r)} \eta = \mathbf{0}_{n-r}$. This is
  because we have assumed that $D^{(r-1)} \theta$ has mean zero. When
  this condition is not true, 
   we would need to add a polynomial sequence of degree $(r-1)$ to
   $\eta$ so that then $D^{(r-1)} \eta$ will have a constant mean
   which is same as saying that $D^{(r)} \eta = \mathbf{0}_{n-r}$.

Note from Lemma \ref{dfi} that $S D^{(r-1)} \theta$ differs
  from $\theta$ by a polynomial of degree at most $r - 2$ so that 
  \begin{equation*}
  \theta -  \eta =  (I - P_X) \theta = (I - P_X) S D^{(r-1)} \theta
  \end{equation*}
  where $P_X$ is the projection matrix on to the column space of
  $X$. As a result
  \begin{equation*}
    \|\theta - \eta\|^2 = \|(I - P_X) S D^{(r-1)} \theta\|^2 \leq \|(I
    - P_X) S\|^2 \|D^{(r-1)} \theta\|^2
  \end{equation*}
  where $\|(I - P_X) S\|$ denotes the operator norm of the matrix $(I
  - P_X) S$. It is clear that the square of the operator norm of $(I -
  P_X) S$ equals the operator norm of $A := S^T (I - P_X) S$ so that
  \begin{equation*}
    \|\theta - \eta\|^2 \leq \|A\| \|D^{(r-1)}  \theta\|^2. 
  \end{equation*}
  Note now that because $A$ is symmetric, its operator norm is bounded
  by its $\|\cdot\|_{\infty}$ norm (see, for example, \citet[Corollary
  2.3.2]{GolubVanLoan}) defined by  
  \begin{equation*}
    \|A\|_{\infty} := \max_{1 \leq i \leq n-r+1} \sum_{j=1}^{n-r+1}
      |a_{ij}| 
  \end{equation*}
  and hence we have
  \begin{equation}\label{mra}
    \|\theta - \eta\|^2 \leq \|A\|_{\infty} \|D^{(r-1)} \theta\|^2. 
  \end{equation}
   It may be noted that the matrix $A$ is the same matrix that
   appeared in the previous section (for example, in Proposition
   \ref{positive} and Proposition \ref{rowsum}) with $r$ replaced by
   $r - 1$. Therefore because all entries of $A$ are positive
   (Proposition \ref{positive}), we deduce that $\|A\|_{\infty} = \|A
   \mathbf{1}_{n-r+1} \|_{\infty}$ (this latter $\|\cdot\|_{\infty}$
     norm refers to the usual $L_{\infty}$ norm for vectors). In
    the proof of Proposition \ref{rowsum}, we gave a precise
    expression for $A \mathbf{1}_{n-r+1}$ (see equation
      \eqref{pill}). Using this, we deduce that (note that $r$ needs
      to be replaced by $r-1$ in \eqref{pill})
      \begin{equation*}
        \|A\|_{\infty} = \max_{1 \leq i \leq n-r+1}
          \frac{\binom{n-i}{r-1}
            \binom{r+i-2}{r-1}}{\binom{2r-2}{r-1}} \leq
              \frac{n^{2r-2}}{\binom{2r-2}{r-1}} \leq n^{2r-2}. 
      \end{equation*}
  Using the above with inequality \eqref{mra}, we obtain
  \begin{equation*}
    \|\theta - \eta\|^2 \leq n^{2r-2} \|D^{(r-1)} \theta\|^2. 
  \end{equation*}
  To bound the right hand side above further, we use \eqref{gaao}
  (note that the mean of the vector $D^{(r-1)} \theta$ is taken to be
  zero) to deduce that 
  \begin{equation*}
    \|\theta - \eta\|^2 \leq n^{2r-1} \|D^{(r)} \theta\|_1^2
  \end{equation*}
   which completes the proof of Lemma \ref{lemco}.  
\end{proof}

\subsection{Proof of the metric entropy bound for $\cras$ (Lemma
  \ref{pgr})} \label{ppgr}  
We shall provide the proof of Lemma \ref{pgr} in this subsection. For
this, we need to bound the metric entropy $\log N(\epsilon, \cras)$ of
the class $\cras$ defined in \eqref{cras}. Our strategy for this
involves the notion of fat shattering dimension. This is a standard
concept from the theory of empirical processes (see
e.g.,~\citet{Pollard90Iowa, rudelson2006combinatorics}) and is
recalled below for the convenience of the reader. 

%\vspace{0.2in} 

\begin{defn}[Fat Shattering Dimension]\label{fsd}
Let $K$ be a subset of $\R^n$. For $t \geq 0$, we say that a subset
$\{i_1, \dots, i_m\}$ of $\{1, \dots, n\}$ is $t$-shattered by $K$ if
there exist real numbers $h_{i_1}, \dots, h_{i_m}$ such that for every
subset $S \subseteq \{i_1, \dots, i_m\}$, there exists a vector
$\theta \in K$ for which $\theta_{i_k} \leq h_{i_k}$ if $i_k \in S$
and $\theta_{i_k} \geq h_{i_k} + t$ if $i_k \notin S$. The fat
shattering dimension of $K$, denoted by $v(K, t)$ is defined as the
maximum cardinality of a set $\{i_1, \dots, i_m\} \subseteq \{1,
\dots, n\}$ that is $t$-shattered by $K$. 
\end{defn}

A deep connection between fat shattering dimension and metric entropy
is given by the following result due to \citet[Corollary
6.4]{rudelson2006combinatorics} which bounds the metric entropy using
the fat shattering dimension. 

%\vspace{0.2in} 

\begin{theorem}[Rudelson and Vershynin] \label{rv}
  Let $K$ be a subset of $\R^n$. Assume that there exists a decreasing
  function $v : (0,\infty) \rightarrow (0, \infty)$ and a real number
  $a > 2$ such that 
  \begin{equation}
    \label{rv.con}
    v(K, s) \leq v(s) ~~~ \text{ and } ~~~ v(as) \leq \frac{1}{2} v(s)
    ~~ \text{ for all } s > 0. 
  \end{equation}
  Then there exists a constant $C$ depending on $a$ alone such that 
  \begin{equation}
    \label{rv.eq}
    \log N(\epsilon, K) \leq C v \left(\frac{\epsilon}{C \sqrt{n}} \right). 
  \end{equation}
\end{theorem}

In order to use Theorem \ref{rv} to prove Lemma \ref{pgr}, it is clear
that we need to bound the fat shattering dimension $v(\cras, t)$ of
$\cras$. The following lemma bounds the fat shattering dimension of
the class $\C_r(a, V)$ defined as: 
\begin{equation}\label{crav}
  \C_r(a, V) := \left\{\theta \in \R^n : a \leq (D^{r-1} \theta)_1
    \leq \dots \leq  (D^{r-1} \theta)_{n-r+1} \leq a + V \right\}
\end{equation}
for $a \in \R$ and $V \geq 0$. Note that $\cras \subseteq
\C_r(a_{r-1}, s_{r-1})$ so that the fat shattering dimension of
$\cras$ is bounded from above by that of $\C_r(a_{r-1}, s_{r-1})$. 

\begin{lemma}\label{cdcr}
  For every $V > 0$, $a \in \R$, $r \geq 1$, $n \geq r$ and 
  $t > 0$,   we have  
  \begin{equation}\label{cdcr.eq}
    v(\C_r(a, V), t) \leq r + \frac{V^{1/r} n^{1 - (1/r)}}{t^{1/r}} C_r
  \end{equation}
  for a positive constant $C_r$ that depends solely on $r$.  
\end{lemma}

Let us first prove Lemma \ref{pgr} assuming that Lemma \ref{cdcr} is
true. The proof of Lemma \ref{cdcr} will be provided following the
next proof. 

\begin{proof}[Proof of Lemma \ref{pgr}]
It turns out that it is enough to prove the following bound on the fat shattering
dimension of $\cras$: 
\begin{equation}\label{clb}
  v(\cras, t) \leq C_r \left(\frac{\sum_{j=1}^r n^{j-1} s_{j-1}}{t}
  \right)^{1/r}.  
\end{equation}
Indeed, Lemma \ref{pgr} is a direct consequence of the above
inequality along with Theorem \ref{rv}. To see this, note that if
inequality \eqref{clb} is true, one can simply take the function
$v(\cdot)$ in Theorem \ref{rv} to be   
\begin{equation*}
  v(s) = C_r \left(\frac{\sum_{j=1}^r n^{j-1} s_{j-1}}{s}
  \right)^{1/r}.   
\end{equation*}
Then the condition \eqref{rv.con} in Theorem \ref{rv} is true with $a
= 2^r$ and Lemma \ref{pgr} is therefore a consequence of inequality
\eqref{rv.eq}. 

The key therefore is to prove \eqref{clb}. For this, note first the
identity (which is a consequence of Lemma \ref{dfi} applied with $r-1$
instead of $r$) 
  \begin{equation*}
    \theta_i = \sum_{j=1}^{i-r+1} \binom{i-j-1}{r-2} (D^{(r-1)} \theta)_j
    + \sum_{j=1}^{r-1} \binom{i-1}{j-1} (D^{(j-1)} \theta)_1. 
  \end{equation*}
 This identity obviously implies the following lower and upper bounds
 on $\theta_i$  for every  $\theta \in \cras$:  
 \begin{equation*}
   \theta_i \geq \sum_{j=1}^{i-r+1} \binom{i-j-1}{r-2} a_{r-1} +
   \sum_{j=1}^{r-1} \binom{i-1}{j-1} a_{j-1}
 \end{equation*}
and
\begin{align*}
   \theta_i &\leq \sum_{j=1}^{i-r+1} \binom{i-j-1}{r-2} a_{r-1} +
   \sum_{j=1}^{r-1} \binom{i-1}{j-1} a_{j-1} + \sum_{j=1}^{r-1}
   \binom{i-1}{j-1} s_{j-1} \\
& + \sum_{j=1}^{i-r+1} \binom{i-j-1}{r-2}
   s_{r-1} . 
 \end{align*}
The last two terms in the expression above can be
combined into one term as follows: 
\begin{equation*}
  \sum_{j=1}^{r-1}
   \binom{i-1}{j-1} s_{j-1} + \sum_{j=1}^{i-r+1} \binom{i-j-1}{r-2} =
   \sum_{j=1}^r \binom{i-1}{j-1} s_{j-1}.   
\end{equation*}
This is a consequence of the fact that 
\begin{equation*}
  \sum_{j=1}^{i-r+1} \binom{i-j-1}{r-2} = \binom{i-1}{r-1}
\end{equation*}
which itself follows from \eqref{bnom} applied to $a = i-2$ and $b =
r-2$. We thus have
\begin{equation*}
  \theta_i \leq \sum_{j=1}^{i-r+1} \binom{i-j-1}{r-2} a_{r-1} +
   \sum_{j=1}^{r-1} \binom{i-1}{j-1} a_{j-1} + \sum_{j=1}^r
   \binom{i-1}{j-1} s_{j-1}. 
\end{equation*}
Combining the upper and lower bounds for $\theta_i$ derived above, we
deduce that 
\begin{equation*}
  \max_{\theta \in \cras}\theta_i - \min_{\theta \in \cras} \theta_i
  \leq \sum_{j=1}^r \binom{n-1}{j-1} s_{j-1} \leq \sum_{j=1}^r n^{j-1}
  s_{j-1}.  
\end{equation*}
The presence of $r  -2$ in the binomial coefficients above might seem
to make the above statement true only for $r \ge 2$. However for $r =
1$, this directly follows from the fact that every vector $\theta$ in
$\C_1(\{a_i\}, \{s_i\})$ satisfies $a_0 \le \theta_1 \leq \dots \le
\theta_n \le a_0 + s_0$. 

 As a consequence, it turns out that $v(\cras, t) = 0$ if $t > \Gamma
 := \sum_{j=1}^r n^{j-1} s_{j-1}$ and hence inequality \eqref{clb} is
 trivially true when $t > \Gamma$. We can therefore assume that $t \le
 \Gamma$. In this case, because $\cras \subseteq \C_r(a_{r-1},
 s_{r-1})$, Lemma \ref{cdcr} gives
 \begin{align*}
   v(\cras, t) &\leq v(\C_r(a_{r-1}, s_{r-1}), t) \\ &\leq r + C_r
   \left(\frac{n^{r-1} s_{r-1}}{t} \right)^{1/r} \\
& \leq r \left(\frac{\Gamma}{t} \right)^{1/r} + C_r
   \left(\frac{\Gamma}{t} \right)^{1/r} = (C_r + r)
   \left(\frac{\Gamma}{t} \right)^{1/r}
 \end{align*}
which proves \eqref{clb} when $t \leq \Gamma$. The completes the proof
of Lemma \ref{ppgr}. 
\end{proof}

We now prove Lemma \ref{cdcr}. For this, we use the
notion of divided differences (see, for example, \citet[Chapter  
15]{kuczma2009introduction}). For $k \geq 1$, indices $1 \leq \ell_1 <
\dots < \ell_k \leq n$ and real  numbers $\alpha_{\ell_1},
\dots,\alpha_{\ell_k}$, the divided difference $[\ell_1, \dots,
\ell_k; \alpha]$ is defined as 
\begin{equation*}
  [\ell_1, \dots, \ell_k ; \alpha] := \sum_{i=1}^k
  \frac{\alpha_{\ell_i}}{\prod_{j \neq i} (\ell_i - \ell_j) } 
\end{equation*}
As examples, note that $[\ell_1; \alpha] = \alpha_{\ell_1}$ and
$[\ell_1, \ell_2; \alpha] = (\alpha_{\ell_2} -
\alpha_{\ell_1})/(\ell_2 - \ell_1)$. 

It is easy to verify that the divided differences satisfy the
recursive relation 
\begin{equation*}
  [\ell_1, \dots, \ell_k; \alpha] = \frac{[\ell_2, \dots, \ell_k ;
    \alpha] - [\ell_1, \dots, \ell_{k-1}; \alpha]}{\ell_k - \ell_1}. 
\end{equation*}

We shall use the following two facts about divided differences for
the proof of Lemma \ref{cdcr}. The first fact is given in Lemma
\ref{subs} below which is a simple consequence of \citet[Theorem
15.3.1]{kuczma2009introduction}.   

%\vspace{0.2in} 

\begin{lemma}\label{subs}
  Fix $r \geq 1$  and $n \geq r$. Suppose $\theta \in \R^n$ satisfies
  $(D^{(r-1)} \theta)_1 \leq \dots \leq (D^{(r-1)}
  \theta)_{n-r+1}$. Then for every choice of indices $1 \leq i_1 <
  \dots < i_{r+1} \leq n$, we have 
  \begin{equation*}
    [i_2, \dots, i_{r+1}; \theta] \geq [i_{1}, \dots, i_r; \theta]. 
  \end{equation*}
\end{lemma}

%\vspace{0.2in} 

\begin{remark}
  When $r = 2$, it is easy to see that Lemma 2.2 reduces to the
  well-known \textit{increasing slopes} property of convex sequences. 
\end{remark}

The second fact about divided differences is given in Lemma \ref{nogs}
below which is a consequence of \citet[Lemma 15.2.5 and Theorem 
15.2.6]{kuczma2009introduction}. 

%\vspace{0.2in} 

\begin{lemma}\label{nogs}
  Fix $r \geq 1$ and $n \geq r$. For every choice of indices $1 \leq
  i_1 < i_2 <  \dots < i_r \leq n$, there exist 
  non-negative real numbers $\{c_i, 1 \leq i \leq n - r + 1\}$
  with $\sum_{i=1}^{n-r+1} c_i = 1$ 
   such that 
   \begin{equation*}
     [i_1, \dots, i_r; \theta] = \frac{1}{(r-1)!}\sum_{i=1}^{n - r +
       1} c_i  (D^{(r-1)} \theta)_i \qt{for every $\theta \in \R^n$}.  
   \end{equation*}
\end{lemma}

We are now ready to give the proof of Lemma \ref{cdcr}. 
\begin{proof}[Proof of Lemma \ref{cdcr}]
Fix $t > 0$ and suppose that $S := \{i_1, \dots, i_m\}$ (with $1
  \leq i_1 < \dots < i_m \leq n$) is a subset of $\{1, \dots, n\}$
  that is $t$-shattered by $\C(V)$. Let $h_{i_1}, \dots, h_{i_m}$
  denote the associated levels and denote by $h$ the vector in $\R^m$
  given by $(h_{i_1}, \dots, h_{i_m})$. We shall then prove that $m$
  is bounded from above by the right hand side of
  \eqref{cdcr.eq}. Note that we can assume that $m \geq r$ (otherwise
  there is nothing to prove). 

We first claim that 
  \begin{equation}\label{fcr}
    [i_j, i_{j+1}, \dots, i_{j+r-1}; h] \geq [i_{j-1}, \dots,
    i_{j+r-2}; h] + t \sum_{k=j-1}^{j+r-1} (- \tau_{k, j}) \{\tau_{k,
      j} < 0\} 
  \end{equation}
 for every $j = 2, \dots, m-r+1$ where
 \begin{equation*}
   \tau_{k,j} := \prod_{j \leq \ell \leq j+r-1 : \ell \neq k}
     \frac{1}{i_k - i_{\ell}} - \prod_{j-1 \leq \ell \leq j+r-2 : \ell \neq k}
     \frac{1}{i_k - i_{\ell}} 
 \end{equation*}
for $k = j , \dots, j+r-2$ and 
\begin{equation*}
  \tau_{j-1, j} := (-1)^r \prod_{\ell = j}^{j+r-2} \frac{1}{i_{\ell} -
    i_{j-1}} ~~ \text{ and } ~~ \tau_{j+r-1, j} := \prod_{\ell =
    j}^{j+r-2} \frac{1}{i_{j+r-1} - i_{\ell}} 
\end{equation*}
In the above, for $r = 1$, we take $\tau_{j-1, j} = -1$ and $\tau_{j,
  j} = 1$.  

To see \eqref{fcr}, note first that because $S$ is $t$-shattered by
$\C_r(a, V)$, there exists $\theta \in \C_r(a, V)$ such that $\theta_{i_k}
\leq h_{i_k}$ whenever $\tau_{k, j} \geq 0$ and $\theta_{i_k} \geq
h_{i_k} + t$ whenever $\tau_{k, j} < 0$. Because $\theta \in
\C_r(a, V)$, Lemma \ref{subs} gives
\begin{equation*}
   [i_j, i_{j+1}, \dots, i_{j+r-1}; \theta] \geq [i_{j-1}, \dots,
    i_{j+r-2}; \theta].
\end{equation*}
It can be checked that the above inequality is equivalent to
$\sum_{k=j-1}^{j+r-1} \tau_{k, j} \theta_{i_k} \geq 0$ which is
further equivalent to 
\begin{equation*}
  \sum_{k=j-1}^{j+r-1} \tau_{k, j} \theta_{i_k} \{\tau_{k, j} \geq 0\}
  \geq \sum_{k=j-1}^{j+r-1} (-\tau_{k, j}) \theta_{i_k} \{\tau_{k, j}
  < 0\}.
\end{equation*}
The above inequality, together with the fact that $\theta_{i_k} \leq
h_{i_k}$ when $\tau_{k, j} \geq 0$ and $\theta_{i_k} \geq h_{i_k}
+ t$ when $\tau_{k, j} < 0$, gives \eqref{fcr}. 

From \eqref{fcr}, it is easy that by recursive application, one
obtains
\begin{equation*}
  [i_j , i_{j+1}, \dots, i_{j+r-1}; h] \geq [i_u, i_{u+1}, \dots,
  i_{u+r-1}; h] + t \sum_{a = u}^{j-1} \sum_{k=a}^{a+r} (- \tau_{k,
    a+1}) \left\{\tau_{k, a+1} < 0 \right\}
\end{equation*}
for every $1 \leq u < j \leq m-r+1$. Taking $u = 1$ and $j = m-r+1$,
we obtain 
\begin{equation}\label{say}
  [i_{m-r+1} , \dots, i_{m}; h] - [i_1,  \dots,
  i_{r}; h]  \geq t T_r. 
\end{equation}
where 
\begin{equation*}
  T_r := \sum_{a = 1}^{m-r} \sum_{k=a}^{a+r} (- \tau_{k,
    a+1}) \left\{\tau_{k, a+1} < 0 \right\}. 
\end{equation*}
We now claim that 
\begin{equation}\label{pnk}
  [i_1,  \dots, i_{r}; h] \geq \frac{a}{(r-1)!} ~~ \text{ and } ~~
  [i_{m-r+1} , \dots, i_{m}; h] \leq \frac{a+V}{(r-1)!}. 
\end{equation}
We shall prove the first inequality in \eqref{pnk} below. The proof of
the second inequality will be similar. One can write $[i_1, \dots,
i_r; h]$ as $\sum_{j=1}^r \beta_j h_{i_j}$ for some real coefficients
$\beta_j$. Because $S$ is $t$-shattered by $\C_r(a, V)$, there exists
$\theta \in \C_r(a, V)$ such that $h_{i_j} \geq \theta_{i_j}$ for
$\beta_j \geq 0$ and $h_{i_j} < \theta_{i_j}$ for $\beta_j < 0$. This
implies that 
\begin{equation*}
[i_1, \dots, i_r; h] = \sum_{j=1}^r \beta_j h_{i_j} \geq \sum_{j=1}^r
\beta_j \theta_{i_j} = [i_1, \dots, i_r; \theta]. 
\end{equation*}
Lemma \ref{nogs} now implies that, for some $c_i \geq 0, 1 \leq i \leq
n-r+1$  with $\sum_{i=1}^{n-r+1} c_i = 1$, we have 
\begin{equation*}
  [i_1, \dots, i_r; \theta] = \frac{1}{(r-1)!} \sum_{i=1}^{n-r+1} c_i
  (D^{(r-1)} \theta)_i \geq \frac{a}{(r-1)!}
\end{equation*}
where the last inequality follows because $\theta \in \C_r(a,
V)$. This proves \eqref{pnk}. 

Combining \eqref{pnk} and \eqref{say}, we obtain
\begin{equation*}
  T_r \leq \frac{V}{t (r-1)!}. 
\end{equation*}
We now claim the following lower bound for $T_r$: 
\begin{equation}\label{rbo}
T_1 = m-1 ~~~ \text{ and } ~~~  T_r \geq \frac{ (m-r)^r}{n^{r-1}
  (r-1)^{r-1}} \qt{for every $r  \geq 2$}.   
\end{equation}
Before we prove \eqref{rbo}, note first that as a consequence of the
above pair of inequalities, inequality \eqref{cdcr.eq} holds with $C_1
= 1$ and 
\begin{equation*}
  C_r = \left(\frac{(r-1)^{r-1}}{(r-1)!} \right)^{1/r} \qt{for $r \geq
    2$}. 
\end{equation*}
Therefore, to complete the proof of Lemma \ref{cdcr}, we only need to
prove inequality \eqref{rbo}. 

To prove \eqref{rbo}, we assume that $r \geq 2$ (the fact that $T_1 =
m-1$ is obvious) and note first that $\tau_{a+r-1, a+1} < 0$ for every
$a = 1, \dots, m-r$ . As a result,  
\begin{equation*}
  T_r \geq \sum_{a=1}^{m-r} \left(- \tau_{a+r-1, a+1} \right) \geq
  \sum_{a=1}^{m-r} \frac{1}{(i_{a+r-1} - i_a) \dots (i_{a+r-1} -
    i_{a+r-2})}. 
\end{equation*}
By the AM-GM inequality, we have
\begin{equation*}
  (i_{a+r-1} - i_a) \dots (i_{a+r-1} - i_{a+r-2}) \leq
  \left(\frac{(i_{a+r-1} - i_a) + \dots + (i_{a+r-1} -
      i_{a+r-2})}{r-1} \right)^{r-1} . 
\end{equation*}
If we define $s_j := i_{j+1} - i_j$ for $j = 1, \dots, m-1$, then it
is easy to see that 
\begin{equation*}
  \frac{(i_{a+r-1} - i_a) + \dots + (i_{a+r-1} -
      i_{a+r-2})}{r-1} = \sum_{j=0}^{r-2}  \frac{j+1}{r-1} s_{a+j}
    \leq \sum_{j=0}^{r-2} s_{a+j}. 
\end{equation*}
We have deduced therefore that 
\begin{equation*}
  T_r \geq \sum_{a=1}^{m-r} \left( \frac{1}{\sum_{j=0}^{r-2} s_{a+j}} \right)^{r-1}.
\end{equation*}
We now use the convexity of the map $x \mapsto (1/x)^{r-1}$ for $x >
0$ to obtain
\begin{equation*}
  T_r \geq \frac{(m-r)^r}{\left( \sum_{a=1}^{m-r} \sum_{j=0}^{r-2}
      s_{a+j} \right)^{r-1}}. 
\end{equation*}
Inequality \eqref{rbo} follows from here because
\begin{equation*}
  \sum_{a=1}^{m-r} \sum_{j=0}^{r-2} s_{a+j} = \sum_{j=0}^{r-2}
  \sum_{a=1}^{m-r} s_{a+j} = \sum_{j=0}^{r-2} \left(i_{m-r+j+1} -
    i_{j+1} \right) \leq n (r-1). 
\end{equation*}
This completes the proof of Lemma \ref{cdcr}. 
\end{proof}

\section{Additional Simulation Results}\label{adsims}
The purpose of this section is to provide additional details for the
main simulation section as well as to provide results for the function
$f_3^*(x) := 14.77 I\{0.1 < x \leq 0.13\} - 3.69 I\{0.13 < x \leq 
  0.15\} + 7.39 I\{0.15 < x \leq 0.23\}  - 7.39 I\{0.23 < x \leq 0.25\}
  + 11.08 I\{0.25 < x \leq 0.4\} - 4.43 I\{0.4 < x \leq 0.44\} + 3.32
  I\{0.44 < x \leq 0.65\} + 19.21 I\{0.65 < x \leq 0.76\} + 7.76
  I\{0.76 < x \leq 0.78\} + 15.51 I\{0.78 < x \leq 0.81\}$. This
  function (plotted in Figure \ref{fr1}) is similar to the blocks
  function of \citet{donoho1994ideal}. 
\begin{figure}[!htbp]
\begin{center}
 \includegraphics[width=\textwidth]{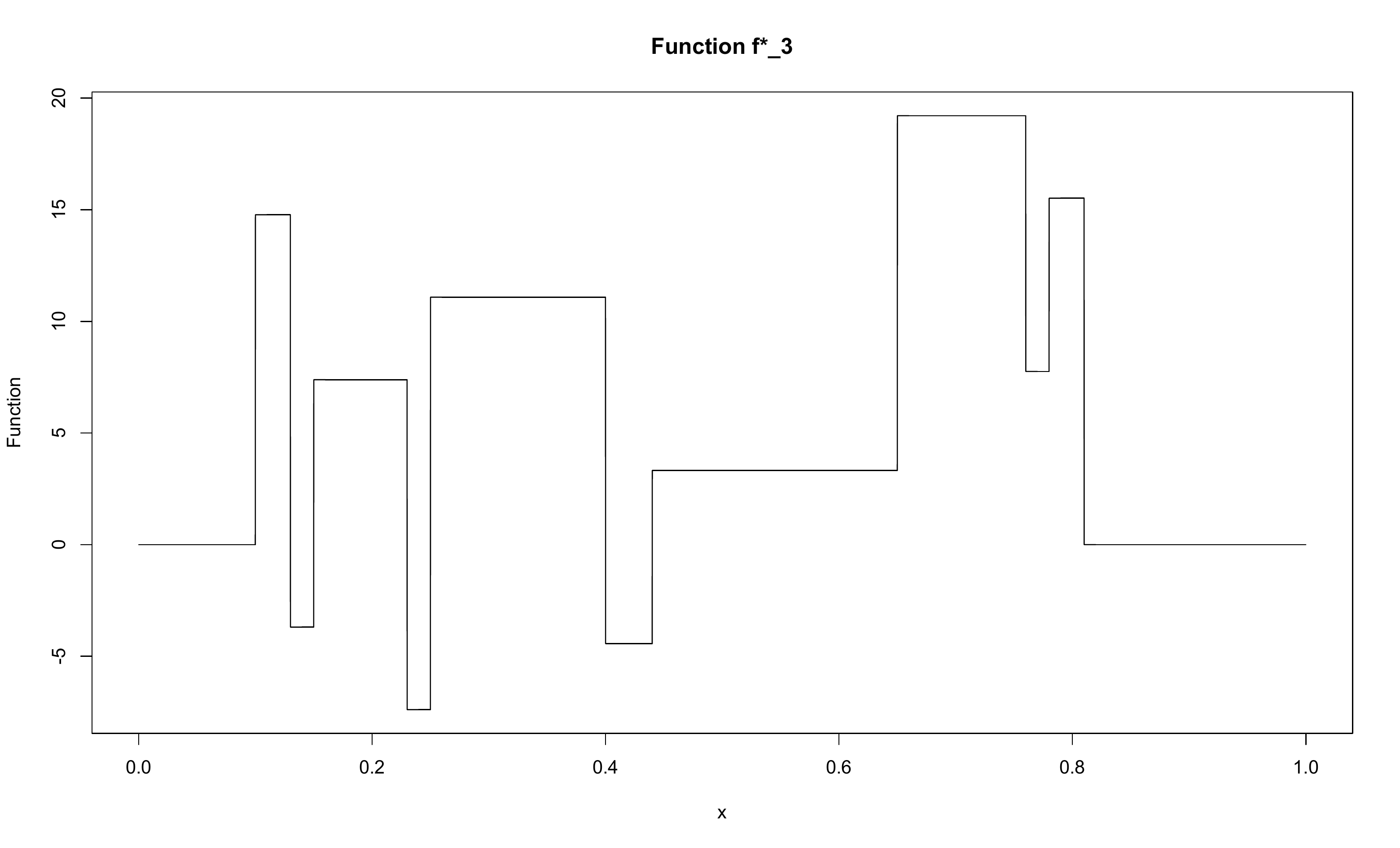}    
\caption{The function $f_3^*$}  
\label{fr1}
\end{center}
\end{figure}

Note that in our simulation results for $f_1^*$, we computed the ideal
penalized estimator with $\lambda$ taken to be $\lambda^*$ defined as
in \eqref{lams}. We mentioned that $\lambda^*$ was computed by
Monte-Carlo averaging based on a convex optimization scheme for
computing $\lambda_{\theta^*}(z)$ for each $z \in \R^n$. Let us
provide more details behind this convex optimization here. For general
$r \geq 1$, it is easy to see (using the definition of 
  $\lambda_{\theta^*}(z)$ and the subdifferential characterization in
  Proposition \ref{characterization}) that $\lambda_{\theta^*}(z)$
  can be read-off as the optimizing value for $\lambda$ in the
  following convex optimization problem:  
\begin{equation*}
\begin{aligned}
& \underset{v_1, \dots, v_n, \lambda}{\text{minimize}}
& & \|z - v\| \\
& \text{subject to}
& & \sum_{i=j}^n {r+i-j-1 \choose r-1} v_i = 0 \text{ for }  j = 1,
\dots, r\\
&&& \sum_{i=j}^n {r+i-j-1 \choose r-1} v_i - \lambda  \leq 0 \text{ for } r < j \leq n  \\
&&& \sum_{i=j}^n {r+i-j-1 \choose r-1} v_i + \lambda  \geq 0 \text{
  for } r < j \leq n \\ 
&&& \sum_{i=j}^n {r+i-j-1 \choose r-1} v_i - \lambda~
\mathrm{sgn}((D^{(r)} \theta)_{j-r})  = 0\text{ for } r < j \leq n \\ &&&\text{
  with } (D^{(r)} \theta)_{j-r} \neq 0. 
\end{aligned}
\end{equation*}
This optimization problem can be solved efficiently by the convex optimization
software MOSEK for $r = 1$. In fact, for computational reasons, it is
easier to solve the dual of this problem. For $r \geq 2$ however, this
problem becomes quite ill-conditioned and MOSEK seems to have trouble
finding the global minimizer. This is why we could not compute the
$\lambda^*$ values for the function $f_2^*$. 

\begin{figure}[h!]
\begin{center}
 \includegraphics[width = \textwidth]{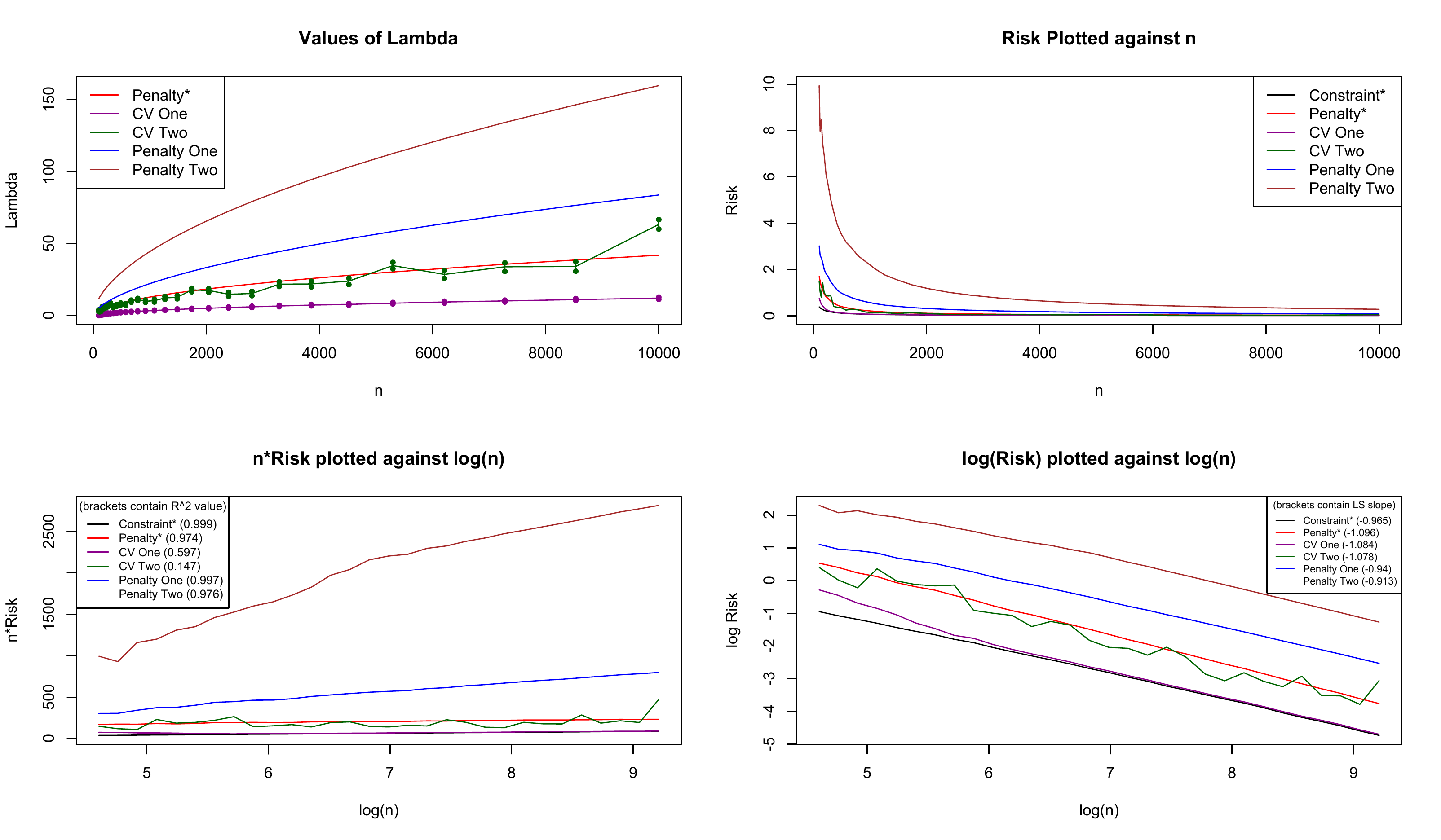}    
\caption{Plots when the true function is $f_3^*$. }  
\label{f2p}
\end{center}
\end{figure}

The simulation results for the function $f_3^*$ (here $r = 1$ as
$f_3^*$ is a piecewise constant function) are given in Figure
\ref{f2p}. It is clear from here that the behavior of the non-CV
estimators is in accordance with our theoretical results. The CV
estimators seem to behave in a complicated manner in the bottom-left
plot. Again, understanding the risk behavior of CV estimates in this
setting is beyond the scope of the present paper.

\bibliographystyle{imsart-nameyear}
\bibliography{AG}

\end{document}